%% file: mainv2.tex
\documentclass[reqno]{amsart}

\input{Preamble.tex}

\usepackage[margin=1.5in]{geometry}
\usepackage{etoolbox}
\newtoggle{draft}
\togglefalse{draft}
\iftoggle{draft} {
\newcommand{\NB}[1]{\todo[color=gray!40]{#1}}
\newcommand{\TODO}[1]{\todo[color=red]{#1}}
}{ % else
\newcommand{\NB}[1]{}
\newcommand{\TODO}[1]{}
\renewcommand{\todo}[1]{}
\renewcommand{\todo}[1]{}
}

\usepackage{tikz}
\usetikzlibrary{tikzmark}
\usepackage{amssymb}
\usepackage{array}
\usepackage{rotating}

\title{The Adams differentials on the classes $h_j^3$}
\author{Robert Burklund}
\address{Department of Mathematical Sciences, University of Copenhagen, Denmark}
\email{rb@math.ku.dk}

\author{Zhouli Xu}
\address{UCLA Department of Mathematics, Los Angeles, CA 90095-1555, USA}
\email{xuzhouli@ucla.edu}

\begin{document}

\begin{abstract}
  In filtration 1 of the Adams spectral sequence, using secondary cohomology operations, Adams \cite{Adams} computed the differentials on the classes $h_j$, resolving the Hopf invariant one problem.
  In Adams filtration 2, using equivariant and chromatic homotopy theory, Hill--Hopkins--Ravenel \cite{HHR} proved that the classes $h_j^2$ support non-trivial differentials for $j 
  \geq 7$, resolving the celebrated Kervaire invariant one problem.
  The precise differentials on the classes $h_j^2$ for $j \geq 7$ and the fate of $h_6^2$ remains unknown. 

  In this paper, in Adams filtration 3, we prove an infinite family of non-trivial $d_4$-differentials on the classes $h_j^3$ for $j \geq 6$, confirming a conjecture of Mahowald.
  Our proof uses two different deformations of stable homotopy theory---$\mathbb{C}$-motivic stable homotopy theory and $\mathbb{F}_2$-synthetic homotopy theory---both in an essential way.
  Along the way, we also show that $h_j^2$ survives to the Adams $E_5$-page and that $h_6^2$ survives to the Adams $E_9$-page.
\end{abstract}

\maketitle

\setcounter{tocdepth}{1}
\tableofcontents
\vbadness 5000

%\todo{The abstract is a bit long at the moment.}
%\todo{Citations won't work right in the arxiv version of abstract.}
%manually delete citations when posting on arXiv

\section{Introduction}
\label{sec:intro}

\input{introv2.tex}

% \section{Outline (temporary)}
% \label{sec:outline}
% \input{outline}

\section{The Miller square and motivic homotopy theory}
\label{sec:review}
\input{reviewv2.tex}

\section{Proof of the main theorem}
\label{sec:proof}
\input{proofv2.tex}

\section{The Cartan--Eilenberg $E_2$-page}
\label{sec:algAH}

\input{algAHv2.tex}

\section{The Cartan--Eilenberg spectral sequence}
\label{sec:cess}
\input{cess2v2.tex}

\section{The key algebraic Novikov differential}
\label{sec:alg-nov-diff}
\input{productv2.tex}

\section{The inductive approach revisited}
\label{sec:kervaire}
\input{kervairev2.tex}

% \section{Eliminating extra terms}
% \label{sec:syn-refinements}
% \input{argument.tex}

\appendix

%\section{The Witt vectors on the moduli of formal groups}
%\label{sec:witt}
%\input{witt.tex}

\section{Code}
\label{app:code}
\input{code.tex}

\bibliographystyle{alpha}
\bibliography{bibliography}

\end{document}

%% file: Preamble.tex
\usepackage{verbatim}
\usepackage[textsize=scriptsize]{todonotes}
\usepackage{tikz-cd}
\usepackage{etex}
\usepackage{hyperref}
\usepackage{cleveref}
\usepackage[T1]{fontenc}

\usepackage[all,cmtip]{xy}
\usepackage{multicol}

\hypersetup{
   colorlinks,
   linkcolor={red},
   citecolor={blue!50!black},
   urlcolor={blue!80!black}
}

\usepackage{chemarr}
\usepackage{amssymb}
\usepackage{amsmath}
\usepackage{comment}
\usepackage{spectralsequences}
\usepackage{cleveref}
\usepackage{mathtools}
\usepackage{rotating}
\usepackage{wrapfig}
\usepackage{outlines}
\usepackage{graphicx}
\usepackage{xcolor}
\usepackage{bbm}
\usepackage{enumitem}

\theoremstyle{definition}
\newtheorem{nul}{}[section]
\newtheorem{dfn}[nul]{Definition}

\newtheorem{rmk}[nul]{Remark}

\newtheorem{cnstr}[nul]{Construction}

\newtheorem{ntn}[nul]{Notation}
\newtheorem{exm}[nul]{Example}

\newtheorem{rec}[nul]{Recollection}

\newtheorem{wrn}[nul]{Warning}
\newtheorem{qst}[nul]{Question}
\newtheorem*{dfn*}{Definition}
\newtheorem*{axm*}{Axiom}
\newtheorem*{ntn*}{Notation}
\newtheorem*{exm*}{Example}
\newtheorem*{exr*}{Exercise}
\newtheorem*{int*}{Intuition}
\newtheorem*{qst*}{Question}
\newtheorem*{rmk*}{Remark}

\theoremstyle{plain}

\newtheorem{thm}[nul]{Theorem}
\newtheorem{prop}[nul]{Proposition}
\newtheorem{lem}[nul]{Lemma}

\newtheorem{cnj}[nul]{Conjecture}
\newtheorem{cor}[nul]{Corollary}

\newtheorem*{thm*}{Theorem}
\newtheorem*{prop*}{Proposition}
\newtheorem*{cor*}{Corollary}
\newtheorem*{lem*}{Lemma}
\newtheorem*{cnj*}{Conjecture}

\let\oldwidetilde\widetilde
\protected\def\widetilde{\oldwidetilde}

\DeclareMathOperator{\Ss}{\mathbb{S}}

\DeclareMathOperator{\F}{\mathbb{F}}

\DeclareMathOperator{\A}{\mathcal{A}}

\DeclareMathOperator{\E}{\mathbb{E}}

\DeclareMathOperator{\Z}{\mathbb{Z}}

\DeclareMathOperator{\Spec}{\text{Spec}}

\DeclareMathOperator{\Ext}{\mathrm{Ext}}

\DeclareMathOperator{\Fil}{\mathbf{Fil}}

\DeclareMathOperator{\Sp}{\mathrm{Sp}}

\DeclareMathOperator{\Sq}{\mathrm{Sq}}

\newcommand{\wt}{\widetilde}

\newcommand{\BP}{\mathrm{BP}}
\newcommand{\Syn}{\mathrm{Syn}}

\def\o{\mathbbm{1}}

\def\P{\mathcal{P}}
\def\CC{\mathbb{C}}

\def\clambda{\Ss\hspace{-2.2pt}/\!\lambda}

\usepackage{tikz}
\usetikzlibrary{matrix,arrows,decorations}
\usepackage{tikz-cd}

\usepackage{adjustbox}

\let\oldtocsection=\tocsection
 
\let\oldtocsubsection=\tocsubsection
 
\let\oldtocsubsubsection=\tocsubsubsection
 
\renewcommand{\tocsection}[2]{\hspace{0em}\oldtocsection{#1}{#2}}
\renewcommand{\tocsubsection}[2]{\hspace{1em}\oldtocsubsection{#1}{#2}}
\renewcommand{\tocsubsubsection}[2]{\hspace{2em}\oldtocsubsubsection{#1}{#2}}

%% file: introv2.tex
% Obtaining a systematic understanding of the differentials in the Adams spectral sequence is one of the basic goals of stable homotopy theory. Fix a prime $p$, the Adams spectral sequence converges to the homotopy groups of $p$-completed sphere spectrum $S^0$ and has signature
% $$\Ext^{s, t}_{\A}(\mathbb{F}_p, \mathbb{F}_p) \cong E_2^{s,t} \Longrightarrow \pi_{t-s} S^0,$$
% $$d_r: E_r^{s,t} \longrightarrow E_r^{s+r, t+r-1}$$
% with $E_2$-page isomorphic to the cohomology of the Steenrod algebra. At the prime 2, the Adams $1$-line is given by 
% \[ \Ext^{1,t}_{\A}(\mathbb{F}_2, \mathbb{F}_2) \cong \begin{cases}
%  	\mathbb{F}_2, \ \text{if} \ t = 2^j \ \text{for some} \ j \geq 0, \\
%  	0, \ \text{otherwise} \end{cases} \]
% where $s=1$ is the cohomological degree and $t$ is the internal degree.
% The generator of the copy of $\mathbb{F}_2$ in bidegree $(s, t)=(1, 2^j)$ is typically denoted $h_j$.\todo{I almost think we can just remove this. It's not like anybody reading this paper hasn't seen the Adams sseq.}

% A systematic understanding of differentials in the Adams spectral sequence is a fundamental problem in stable homotopy theory. Many questions in geometry and topology are closely related to the computations of the Adams spectral sequence, especially on the Adams differentials of certain classes.

One of the basic goals of stable homotopy theory is to obtain a systematic understanding of the differentials in the Adams spectral sequence.
Perhaps surprisingly, many questions in differential topology also require and reduce to a core Adams spectral sequence computation. As an example, let us consider the following question: When is the $n$-sphere parallelizable?
Classical constructions using division algebras provide parallelizations of $S^1$, $S^3$ and $S^7$ and the remaining problem is to show that all other spheres are not parallelizable.
For this it suffices to show that there does not exist a class with Hopf invariant one in the $n$-stem. %\footnote{In fact, one only needs to show that there is no Hopf invariant one class in the image of $J$, which is easier than the full Hopf invariant one problem (see \cite{BottMilnor}).}
In order to approach this problem Adams began by reformulating earlier work of Adem into the following theorem.

% Adem showed that a necessary condition is that $n$ be of the form $2^j-1$ and his argument can be rephrased as the computation of the $1$-line of the $E_2$-page of the Adams spectral sequence. 
% Specifically, 
% As motivation for introducing the Adams spectral sequence

\begin{thm}[Adem, Adams \cite{AdamsSseq}]\label{thm: Hopf inv}
  Hopf invariant one classes can only exist in dimensions of the form $2^j-1$.
  Moreover, there exists a Hopf invariant one class in the $(2^j-1)$-stem
  if and only if the class $h_j$ on the $1$-line of the Adams spectral sequence is a permanent cycle.
\end{thm}

% It was known that the Hopf invariant 1 problem has a negative answer in dimensions not of the form $2^j-1$.
% Due to the existence of the first several Hopf maps, the Hopf classes $h_j$ are permanent cycles for $0 \leq j \leq 3$.

Then, using secondary cohomology operations, Adams computed the first infinite family of nonzero differentials in the Adams spectral sequence, resolving the Hopf invariant one problem.
% then proved the following theorem and therefore completely solved the Hopf invariant 1 problem.

\begin{thm}[Adams \cite{Adams}] \label{thm: hopf inv diff}
$d_2(h_j) = h_0 h_{j-1}^2 \neq 0$ for all $j \geq 4$.	
\end{thm}

\begin{rmk}
We now have many other methods proving Adams' theorem, for example, using Adams operations in K-theory, using the comparison map between the Adams and the Adams-Novikov spectral sequence, using product structure in higher Adams filtrations, and using Bruner's power operation in the Adams spectral sequence.
\end{rmk}

%Consequences of Theorem~\ref{thm: hopf inv diff} includes that the tangent bundle of the $n$-sphere is trivial if and only if $n = 0, 1, 3, 7$, and that the only finite dimensional division algebra over the real numbers are of dimensions $1, 2, 4, 8$. \todo{I think many of these were known without Adams, due to Milnor's work.}

On the Adams $2$-line we have classes $h_j^2$ and Browder showed that they too are intimately connected with a fundamental problem in differential topology.

\begin{thm}[Browder \cite{Browder}]\label{thm: Ker inv}
  A smooth framed manifold with Kervaire invariant one can only exist in dimensions of the form $2(2^j-1)$.
  Moreover, the following statements are equivalent.
  \begin{itemize}
  \item The class $h_j^2$ is a permanent cycle in the Adams spectral sequence.
  \item There exists a smooth framed manifold of dimension $2(2^j-1)$ with Kervaire invariant one.
  \end{itemize}
\end{thm}

Using equivariant and chromatic methods, Hill, Hopkins and Ravenel 
proved the following celebrated theorem, resolving the Kervaire invariant problem in large dimensions.

\begin{thm}[Hill--Hopkins--Ravenel \cite{HHR}] \label{thm: HHR}
The classes $h_j^2$ support nonzero Adams differentials for $j \geq 7$.%\footnote{Note that the targets of the Adams differentials on the Kervaire classes $h_j^2$ have not been identified.}
\end{thm}

\begin{rmk}
Unlike the solution of the Hopf invariant one problem, the Hill--Hopkins--Ravenel proof is still the only proof to the solution of the Kervaire invariant one problem. Also note that the targets of the Adams differentials on the Kervaire classes $h_j^2$ have not been identified.
\end{rmk}

Computations of the homotopy groups of spheres show that the Kervaire classes $h_j^2$ are permanent cycles for $0 \leq j \leq 5$ \cite{BMT, BJM62, Xu}. The final case, the fate of $h_6^2$ in dimension 126, remains open.

Continuing in this manner we are led to consider the classes $h_j^3$ on the Adams $3$-line.
We begin by asking whether they too lie at the heart of some geometric problem.

\begin{qst}\label{qst: hj3 geo}
  Do the classes $h_j^3$ %, like the classes $h_j$ and $h_j^2$, 
  also have an interpretation in differential topology?
\end{qst}

Laures and Miller \cite{Laures, Miller} have interpreted differentials in the Adams spectral sequence in terms of certain characteristic numbers of framed manifolds with corners and worked out the cases for $h_j$ and $h_j^2$. We plan to study the geometric meaning of $h_j^3$ through their frame work and our Theorem~\ref{thm:main} in a future project. 

% Based on \cite{KM, Browder}, a consequence of Theorem~\ref{thm: HHR} and recent work of Wang and the second author \cite{WX} is that the only odd dimensional spheres that have a unique smooth structure is of dimensions $1, 3, 5, 61$.

% Given the connections in Theorems~\ref{thm: Hopf inv} and \ref{thm: Ker inv} from differential topology regarding the destiny of the classes $h_j$ and $h_j^2$ in the Adams spectral sequences, we propose the following two questions.

% \begin{qst}\label{qst: hj3}
% When are the classes $h_j^n$ permanent cycles in the Adams spectral sequence for $n \geq 3$?	
% \end{qst}

% \begin{qst}\label{qst: hj3 geo}
%   What is the differential topology meaning of Question~\ref{qst: hj3}?
% \end{qst}

% In this article, we give a complete solution to Question~\ref{qst: hj3}. 

% Before we discuss our result and the methodology, here is an observation. Adams proved that $h_j^4 = 0$ for all $j \geq 1$, and that $h_0^n$ are nonzero permanent cycles for all $n$. Therefore, the only nontrivial case for Questions~\ref{qst: hj3} and \ref{qst: hj3 geo} is when $n=3$.

Barratt, Mahowald and Tangora \cite{BMT} proved that the classes $h_j^3$ are permanent cycles for $j \leq 4$. At Toda's $60^{\mathrm{th}}$ birthday conference in 1988 Mahowald then made the following conjecture (see \cite[Remark~3.1]{Minami}) regarding the $h_j^3$ family:

% Minami \cite[Remark~3.1]{Minami} stated that during the Toda's conference in Kinosaki, Japan, in the summer of 1988, Mahowald made the following conjecture.

\begin{cnj}[Mahowald] \label{conj Mahowald}
The classes $h_j^3$ are not permanent cycles for $j$ large.
\end{cnj}

Recently, Isaksen, Wang and the second author proved that $h_5^3$ supports a nonzero Adams differential using motivic stable homotopy theory. Let $g_j$ be the generator of the group $\Ext^{4, 3 \cdot 2^{j+2}}_{\A}(\mathbb{F}_2, \mathbb{F}_2) \cong \mathbb{F}_2$.
    
\begin{thm}[Isaksen--Wang--Xu \cite{IWX}]\label{thm: d4 h53}
  $d_4(h_5^3) = h_0^3 g_3 \neq 0$.
  %where $g_3$ is a generator of $\Ext^{4, 96}_{\A}(\mathbb{F}_2, \mathbb{F}_2) \cong \mathbb{F}_2$.
\end{thm}
%% Here $g_j$ is the generator of the group $\Ext^{4, 2^{j+3}+2^{j+2}}_{\A}(\mathbb{F}_2, \mathbb{F}_2) = \mathbb{F}_2$.% forms a $\Sq^0$-family on the Adams $4$-line.

Previously, Wu and Lin had, respectively, shown that $h_6^3$ and $h_7^3$ support nonzero Adams differentials.

\begin{thm}[Lin \cite{Lin}, Wu \cite{Wu}]\label{thm: d4 h63}
  $d_4(h_6^3) = h_0^3 g_4 \neq 0$ and $d_4(h_7^3) \neq 0$.
\end{thm}

Lin and Wu analyze $h_6^3$ and $h_7^3$ via a study of the Kahn--Priddy transfer map.
Lin also went on to conjecture that $d_4(h_j^3) = h_0^3g_{j-2}$ for $j \geq 6$ \cite{Lin} and Wu verified this conjecture in the case $j=6$.

In this article we prove Lin's refinement of Mahowald's Conjecture, 
thereby determining the fate the $h_j^3$ family.

% \begin{rmk}
%   % In \cite{Lin}, Lin established that the $d_4$ differential on $h_7^3$ is non-zero (though the precise target was not identified) and stated the conclusion of \Cref{thm: d4 hj3} as a conjecture.
%   Lin's method for analyzing $h_7^3$ proceeds via a study of the Kahn--Priddy transfer map and is rather elaborate.
%   We remark that we found the contents of \cite{Lin} and \cite{Wu} exceedingly difficult to evaluate.
%   %  as several key arguments are omitted.
%   %Wu extended Lin's methods in order to compute $d_4(h_6^3)$ in \cite{Wu}.  
%   % \cite{Lin} conjectured that the classes $h_j^3$ supports nonzero $d_4$ differentials, and equal to $h_0^3 g_{j-2}$ plus possible other terms. For more details, see Remark later.
%   % The differential on $h_6^3$ was previously established in \cite{Wu} by studying the $C_2$ transfer map.
%   % This builds on the main theorem of \cite{Lin} which asserts that there is a non-zero $d_4$ differential on $h_7^3$.   
% \end{rmk}

\begin{thm}\label{thm: d4 hj3} \label{thm:main}
  $d_4(h_j^3) = h_0^3 g_{j-2} \neq 0$
  for $j \geq 6$. % where $g_j$ is a generator of $\Ext^{4, 2^{j+3}+2^{j+2}}_{\A}(\mathbb{F}_2, \mathbb{F}_2) \cong \mathbb{F}_2$.
\end{thm}

\begin{rmk}
We would like to comment that we don't expect the fate of $h_j^3$ can be obtained via the Hill--Hopkins--Ravenel's detection method for cyclic groups $C_{2^n}$. This is partly because we will show that for $j$ large the class $h_j^3$ does not survive to the Adams--Novikov $E_2$-page in the algebraic Novikov spectral sequence so it is difficult to directly compare with the homotopy fixed point spectral sequence and the slice spectral sequence; This is also suggested by a theorem of Hill \cite{Hill} that $\eta^3$ maps to zero under the Hurewicz map for all fixed points of all norms to cyclic 2-groups of the Landweber-Araki Real bordism spectrum, and the work of Li--Shi--Wang--Xu \cite{LSWX} that the Hurewicz image of the $C_2$-fixed point of the Real bordism spectrum behaves well under the $\Sq^0$-operation (see also \Cref{rec:sq0}).
\end{rmk}

The proof of \Cref{thm: d4 hj3} uses two different deformations of the category of spectra---$\mathbb{C}$-motivic spectra and $\F_2$-synthetic spectra---in an essential way. 
We prove \Cref{thm: d4 hj3} as the Betti realization of a corresponding differential in the motivic Adams spectral sequence. This differential in turn is lifted from the the motivic Adams spectral sequence of a certain 2 cell complex---the cofiber of $\tau$. The map $\tau$ here is a map between two p-completed motivic sphere spectra in bidegree $(0,-1)$ that induces a nonzero map on mod $p$ motivic homology (\cite{Voe032}, \cite{HuKrizOrmsby}), which has played a significant role in the $\mathbb{C}$-motivic stable homotopy theory \cite{Isaksen}, \cite{IWX} (see more details in Section~\ref{section2.2}), in contrast to the story over a general base field \cite{rmot}, \cite{BKWX}.  In particular, 
by a theorem of Gheorghe, Wang and the second author \cite[Theorem~1.3]{GWX},
the motivic Adams spectral sequence for the cofiber of $\tau$ is isomorphic to
the algebraic Novikov spectral sequence.
The crucial step in our proof is then to establish a certain non-trivial differential in the algebraic Novikov spectral sequence.

The most delicate part of our argument is in lifting the algebraic Novikov differential proved in \Cref{sec:alg-nov-diff} to the motivic Adams spectral sequence and it is at this point where we need several auxiliary inputs from $\F_2$-synthetic spectra (\cite{Pstragowski}). In \Cref{sec:kervaire} we analyze the differentials on the classes $h_j^2$ using a synthetic refinement of the ``inductive approach to Kervaire invariant one'' from \cite{Inductive}. Roughly speaking, the sphere spectrum in the category of $\F_2$-synthetic spectra has bigraded homotopy groups that encode the Adams spectral sequence information (including differentials) of the classical sphere spectrum. There is a map $\lambda$ between sphere spectra in bidegree $(0,-1)$ such that the bigraded homotopy groups of the cofiber of $\lambda^r$, i.e., $\pi_{**}(\clambda^r)$ encodes information of the Adams spectral sequence of the sphere up to the $E_{r+1}$-page. See \Cref{sec:kervaire} for an introduction and our conventions regarding synthetic spectra.

As a consequence of this we obtain the following theorem.

% Theorem~\ref{thm: d4 hj3} then gives a complete solution to Question~\ref{qst: hj3}. 

% A crucial step in our proof of Theorem~\ref{thm: d4 hj3} is to establish a nonzero differential in the algebraic Novikov spectral sequence, which, by a theorem of Gheorghe, Wang and the second author \cite[Theorem~1.3]{GWX}, is isomorphic to the motivic Adams spectral sequence of a 2 cell complex -- the cofiber of $\tau$. We then pullback this motivic Adams differential from the cofiber of $\tau$ to the motivic sphere spectrum, then use the Betti realization to push forward it to get the classical Adams differentials as stated in Theorem~\ref{thm: d4 hj3}. Background on relevant motivic homotopy theory is included in Section~2.

\begin{thm} \label{thm:inductive-intro} %\!\!\footnote{See \Cref{sec:kervaire} for our conventions regarding synthetic spectra.}
  Fix an $r \geq 2$ and suppose that $\theta_j$ is a lift of $h_j^2$ from $\clambda$ to $\clambda^r$.
  If $2\theta_j = 0$ and $\lambda^2 \theta_j^2 = 0$ in $\pi_{**}(\clambda^r)$,
  then there exists a class $\theta_{j+1}$ lifting $h_{j+1}^2$ to $\clambda^r$ such that $2\theta_{j+1} = 0$ in $\pi_{**}(\clambda^r)$.
  %As a corollary we show that
 % $ d_4(h_j^2) = 0 $
%  for all $j \geq 0$.
\end{thm}

As a corollary of \Cref{thm:inductive-intro} we show that
\begin{cor}
The class $h_j^2$ survives to the Adams $E_5$-page for all $j \geq 0$. In other words,
  \[ d_r(h_j^2) = 0 \quad\text{for}\quad 2 \leq r \leq 4.\]
\end{cor}

Specializing to the case $j=6$ we are able to %use the existence of $\Theta_5 \in \pi_{62}S^0$ to 
refine our arguments to obtain partial progress towards understanding the fate of $h_6^2$.

\begin{thm} \label{thm:h62-intro}
  The class $h_6^2$ survives to the Adams $E_9$-page. In other words,
  \[ d_r(h_6^2) = 0 \quad\text{for}\quad 2 \leq r \leq 8.\]
\end{thm}

\begin{rmk}
  Adams proved that $h_j^4 = 0$ for all $j \geq 1$ and $h_0^n$ is a nonzero permanent cycles detecting $2^n$, therefore as a consequence of \Cref{thm: d4 hj3} the only remaining class of the form $h_j^n$ whose fate we do not know is $h_6^2$.
\end{rmk}

\subsection*{The New Doomsday Conjecture}\hfill

The infinite families of Adams differentials in Theorems~\ref{thm: hopf inv diff} and \ref{thm: d4 hj3} share a key structural feature: their sources and targets naturally fit into $\Sq^0$-families. We expect that this is not a coincidence.

\begin{rec} \label{rec:sq0}
The commutative $\F_2$-algebra structure on the stack of additive formal groups 
provides us with algebraic Steenrod operations (see Chapters IV--VI of \cite{BMMS})
acting on the cohomology of the Steenrod algebra. 
In particular, there is an operation
$$\Sq^0: \Ext^{s, t}_{\A}(\mathbb{F}_2, \mathbb{F}_2) \longrightarrow \Ext^{s, 2t}_{\A}(\mathbb{F}_2, \mathbb{F}_2).$$

Given a class $x$ in the cohomology of the Steenrod algebra we obtain an infinite $\Sq^0$-family of classes $\{x,\ \Sq^0(x),\ \Sq^0(\Sq^0(x)),\ \dots\}$ by iterating $\Sq^0$. 
We say that a $\Sq^0$-family is non-trivial if all members of the family are non-zero.
\end{rec}

\begin{exm}
On the Adams 1-line the Hopf classes $\{h_j\}_{j \geq 0}$ form a $\Sq^0$-family as 
$$\Sq^0 (h_j) = h_{j+1}.$$
Similarly, because $\Sq^0$ is a ring map, we have $\Sq^0$-families $\{h_j^2\}_{j\geq0}$ and $\{h_j^3\}_{j \geq 0}$ as well. 
\end{exm}

\begin{exm}
The classes 
$g_j \in \Ext^{4, 2^{j+3}+2^{j+2}}_{\A}(\mathbb{F}_2, \mathbb{F}_2)$ (see \cite{BMMS}, \cite{LinExt}, \cite{Ext5} for example)
in the statement of Theorem~\ref{thm: d4 hj3} also form a $\Sq^0$-family. In fact, the $g_j$ are defined to be the classes in the $\Sq^0$-family generated by $g_1 \coloneqq g$, which is an element that plays a significant role in the homotopy groups of tmf \cite{TMFbook}. 
\end{exm}

In 1995, Minami \cite{Minami} made the following conjecture regarding $\Sq^0$-families:
%which can be interpretted as a far reaching generalization of the non-existence of Hopf invariant 1 and Kervaire invariant 1 classes.

\begin{cnj}[New Doomsday Conjecture]\label{conj: doomsday}
  For any $\Sq^0$-family $\{x_j\}_{j \geq 0}$ on the Adams $E_2$-page
  only finitely many classes survive to the $E_\infty$-page.
\end{cnj}

Adams's solution of the Hopf invariant one problem (Theorem~\ref{thm: hopf inv diff}) shows that the New Doomsday Conjecture is true on the Adams 1-line. Similarly, Hill--Hopkins--Ravenel's solution of the Kervaire invariant one problem (Theorem~\ref{thm: HHR}) was the last, hardest, case of the New Doomsday Conjecture on the Adams 2-line. 

\begin{rmk}
On the Adams 3-line and above, Conjecture~\ref{conj: doomsday} remains open. With our Theorem~\ref{thm: d4 hj3}, the remaining open cases are the $\Sq^0$-families $\{ h_j^2 h_{j+k+1} + h_{j+1} h_{j+k}^2 \}_{j \geq 0}$ for $k \geq 2$.
\end{rmk}

Based on the uniformity of the differentials in Theorems~\ref{thm: hopf inv diff} and \ref{thm:main} we propose the following refinement of Minami's conjecture:
%Assuming Conjecture~\ref{conj: doomsday} is true, we make another conjecture regarding the $\Sq^0$-families, in an even stronger form.

\begin{cnj}[Uniform Doomsday Conjecture]\label{conj: stable length}
  Let $\{a_j\}_{j \geq 0}$ be a non-trivial $\Sq^0$-family on the Adams $E_2$-page.
  Then, there exists another $\Sq^0$-family $\{b_j\}_{j \geq 0}$, an $r \geq 2$ and a class $c$ such that
  \[ d_r(a_j) = c \cdot b_j \neq 0 \]
  for $j \gg 0$.
\end{cnj}
%$\Sq^0$-stabilization conjecture
\todo{sq0 upper vs lower indexing}

\todo{more on this.}
\begin{exm} \label{exm:Bruner formula}
  Bruner's formulas for power operations on Adams differentials naturally produces families of differentials of the form predicted by \Cref{conj: stable length}. For example one obtains families of differentials 
  \setlength{\columnsep}{-30pt}
  \begin{multicols}{2}
    \begin{enumerate}[leftmargin=30pt, rightmargin=0pt]
    \item $d_2(h_{j+1}) = h_0 \cdot \mathrm{Sq}^1(h_j) = h_0h_j^2$,
    \item $d_2(h_{j+1}^2) = h_0 \cdot \mathrm{Sq}^1(h_j^2) = 0$,
    \item $d_2(c_{j+1}) = h_0 \cdot \mathrm{Sq}^1(c_j) = h_0f_j$,
    \item $d_2(d_{j+1}) = h_0 \cdot \mathrm{Sq}^1(d_j) = 0$,
    \item $d_2(e_{j+1}) = h_0 \cdot \mathrm{Sq}^1(e_j) = h_0x_j$,
    \item $d_2(f_{j+1}) = h_0 \cdot \mathrm{Sq}^1(f_j) = 0$,
    \item $d_2(g_{j+1}) = h_0 \cdot \mathrm{Sq}^1(g_j) = 0$,
    \item $d_2(p_{j+1}) = h_0 \cdot \mathrm{Sq}^1(p_j) = h_0h_jp'_j$,
    \item $d_2(D_3(j+1)) = h_0 \cdot \mathrm{Sq}^1(D_3(j)) = h_0K_j$,
    \item $d_2(p'_{j+1}) = h_0 \cdot \mathrm{Sq}^1(p'_{j}) = h_0T_j$,
    \end{enumerate}    
  \end{multicols}
  \noindent for $j \geq 1$.
  \setlength{\columnsep}{0pt}
  These $Sq^0$-families and these families of differentials are discussed in \cite{BMMS}, \cite{LinExt}, \cite{Ext5}.
\end{exm}

Prior to the present work, all known cases of \Cref{conj: stable length} followed as corollaries of Bruner's power operation formulas
and the authors' core motivation in undertaking this project was to provide a more substantiative test of this conjecture.
The differentials of \Cref{thm:main} cannot be obtained from Bruner's formulas. 
In fact, reversing the flow of information, the differentials of \Cref{thm:main} can be interpreted as the first computation of an infinite family of \emph{hidden} power operations in the Adams spectral sequence for the sphere.

\subsection{An outline of the paper}\hfill

In \Cref{sec:review} we review the Miller square, motivic homotopy theory and set up notation for many of the spectral sequences we will use throughout the paper.
In \Cref{sec:proof} we reduce the proof of \Cref{thm:main} to a collection of propositions which we will verify across the remaining sections of the paper.
In \Cref{sec:algAH} we study the $E_2$-page of the Cartan--Eilenberg spectral sequence near $h_j^3$.
In \Cref{sec:cess} we show that there are no Cartan--Eilenberg differentials near $h_j^3$ and in particular it is at this point that we show that the target of the Adams differential in \Cref{thm:main} is non-trivial.
%% In \Cref{sec:mcess} we introduce the motivic Cartan--Eilenberg spectral sequence and use it to obtain information about the $E_2$-page of the motivic Adams spectral sequence near $h_j^3$.
%% In \Cref{sec:e4}, using the material developed in the previous pair of sections, we prove that $d_2(h_j^3) = 0$ and $d_3(h_j^3) = 0$.
In \Cref{sec:alg-nov-diff} we prove the key family of algebraic Novikov differentials at the heart of our proof.
In \Cref{sec:kervaire} we use $\F_2$-synthetic homotopy theory to analyze the classes $h_j^2$, proving \Cref{thm:inductive-intro}, \Cref{thm:h62-intro} and the final input necessary for the proof of \Cref{thm:main}.

%% \todo{summary copied from elsewhere. Merge this in.}
%% In the rest of this paper, we prove Theorems~\ref{E2 descriptions}, \ref{diff classical}, \ref{diff motivic sphere}, \ref{diff motivic ctau}. In Section~\ref{sec:cess}, we set up the Cartan-Eilenberg spectral sequence and prove part $(1)$ of Theorem~\ref{E2 descriptions}. We also prove Theorem~\ref{diff motivic ctau}~$(3)$ at the end of Section~\ref{sec:cess}. In Section~\ref{sec:mcess}, we set up the motivic analog of the Cartan-Eilenberg spectral sequence and prove parts $(2)(3)$ of Theorem~\ref{E2 descriptions}. In Section~\ref{sec:e4}, we prove Theorem~\ref{diff classical}~$(3)$ and we use Theorem~\ref{diff motivic ctau}~$(1)$ and the motivic zigzag strategy to prove Theorem~\ref{diff classical}~$(1)$ and Theorem~\ref{diff motivic sphere}. These proofs are similar to but easier than the proof of our main Theorem~\ref{thm: d4 hj3}. What remains to be proved are Theorem~\ref{diff motivic ctau}~$(1)(2)$ and Theorem~\ref{diff classical}~$(2)$. In Section~\ref{sec:alg-nov-diff}, we study the key algebraic Novikov differential through the cobar complex, and prove Theorem~\ref{diff motivic ctau}~$(1)(2)$. In Section~\ref{sec:kervaire}, we revisit Barratt--Jones--Mahowald's inductive approach on the Kervaire invariant classes and prove Theorem~\ref{diff classical}~$(2)$.

\subsection{Notations and conventions}\hfill

\begin{enumerate}
\item We write $S^n$ for the $n$-sphere in spectra, \\
  $S^{s,w}$ for the $\CC$-motivic $(s,w)$-sphere $\Sigma^{s-w}(\mathbb{G}_m)^{\otimes w}$ and \\
  $\Ss^{k,s}$ for the $\F_2$-synthetic $(k,s)$-sphere $\Sigma^{-s}\nu( S^{k+s} )$.
\item The indices we use in the various spectral sequences we consider are as follows:
  In the classical Adams $E_2$-page $\Ext_{\A}^{a,t}(\mathbb{F}_2, \mathbb{F}_2)$, we use $a$ for the Adams filtration and $t$ for the internal degree.
  In the motivic Adams $E_2$-page $\Ext^{a,t,w}_{\A^\textup{mot}}(\mathbb{F}_2[\tau], \mathbb{F}_2[\tau])$, we use $a$ for the Adams filtration and $t$ for the internal degree and $w$ for weight, which corresponds to the dimension of $\mathbb{G}_m$-twists (see Section~\ref{section2.2} for more details).
  Moreover, from the motivic Cartan--Eilenberg spectral sequence, every element in $\Ext^{a,t,w}_{\A^\textup{mot}}(\mathbb{F}_2[\tau], \mathbb{F}_2[\tau])$ has a Cartan--Eilenberg filtration, denoted by $k$, and an Adams--Novikov filtration $s$, satisfying that $a = s+k$.
\item In order to avoid notational collisions we write $\tau$ for the usual $\CC$-motivic map $\tau$ and $\lambda$ for the $\F_2$-synthetic map usually denoted $\tau$.
\item In the final pair of sections it becomes important to distinguish between several different classes connected with Kervaire invariant one. 
  We use $h_j^2$ for the classes on the Adams $E_2$-page.
  We use $\vartheta_j$ for certain class on the Adams--Novikov $E_2$-page which map to $h_j^2$ under the Thom reduction map (see \Cref{dfn:Tj}).
  We use $\Theta_j$ for the Kervaire invariant one classes in the stable homotopy groups of spheres.
  We use $\theta_j$ for a choice of class in the synthetic homotopy groups of $\clambda^k$ (which exists when $h_j^2$ survives to the $E_{k+1}$-page of the Adams spectral sequence).
\item All objects are $p$-complete unless otherwise noted.
\item Other than some general introduction in \Cref{sec:review} we work entirely at $p=2$.
\end{enumerate}
  
\subsection*{Acknowledgments}\hfill

The authors would like to thank Hood Chatham, Lars Hesselholt, Mike Hopkins, Hana Jia Kong, Ishan Levy, Haynes Miller, Andy Senger and Guozhen Wang for helpful conversations regarding the content of this paper. We would also like to thank anonymous referee for helpful comments.
During the course of this work the first author was supported by NSF grant DMS-2202992, by the DNRF through the Copenhagen center for Geometry and Topology (DNRF151) and by Villum Fonden through the Young Investigator Program.
The second author is partially supported by NSF grant DMS 2105462.

%%% Local Variables:
%%% mode: latex
%%% TeX-master: "main"
%%% End:

%% file: reviewv2.tex
In this section we review necessary background for the proof of our main theorem---Theorem~\ref{thm: d4 hj3}. We begin by recalling the Miller square, introduced in \cite{MillerSquare}, which captures the interplay between the Adams spectral sequence and the Adams--Novikov spectral sequence. We then discuss the recent cofiber of $tau$ method developed by Gheorghe, Isaksen, Wang and the second author \cite{GWX, IWX, IWX2}, which uses the motivic stable homotopy category over $\mathbb{C}$ and the motivic Adams spectral sequence to categorify the Miller square. %Reader who are familiar with both material should feel free to skip this section.

\subsection{The Miller square}\hfill

The Adams spectral sequence and the Adams--Novikov spectral sequence are two of the most effective methods of computing the homotopy groups of the $p$-completed sphere spectrum, $S^0$. They are spectral sequences of the form:
\begin{align*}
  \Ext_{\A}^{s,t}(\F_p,\, \F_p) &\cong E_2^{s,t} \Longrightarrow \pi_{t-s}S^0, \ &&  d_r:E_r^{s,t} \rightarrow E_r^{s+r, t+r-1} \\
  \Ext_{\BP_*\BP}^{s,t}(\BP_*,\, \BP_*) &\cong E_2^{s,t} \Longrightarrow \pi_{t-s}S^0, \quad &&  d_r:E_r^{s,t} \rightarrow E_r^{s+r, t+r-1}
\end{align*}
where $\A$ is the mod $p$ dual Steenrod algebra and $\BP$ is the Brown--Peterson spectrum at the prime $p$. For degrees, $s$ is the homological degree and is referred as the Adams filtration (resp. the Adams--Novikov filtration), and $t$ is the internal degree.

It is important to understand connections between them. A first connection is given by the Thom reduction map $\BP\rightarrow \F_p$, which induces a map of spectral sequences
$$\Ext_{\BP_*\BP}^{s,t}(\BP_*,\, \BP_*) \longrightarrow \Ext_{\A}^{s,t}(\F_p,\, \F_p)$$
that preserves the $(s,t)$-degrees.
However, a general homotopy class in $\pi_{*}S^0$ doesn't usually have the same Adams filtration as the Adams--Novikov filtration. So this map is not very useful for comparison of the Adams filtration and the Adams--Novikov filtration of a surviving homotopy class---it only tells us the latter is less or equal to the former.

A fundamental connection is the Miller square. We have an algebraic Novikov spectral sequence converging to the Adams--Novikov $E_2$-page, and a Cartan--Eilenberg spectral sequence converging to the Adams $E_2$-page. It turns out the $E_2$-pages of these two algebraic spectral sequences are isomorphic.

The algebraic Novikov spectral sequence comes the filtration of powers of the augmentation ideal $I = (p, v_1, v_2, \cdots) \subset \BP_*$. It has the form:
$$ \Ext_{\BP_*\BP/I}^{s,t'}(\BP_*/I,\, I^k/I^{k+1}) \cong E_2^{s,k, t'} \Longrightarrow \Ext_{\BP_*\BP}^{s,t'}(\BP_*,\, \BP_*)$$
$$d_r: E_r^{s,k,t'} \longrightarrow E_r^{s+1,k+r-1,t'}$$
where $s$ and $k$ are homological degrees, and $t'$ is internal degree. In particular, in the Adams--Novikov gradings, all differentials in the algebraic Novikov spectral sequence look like Adams--Novikov $d_1$-differentials.

For $p=2$, let $\P$ be the sub-Hopf algebra of squares inside $\A$
\[ \P \cong \F_2[\xi_1^2, \xi_2^2, \cdots] \subseteq \F_2[\xi_1, \xi_2, \cdots] \cong \A \]
and let $\mathcal{Q}$ be the quotient Hopf algebra,
$\mathcal{Q} \cong \A \otimes_{\P} \F_2 \cong \Lambda_{\F_2}[\xi_1, \xi_2, \cdots]$.
The Cartan-Eilenberg spectral sequence comes from the extension of Hopf algebras
$\P \to \A \to \mathcal{Q}$.
It has the form:
$$E_2^{s,k,t} = \Ext^{s,t}_{\P}(\F_2, \Ext_{\mathcal{Q}}^k(\F_2, \F_2)) \Longrightarrow \Ext_{\A}^{s+k, t}(\F_2, \F_2)$$
$$d_r: E_r^{s,k,t} \longrightarrow E_r^{s+r,k-r-1,t}$$
where $s$ and $k$ are homological degrees, and $t$ is internal degree. In particular, in the Adams gradings, all differentials in the Cartan--Eilenberg spectral sequence look like Adams $d_1$-differentials.
The Hopf algebra $\mathcal{Q}$ is cocommutative and primitively generated by the exterior classes $\xi_{i+1}$, therefore we have
$$\Ext_{\mathcal{Q}}^{*,*}(\F_2, \F_2) \cong \F_2[q_0, q_1, \cdots],$$
where $q_i$ corresponds to $[\xi_{i+1}]$ and has $(k,t)$ bidegree $(1, 2^{i+1}-1)$.

We identify the $E_2$-pages of the Cartan--Eilenberg spectral sequence and the algebraic Novikov spectral sequence by using the isomorphism of Hopf algebroids
\[ (\BP_*/I,\ \BP_*\BP/I) \cong (\F_2,\ \P) \]
and the associated isomorphism of $\P$-comodule algebras
\[ \Ext_{\mathcal{Q}}^*(\F_2, \F_2) \cong \F_2[q_0, q_1, \cdots] \cong \F_2[v_0, v_1, \cdots] \cong \oplus_* I^*/I^{*+1} \]
%% \begin{align*}
%% \BP_*/I & \cong \F_p,\\
%% \BP_*\BP/I & \cong P,\\
%% \Ext_Q^*(\F_p, \F_p) & \cong \F_p[q_0, q_1, \cdots],\\
%% I^*/I^{*+1} & \cong \F_p[v_0, v_1, \cdots].
%% \end{align*}
where the middle isomorphisms identifies $q_i$ with $v_i$.
Taking into account that $q_i$ has $(s,k,t)$-degree $(0,1,2^{i+1} - 1)$ and $v_i$ has $(s,k,t')$-degree $(0,1,2^{i+1} - 2)$ the isomorphisms above provide an isomorphism
$$\xymatrix{
\Ext^{s,t}_{\P}(\F_2, \Ext_{\mathcal{Q}}^k(\F_2, \F_2)) \ar[rr]^-{\cong} & & \Ext_{\BP_*\BP/I}^{s,t'}(\BP_*/I, I^k/I^{k+1})
}$$
by sending $s$ to $s$, $k$ to $k$, and $t$ to $t'+k$.
Altogether, we have introduced the Miller square:

%% We further identify $q_i$ with $v_i$ keeping in mind that their $s$ and $k$-degrees are 0 and 1, and their $t$ and $t'$-degrees are $2^{i+1} - 1$ and $2^{i+1} - 2$. Then we have an isomorphism
%% $$\xymatrix{
%% \Ext^{s,t}_P(\F_p, \Ext_Q^k(\F_p, \F_p)) \ar[rr]^-{\cong} & & \Ext_{\BP_*\BP/I}^{s,t'}(\BP_*/I, I^k/I^{k+1})
%% }$$
%% by sending $s$ to $s$, $k$ to $k$, and $t$ to $t'+k$. Therefore, we have the Miller square.

\begin{displaymath}
    \xymatrix{
  & & \Ext^{s,t}_{\P}(\F_2, \Ext_{\mathcal{Q}}^k(\F_2, \F_2)) \ar@{=>}[ddll]|{\mathbf{Cartan\textup{-}Eilenberg \ SS}} \ar@{=>}[ddrr]|{\mathbf{Algebraic \ Novikov \ SS}} & &  \\
  & & & & \\
 \Ext_{\A}^{s+k,t}(\F_2, \F_2) \ar@{=>}[ddrr]|{\mathbf{Adams \ SS}} & & & & \Ext^{s,t-k}_{\BP_*\BP}(\BP_*,\BP_*) \ar@{=>}[ddll]|{\mathbf{Adams\textup{-}Novikov \ SS}}\\
  & & & & \\
  & & \pi_{t-s-k}S^0  & &
    }
\end{displaymath}

\begin{rmk}
  Miller's original motivation for studying his square was in order to deduce $d_2$-differentials in the Adams spectral sequence from $d_2$-differentials in the algebraic Novikov spectral sequence (see \cite{MillerSquare}). Using motivic homotopy theory, one can generalize this method to obtain information about $d_r$-differentials for $r \geq 2$. We will explain this in the next subsection.
\end{rmk}

\begin{rmk}
  At odd primes, the dual Steenrod algebra admits an additional \emph{Cartan grading} which places $\xi_i$ in degree $0$ and $\tau_i$ in degree $1$. As the differentials in the Cartan--Eilenberg spectral sequence must respect the Cartan grading, this spectral sequence collapses at the $E_2$-page (see \cite{RavenelGreenBook} for example).
\end{rmk}

\subsection{Motivic homotopy theory} \label{section2.2}   \hfill

We work with the motivic stable homotopy category over $\mathbb{C}$. For the gradings, we denote by $S^{1,0}$ the simplicial sphere, and by $S^{1,1}$ the multiplicative group $\mathbb{G}_m = \mathbb{A}^1 - 0$. We use the same notation for their suspension spectra. We denote by $\F_p^{\textup{mot}}$ the mod $p$ motivic Eilenberg-Mac Lane spectrum that represents the mod $p$ motivic cohomology, and by $\textup{BPGL}$ the motivic Brown-Peterson spectrum at the prime $p$. When the context is clear, we abuse notation and also write $S^{n,w}$ for the $\F_p^{\textup{mot}}$-completed motivic sphere spectrum in bidegree $(n,w)$. 

There is a map
$$\tau: S^{0,-1} \longrightarrow S^{0,0}$$
that induces a nonzero map on mod $p$ motivic homology (\cite{Voe032}, \cite{HuKrizOrmsby}), which has played a significant role in the $\mathbb{C}$-motivic stable homotopy theory \cite{Isaksen}, \cite{GWX} \cite{IWX}. We denote by $S^{0,0}/\tau$ the cofiber of $\tau$.
$$\xymatrix{   
S^{0,-1} \ar[r]^\tau & S^{0,0} \ar[r] & S^{0,0}/\tau \ar[r] & S^{1,-1}.  
}$$

There is a Betti Realization functor $\mathbf{Re}$ from the motivic stable homotopy category over $\mathbb{C}$ to the classical stable homotopy category. We have 
\begin{align*}
\mathbf{Re}(S^{n,w}) & \simeq S^n,\\
\mathbf{Re}(\F_p^\textup{mot}) & \simeq \F_p,\\ % \ \tau \mapsto 1,\\
\mathbf{Re}(\textup{BPGL}) & \simeq \BP.
\end{align*}

%(\cite{HuKrizOrmsby, DuggerIsaksen})
The motivic dual Steenrod algebra over $\mathbb{C}$ is a Hopf algebra over $\pi_{*,*}(\F_p^{\textup{mot}}) = \F_p[\tau]$
which takes the following form at $p=2$ (see \cite[Sec. 12]{Voe03})
$$\A^\textup{mot} \cong \F_2[\tau][\tau_1, \ \tau_2, \cdots][\xi_1^2, \ \xi_2^2, \cdots]/\tau_i^2 = \tau \xi_i^2,$$
$$\mathbf{Re}(\tau)=1, \ \mathbf{Re}(\tau_i) = \xi_i, \ \mathbf{Re}(\xi_i^2) = \xi_i^2.$$
We then have the motivic Adams spectral sequence at every prime (\cite{DuggerIsaksen}) of the form
$$ \Ext_{\A^\textup{mot}}^{a,t,w}(\F_p[\tau], \F_p[\tau]) \cong E_2^{a,t,w} \Longrightarrow \pi_{t-a,w}S^{0,0}$$
$$d_r:E_r^{a,t,w} \rightarrow E_r^{a+r, t+r-1,w}$$
% Here $\A^{\textup{mot}}$ is motivic mod $p$ dual Steenrod algebra, $\pi_{*,*}(\F_p^{\textup{mot}}) = \F_p[\tau]$, and $\mathbf{Re}(\tau) = 1$.
The quotient map $S^{0,0} \rightarrow S^{0,0}/\tau$ induces a map of the motivic Adams spectral sequences.
$$\Ext_{\A^\textup{mot}}^{a,t,w}(\F_p[\tau],\, \F_p[\tau]) \longrightarrow \Ext_{\A^\textup{mot}}^{a,t,w}(\F_p[\tau],\, \F_p).$$

The following theorem is crucial in the computation of classical and motivic stable homotopy groups of spheres over $\mathbb{C}$.
\begin{thm}[{\cite[Theorem 1.17]{GWX}}] \label{MASS algNSS}
  There is an isomorphism of tri-graded spectral sequences between
  the motivic Adams spectral sequence for $S^{0,0}/\tau$ and
  the algebraic Novikov spectral sequence.
  \begin{displaymath}
    \xymatrix{
      \Ext_{{\A}^{\textup{mot}}}^{s+k, t'-k, \frac{t'}{2}}(\F_p[\tau],\, \F_p) \ar@{=>}[dd]|{\mathbf{Motivic \ Adams \ SS}}  \ar[rr]^{\cong} & &  \Ext_{\BP_*\BP/I}^{s,t'}(\F_p,\, I^{k}/I^{k+1}) \ar@{=>}[dd]|{\mathbf{Algebraic \ Novikov \ SS}} \\
      & & \\
      \pi_{t'-s, \frac{t'}{2}}(S^{0,0}/\tau)  \ar[rr]^-{\cong} & &   \Ext_{\BP_*\BP}^{s, t'}(\BP_*,\, \BP_*).
    }
  \end{displaymath}
\end{thm}

\begin{rmk}  
  Note that the $\Ext$-groups in the left column in Theorem~\ref{MASS algNSS} are only defined when $t'$ is even, meanwhile in the right column the $\Ext$-groups vanish for $t'$ odd for sparsity reasons.
\end{rmk}

As explained in \cite[Subsection~1.3]{GWX}, the motivic deformation and the naturality of the Adams spectral sequences give us a zig-zag diagram.
\begin{displaymath}
  \xymatrix{
\Ext_{\A}^{a,t}(\F_p,\, \F_p) \ar@{=>}[dd]|{\mathbf{Adams \ SS}}  & & 
  \Ext^{a,t,w}_{\A^\textup{mot}}(\F_p[\tau],\, \F_p[\tau]) \ar[ll]_-{\mathbf{Re}} \ar[rr] \ar@{=>}[dd]|{\mathbf{Motivic \ Adams \ SS}} & & 
   \Ext^{a,t,w}_{\A^\textup{mot}}(\F_p[\tau],\, \F_p)   \ar@{=>}[dd]|{\mathbf{Motivic \ Adams \ SS}} \\
 & & & & \\
 \pi_{t-a}S^0 & & \pi_{t-a,w}S^{0,0} \ar[ll]_-{\mathbf{Re}} \ar[rr] & & \pi_{t-a,w}S^{0,0}/\tau
}
\end{displaymath}
The right-hand horizontal maps are induced by the quotient map $S^{0,0} \rightarrow S^{0,0}/\tau$, and
the left-hand horizontal maps are given by the Betti realization functor.
The diagram of spectral sequences allow us to build up connections between the differentials in the classical Adams spectral sequence and the algebraic Novikov spectral sequence (by Theorem~\ref{MASS algNSS}) through the motivic world.

\begin{rmk} \label{tridegree}
At $p=2$, it is often useful to combine the isomorphisms in the Miller square and in Theorem~\ref{MASS algNSS} and obtain the following isomorphism between the $E_2$-pages of the Cartan-Eilenberg spectral sequence and the motivic Adams spectral sequence for $S^{0,0}/\tau$:
\[\xymatrix{
  \Ext^{s,t}_{\P}(\F_2, \Ext_{\mathcal{Q}}^k(\F_2, \F_2)) \ar[rr]^-{\cong} & & \Ext^{s+k, t, \frac{t-k}{2}}_{\A^{\textup{mot}}}(\F_2[\tau], \F_2).
}\]
\end{rmk}

We also have the motivic Adams-Novikov spectral sequence (\cite{HuKrizOrmsby, Isaksen}) of the form.
$$ \Ext_{\textup{BPGL}_{*,*}\textup{BPGL}}^{s,t',w}(\textup{BPGL}_{*,*},\, \textup{BPGL}_{*,*}) \cong E_2^{s,t',w} \Longrightarrow \pi_{t'-s,w}S^{0,0}$$
$$d_r:E_r^{s,t',w} \rightarrow E_r^{s+r, t'+r-1,w}$$

In \cite[Sections~6.1, 6.2]{Isaksen}, Isaksen proves the following rigidity theorem.
\begin{thm} \label{MANSS Isaksen}
After a re-grading, the motivic Adams-Novikov spectral sequence for $\pi_{*,*}S^{0,0}$ is isomorphic to a $\tau$-bockstein spectral sequence.
\end{thm}

Finally, for some of our later arguments we need a motivic version of the Cartan--Eilenberg spectral sequence.
Notably, this spectral sequence satisfies a rigidity theorem analogous to \Cref{MANSS Isaksen}.

\begin{cnstr}
  At $p=2$, we define 
  \begin{align*}
    \P^\textup{mot} & \coloneqq \F_2[\tau][\xi_1^2, \ \xi_2^2, \cdots] \cong \P[\tau], \\ 
    \mathcal{Q}^{\textup{mot}} & \coloneqq \A^{\textup{mot}}\otimes_{\P^{\textup{mot}}} \F_2[\tau] \cong \F_2[\tau][\tau_1, \ \tau_2, \cdots]
  \end{align*}
  The associated extension of Hopf algebroids
  $\P^\textup{mot} \rightarrow \A^\textup{mot} \rightarrow \mathcal{Q}^\textup{mot}$
  gives us a motivic version of the classical Cartan-Eilenberg spectral sequence.
  It has the form
  $$ \Ext^{s,t,w}_{\P^\textup{mot}}(\F_2[\tau],\, \Ext^k_{\mathcal{Q}^\textup{mot}}(\F_2[\tau],\, \F_2[\tau])) \cong E_2^{s,k,t,w} \Rightarrow \Ext^{s+k, t, w}_{\A^\textup{mot}}(\F_2[\tau],\, \F_2[\tau]), $$
$$d_r: E_r^{s,k,t,w} \to E_r^{s+r, k-r+1, t, w}.$$
  Note that in this spectral sequence all $d_r$-differentials look like motivic Adams $d_1$-differential in the tri-gradings.
\end{cnstr}

%% In Section~\ref{sec:mcess}, from an extension of Hopf algebroids $\P^{\textup{mot}} \rightarrow \A^{\textup{mot}} \rightarrow \mathcal{Q}^{\textup{mot}}$, we set up a motivic Cartan-Eilenberg spectral sequence that converges to the motivic Adams $E_2$-page, and prove the following rigidity theorem in analogue of Theorem~\ref{MANSS Isaksen}.

\begin{thm}[Rigidity Theorem] \label{MCESS tau Bockstein}
At $p=2$, after a re-grading, the motivic Cartan-Eilenberg spectral sequence for $ \Ext_{{\A}^{\textup{mot}}}^{*, *, *}(\F_2[\tau], \F_2[\tau])$ is isomorphic to a $\tau$-Bockstein spectral sequence.
\end{thm}

\begin{proof}
We have 
\begin{align*}
\Ext^{s,t,w}_{\P^{\textup{mot}}}(\F_2[\tau], \Ext^k_{\mathcal{Q}^\textup{mot}}(\F_2[\tau], \F_2[\tau])) & \cong \Ext^{s,t,w}_{\P^\textup{mot}}(\F_2[\tau], \Ext^k_{\mathcal{Q}}(\F_2, \F_2))[\tau])  \\ 
& \cong \Ext_{\P}^{s,t}(\F_2, \Ext^k_{\mathcal{Q}}(\F_2, \F_2))[\tau].
\end{align*}
Here $\tau$ has $(s,k,t,w)$-degrees $(0,0,0, -1)$, and every element in $\Ext_{\P}^{s,t}(\F_2, \Ext^k_{\mathcal{Q}}(\F_2, \F_2))$ satisfies $w = \frac{t-k}{2}$ by Remark~\ref{tridegree}. Note that by sparseness this group is nonzero only if $t-k$ is even.

Consider $S^{0,0}/\tau$, we have an isomorphism of the $E_2$-page of its MCESS 
\begin{align*}
\Ext^{s,t,w}_{\P^{\textup{mot}}}(\F_2[\tau], \Ext^k_{\mathcal{Q}^\textup{mot}}(\F_2[\tau], \F_2)) & \cong \Ext^{s,t,w}_{\P^\textup{mot}}(\F_2[\tau], \Ext^k_{\mathcal{Q}}(\F_2, \F_2)))  \\ 
& \cong \Ext_{\P}^{s,t}(\F_2, \Ext^k_{\mathcal{Q}}(\F_2, \F_2)).
\end{align*}
Since all differentials in MCESS preserve the $w$ and $t$-degrees, it must preserve the $k$-degree in the case for $S^{0,0}/\tau$, so its MCESS collapses at the $E_2$-page.

Back to $S^{0,0}$, again due to degree reasons, all $d_r$-differentials have the form
$$d_{2n+1} x = \tau^n y, \ \ x, y \in \Ext_{\P}^{s,t}(\F_2, \Ext^k_{\mathcal{Q}}(\F_2, \F_2)),$$
and $d_{2n} = 0$. One can check that there are no $\tau$-extensions in MCESS. This is precisely saying that after a regrading, the MCESS for $\Ext^{*, *, *}_{\A^\textup{mot}}(\F_2[\tau], \F_2[\tau])$ is isomorphic to a $\tau$-Bockstein spectral sequence and completes the proof.
\end{proof}

\begin{rmk}
The important consequence of Theorem~\ref{MANSS Isaksen} is that, the differentials in the classical Adams-Novikov spectral sequence completely determine the differentials in the motivic Adams-Novikov spectral sequence, and vice versa. Similarly, our Theorem~\ref{MCESS tau Bockstein} tells us that the differentials in the classical and motivic Cartan-Eilenberg spectral sequences determine each other.
\end{rmk}

\begin{cor} \label{cor:no entering}
  Suppose there are no non-trivial Cartan--Eilenberg differentials entering tridegrees
  $(s,k,t) = (s,*,t)$.
  Let $\{\overline{x}_i\}_{i \in I}$ be a collection of generators of the permanent cycles on the Cartan--Eilenberg $E_2$-page in tridegree $(s,k,t) = (s,*,t)$.
  Then, we can lift each $\overline{x}_i$ to a class $x_i$ on the $E_2$-page of the motivic Adams spectral sequence for $S^{0,0}$ and any choice of lifts $\{x_i\}_{i \in I}$ provides a basis for $E_2^{s,*,t}$ as a free $\F_2[\tau]$-module.
\end{cor}

%This fact is very useful in our main arguments in Section~\ref{sec:proof}.

%% file: proofv2.tex
In this section, we give the proof of our main theorem, \Cref{thm: d4 hj3}, modulo three inputs from the later sections of the paper. Our core strategy is to use the motivic zig-zag to gain information about the classical Adams $d_4$ differential on $h_j^3$ from the associated algebraic Novikov $d_4$ differential.

%For indexes, recall from Section~\ref{sec:review} that, for elements in the classical Adams $E_2$-page $\Ext_{A}^{a,t}(\mathbb{F}_2, \mathbb{F}_2)$, we have $a$ is the Adams filtration, $t$ is the internal degree; for elements in the motivic Adams $E_2$-page $\Ext^{a,t,w}_{A^\textup{mot}}(\mathbb{F}_2[\tau], \mathbb{F}_2[\tau])$, we have $w$ is the motivic filtration. Moreover, from the motivic Cartan-Eilenberg spectral sequence, every element in $\Ext^{a,t,w}_{A^\textup{mot}}(\mathbb{F}_2[\tau], \mathbb{F}_2[\tau])$ has a Cartan-Eilenberg filtration, denoted by $k$, and an Adams-Novikov filtration $s$, satisfying that $a = s+k$.

\begin{rec}
  As in the classical Steenrod algebra, $\xi_1^{2^j}$ is a primitive element in the motivic dual Steenrod algebra. We let $h_j$ denote the associated class on the $E_2$-page of the motivic Adams spectral sequence. $h_j$ lives in $(a,s,t,w)$-degree $(1, \ 1, \ 2^j, \ 2^{j-1})$. Similarly, we also denote the induced classes on the $E_2$-pages of the classical Adams sseq and motivic Adams sseq for $S^{0,0}/\tau$ by $h_j$.
\end{rec}

In fact, the main result of this section is a computation of $d_4(h_j^3)$ in the motivic Adams spectral sequence (with \Cref{thm: d4 hj3} being obtained from this by Betti realization). Before proceeding let us summarize what we need from later sections.

\begin{itemize}
\item In \Cref{E2 descriptions} we describe the $E_2$-pages of all three spectral sequences in the motivic zig-zag (classical Adams, motivic Adams, motivic Adams for $S^{0,0}/\tau$) in a neighborhood of the classes $h_j^3$. In particular, the Adams filtration 6 and 7 lines in \Cref{E2 descriptions}~(1) are the crucial new information. This theorem is proved in Sections \ref{sec:algAH} and \ref{sec:cess} and summarized in \Cref{fig:3charts}.
\item In \Cref{diff motivic ctau} we compute the differentials on $h_j^3$ in the motivic Adams sseq for $S^{0,0}/\tau$. This theorem is proved in \Cref{sec:alg-nov-diff}.
\item In \Cref{prop:diff classical} we give partial information on the classical Adams differentials on $h_j^3$.
  This theorem is proved in \Cref{sec:kervaire} using $\F_2$-synthetic homotopy theory.
\end{itemize}

\begin{thm} \label{E2 descriptions}
  Let $j \geq 6$.
  \begin{enumerate}
  \item The $E_2$-page of the classical Adams spectral sequence for $S^{0}$, 
    $$ \Ext_{\A}^{a,t}(\mathbb{F}_2, \mathbb{F}_2) \cong E_2^{a,t} \Rightarrow \pi_{t-a}S^{0},$$
    takes the following form near $h_j^3$:
    \begin{center}{\renewcommand{\arraystretch}{1.2}
        \begin{tabular}{|ll|c|}\hline
          $(a,$ & $t-a)$ & $\Ext^{a,\,t}_{\A}(\mathbb{F}_2,\, \mathbb{F}_2)$ \\\hline\hline
          $(4,$ & $ 3 \cdot 2^j-4)$ & $\F_2\{g_{j-2}\}$ \\\hline
          $(5,$ & $ 3 \cdot 2^j-4)$ & $\F_2\{h_0g_{j-2}\}$ \\\hline
          $(6,$ & $ 3 \cdot 2^j-4)$ & contains $\F_2\{h_0^2g_{j-2}\}$ \\\hline
          $(7,$ & $ 3 \cdot 2^j-4)$ & contains $\F_2\{h_0^3g_{j-2}\}$ \\\hline
          $(3,$ & $ 3 \cdot 2^j-3)$ & $\F_2\{h_j^3\}$ \\\hline
          $(4,$ & $ 3 \cdot 2^j-3)$ & $\F_2\{h_0h_j^3\}$ \\\hline
          $(5,$ & $ 3 \cdot 2^j-3)$ & $\F_2\{h_0^2h_j^3, h_1g_{j-2}\}$ \\\hline    
        \end{tabular}}
    \end{center}
    with the exception that in the case $j=6$ degree $(a,t-s)=(5,3 \cdot 2^{6} - 4)$ contains the additional class $h_7D_3(0)$, and that in the case $j=7$ degree $(a,t-s)=(5,3 \cdot 2^{7} - 3)$ contains the additional class $h_8D_3(1)$.
    
  \item The $E_2$-page of the motivic Adams spectral sequence for $S^{0,0}$, 
    $$\Ext^{a,t,w}_{\A^\textup{mot}}(\mathbb{F}_2[\tau], \mathbb{F}_2[\tau]) \cong E_2^{a,t,w} \Rightarrow \pi_{t-a, w}S^{0,0},$$
    takes the following form near $h_j^3$:
    \begin{center}{\renewcommand{\arraystretch}{1.2}
        \begin{tabular}{|ll|c|}\hline
          $(a,$ & $t-a)$ & $\Ext^{a,\,t,\,w}_{\A^\textup{mot}}(\mathbb{F}_2[\tau],\, \mathbb{F}_2[\tau])$ \\\hline\hline
          $(4,$ & $ 3 \cdot 2^j-4)$ & $\F_2[\tau]\{g_{j-2}\}$ \\\hline
          $(5,$ & $ 3 \cdot 2^j-4)$ & $\F_2[\tau]\{h_0g_{j-2}\}$ \\\hline
          $(6,$ & $ 3 \cdot 2^j-4)$ & contains $\F_2[\tau]\{h_0^2g_{j-2}\}$, $\tau$-torsion free \\\hline
          $(7,$ & $ 3 \cdot 2^j-4)$ & contains $\F_2[\tau]\{h_0^3g_{j-2}\}$, $\tau$-torsion free \\\hline
          $(3,$ & $ 3 \cdot 2^j-3)$ & $\F_2[\tau]\{h_j^3\}$ \\\hline
          $(4,$ & $ 3 \cdot 2^j-3)$ & $\F_2[\tau]\{h_0h_j^3\}$ \\\hline
          $(5,$ & $ 3 \cdot 2^j-3)$ & $\F_2[\tau]\{h_0^2h_j^3, h_1g_{j-2}\}$ \\\hline    
        \end{tabular}}
    \end{center}
    where each of the generators
    $g_{j-2}, h_0g_{j-2}, h_0^2g_{j-2}, h_0^3g_{j-2}, h_j^3, h_0h_j^3, h_0^2h_j^3$
    has weight $w = 3 \cdot 2^{j-1}$ and $h_1g_{j-2}$ has weight $w = 3 \cdot 2^{j-1}+1$, with the exception that
    in the case $j=6$ degree $(a,t-a)=(5,3 \cdot 2^{6} - 4)$ contains an additional $\F_2[\tau]$ summand with generator $h_7D_3(0)$, and that
    in the case $j=7$ degree $(a,t-a)=(5,3 \cdot 2^{7} - 3)$ contains an additional $\F_2[\tau]$ summand with generator $h_8D_3(1)$ of weight $w = 3 \cdot 2^{6}+1$.
    
  \item The $E_2$-page of the motivic Adams spectral sequence for $S^{0,0}/\tau$, 
    $$\Ext^{a,t,w}_{\A^\textup{mot}}(\mathbb{F}_2[\tau], \mathbb{F}_2) \cong E_2^{a,t,w} \Rightarrow \pi_{t-a, w}S^{0,0}/\tau,$$
    takes the following form near $h_j^3$:
    \begin{center}{\renewcommand{\arraystretch}{1.2}
        \begin{tabular}{|lll|c|}\hline
          $(a,$ & $2w-t+a,$ & $t-a)$ & $\Ext^{a,\,t,\,w}_{\A^\textup{mot}}(\mathbb{F}_2[\tau],\, \mathbb{F}_2)$ \\\hline\hline
          $(4,$ & $4,$ & $ 3 \cdot 2^j-4)$ & $\F_2\{g_{j-2}\}$ \\\hline
          $(5,$ & $4,$ & $ 3 \cdot 2^j-4)$ & $\F_2\{h_0g_{j-2}\} $ \\\hline
          $(6,$ & $4,$ & $ 3 \cdot 2^j-4)$ & contains $\F_2\{h_0^2g_{j-2}\} $ \\\hline
          $(7,$ & $4,$ & $ 3 \cdot 2^j-4)$ & contains $\F_2\{h_0^3g_{j-2}\} $ \\\hline
          $(3,$ & $1,$ & $ 3 \cdot 2^j-3)$ & $0$ \\\hline
          $(4,$ & $1,$ & $ 3 \cdot 2^j-3)$ & $0$ \\\hline
          $(5,$ & $1,$ & $ 3 \cdot 2^j-3)$ & $0$ \\\hline
          $(6,$ & $1,$ & $ 3 \cdot 2^j-3)$ & $0$ \\\hline          
          $(3,$ & $3,$ & $ 3 \cdot 2^j-3)$ & $\F_2\{h_j^3\}$ \\\hline
          $(4,$ & $3,$ & $ 3 \cdot 2^j-3)$ & $\F_2\{h_0h_j^3\}$ \\\hline
          $(5,$ & $3,$ & $ 3 \cdot 2^j-3)$ & $\F_2\{h_0^2h_j^3\}$ \\\hline    
          $(6,$ & $3,$ & $ 3 \cdot 2^j-3)$ & $\F_2\{h_0^3h_j^3\}$ or $0$ \\\hline
          $(5,$ & $5,$ & $ 3 \cdot 2^j-3)$ & $\F_2\{h_1g_{j-2}\}$ \\\hline          
        \end{tabular}}
    \end{center}
    where the generators have the same weights as in (2), with the exception that
    in the case $j=6$ degree $(a,2w-t+a,t-a)=(5,4,3 \cdot 2^{6} - 4)$ possibly contains an additional class $h_7D_3(0)$, and that
    in the case $j=7$ degree $(a,2w-t+a,t-a)=(5,5,3 \cdot 2^{7} - 3)$ contains an additional class $h_8D_3(1)$.

  \item Under the Betti realization and reduction mod $\tau$ maps each generator maps to the generator of the same name.
  \end{enumerate}
\end{thm}

\input{proof-chart.tex}

\begin{figure}[h]
  \centering
  The classical and motivic Adams spectral sequences \\
  for $S^0, \ S^{0,0}$ and $S^{0,0}/\tau$ near $h_j^3$
   %\\\vspace{6pt}
  \scalebox{1.00}{
    \printpage[ name = 3charts1hj3, page = 3, y axis gap = 0.8cm, right clip padding = 0cm, x axis extend end = 0.75cm ] \quad
    \printpage[ name = 3charts2hj3, page = 3, no y ticks, x axis tail = 0cm, y axis gap = 0.8cm, right clip padding = 0cm, x axis extend end = 0.75cm ] \quad
    \printpage[ name = 3charts3hj3, page = 3, no y ticks, x axis tail = 0cm, y axis gap = 0.8cm, right clip padding = 0cm, x axis extend end = 0.75cm ]
  }
  \caption{ 
    The horizontal degree is the topological stem $t-a$, shifted by $3\cdot 2^j$.     
    The vertical degree is the Adams filtration $a$. 
    Left and Right:
    Each dot $\bullet$ denotes a copy of $\F_2$. 
    Circles $\circ$ denote possible copies of $\F_2$ and
    numbers inside the circle indicate the $s$-degree, the Cartan--Eilenberg Ext-degree.
    The red dot indicates a class that is possibly zero.
    Middle: 
    Each solid square $\blacksquare$ denotes a copy of $\F_2[\tau]$. 
    Hollowed squares $\square$ denotes possible copies of $\F_2[\tau]$.}
  \label{fig:3charts}    
\end{figure}

% \begin{itemize}
%     \item
%       The horizontal degree is the topological stem $t-a$, shifted by $3\cdot 2^j$. 
%     \item
%       The vertical degree is the Adams filtration $a$. 
%     \item Left and Right: Each dot $\bullet$ denotes a copy of $\F_2$. 
%     \item Middle: 
%       \begin{itemize}
%       \item Each solid square $\blacksquare$ denotes a copy of $\F_2[\tau]$. 
%       \item Hollowed square $\square$ denotes possible copies of $\F_2[\tau]$. 
%       \item Circle $\circ$ and double circle $\circledcirc$ denote possible copies of $\F_2$ and $\F_2[\tau]/\tau^2$,\\ with numbers inside indicating their $s$-degrees.
%       \end{itemize}
%     \end{itemize}

% \begin{rmk}
%   In fact, using arguments that are similar to the proof of Theorem~\ref{E2 descriptions}~$(2)$~$(c)$, we can further rule out a few more possible groups in the statement of Theorem~\ref{E2 descriptions}~$(2)$, such as the possible copies of $\F_2$ in degree $(a, \ t-a, \ s, \ w) = (4, \ 3\cdot 2^j - 3, \ 3, \ 3\cdot 2^{j-1})$ of Theorem~\ref{E2 descriptions}~$(2)$~$(a)$. But for the purpose of proving our main Theorem~\ref{thm: d4 hj3}, we do not need such arguments. 
% \end{rmk}

As an amplification of the first four lines of the table in \Cref{E2 descriptions}(2) we have the following corollary.

\begin{cor} \label{cor:E2 inj}
  The Betti realization map from 
  the $E_2$-page of the motivic Adams sseq for $S^{0,0}$ to
  the $E_2$-page of the Adams sseq for $S^{0}$ is
  injective in the tridegrees of $h_0g_{j-2}$, $h_0^2 g_{j-2}$, $h_0^3 g_{j-2}$ and their $\tau$-multiples.
\end{cor}

%% In order to proceed we need two more inputs from the later sections of this paper.
%% In \Cref{sec:algN} we will prove \Cref{diff motivic ctau} which computes 
%% prove Theorem~\ref{thm: d4 hj3}, we also need the following theorems.

% section 7 input
\begin{thm} \label{diff motivic ctau}
  Let $j \geq 6$.
  In the motivic Adams sseq for $S^{0,0}/\tau$, 
  \begin{enumerate}
  \item $d_2(h_j^3) = 0$,
  \item $d_3(h_j^3) = 0$ and 
  \item $d_4(h_j^3) = h_0^3 g_{j-2}$.
  \end{enumerate}
\end{thm}

% section 8 input
\begin{thm} \label{prop:diff classical}
  Let $j \geq 6$. In the classical Adams sseq for $S^{0}$, 
  \begin{enumerate}
  \item $d_2(h_j^3) = 0$,
  \item $d_3(h_j^3)$ is either $0$ or $h_0^2g_{j-2}$ and
  \item $d_4(h_j^3)$ is either $0$ or $h_0^3g_{j-2}$ (if defined).
  \end{enumerate}  
\end{thm}

With all our inputs ready we now begin proving our main theorem.

% alg N entering diffs
\begin{lem} \label{lem:ctau-entering-diffls}
  Let $j \geq 6$.
  In the motivic Adams sseq for $S^{0,0}/\tau$, 
  there are no non-zero $d_2$ or $d_3$-differentials entering the tridegrees of
  $h_0g_{j-2}$, $h_0^2g_{j-2}$ or $h_0^3 g_{j-2}$.
\end{lem}

\begin{proof}
  Using \Cref{E2 descriptions}(3) we can read off that the only potential sources for a $d_2$ or $d_3$-differential entering the tridegree of one of $h_0g_{j-2}$, $h_0^2g_{j-2}$ or $h_0^3 g_{j-2}$ are
  $h_j^3$, $h_0h_j^3$ and $h_0^2h_j^3$.
  The lemma now follows as a corollary of \Cref{diff motivic ctau}(1,2).
\end{proof}

% d2 differentials
\begin{lem} \label{lem:d2 diffls}
  Let $j \geq 6$.
  \begin{enumerate}
  \item In the motivic Adams sseq for $S^{0,0}$, $d_2(h_j^3) = 0$.
  \item For $j \neq 7$, in the Adams sseq for $S^0$, there are no non-zero $d_2$-differentials entering the bidegrees of $h_0g_{j-2}$, $h_0^2 g_{j-2}$ and $h_0^3 g_{j-2}$.
  \item For $j \neq 7$, in the motivic Adams sseq for $S^{0,0}$, there are no non-zero $d_2$-differentials entering the tridegrees of $h_0g_{j-2}$, $h_0^2 g_{j-2}$, $h_0^3 g_{j-2}$ or their $\tau$-multiples.
  \item The Betti realization map from
    the $E_3$-page of the motivic Adams sseq for $S^{0,0}$ to
    the $E_3$-page of the Adams sseq for $S^{0}$ is
    injective in the tridegrees of $h_0g_{j-2}$, $h_0^2 g_{j-2}$, $h_0^3 g_{j-2}$ and their $\tau$-multiples.
  \end{enumerate}  
\end{lem}

\begin{proof}
  Using \Cref{cor:E2 inj} we can deduce (1) from \Cref{prop:diff classical}(1).  
  For $j \neq 7$  we can read off from \Cref{E2 descriptions}(1) that the only potential sources for Adams differential entering the bidegrees of $h_0g_{j-2}$, $h_0^2 g_{j-2}$ and $h_0^3 g_{j-2}$ are
  the classes $h_j^3$, $h_0h_j^3$, $h_0^2h_j^3$ and $h_1g_{j-2}$.
  Thus, (2) follows from \Cref{prop:diff classical}(1) and \Cref{exm:Bruner formula} (which tells us that $d_2(g_{j-2}) = 0$).
  The injectivity from \Cref{cor:E2 inj} implies that any entering motivic $d_2$-differential would induce an entering classical $d_2$-differential, therefore (2) implies (3).  
  For $j \neq 7$, (4) is obtained by combining (2), (3) and \Cref{cor:E2 inj}.

  For the $j=7$ case of (4) we observe that the additional generator $h_8D_3(1)$ has degree
  $(a,t,w) = (5, 3 \cdot 2^7 + 2, 3 \cdot 2^6 + 1)$ and that
  in order for this differential to create $\tau$-torsion on the motivic Adams $E_3$-page its target must be $\tau$-divisible.
  This would force us to have a non-trivial class in degree
  $(a,t,w) = (7, 3 \cdot 2^7 + 3, 3 \cdot 2^6 + 2)$.
  On the other hand, such a class would contradict the fact that
  the motivic Adams $E_2$-page vanishes when $t < 2w$.
\end{proof}

% d3 differentials
\begin{lem} \label{lem:d3 diffls}
  Let $j \geq 6$.
  \begin{enumerate}
  \item In the classical Adams sseq for $S^{0}$, $d_3(h_j^3) = 0$.
  \item In the motivic Adams sseq for $S^{0,0}$, $d_3(h_j^3) = 0$.
  \item In the classical Adams sseq for $S^0$, there are no non-zero $d_3$-differentials entering the bidegrees of $h_0^2 g_{j-2}$ and $h_0^3 g_{j-2}$.
  \item In the motivic Adams sseq for $S^{0,0}$, there are no non-zero $d_3$-differentials entering the tridegrees of $h_0^2 g_{j-2}$, $h_0^3 g_{j-2}$ or their $\tau$-multiples.
  \item The Betti realization map from
    the $E_4$-page of the motivic Adams sseq for $S^{0,0}$ to
    the $E_4$-page of the Adams sseq for $S^{0}$ is
    injective in the tridegrees of $h_0g_{j-2}$, $h_0^2 g_{j-2}$, $h_0^3 g_{j-2}$ and their $\tau$-multiples.
  \end{enumerate}  
\end{lem}

\begin{proof}
  We begin with (2).
  The injectivity statement from \Cref{lem:d2 diffls}(4) allows us to upgrade
  \Cref{prop:diff classical} to the claim that
  in the motivic Adams sseq for $S^{0,0}$, $d_3(h_j^3)$ is either $0$ or $h_0^2g_{j-2}$.
  Next we examine the reduction mod $\tau$ map to the motivic Adams sseq for $S^{0,0}/\tau$.  
  From \Cref{E2 descriptions}(3) and \Cref{lem:ctau-entering-diffls}
  we know that $h_0^2g_{j-2}$ is non-zero
  on the $E_3$-page of the motivic Adams sseq for $S^{0,0}/\tau$.
  Thus, (2) follows from \Cref{diff motivic ctau}(2) which tells us that
  $d_3(h_j^3) = 0$ in the motivic Adams sseq for $S^{0,0}/\tau$.
  (1) follows from (2) by Betti realization.
        
  From \Cref{E2 descriptions}(1) we can read off that the only potential sources for Adams differential entering the bidegrees of $h_0g_{j-2}$, $h_0^2 g_{j-2}$ and $h_0^3 g_{j-2}$ are
  the classes $h_j^3$ and $h_0h_j^3$.
  Thus, (3) follows from (1).
  The injectivity from \Cref{lem:d2 diffls}(4) implies that any entering motivic $d_3$-differential would induce an entering classical $d_3$-differential, therefore (3) implies (4).  
  (5) is obtained by combining (3), (4) and \Cref{lem:d2 diffls}(4).
\end{proof}

\begin{thm} \label{thm:motivic d4}
  Let $j \geq 6$.
  In the motivic Adams sseq for $S^{0,0}$ we have
  \[ d_4(h_j^3) = h_0^3g_{j-2} \neq 0 \]
  which Betti realizes to the non-trivial Adams differential of \Cref{thm: d4 hj3}
\end{thm}

\begin{proof}
  From Lemmas \ref{lem:d2 diffls}(1) and \ref{lem:d3 diffls}(2) we know that the class $h_j^3$ survives to the $E_4$-page of the motivic Adams sseq.
  Under the reduction mod $\tau$ map to $E_4$-page of the motivic Adams sseq for $S^{0,0}/\tau$
  the class $d_4(h_j^3)$ is sent to a non-zero class by \Cref{diff motivic ctau}(3) and \Cref{lem:ctau-entering-diffls}, therefore $d_4(h_j^3) \neq 0$.
  On the other hand, the injectivity statement from \Cref{lem:d3 diffls}(5) allows us to upgrade
  \Cref{prop:diff classical}(3) to the claim that
  in the motivic Adams sseq for the sphere $d_4(h_j^3)$ is either $0$ or $h_0^3g_{j-2}$.
  It follows that
  \[ d_4(h_j^3) = h_0^3g_{j-2} \neq 0. \]
  
  Using the injectivity from \Cref{lem:d3 diffls}(5) again we obtain the desired non-trivial classical Adams $d_4$-differential from the motivic one.
  \qedhere
  
  % Using Theorem~\ref{diff motivic sphere}, the element $h_j^3$ survives to the $E_4$-page of the motivic Adams spectral sequence, in particular we obtain a class $d_4(h_j^3)$ on the $E_4$-page of the motivic Adams spectral sequence. Now consider the map between motivic Adams spectral sequences induced by the quotient map $S^{0,0} \rightarrow S^{0,0}/\tau$, using Theorem~\ref{diff motivic ctau} we can read off that
  % $d_4(h_j^3)$ maps to $h_0^3g_{j-2}$ under this reduction map.
  % Now we apply \Cref{lem:zig-zag-d4} in order to conclude that the Betti realization of $d_4(h_j^3)$ is non-zero. As a consequence we learn that in the classical Adams spectral sequence we have $d_4(h_j^3) \neq 0$.

  % To complete the proof we now need to identify the target of this classical Adams $d_4$-differential.
  % By Theorem~\ref{diff classical}~$(2)$, we have $h_0^2$ divides $d_4(h_j^3)$. From the description of the classical $\Ext_{\A}$ in Theorem~\ref{E2 descriptions}~$(1)$, the only non-zero class that is divisible by $h_0^2$ is $h_0^3g_{j-2}$. In other words, $d_4(h_j^3) = h_0^3g_{j-2}$.
\end{proof}

\begin{rmk}
One may prove directly that the element $h_0^3g_{j-2}$ survives to the classical Adams $E_4$-page, but it is not logically necessary for the proof of \Cref{thm:motivic d4}---\Cref{thm:motivic d4} in fact implies that $h_0^3g_{j-2}$ does not support a nonzero $d_2$ or $d_3$-differential. %If , then it would contradict to conclusion of Theorem~\ref{thm: d4 hj3} that $h_0^3g_{j-2}$ is a permanent cycle and being nonzero on the $E_4$-page.
\end{rmk}

%% file: proof-chart.tex
\begin{sseqdata}[ name = 3charts1hj3, xscale=1.5, yscale=1.5, x range = {-4}{-3}, y range = {2}{7}, x tick step = 1, y tick step = 1, grid = crossword, Adams grading, lax degree]
  
  \class[circle, inner sep=0.5pt, "?", {font = \tiny}](-4,6)
  \class[circle, inner sep=0.5pt, "?", {font = \tiny}](-4,7)
  
  \class[fill, "g_{j-2}" {below} ](-4,4)
  \class[fill](-4,5) \structline
  \class[fill](-4,6) \structline
  \class[fill](-4,7) \structline
  
  \class[fill](-3,5) \structline(-4,4)
  
  \class[fill, "h_j^3" {below} ](-3,3)
  \class[fill](-3,4) \structline
  \class[fill](-3,5) \structline
  
\end{sseqdata}

% \begin{sseqdata}[ name = 3charts2hj3, xscale=1.5, yscale=1.5, x range = {-4}{-3}, y range = {2}{7}, x tick step = 1, y tick step = 1, grid = crossword, Adams grading, lax degree]

% \class[rectangle, inner sep=3pt](-4,6)
%  \class[rectangle, inner sep=3pt](-4,7)
 
%  \class[white ,circlen=2](-4,6)
%   \class[circle, inner sep=0.5pt, "4", {font = \tiny}](-4,6)
%   \class[circle, inner sep=0.5pt, "6", {font = \tiny}](-4,6)
    
%   \class[inner sep=0.5pt, "6", {font = \tiny},circlen=2](-4,7)
%   \class[circle, inner sep=0.5pt, "4", {font = \tiny}](-4,7)
%   \class[circle, inner sep=0.5pt, "6", {font = \tiny}](-4,7)

%   \class[rectangle, fill, inner sep=3pt, "g_{j-2}" {below} ](-4,4)
%   \class[rectangle, fill, inner sep=3pt](-4,5) \structline
%   \class[rectangle, fill, inner sep=3pt](-4,6) \structline
%   \class[rectangle, fill, inner sep=3pt](-4,7) \structline
  
%   \class[rectangle, fill, inner sep=3pt](-3,5) \structline(-4,4)
  
%   \class[circle, inner sep=0.5pt, "3", {font = \tiny}](-3,4)
%   \class[circle, inner sep=0.5pt, "4", {font = \tiny}](-4,5)
  
%   \class[inner sep=0.5pt, "5", {font = \tiny}, circlen=2](-3,5)
%   \class[circle, inner sep=0.5pt, "3", {font = \tiny}](-3,5)
%   \class[circle, inner sep=0.5pt, "5", {font = \tiny}](-3,5)

%    \class[rectangle, fill, inner sep=3pt, "h_j^3" {below} ](-3,3)
%   \class[rectangle, fill, inner sep=3pt](-3,4) \structline
%   \class[rectangle, fill, inner sep=3pt](-3,5) \structline
      
% \end{sseqdata}

\begin{sseqdata}[ name = 3charts2hj3, xscale=1.5, yscale=1.5, x range = {-4}{-3}, y range = {2}{7}, x tick step = 1, y tick step = 1, grid = crossword, Adams grading, lax degree]

  \class[rectangle, inner sep=1.5pt, "?", {font = \tiny}](-4,6)
  \class[rectangle, inner sep=1.5pt, "?", {font = \tiny}](-4,7)
       
  \class[rectangle, fill, inner sep=3pt, "g_{j-2}" {below} ](-4,4)
  \class[rectangle, fill, inner sep=3pt](-4,5) \structline
  \class[rectangle, fill, inner sep=3pt](-4,6) \structline
  \class[rectangle, fill, inner sep=3pt](-4,7) \structline
  
  \class[rectangle, fill, inner sep=3pt](-3,5) \structline(-4,4)
    
  \class[rectangle, fill, inner sep=3pt, "h_j^3" {below} ](-3,3)
  \class[rectangle, fill, inner sep=3pt](-3,4) \structline
  \class[rectangle, fill, inner sep=3pt](-3,5) \structline
  
\end{sseqdata}

\begin{sseqdata}[ name = 3charts3hj3, xscale=1.5, yscale=1.5, x range = {-4}{-3}, y range = {2}{7}, x tick step = 1, y tick step = 1, grid = crossword, Adams grading, lax degree]

  \class[circle, inner sep=0.5pt, "?", {font = \tiny}](-4,6)
  \class[circle, inner sep=0.5pt, "?", {font = \tiny}](-4,7)
  
  \class[fill, "g_{j-2}" {below} ](-4,4)
  \class[fill](-4,5) \structline
  \class[fill](-4,6) \structline
  \class[fill](-4,7) \structline
  
  \class[fill](-3,5) \structline(-4,4)

  \class[circle, inner sep=0.5pt, "5", {font = \tiny}](-3,6)
  
   \class[fill, "h_j^3" {below} ](-3,3)
  \class[fill](-3,4) \structline
  \class[fill](-3,5) \structline
  \class[fill, red](-3,6) \structline
    
\end{sseqdata}

%% file: algAHv2.tex
%% We use the Cartan-Eilenberg spectral sequence, which converges to the classical Adams $E_2$-page. We further study an algebraic Atiyah-Hirzebruch spectral sequence, which converges to the Cartan-Eilenberg $E_2$-page. We have the following diagram of spectral sequences.

Our goal in the next pair of sections is to prove \Cref{E2 descriptions}.
In this section we prove a weak form of \Cref{E2 descriptions}(3).
In fact, as explained in \Cref{tridegree} the $E_2$-page of the motivic Adams sseq for $S^{0,0}/\tau$ is isomorphic to the $E_2$-page of the $\textsl{classical}$ Cartan--Eilenberg sseq, therefore the main object of this section is the Cartan--Eilenberg $E_2$-page.
In order to gain access to this $E_2$-page we construct an algebraic Atiyah--Hirzebruch sseq, converging to the Cartan-Eilenberg $E_2$-page.
Then, in \Cref{sec:cess} we study the differentials in the Cartan--Eilenberg sseq and use this information to complete the proof of \Cref{E2 descriptions}.

In order to motivate our strategy,  let us focus on a particular claim from \Cref{E2 descriptions}(1):
the class $h_0^3g_{n+1}$ on the Adams $7$-line is non-trivial for $n \geq 3$.
On the $E_2$-page of the Cartan--Eilenberg sseq this class is detected by $q_0^3 \cdot g_n$ (see \Cref{ntn:frob names} below for the index shifting from $g_{n+1}$ to $g_n$) and therefore it suffices for us to do two things (1) show that $q_0^3 \cdot g_n$ is non-trivial on the Cartan--Eilenberg $E_2$-page and (2) show that $q_0^3 \cdot g_n$ is not the target of a Cartan--Eilenberg differential.
The class $q_0^3 \cdot q_n$ is detected by a non-trivial class of the same name on the algebraic Atiyah--Hirzebruch $E_1$-page and we prove (1) by showing that this class is not the target of an algebraic Atiyah--Hirzebruch differential. We depict this strategy in the diagram of spectral sequences below.

$$\xymatrix{
q_0^3 \cdot g_n \in & \Ext_{\P}^{*,*}(\F_2, \F_2) \otimes \F_2[q_0, q_1, \cdots] \ar@{=>}[d]^{\textup{algebraic Atiyah-Hirzebruch SS}} \\
q_0^3 \cdot g_n \in & \Ext^{*,*}_{\P}(\F_2, \F_2[q_0, q_1, \cdots]) \ar@{=>}[d]^{\textup{Cartan-Eilenberg SS}} \\
h_0^3 g_{n+1} \in & \Ext_{\A}^{*,*}(\F_2, \F_2)
}$$

\begin{ntn} \label{ntn:frob names}
  In order to give names to elements on the $E_2$-page we use the injective map
  \[ \Ext_{\P}^{*,*}(\F_2, \F_2) \to \Ext_{\P}^{*,*}(\F_2, \F_2[q_0,q_1,\dots]) \]
  together with the Frobenius isomorphism $\P \cong \A$ so that we can use the familiar names from $\Ext_{\A}$.
  Note that this use of the Frobenius in naming classes means that the class $a$ on Cartan--Eilenberg $E_2$-page will detect the class $\Sq^0(a)$ in $\Ext_{\A}$.
  In particular, this means that the class $h_j^3$ on the motivic Adams $E_2$-page corresponds to the class $h_{j-1}^3$ on the Cartan--Eilenberg $E_2$-page.
\end{ntn}

The following \Cref{prop:weak CE E2} is the main result of this section.

\begin{prop} \label{prop:weak CE E2}
  Let $j \geq 5$. 
  The $E_2$-page of the Cartan--Eilenberg spectral sequence
  takes the following form near $h_{j}^3$:
    \begin{center}{\renewcommand{\arraystretch}{1.2}
        \begin{tabular}{|lll|c|c|}\hline
          $(s,$ & $k,$ & $t)$ & $\Ext^{s,\,t}_{\P}(\mathbb{F}_2,\, \Ext_{\mathcal{Q}}^k(\F_2,\F_2))$ & $\textup{proof}$ \\\hline\hline
          $(4,$ & $1,$ & $ 3 \cdot 2^{j+1} + 1)$ & contains $\F_2\{q_0g_{j-2}\} $ & \Cref{lem:q03g-nonzero} \\\hline
          $(4,$ & $2,$ & $ 3 \cdot 2^{j+1} + 2)$ & contains $\F_2\{q_0^2g_{j-2}\} $ & \Cref{lem:q03g-nonzero} \\\hline
          $(4,$ & $3,$ & $ 3 \cdot 2^{j+1} + 3)$ & contains $\F_2\{q_0^3g_{j-2}\} $ & \Cref{lem:q03g-nonzero} \\\hline
          $(1,$ & $2,$ & $ 3 \cdot 2^{j+1})$ & $0$ & \Cref{lem:skt120} \\\hline
          $(1,$ & $3,$ & $ 3 \cdot 2^{j+1} + 1)$ & $0$ & \Cref{lem:skt131}\\\hline
          $(1,$ & $4,$ & $ 3 \cdot 2^{j+1} + 2)$ & $0$ & \Cref{lem:skt142} \\\hline
          $(1,$ & $5,$ & $ 3 \cdot 2^{j+1} + 3)$ & $0$ & \Cref{lem:skt153} \\\hline                    
          $(3,$ & $1,$ & $ 3 \cdot 2^{j+1} + 1)$ & $\F_2\{q_0h_j^3\}$ or $0$ & \Cref{lem:q01hj3} \\\hline
          $(3,$ & $2,$ & $ 3 \cdot 2^{j+1} + 2)$ & $\F_2\{q_0^2h_j^3\}$ or $0$ & \Cref{lem:q02hj3} \\\hline
          $(3,$ & $3,$ & $ 3 \cdot 2^{j+1} + 3)$ & $\F_2\{q_0^3h_j^3\}$ or $0$ & \Cref{lem:q03g-nonzero} \\\hline
      \end{tabular}}
    \end{center}
\end{prop}

\begin{proof}%[Proof (of \Cref{prop:weak CE E2}).]
This proposition follows from various lemmas in this subsection. Note that the proposition is stated for $j \geq 5$ around $t$-degree $3\cdot 2^{j+1}$, and the lemmas are stated for $n \geq 3$ around $t$-degree $3\cdot 2^{n+3}$ so they are compatible.
\end{proof}

As discussed in \Cref{sec:review} we have an isomorphism
$\Ext_{\mathcal{Q}}(\F_2,\F_2) \cong \F_2[q_0,q_1,\dots]$
where the polynomial generator $q_i$ has degree $(k,t) = (1,2^{i+1}-1)$.
The $\P$-comodule structure on this polynomial algebra is given by 
\begin{equation}
\psi: \F_2[q_0, q_1, \cdots] \to \P \otimes \F_2[q_0, q_1, \cdots] 
\end{equation}
$$ \psi(q_n) = \sum_{i=0}^n \xi_{n-i}^{2^{i+1}} \otimes q_{i}, \ \text{where} \ \xi_0 = 1.$$
(see \cite[Theorem~4.3.3]{RavenelGreenBook}).
For further discussion of the Cartan--Eilenberg sseq see \cite{AndrewsMiller, MillerSquare, RavenelGreenBook}.

\subsection{An algebraic Atiyah-Hirzebruch spectral sequence} \label{sec algAHSS}\hfill

In this subsection we construct an algebraic Atiyah--Hirzebruch spectral sequence which converges to the Cartan--Eilenberg $E_2$-page.
The algAH sseq is the main tool we use in our proof of \Cref{prop:weak CE E2}.

\begin{cnstr} \label{cnstr:algAH}
  The $\P$-comodule $\F_2[q_0,\dots]$ can be described as the free polynomial algebra on the $\P$-comodule $\F_2\{q_i\}_{i \geq 0}$. We place an increasing filtration on $\F_2\{q_i\}_{i \geq 0}$ where $q_i$ is in filtration $i$. Passing to polynomial algebras we obtain a filtered commutative algebra in $\P$-comodules whose underlying object is $\F_2[q_0,\dots]$. Associated to this filtration is a multiplicative spectral sequence computing $\Ext_{\P}(\F_2, \F_2[q_0,\dots])$ which we will call the algebraic Atiyah--Hirzebruch spectral sequence.
\end{cnstr}

\begin{lem} \label{lem:algAH E1}
  The associated graded of the filtration on $\F_2[q_0,q_1,\cdots]$ from \Cref{cnstr:algAH} is given by the graded $\P$-comodule algebra $\F_2[q_0,q_1, \cdots]$ with trivial $\P$-comodule structure where the monomial $q_{i_1}q_{i_2}\cdots q_{i_k}$ lives in filtration $i_1 + i_2 + \cdots + i_k$.
\end{lem}

\begin{proof}
  The associated graded of a polynomial algebra on a filtered $P$-comodule $M$ is the polynomial algebra on the associated graded of $M$. Therefore, it suffices for us to observe that since we used the cellular filtration on $\F_2\{q_i\}_{i \geq 0}$ the associated graded has trivial $P$-comodule structure.
\end{proof}

As a consequence of \Cref{lem:algAH E1}
the algAH sseq of \Cref{cnstr:algAH} has the form
\begin{align*}
  \Ext_{\P}(\F_2, \F_2) \otimes \F_2[q_0, q_1, \cdots] \cong E_1
  &\Longrightarrow \Ext_{\P}(\F_2, \F_2[q_0, q_1, \cdots]) \\
  d_r : E_r^{s,k,t,i} \to E_r^{s+1,k,t,i-r}
\end{align*}
where we give classes in $\Ext_{\P}^{s,t}(\F_2, \F_2)$ degree $(s,k,t,i) = (s,0,t,0)$
and $q_n$ has degree $(s,k,t,i) = (0, 1, 2^{n+1}-1, n)$. 

\begin{rmk}
  The multiplicative structure on the algAH sseq coming from its construction via a filtered commutative algebra includes compatibility of the product structure on the $E_1$-page with products in $\Ext_\P(\F_2, \F_2[q_0,q_1,\dots])$ and a Liebniz rule for differentials.
\end{rmk}

\begin{rmk}
  The $E_1$-page of the algAH sseq has a basis of elements of the form
  $q_{i_1}q_{i_2}\cdots q_{i_k} \cdot a$, where $a\in \Ext_{\P}^{*,*}(\F_2, \F_2)$.
  Furthermore, since the $k$-degree just records the number of $q$'s,
  algAH differentials preserve the number of $q$'s.
\end{rmk}

%% For the Ext groups $\Ext_P^{*,*}(\F_2, \F_2)$, since $P$ is a ``double up" of $A$, we have 
%% \begin{align*}
%%  \Ext_P^{s,2t}(\F_2, \F_2) & \cong \Ext_A^{s,t}(\F_2, \F_2),\\
%%  \Ext_P^{s,2t+1}(\F_2, \F_2) & = 0.
%%  \end{align*}
%% For an element  $a \in \Ext_A^{s,t}(\F_2, \F_2)$, we abuse the notation and use the same name $a$ to denote the element in $\Ext_P^{s,2t}(\F_2, \F_2)$ under this isomorphism, when there are no ambiguities in the context.

%% Finally, for an element in the $E_2$-page of CESS, we use the element in the $E_1$-page of algAHSS detecting it to denote it. So every element in the $E_2$-page of CESS has a name of the form
%% $$q_{i_1}q_{i_2}\cdots q_{i_k} \cdot a,$$
%% where $a \in \Ext_A^{s,t}(\F_2, \F_2) \cong \Ext_P^{s,2t}(\F_2, \F_2)$.

In general, the differentials in algAH sseq can be computed by embedding into the cobar complex that computes the $E_2$-page of Cartan--Eilenberg sseq. So in particular, the primary algAH differentials can be computed by using the $\P$-comodule structure map $\psi$.

\begin{exm}
  We have 
  $$\psi(q_1) = \xi_1^2 \otimes q_0 + 1 \otimes q_1.$$
  Since $\xi_1$ detects $h_0$ in $\Ext_{\A}^{1,1}(\F_2, \F_2)$,
  we have $\xi_1^2$ detects $h_0$ in  $\Ext_{\P}^{1,2}(\F_2, \F_2)$ (in the naming convention of \Cref{ntn:frob names}). This gives us algAH differentials
  $$d_1 (q_1\cdot a) =  q_0 \cdot h_0 a,$$
  for $a \in \Ext_{\P}^{*,*}(\F_2, \F_2)$.
\end{exm}

\begin{exm}
  We have
  $$\psi(q_2) = \xi_2^2 \otimes q_0 + \xi_1^4 \otimes q_1 + 1 \otimes q_0.$$
  Examining the top AH filtration terms we obtain
  $$d_1 (q_2 \cdot a) =  q_1 \cdot h_1 a.$$
  If $h_1 a = 0$ in $\Ext_{\P}^{*,*}(\F_2, \F_2)$, then we have $d_1 (q_2\cdot a) = 0$.
  In this case, we have 
  $$d_2 (q_2\cdot a) = q_0 \cdot \langle h_0, h_1, a \rangle.$$
  This is due to the fact that in cobar complex for $\Ext_{\A}^{*,*}(\F_2, \F_2)$, the nullhomotopy of $h_0h_1$ is $\xi_2$. Note that the indeterminacy of the set $q_0 \cdot \langle h_0, h_1, a \rangle$ is killed by $d_1$-differentials.
\end{exm}

\begin{exm}
  $$\psi(q_1^2) = \psi(q_1)^2 = \xi_1^4 \otimes q_0^2 + 1 \otimes q_1^2,$$
  therefore we have
  $$d_2 (q_1^2\cdot a) =  q_0^2 \cdot h_1 a.$$
\end{exm}

Generalizing these examples we have the following lemma.

\begin{lem} \label{lem AH diff}
We have the following differentials in $\textup{algAHSS}$ for $n \geq 0$.
\begin{enumerate}
\item $d_1 (q_{n+1} \cdot a) = q_n \cdot h_n a$.
\item $d_2 (q_{n+1}^2 \cdot a) = q_n^2 \cdot h_{n+1} a$.
\item $d_2 (q_{n+2} \cdot a) = q_{n} \cdot  \langle h_n, h_{n+1}, a \rangle$, if $h_{n+1}a = 0$.
\item $d_4 (q_{n+2}^2 \cdot a) = q_{n}^2 \cdot  \langle h_{n+1}, h_{n+2}, a \rangle$, if $h_{n+2}a = 0$.
\item $d_3 (q_{n+3} \cdot a) = q_{n} \cdot  \langle h_n, h_{n+1}, h_{n+2}, a \rangle$, if $h_{n+2}a = 0$ and $0\in\langle h_{n+1}, h_{n+2}, a \rangle$.
\item $d_6 (q_{n+3}^2 \cdot a) = q_{n}^2 \cdot  \langle h_{n+1}, h_{n+2}, h_{n+3}, a \rangle$, if $h_{n+3}a = 0$ and $0\in\langle h_{n+2}, h_{n+3}, a \rangle$.
\end{enumerate}
\end{lem}

\begin{proof}
The differentials in $(1)$ and $(2)$ follow from the comodule structure map $\psi$ modulo elements in lower AH filtration
$$\psi(q_{n+1}) \equiv \xi_1^{2^{n+1}} \otimes q_n + 1 \otimes q_{n+1},$$
$$\psi(q_{n+1}^2) = \psi(q_{n+1})^2 \equiv \xi_1^{2^{n+2}} \otimes q_n^2 + 1 \otimes q_{n+1}^2,$$
and the facts that $ \xi_1^{2^{n+1}}$ and $ \xi_1^{2^{n+2}}$ detects $h_n$ and $h_{n+1}$ in the cobar complex that computes $\Ext_{\P}^{*,*}(\F_2, \F_2)$.

The differentials in $(3)$ and $(4)$ follow from the comodule structure map $\psi$ modulo elements in lower AH filtration
$$\psi(q_{n+2}) \equiv \xi_2^{2^{n+1}} \otimes q_{n} + \xi_1^{2^{n+2}} \otimes q_{n+1} + 1 \otimes q_{n+2},$$
$$\psi(q_{n+2}^2) = \psi(q_{n+2})^2 \equiv  \xi_2^{2^{n+2}} \otimes q_{n}^2 + \xi_1^{2^{n+3}} \otimes q_{n+1}^2 + 1 \otimes q_{n+2}^2,$$
and the facts that $\xi_2^{2^{n+1}}$ and $ \xi_2^{2^{n+2}}$ detects the nullhomotopy of $h_nh_{n+1}$ and $h_{n+1}h_{n+2}$ in the cobar complex that computes $\Ext_{\P}^{*,*}(\F_2, \F_2)$.

Similarly, the elements $\xi_3^{2^{n+1}}$ and $ \xi_3^{2^{n+2}}$ detects the null homotopy of the Massey products $\langle h_n, h_{n+1}, h_{n+2} \rangle$ and  $\langle h_{n+1}, h_{n+2}, h_{n+3} \rangle$, and the differentials in $(5)$ and $(6)$ follow.

Note that indeterminacies of the expressions $q_{n} \cdot  \langle h_n, h_{n+1}, a \rangle$, $q_{n}^2 \cdot  \langle h_{n+1}, h_{n+2}, a \rangle$,\\  $q_{n} \cdot  \langle h_n, h_{n+1}, h_{n+2}, a \rangle$, and $q_{n}^2 \cdot  \langle h_{n+1}, h_{n+2}, h_{n+3}, a \rangle$ are killed by shorter differentials.
\end{proof}

\subsection{The algebraic Atiyah--Hirzebruch $E_1$-page}
\label{subsec q03gn algAHSS}\hfill

In this subsection we compute the $E_1$-page of the algAH sseq in the degrees we will need for proving \Cref{prop:weak CE E2}.

\begin{rec}
  Recall that we have the following generators in the cohomology of the Steenrod algebra
  \begin{align*}
    h_n & \in \Ext_{\A}^{1, 2^n}(\F_2, \F_2) \cong \Ext_{\P}^{1,2^{n+1}}(\F_2, \F_2) \cong \F_2, \\ 
    c_n & \in \Ext_{\A}^{3, 11\cdot 2^n}(\F_2, \F_2) \cong \Ext_{\P}^{3, 11 \cdot 2^{n+1}}(\F_2, \F_2) \cong \F_2, \\
    g_n & \in \Ext_{\A}^{4, 3\cdot 2^{n+2}}(\F_2, \F_2) \cong \Ext_{\P}^{4, 3 \cdot 2^{n+3}}(\F_2, \F_2) \cong \F_2.
  \end{align*}
\end{rec}
  
\begin{lem} \label{candidates algAHSS}
  In degree $(s,k,t) = (3,3,3 \cdot 2^{n+3}+3)$ the $E_1$-page of the algAH sseq consists of the following classes for $n \geq 3$,
  % For degree reasons, the following list consists of all elements in the $E_1$-page of algAHSS that could kill $q_0^3 \cdot g_n$ for $n \geq 4$.
\begin{multicols}{2}
\begin{enumerate}
\setcounter{enumi}{-1}
\item $q_3q_4q_6 \cdot c_0$ for $n=3$,
\item $q_0q_1q_n \cdot c_n$,
\item $q_0^3 \cdot h_{n+1}^2h_{n+3} = q_0^3 \cdot h_{n+2}^3$,
\item $q_0^2q_{n+1} \cdot h_0h_{n+1}h_{n+3}$,
\item $q_0^2q_{n+3} \cdot h_0h_{n+1}^2$,
\item $q_0q_{n+1}^2 \cdot h_0^2h_{n+3}$,
\item $q_0q_{n+1}q_{n+3} \cdot h_0^2h_{n+1}$,
\item $q_{n+1}^2q_{n+3} \cdot h_0^3$,
\item $q_0^2q_{n+2} \cdot h_0h_{n+2}^2$,
\item $q_0q_{n+2}^2 \cdot h_0^2h_{n+2}$,
\item $q_{n+2}^3 \cdot h_0^3$,
\item $q_0q_1q_{n+1} \cdot h_{n}^2h_{n+3}$,
\item $q_0q_{n}^2 \cdot h_1h_{n+1}h_{n+3}$,
\item $q_0q_nq_{n+1} \cdot h_1h_{n}h_{n+3}$,
\item $q_0q_{n+1}q_{n+3} \cdot h_1h_{n}^2$,
\item $q_1q_{n}^2 \cdot h_0h_{n+1}h_{n+3}$,
\item $q_1q_nq_{n+1} \cdot h_0h_{n}h_{n+3}$,
\item $q_1q_{n+1}q_{n+3} \cdot h_0h_{n}^2$,
\item $q_0q_{n+1}^2 \cdot h_1h_{n+2}^2$,
\item $q_0q_{n+2}^2 \cdot h_1h_{n+1}^2$,
\item $q_1q_{n+1}^2 \cdot h_0h_{n+2}^2$,
\item $q_1q_{n+2}^2 \cdot h_0h_{n+1}^2$.
\end{enumerate}
\end{multicols}
\end{lem}

\begin{proof}
  As discussed in subsection~\ref{sec algAHSS},
  the $E_1$-page of the algAH sseq in degree $(s,k,t) = (3,3,3 \cdot 2^{n+3}+3)$
  has a basis of elements of the form $q_iq_jq_k \cdot a$, where
  $$a \in \Ext_{\P}^{3, 2*}(\F_2, \F_2) \cong \Ext_{\A}^{3, *}(\F_2, \F_2).$$ 
  The classical Adams $3$-line $\Ext^{3,*}_{\A}$ is generated by the elements $c_{\ell}$ and $h_ah_bh_c$. So the candidates are of the following two forms
  \begin{itemize}
  \item $q_iq_jq_k \cdot c_{\ell}$,
  \item $q_iq_jq_k \cdot h_ah_bh_c$.
  \end{itemize}
We may also assume that $i \leq j \leq k$ and $a \leq b \leq c$.

%% The $t$-degree of $q_0^3 \cdot g_n$ is $3 + 3 \cdot 2^{n+3}$.

We start by considering classes of the form $q_iq_jq_k \cdot c_{\ell}$.
The class $q_iq_jq_k \cdot c_{\ell}$ has $t$-degree
$$ (2^{i+1} - 1) + (2^{j+1} - 1) + (2^{k+1} - 1) + 11 \cdot 2^{\ell+1} $$
from this we obtain the equation
$$3 + 3 \cdot 2^{n+3} = (2^{i+1} - 1) + (2^{j+1} - 1) + (2^{k+1} - 1) + 11 \cdot 2^{\ell+1},$$
which can be simplified to the equation
\begin{equation}
1 + 2 + 2^{n+2} + 2^{n+3} = 2^i + 2^j + 2^k + 2^l + 2^{l+1} + 2^{l+3}. \tag{$*$}
\end{equation}
Since we are considering $n \geq 3$, the right hand side of $(*)$ is 3 mod 32, and is at least 99.

When $\ell=0$, the equation $(*)$ becomes
$$2^{n+2} + 2^{n+3} = 2^i + 2^j + 2^k + 8,$$
so we must have $(i,j,k) = (3,4,6)$ and $n=3$. This is case $(0): q_3q_4q_6 \cdot c_0$ for $n=3$.
Not possible for $n \geq 4$.

When $\ell=1$, we must have $i=0$, then equation $(*)$ becomes
$$2^{n+2} + 2^{n+3} = 2^j + 2^k + 4 + 16.$$ 
Not possible for $n \geq 3$.

When $\ell=2$, we must have $i=0, j=1$, then equation $(*)$ becomes
$$2^{n+2} + 2^{n+3} = 2^k + 4 + 8 + 32.$$ 
Not possible for $n \geq 3$.

When $\ell \geq 3$, we must have $i=0, j=1$, then equation $(*)$ becomes
$$2^{n+2} + 2^{n+3} = 2^k + 2^{\ell} + 2^{\ell+1} + 2^{\ell+3}.$$ 
We must have $k=\ell=n$, which is at least 3. This is case $(1): q_0q_1q_n \cdot c_n$.

% This completes the discussion of elements of the form $q_iq_jq_k \cdot c_{\ell}$.

Next we consider classes of the form $q_iq_jq_k \cdot h_ah_bh_c$.
The class $q_iq_jq_k \cdot h_ah_bh_c$ has $t$-degree
$$(2^{i+1} - 1) + (2^{j+1} - 1) + (2^{k+1} - 1) + 2^{a+1} + 2^{b+1} +2^{c+1}.$$
from this we obtain the equation
$$3 + 3 \cdot 2^{n+3} = (2^{i+1} - 1) + (2^{j+1} - 1) + (2^{k+1} - 1) + 2^{a+1} + 2^{b+1} +2^{c+1},$$
which can be simplified to the equation
\begin{equation}
1 + 2 + 2^{n+2} + 2^{n+3} = 2^i + 2^j + 2^k + 2^a + 2^b + 2^c. \tag{$**$}
\end{equation}
Either two or three terms on the right hand side of $(**)$ contribute to $1+2$. So we only have the following 4 possibilities for the unordered set $\{2^i, \ 2^j, \ 2^k, \ 2^a, \ 2^b, \ 2^c\}$.
\begin{itemize}
\item $\{1, \ 1, \ 1, \ 2^{n+1}, \ 2^{n+1}, \ 2^{n+3}\}$,
\item $\{1, \ 1, \ 1, \ 2^{n+2}, \ 2^{n+2}, \ 2^{n+2}\}$,
\item $\{1, \ 2, \ 2^n, \ 2^{n}, \ 2^{n+1}, \ 2^{n+3}\}$,
\item $\{1, \ 2, \ 2^{n+1}, \ 2^{n+1}, \ 2^{n+2}, \ 2^{n+2}\}$.
\end{itemize}
For the set $\{1, \ 1, \ 1, \ 2^{n+1}, \ 2^{n+1}, \ 2^{n+3}\}$, we have candidates 
$$q_0^3 \cdot h_{n+1}^2h_{n+3}, \ q_0^2q_{n+1} \cdot h_0h_{n+1}h_{n+3}, \ q_0^2q_{n+3} \cdot h_0h_{n+1}^2,$$
$$q_0q_{n+1}^2 \cdot h_0^2h_{n+3} , \ q_0q_{n+1}q_{n+3} \cdot h_0^2h_{n+1}, \ q_{n+1}^2q_{n+3} \cdot h_0^3.$$
These are the cases $(2)-(7)$.

For the set $\{1, \ 1, \ 1, \ 2^{n+2}, \ 2^{n+2}, \ 2^{n+2}\}$, we have candidates
$$q_0^3 \cdot h_{n+2}^3, \ q_0^2q_{n+2} \cdot h_0h_{n+2}^2, \ q_0q_{n+2}^2 \cdot h_0^2h_{n+2}, \ q_{n+2}^3 \cdot h_0^3.$$
Note that we have a relation $h_{n+2}^3 = h_{n+1}^2h_{n+3}$ in $\Ext_{\A}^{3,*}$ and hence in $\Ext_{\P}^{3,2*}$. These are the cases $(2)$ and $(8)-(10)$.

For the set $\{1, \ 2, \ 2^n, \ 2^{n}, \ 2^{n+1}, \ 2^{n+3}\}$, we have candidates
$$q_0q_1q_{n+1} \cdot h_{n}^2h_{n+3}, \ q_0q_{n}^2 \cdot h_1h_{n+1}h_{n+3}, \ q_0q_nq_{n+1} \cdot h_1h_{n}h_{n+3}, \ q_0q_{n+1}q_{n+3} \cdot h_1h_{n}^2,$$
$$q_1q_{n}^2 \cdot h_0h_{n+1}h_{n+3}, \ q_1q_nq_{n+1} \cdot h_0h_{n}h_{n+3}, \ q_1q_{n+1}q_{n+3} \cdot h_0h_{n}^2.$$
Note that we have a relation $h_{n} h_{n+1}=0$ in $\Ext_{\A}^{2,*}$ and hence in $\Ext_{\P}^{2,2*}$. So certain candidates are already zero (e.g. $q_0q_1q_{n} \cdot h_{n}h_{n+1}h_{n+3}$).  These are the cases $(11)-(17)$.

For the set $\{1, \ 2, \ 2^{n+1}, \ 2^{n+1}, \ 2^{n+2}, \ 2^{n+2}\}$, we have candidates
$$q_0q_{n+1}^2 \cdot h_1h_{n+2}^2, \ q_0q_{n+2}^2 \cdot h_1h_{n+1}^2, \ q_1q_{n+1}^2 \cdot h_0h_{n+2}^2, \ q_1q_{n+2}^2 \cdot h_0h_{n+1}^2.$$
Again, due the relation $h_{n} h_{n+1}=0$, certain candidates are already zero so we don't list them. These are the cases $(18)-(21)$.

This completes the discussion of elements of the form $q_iq_jq_k \cdot h_ah_bh_c$ and therefore the proof of this lemma.
\end{proof}

\begin{lem} \label{candidates CESS}
  In degree $(s,k,t) = (1,5,3 \cdot 2^{n+3}+3)$ the $E_1$-page of the algAHSS consists of the following classes for $n \geq 3$,  
% For degree reasons, the following list consists of all elements in the $E_1$-page of algAHSS that, if survive in algAHSS, could support a nonzero differential that kills $q_0^3 \cdot g_n$ for $n \geq 4$ in CESS.
\begin{multicols}{2}
\begin{enumerate}
\item $q_0^3q_{n+2}q_{n+3} \cdot h_0$,
\item $q_0^4q_{n+3} \cdot h_{n+2}$,
\item $q_0^4q_{n+2} \cdot h_{n+3}$,
\item $q_0q_1q_{n+1}^2q_{n+3} \cdot h_0$,
\item $q_0^2q_{n+1}^2q_{n+3} \cdot h_1$,
\item $q_0^2q_1q_{n+1}q_{n+3} \cdot h_{n+1}$,
\item $q_0^2q_1q_{n+1}^2 \cdot h_{n+3}$,
\item $q_0q_1q_{n+2}^3 \cdot h_{0}$,
\item $q_0^2q_{n+2}^3 \cdot h_{1}$,
\item $q_0^2q_1q_{n+2}^2 \cdot h_{n+2}$,
\item $q_1q_n^2q_{n+1}q_{n+3} \cdot h_{1}$,
\item $q_1^2q_nq_{n+1}q_{n+3} \cdot h_{n}$,
\item $q_1^2q_n^2q_{n+3} \cdot h_{n+1}$,
\item $q_1^2q_n^2q_{n+1} \cdot h_{n+3}$,
\item $q_1q_{n+1}^2q_{n+2}^2 \cdot h_{1}$,
\item $q_1^2q_{n+1}q_{n+2}^2 \cdot h_{n+1}$,
\item $q_1^2q_{n+1}^2q_{n+2} \cdot h_{n+2}$,
\item $q_{n-1}^2q_nq_{n+1}q_{n+3} \cdot h_{2}$,
\item $q_{2}q_{n-1}q_nq_{n+1}q_{n+3} \cdot h_{n-1}$,
\item $q_{2}q_{n-1}^2q_{n+1}q_{n+3} \cdot h_{n}$,
\item $q_2q_{n-1}^2q_nq_{n+3} \cdot h_{n+1}$,
\item $q_2q_{n-1}^2q_nq_{n+1} \cdot h_{n+3}$,
\item $q_{n}^4q_{n+3} \cdot h_{2}$,
\item $q_2q_{n}^3q_{n+3} \cdot h_{n}$,
\item $q_2q_{n}^4 \cdot h_{n+3}$,
\item $q_{n}^2q_{n+1}q_{n+2}^2 \cdot h_{2}$,
\item $q_2q_nq_{n+1}q_{n+2}^2 \cdot h_{n}$,
\item $q_2q_{n}^2q_{n+2}^2 \cdot h_{n+1}$,
\item $q_2q_{n}^2q_{n+1}q_{n+2} \cdot h_{n+2}$,
\item $q_{n+1}^4q_{n+2} \cdot h_{2}$,
\item $q_2q_{n+1}^3q_{n+2} \cdot h_{n+1}$,
\item $q_2q_{n+1}^4 \cdot h_{n+2}$.
\end{enumerate}
\end{multicols}
For $n=3$, the cases $(18)$ and $(19)$ are identical, and for $n \geq 4$ all cases are different. 
\end{lem}

\begin{proof}
  The $E_1$-page of the algAH sseq in degree $(s,k,t) = (1,5,3 \cdot 2^{n+3}+3)$
  has a basis of elements of the form
  $$ q_{i_1}q_{i_2}q_{i_3}q_{i_4}q_{i_5} \cdot h_m $$
  where $i_1\leq i_2\leq i_3\leq i_4\leq i_5$.

  Considering the $t$-degrees, we have an equation
  $$3 + 3 \cdot 2^{n+3} = 2^{m+1} + (2^{i_1+1} - 1) + (2^{i_2+1} - 1) + (2^{i_3+1} - 1) + (2^{i_4+1} - 1) + (2^{i_5+1} - 1),$$
  which can be simplified to the equation
  $$ 4 + 2^{n+2} + 2^{n+3} = 2^m + 2^{i_1} + 2^{i_2} + 2^{i_3} + 2^{i_4} + 2^{i_5}. $$
  Up to four terms on the right hand side contribute to $4$. So we have the following 9 possibilities for the unordered set $\{2^m, \ 2^{i_1}, \ 2^{i_2}, \ 2^{i_3}, \ 2^{i_4}, \ 2^{i_5}\}$.
  \begin{multicols}{2}
    \begin{itemize}
    \item $\{1, \ 1, \ 1, \ 1, \ 2^{n+2}, \ 2^{n+3}\}$,
    \item $\{1, \ 1, \ 2, \ 2^{n+1}, \ 2^{n+1}, \ 2^{n+3}\}$,
    \item $\{1, \ 1, \ 2, \ 2^{n+2}, \ 2^{n+2}, \ 2^{n+2}\}$,
    \item $\{2, \ 2, \ 2^n, \ 2^{n}, \ 2^{n+1}, \ 2^{n+3}\}$,
    \item $\{2, \ 2, \ 2^{n+1}, \ 2^{n+1}, \ 2^{n+2}, \ 2^{n+2}\}$,
    \item $\{4, \ 2^{n-1}, \ 2^{n-1}, \ 2^{n}, \ 2^{n+1}, \ 2^{n+3}\}$,
    \item $\{4, \ 2^{n}, \ 2^{n}, \ 2^{n}, \ 2^{n}, \ 2^{n+3}\}$,
    \item $\{4, \ 2^{n}, \ 2^{n}, \ 2^{n+1}, \ 2^{n+2}, \ 2^{n+2}\}$,
    \item $\{4, \ 2^{n+1}, \ 2^{n+1}, \ 2^{n+1}, \ 2^{n+1}, \ 2^{n+2}\}$.
    \end{itemize}
  \end{multicols}
  The rest of proof works similarly to the proof of Lemma~\ref{candidates algAHSS}.
\end{proof}

\begin{rmk} \label{rmk:q0 remove}
  Note that since $q_0$ acts injectively on the algAH $E_1$-page we can read off
  the collection of elements on the $E_1$-page in degree
  $(s,k,t) = (3,2,3 \cdot 2^{n+3} + 2)$
  by extracting the subset of elements divisible by $q_0$ from \Cref{candidates algAHSS}, and elements in degree $(s,k,t) = (1,4,3 \cdot 2^{n+3} + 2)$ from \Cref{candidates CESS}, etc.
\end{rmk}

\subsection{algebraic Atiyah--Hirzebruch differentials}\hfill

In this subsection we complete the proof of \Cref{prop:weak CE E2} by computing the 
relevant degrees of the algAH sseq.
Our computation of the algAH differentials uses multiplicative and Massey product structures in $\Ext_{\P}$ together with descriptions of differentials in the algAH sseq from \Cref{lem AH diff}.

%% We prove Proposition~\ref{q03gn algAHSS} in the subsection that the element $q_0^3 \cdot g_n$ for $n \geq 4$ survives the algebraic Atiyah-Hirzebruch spectral sequence. 

%% First of all, it is clear that this element $q_0^3 \cdot g_n$ is a permanent cycle. In fact, since there is a splitting of algAHSS's, this element $q_0^3 \cdot g_n$ lives in the $E_1$-page of algAHSS 
%% $$\Ext_P^{*,*}(\F_2, \F_2) \otimes \bigoplus_{i,j,k}\F_2\{q_iq_jq_k\} \Longrightarrow \Ext_P^{*,*}(\F_2, \bigoplus_{i,j,k}\F_2\{q_iq_jq_k\}).$$ 
%% The inclusion of $P$-comodules $\F_2\{q_0^3\} \rightarrow \bigoplus_{i,j,k} \F_2\{q_iq_jq_k\}$ induces a map of $\Ext$ groups
%% $$\Ext_P^{*,*}(\F_2, \F_2\{q_0^3\}) \longrightarrow \Ext_P^{*,*}(\F_2, \bigoplus_{i,j,k}\F_2\{q_iq_jq_k\}),$$
%% making all elements $q_0^3 \cdot a$ in $\Ext_P^{*,*}(\F_2, \F_2)\{q_0^3\}$ permanent cycles in algAHSS.

%% So it remains to show that in algAHSS the element $q_0^3 \cdot g_n$ is not killed by other elements. The strategy of the proof is the following. We first use the constraints on degrees that were discussed in subsection~\ref{sec algAHSS} to list all possible candidates that could kill $q_0^3 \cdot g_n$. We then rule them out one by one, by proving that they either support or are the targets of shorter differentials. The proofs use multiplicative and Massey product structures in $\Ext_A^{*,*}(\F_2, \F_2) \cong \Ext_P^{*,2*}(\F_2, \F_2)$ and the descriptions of differentials in algAHSS in subsection~\ref{sec algAHSS}.

We begin with the following lemma which we will need in order to conclude that the targets of various algAH differentials are non-trivial.

\begin{lem} \label{fact on Ext4}
The following elements are nonzero in $\Ext_{\A}^{4,*}(\F_2, \F_2)$.
\begin{multicols}{2}
\begin{itemize}
\item $h_0c_n$, \ for \ $n \geq 2$,
\item $h_0^3 h_n$, \ for \ $n \geq 3$,
\item $h_0h_n^2h_{n+3}$, \ for \ $n \geq 3$,
\item $h_1h_n^2h_{n+3}$, \ for \ $n \geq 4$,
\item $h_1h_n^3$, \ for \ $n \geq 5$,
\item $h_0^2h_{n+1}h_{n+3}$, \ for \ $n \geq 1$,
\item $h_0^2h_nh_{n+3}$, \ for \ $n \geq 2$,
\item $h_0^2 h_n^2$, \ for \ $n \geq 4$,
\item $h_1c_n$, \ for \ $n \geq 3$,
\item $h_5c_0$.
\end{itemize}
\end{multicols}
\end{lem}

\begin{proof}
The complete description of $\Ext_{\A}^{\leq 4,*}(\F_2, \F_2)$ is given as Theorems~1.2 and 1.3 in \cite{LinExt}. In particular, all relations that are satisfied among the elements $h_n$ and $c_n$ up to the Adams filtration 4 are the following:
$$h_nh_{n+1} = 0, \ h_n h_{n+2}^2 = 0, \ h_n^2 h_{n+2} = h_{n+1}^3, \ h_n^2 h_{n+3}^2 = 0,$$
$$h_j c_n = 0 \ \text{for} \ j = n-1, n, n+2 \ \text{and} \ n+3.$$
One checks that this lemma is true.
\end{proof}

% (18): \ \ \ & d_2(q_0q_{n+2}^2 \cdot h_{1}h_{n+2})  = q_0q_{n+1}^2 \cdot h_1h_{n+2}^2. 

\begin{figure}[t]
\begin{center}
  {\renewcommand{\arraystretch}{1.2}
    \begin{tabular}{|c||c|c|c|}\hline
      & $s=2$ & $s=3$ & $s=4$ \\\hline\hline
    $3n+6$ & & $q_{n+2}^3 \cdot h_0^3$ \tikzmark{As10} & \\\hline
    $3n+5$ & & $q_{n+1}^2q_{n+3} \cdot h_0^3$ \tikzmark{As7} & \tikzmark{At10} $q_{n+1}q_{n+2}^2 \cdot h_0^3h_{n+1}$ \\\hline
    $3n+4$ & &  & \tikzmark{At7} $q_{n+1}^2q_{n+2} \cdot h_0^3h_{n+2}$ \\\hline
    $\cdots$ & & & \\\hline
    $2n+5$ & $q_1q_{n+1}q_{n+3} \cdot h_0h_{n+1}$ \tikzmark{As6} & $q_1q_{n+2}^2 \cdot h_0h_{n+1}^2$ \tikzmark{As21} & \\
    & $q_1q_{n+2}^2 \cdot h_0h_{n+2}$ \tikzmark{As9} & $q_1q_{n+1}q_{n+3} \cdot h_0h_{n}^2$ \tikzmark{As17a}\tikzmark{As17b}\tikzmark{As17c} & \\\hline
    $2n+4$ & & \tikzmark{At6} $q_0q_{n+1}q_{n+3} \cdot h_0^2h_{n+1}$ & \tikzmark{At21} $q_{0}q_{n+2}^2 \cdot h_0^2h_{n+1}^2$ \\
           & & \tikzmark{At9} $q_0q_{n+2}^2 \cdot h_0^2h_{n+2}$ & \tikzmark{At17b} $q_{1}q_{n}q_{n+3} \cdot h_0h_{n}^3$ \\
           & $ q_0q_{n+2}^2 \cdot h_{1}h_{n+2}$ \tikzmark{As18} & $q_0q_{n+1}q_{n+3} \cdot h_1h_{n}^2$ \tikzmark{As14a}\tikzmark{As14b} & \tikzmark{At17a} $q_{0}q_{n+1}q_{n+3} \cdot h_0^2h_{n}^2$ \\
    & & $ q_0q_{n+2}^2 \cdot h_1h_{n+1}^2$ \tikzmark{As19} & \tikzmark{At17c} $q_{1}q_{n+1}q_{n+2} \cdot h_0h_{n}^2h_{n+2}$ \\\hline
    $2n+3$ & $q_1q_{n+1}^2 \cdot h_0h_{n+3}$ \tikzmark{As5} & $q_1q_{n+1}^2 \cdot h_0h_{n+2}^2$ \tikzmark{As20} & \tikzmark{At14a} $q_{0}q_{n}q_{n+3} \cdot h_1h_n^3$ \\
           & & & \tikzmark{At14b} $q_{0}q_{n+1}q_{n+2} \cdot h_1h_n^2h_{n+2}$ \\\hline
    $2n+2$ & & \tikzmark{At18} $ q_0q_{n+1}^2 \cdot h_1h_{n+2}^2$ & \\
           & $ q_0q_{n+1}^2 \cdot h_1h_{n+3}$ \tikzmark{As12} & \tikzmark{At5} $q_0q_{n+1}^2 \cdot h_0^2h_{n+3}$ & \tikzmark{At20} $q_{0}q_{n+1}^2 \cdot h_0^2h_{n+2}^2$ \\
    & & $q_1q_{n}q_{n+1} \cdot h_0h_{n}h_{n+3}$ \tikzmark{As16a}\tikzmark{As16b} & \\\hline
    $2n+1$ & & $q_0q_nq_{n+1} \cdot h_1h_{n}h_{n+3}$ \tikzmark{As13} & \tikzmark{At16a} $q_{0}q_{n}q_{n+1} \cdot h_0^2h_{n}h_{n+3}$ \\
    & & $q_1q_{n}^2 \cdot h_0h_{n+1}h_{n+3}$ \tikzmark{As15} & \tikzmark{At16b} $q_1q_n^2 \cdot h_0h_n^2h_{n+3}$ \\\hline
    $2n$ & & \tikzmark{At12} $ q_0q_n^2 \cdot h_1h_{n+1}h_{n+3}$ & \tikzmark{At13} $q_{0}q_{n}^2 \cdot h_1h_n^2h_{n+3}$ \\
    & & & \tikzmark{At15} $q_{0}q_{n}^2 \cdot h_0^2h_{n+1}h_{n+3}$ \\\hline
    $2n-1$ & & & \\\hline
    $2n-2$ & & & \tikzmark{At19} $ q_0q_{n-1}^2 \cdot h_1c_n$ \\\hline
    $\cdots$ & & & \\\hline
    $n+4$ & $q_0q_1q_{n+3} \cdot h_{n+1}^2$ \tikzmark{As4} & & \\\hline
    $n+3$ & $q_0q_1q_{n+2} \cdot h_{n+2}^2$ \tikzmark{As8} & \tikzmark{At4} $q_0^2q_{n+3} \cdot h_0h_{n+1}^2$ & \\\hline
    $n+2$ & $q_0q_1q_{n+1} \cdot h_{n+1}h_{n+3}$ \tikzmark{As3} & \tikzmark{At8} $q_0^2q_{n+2} \cdot h_0h_{n+2}^2$ & \\
    & & $q_0q_1q_{n+1} \cdot h_{n}^2h_{n+3}$ \tikzmark{As11a}\tikzmark{As11b} & \\\hline
    $n+1$ & & $q_0q_1q_n \cdot c_n$ \tikzmark{As1} & \tikzmark{At11a} $q_{0}^2q_{n+1} \cdot h_0h_n^2h_{n+3}$ \\
    & & \tikzmark{At3} $q_0^2q_{n+1} \cdot h_0h_{n+1}h_{n+3}$ & \tikzmark{At11b} $q_0q_1q_n \cdot h_{n}^3h_{n+3}$ \\\hline
    $n$ & & & \tikzmark{At1} $q_0^2 q_n \cdot h_0c_n$ \\\hline
    $n-1$ & & & \\\hline
    $\cdots$ & & & \\\hline
    $0$ & & $q_0^3 \cdot h_{n+2}^3$ & $q_0^3 \cdot g_n$ \\\hline
  \end{tabular}
  \begin{tikzpicture}[overlay, remember picture, shorten >=.5pt, shorten <=.5pt, transform canvas={yshift=.25\baselineskip}]
    \draw [->] ({pic cs:As1}) [red] to ({pic cs:At1});
    % 2 is added
    \draw [->] ({pic cs:As3}) [red] to ({pic cs:At3});
    \draw [->] ({pic cs:As4}) [red] to ({pic cs:At4});
    \draw [->] ({pic cs:As5}) [red] to ({pic cs:At5});
    \draw [->] ({pic cs:As6}) [red] to ({pic cs:At6});
    \draw [->] ({pic cs:As7}) [red] to ({pic cs:At7});
    \draw [->] ({pic cs:As8}) [red] to ({pic cs:At8});
    \draw [->] ({pic cs:As9}) [red] to ({pic cs:At9});
    \draw [->] ({pic cs:As10}) [red] to ({pic cs:At10});
    \draw [->] ({pic cs:As11a}) [red] to ({pic cs:At11a});
    \draw [->] ({pic cs:As11b}) [red] to ({pic cs:At11b});
    \draw [->] ({pic cs:As12}) [blue] to ({pic cs:At12});
    \draw [->] ({pic cs:As13}) [red] to ({pic cs:At13});
    \draw [->] ({pic cs:As14a}) [bend right, red] to ({pic cs:At14a});
    \draw [->] ({pic cs:As14b}) [red] to ({pic cs:At14b});
    \draw [->] ({pic cs:As15}) [red] to ({pic cs:At15});
    \draw [->] ({pic cs:As16a}) [red] to ({pic cs:At16a});
    \draw [->] ({pic cs:As16b}) [red] to ({pic cs:At16b});
    \draw [->] ({pic cs:As17a}) [red] to ({pic cs:At17a});
    \draw [->] ({pic cs:As17b}) [red] to ({pic cs:At17b});    
    \draw [->] ({pic cs:As17c}) [red] to ({pic cs:At17c});
    \draw [->] ({pic cs:As18}) [blue] to ({pic cs:At18});
    \draw [->] ({pic cs:As19}) [bend right=7, green!70!black] to ({pic cs:At19});
    \draw [->] ({pic cs:As20}) [red] to ({pic cs:At20});
    \draw [->] ({pic cs:As21}) [red] to ({pic cs:At21});
  \end{tikzpicture}}
\end{center}
\caption{ The algebraic Atiyah--Hirzebruch spectral sequence in degree $(s,k,t) = (3,3, 3\cdot 2^{n+3}+3)$ for \Cref{lem:q03g-nonzero}. In this chart the vertical axis is the algebraic Atiyah--Hirzebruch filtration and the horizontal axis is the $s$-degree. $d_1$-differentials are red, $d_2$-differentials are blue and $d_6$-differentials are green.}
\label{fig:algA 1}
\end{figure}

\begin{lem} \label{lem:q03g-nonzero}%\hfill
For $n \geq 3$,
  \begin{itemize}
  \item $\Ext^{3,\ 3\cdot 2^{n+3} + 3}_{\P}(\F_2, \Ext_{\mathcal{Q}}^3(\F_2, \F_2))$ is either $\F_2\{q_0^3 \cdot h_{n+2}^3\}$ or $0$.
  \item $\Ext^{4,\ 3\cdot 2^{n+3} + 3}_{\P}(\F_2, \Ext_{\mathcal{Q}}^3(\F_2, \F_2))$ contains $q_0^3 \cdot g_{n} \neq 0$.
  \item $\Ext^{4,\ 3\cdot 2^{n+3} + 2}_{\P}(\F_2, \Ext_{\mathcal{Q}}^2(\F_2, \F_2))$ contains $q_0^2 \cdot g_{n} \neq 0$.
  \item $\Ext^{4,\ 3\cdot 2^{n+3} + 1}_{\P}(\F_2, \Ext_{\mathcal{Q}}^1(\F_2, \F_2))$ contains $q_0 \cdot g_{n} \neq 0$.
  \end{itemize}  
\end{lem}

\begin{proof}
  We begin by noting that $q_0^3 \cdot g_n$, $q_0^2 \cdot g_n$ and $q_0^1 \cdot g_n$ are permanent cycles in the algAH sseq (since they are each a product of permanent cycles).
  In order to prove the second bullet point, we will show that $q_0^3 \cdot g_n$ is not hit by an algAH differential.
  Note that if $q_0^3 \cdot g_n$ is non-zero on the $E_\infty$-page of the algAH sseq, then so are
  $q_0^2 \cdot g_n$ and $q_0 \cdot g_n$. 
  Therefore, the third and fourth bullet points will follow.
  We prove the first bullet point by showing that in the algAH sseq this degree has only $q_0^3 \cdot h_{n+2}^3$ by the $E_7$-page.

  %% In fact, since there is a splitting of algAHSS's, this element $q_0^3 \cdot g_n$ lives in the $E_1$-page of algAHSS 
%% $$\Ext_{\P}^{*,*}(\F_2, \F_2) \otimes \bigoplus_{i,j,k}\F_2\{q_iq_jq_k\} \Longrightarrow \Ext_{\P}^{*,*}(\F_2, \bigoplus_{i,j,k}\F_2\{q_iq_jq_k\}).$$ 
%% The inclusion of $P$-comodules $\F_2\{q_0^3\} \rightarrow \bigoplus_{i,j,k} \F_2\{q_iq_jq_k\}$ induces a map of $\Ext$ groups
%% $$\Ext_{\P}^{*,*}(\F_2, \F_2\{q_0^3\}) \longrightarrow \Ext_{\P}^{*,*}(\F_2, \bigoplus_{i,j,k}\F_2\{q_iq_jq_k\}),$$
%% making all elements $q_0^3 \cdot a$ in $\Ext_{\P}^{*,*}(\F_2, \F_2)\{q_0^3\}$ permanent cycles in algAHSS.

  In \Cref{candidates algAHSS} we determined the $E_1$-page of the algAH sseq in degree 
  $ (s, \ k, \ t)=(3, \ 3, \ 3\cdot 2^{n+3} + 3) $.
  What we must do now is show that 
  for all 21 elements (22 elements for $n=3$) in \Cref{candidates algAHSS}, other than case $(2)$, each element either supports or is killed by a short algAH differential, so none of them can kill $q_0^3 \cdot g_n$. %(In particular, we also check that in the cases that some of them supports nonzero Atiyah-Hirzebruch differential, the targets are all linearly independent, so they all disappear in the Cartan--Eilenberg $E_2$-page.)

It is clear that being an element in $\Ext_{\P}^{*,*}(\F_2, \F_2)\{q_0^3\}$, $q_0^3 \cdot h_{n+2}^3$ is a permanent cycle so it cannot kill $q_0^3 \cdot g_n$. This accounts for the case $(2)$.

For the case $(0)$, by \Cref{lem AH diff}(1) and \Cref{fact on Ext4}, we have a nonzero algAH $d_1$-differential.
$$(0): \ \ \ d_1(q_3q_4q_6\cdot c_0) = q_3q_4q_5 \cdot h_5c_0.$$

For $n \geq 4$, using \Cref{lem AH diff}(1) and the Leibniz rule, we obtain the algAH $d_1$-differentials displayed in \Cref{fig:algA 1}.

For the cases $(1), (7), (10), (11), (13)-(17), (20), (21)$, by \Cref{fact on Ext4}, the targets are all nonzero for $n \geq 4$. Note that in the case $(14)$ for $n=4$, the first term of the target $q_{0}q_{4}q_{7} \cdot h_1h_4^3$ is zero, while the second term $q_{0}q_{5}q_{6} \cdot h_1h_4^2h_{6}$ is nonzero. For $n \geq 5$, both terms are nonzero. 

Using \Cref{lem AH diff}(2) and the Liebniz rule, we have the following algAH $d_2$-differentials.
\begin{align*}
(12): \ \ \ & d_2(q_0q_{n+1}^2 \cdot h_1h_{n+3})  = q_0q_n^2 \cdot h_1h_{n+1}h_{n+3}, \\
(18): \ \ \ & d_2(q_0q_{n+2}^2 \cdot h_{1}h_{n+2})  = q_0q_{n+1}^2 \cdot h_1h_{n+2}^2. 
\end{align*}
One observes that the sources and targets of the above two $d_2$-differentials survives to the $E_2$-page of the algAH sseq. In fact, for the case $(18)$, we have a $d_1$-differential.
$$d_1(q_0q_{n+2}q_{n+3} \cdot h_1) = q_0q_{n+2}^2 \cdot h_1h_{n+2} + q_0q_{n+1}q_{n+3} \cdot h_1h_{n+1}.$$
So in the $E_2$-page, we have a relation that $q_0q_{n+2}^2 \cdot h_1h_{n+2} = q_0q_{n+1}q_{n+3} \cdot h_1h_{n+1}$. One may also use \Cref{lem AH diff}(3) to prove a $d_2$-differential 
\begin{align*}
(18): \ \ \ d_2(q_0q_{n+1}q_{n+3} \cdot h_1h_{n+1}) & = q_0q_{n+1}^2 \cdot h_1 \langle h_{n+1}, h_{n+2}, h_{n+1} \rangle \\
& = q_0q_{n+1}^2 \cdot h_1h_{n+2}^2,
\end{align*}
which is equivalent to the above one.

Using \Cref{lem AH diff}(6), we have the following algAH $d_6$-differential
\begin{align*}
(19): \ \ \ d_6(q_0q_{n+2}^2 \cdot h_1h_{n+1}^2) & = q_0q_{n-1}^2 \cdot h_1 \langle h_n, h_{n+1}, h_{n+2}, h_{n+1}^2 \rangle \\
& = q_0q_{n-1}^2 \cdot h_1c_n.
\end{align*}

One observes that all targets of the above algAH differentials are linearly independent.
Since most of the targets only have one term, this is not hard to check. This completes the proof.
\end{proof}

\begin{lem} \label{lem:q02hj3}
 For $n \geq 3$, $\Ext^{3,\ 3\cdot 2^{n+3} + 2}_{\P}(\F_2, \Ext_{\mathcal{Q}}^2(\F_2, \F_2))$ is either $\F_2\{q_0^2 \cdot h_{n+2}^3\}$ or $0$.
\end{lem}  

\begin{figure}[h]
\begin{center}
  {\renewcommand{\arraystretch}{1.2}
    \begin{tabular}{|c||c|c|c|}\hline
      & $s=2$ & $s=3$ & $s=4$ \\\hline\hline
      $2n+5$ & $q_{n+2}q_{n+3} \cdot h_0^2 $ \tikzmark{BsQ}\tikzmark{Bsa} & & \\\hline
      $2n+4$ & & $q_{n+1}q_{n+3} \cdot h_1h_{n}^2$ \tikzmark{Bs14a}\tikzmark{Bs14b} &  \\           
             & & \tikzmark{Bta} $q_{n+1}q_{n+3} \cdot h_0^2h_{n+1}$ \tikzmark{BsP} & \\
             & & \tikzmark{BtQ} $q_{n+2}^2 \cdot h_0^2h_{n+2}$ \tikzmark{Bsc} & \\
             & $ q_{n+2}^2 \cdot h_{1}h_{n+2}$ \tikzmark{Bs18} & $ q_{n+2}^2 \cdot h_1h_{n+1}^2$ \tikzmark{Bs19} & \\\hline
      $2n+3$ &  &  &  $q_{n}{\tikzmark{Bt14a}}q_{n+3} \cdot h_1h_n^3$ \\
             & & &  $q_{n+1}{\tikzmark{Bt14b}}q_{n+2} \cdot h_1h_n^2h_{n+2}$ \\\hline
      $2n+2$ & & \tikzmark{Bt18} $ q_{n+1}^2 \cdot h_1h_{n+2}^2$ & \\
           & $ q_{n+1}^2 \cdot h_1h_{n+3}$ \tikzmark{Bs12} & $q_{n+1}^2 \cdot h_0^2h_{n+3}$ \tikzmark{Bsb} & \tikzmark{Btc} \tikzmark{BtP} $q_{n+1}^2 \cdot h_0^2h_{n+2}^2$ \\\hline
    $2n+1$ & & $q_nq_{n+1} \cdot h_1h_{n}h_{n+3}$ \tikzmark{Bs13} & $q_{n}q_{n+1} \cdot h_0^2h_{n}h_{n+3}$ \\\hline
    $2n$ & & \tikzmark{Bt12} $ q_n^2 \cdot h_1h_{n+1}h_{n+3}$ & \tikzmark{Btb} $q_{n}^2 \cdot h_0^2h_{n+1}h_{n+3}$ \\
    & & & \tikzmark{Bt13} $q_{n}^2 \cdot h_1h_n^2h_{n+3}$ \\\hline
    $2n-1$ & & & \\\hline
    $2n-2$ & & & \tikzmark{Bt19} $ q_{n-1}^2 \cdot h_1c_n$ \\\hline
    $\cdots$ & & & \\\hline
    $n+4$ & $q_1q_{n+3} \cdot h_{n+1}^2$ \tikzmark{Bs4} & & \\\hline
    $n+3$ & $q_1q_{n+2} \cdot h_{n+2}^2$ \tikzmark{Bs8} & \tikzmark{Bt4} $q_0q_{n+3} \cdot h_0h_{n+1}^2$ & \\\hline
    $n+2$ & $q_1q_{n+1} \cdot  h_{n+1} \tikzmark{Bs3} h_{n+3}$  & \tikzmark{Bt8} $q_0q_{n+2} \cdot h_0h_{n+2}^2$ & \\
    & & $q_1q_{n+1} \cdot h_{n}^2h_{n+3}$ \tikzmark{Bs11a}\tikzmark{Bs11b} & \\\hline
    $n+1$ & & $q_1q_n \cdot c_n$ \tikzmark{Bs1} & \tikzmark{Bt11a} $q_{0}q_{n+1} \cdot h_0h_n^2h_{n+3}$ \\
    & & \tikzmark{Bt3} $q_0q_{n+1} \cdot h_0h_{n+1}h_{n+3}$ & \tikzmark{Bt11b} $q_1q_n \cdot h_{n}^3h_{n+3}$ \\\hline
    $n$ & & & \tikzmark{Bt1} $q_0 q_n \cdot h_0c_n$ \\\hline
    $n-1$ & & & \\\hline
    $\cdots$ & & & \\\hline
    $0$ & & $q_0^2 \cdot h_{n+2}^3$ & $q_0^2 \cdot g_n$ \\\hline
  \end{tabular}
  \begin{tikzpicture}[overlay, remember picture, shorten >=.5pt, shorten <=.5pt, transform canvas={yshift=.25\baselineskip}]
    \draw [->] ({pic cs:Bsa}) [red] to ({pic cs:Bta});
    \draw [->] ({pic cs:BsQ}) [red, bend right=15] to ({pic cs:BtQ});
    \draw [->] ({pic cs:BsP}) [blue, bend right=20] to ({pic cs:BtP});
    \draw [->] ({pic cs:Bs12}) [blue] to ({pic cs:Bt12});
    \draw [->] ({pic cs:Bsb}) [blue] to ({pic cs:Btb});
    \draw [->] ({pic cs:Bsc}) [blue, bend right=15] to ({pic cs:Btc});
    \draw [->] ({pic cs:Bs14a}) [red] to ({pic cs:Bt14a});
    \draw [->] ({pic cs:Bs14b}) [red, bend left=5] to ({pic cs:Bt14b});
    \draw [->] ({pic cs:Bs18}) [blue] to ({pic cs:Bt18});
    \draw [->] ({pic cs:Bs13}) [red] to ({pic cs:Bt13});
    \draw [->] ({pic cs:Bs19}) [bend right=7, green!70!black] to ({pic cs:Bt19});
    \draw [->] ({pic cs:Bs8}) [red] to ({pic cs:Bt8});
    \draw [->] ({pic cs:Bs4}) [red] to ({pic cs:Bt4});    
    \draw [->] ({pic cs:Bs1}) [red] to ({pic cs:Bt1});
    \draw [->] ({pic cs:Bs11a}) [red] to ({pic cs:Bt11a});
    \draw [->] ({pic cs:Bs11b}) [red] to ({pic cs:Bt11b});
    \draw [->] ({pic cs:Bs3}) [red, shorten <=0.23cm] to ({pic cs:Bt3});
  \end{tikzpicture}}
\end{center}
\caption{The algebraic Atiyah--Hirzebruch spectral sequence in degree $(s,k,t) = (3,2, 3\cdot 2^{n+3}+2)$ for \Cref{lem:q02hj3}. In this chart the vertical axis is the algAH filtration and the horizontal axis is the $s$-degree. $d_1$-differentials are red, $d_2$-differentials are blue and $d_6$-differentials are green.}
\label{fig:algA 3}
\end{figure}

\begin{proof}
  We consider all elements in the $E_1$-page of algAH sseq in degree 
  $ (s, \ k, \ t)=(3, \ 2, \ 3\cdot 2^{n+3} + 2) $.
  In \Cref{rmk:q0 remove} we described how each such element can be obtained by dividing one of the elements from \Cref{candidates algAHSS} by $q_0$.
  From this we obtain the following list of classes in degree $ (s, \ k, \ t)=(3, \ 2, \ 3\cdot 2^{n+3} + 2) $.
%%   Let $x$ be such an element, then $q_0 \cdot x$ is in tri-degree 
%%   $(s, \ k, \ t)=(3, \ 3, \ 3\cdot 2^j + 3)$.
%%   Let $n=j-2$, this is the same tri-degree as the one in Lemma~\ref{candidates algAHSS} (the candidates in algAHSS that could kill $q_0^3 \cdot g_n$ in CESS). 

%% Removing the ones not of the form $q_0 \cdot x$ in the statement of Lemma~\ref{candidates algAHSS}, we have the following candidates for $x$ (with the same index):
\begin{multicols}{2}
\begin{enumerate}
\item $q_1q_n \cdot c_n$,
\item $q_0^2 \cdot h_{n+1}^2h_{n+3} = q_0^2 \cdot h_{n+2}^3$,
\item $q_0q_{n+1} \cdot h_0h_{n+1}h_{n+3}$,
\item $q_0q_{n+3} \cdot h_0h_{n+1}^2$,
\item $q_{n+1}^2 \cdot h_0^2h_{n+3}$,
\item $q_{n+1}q_{n+3} \cdot h_0^2h_{n+1}$,
\stepcounter{enumi}%\item $q_{n+1}^2q_{n+3} \cdot h_0^3$,
\item $q_0q_{n+2} \cdot h_0h_{n+2}^2$,
\item $q_{n+2}^2 \cdot h_0^2h_{n+2}$,
\stepcounter{enumi}%\item $q_{n+2}^3 \cdot h_0^3$,
\item $q_1q_{n+1} \cdot h_{n}^2h_{n+3}$,
\item $q_{n}^2 \cdot h_1h_{n+1}h_{n+3}$,
\item $q_nq_{n+1} \cdot h_1h_{n}h_{n+3}$,
\item $q_{n+1}q_{n+3} \cdot h_1h_{n}^2$,
\stepcounter{enumi}%\item $q_1q_{n}^2 \cdot h_0h_{n+1}h_{n+3}$,
\stepcounter{enumi}%\item $q_1q_nq_{n+1} \cdot h_0h_{n}h_{n+3}$,
\stepcounter{enumi}%\item $q_1q_{n+1}q_{n+3} \cdot h_0h_{n}^2$,
\item $q_{n+1}^2 \cdot h_1h_{n+2}^2$,
\item $q_{n+2}^2 \cdot h_1h_{n+1}^2$.
\stepcounter{enumi}%\item $q_1q_{n+1}^2 \cdot h_0h_{n+2}^2$,
\stepcounter{enumi}%\item $q_1q_{n+2}^2 \cdot h_0h_{n+1}^2$.
\end{enumerate}
\end{multicols}
Other than the cases $(2)$, $(5)$, $(6)$ and $(9)$, the differentials in the proof of \Cref{lem:q03g-nonzero} are $q_0$-divisible, therefore these candidates do not survive the algAH sseq. 
For the cases $(6)$ and $(9)$, we have an algAH $d_1$-differential
\[ (6,9): \ \ \ d_1(q_{n+2}q_{n+3} \cdot h_0^2) = q_{n+1}q_{n+3} \cdot h_0^2h_{n+1} + q_{n+2}^2 \cdot h_0^2h_{n+2} \]
and by \Cref{lem AH diff}(2), we have the following algAH $d_2$-differential.
$$(9): \ \ \ d_2(q_{n+2}^2 \cdot h_0^2h_{n+2})  = q_{n+1}^2 \cdot h_0^2h_{n+2}^2.$$
For case $(5)$, we have the following algAH $d_2$-differential by \Cref{lem AH diff}(2),
\[ (5): \ \ \ d_2(q_{n+1}^2 \cdot h_0^2h_{n+3}) = q_{n}^2 \cdot h_0^2h_{n+1}h_{n+3}. \]
This completes the proof.
\end{proof}

\begin{lem} \label{lem:q01hj3}
  For $n\geq 3$, $\Ext^{3,\ 3\cdot 2^{n+3} + 1}_{\P}(\F_2, \Ext_{\mathcal{Q}}^1(\F_2, \F_2))$ is either $\F_2\{q_0 \cdot h_{n+2}^3\}$ or $0$.
\end{lem} 

\begin{figure}[h]
\begin{center}
  {\renewcommand{\arraystretch}{1.2}
    \begin{tabular}{|c||c|c|c|}\hline
      & $s=2$ & $s=3$ & $s=4$ \\\hline\hline
      $n+3$ & $q_{n+3} \cdot h_0h_{n+2}$ \tikzmark{Cs2} &  $q_{n+3} \cdot h_0h_{n+1}^2$ \tikzmark{Cs3} & \\\hline
      $n+2$ & $q_{n+2} \cdot h_0h_{n+3}$ \tikzmark{Cs1}  & \tikzmark{Ct2} $q_{n+2} \cdot h_0h_{n+2}^2$ & \\\hline
      $n+1$ & & \tikzmark{Ct1} $q_{n+1} \cdot h_0h_{n+1}h_{n+3}$ & \\\hline
      $n$ & & &  \tikzmark{Ct3} $ q_n \cdot h_0c_n$ \\\hline
      $\cdots$ & & & \\\hline
      $0$ & & $q_0 \cdot h_{n+2}^3$ & $q_0 \cdot g_n$ \\\hline
  \end{tabular}
  \begin{tikzpicture}[overlay, remember picture, shorten >=.5pt, shorten <=.5pt, transform canvas={yshift=.25\baselineskip}]
    \draw [->] ({pic cs:Cs1}) [red] to ({pic cs:Ct1});
    \draw [->] ({pic cs:Cs2}) [red] to ({pic cs:Ct2});
    \draw [->] ({pic cs:Cs3}) [green!70!black] to ({pic cs:Ct3});
    % \draw [->] ({pic cs:Cs12}) [blue] to ({pic cs:Ct12});
    % \draw [->] ({pic cs:Csb}) [blue] to ({pic cs:Ctb});
    % \draw [->] ({pic cs:Csc}) [blue, bend right=15] to ({pic cs:Ctc});
    % \draw [->] ({pic cs:Cs14a}) [bend right=15, red] to ({pic cs:Ct14a});
    % \draw [->] ({pic cs:Cs14b}) [red] to ({pic cs:Ct14b});
    % \draw [->] ({pic cs:Cs18}) [blue] to ({pic cs:Ct18});
    % \draw [->] ({pic cs:Cs13}) [red] to ({pic cs:Ct13});
    % \draw [->] ({pic cs:Cs19}) [bend right=7, green!70!black] to ({pic cs:Ct19});
    % \draw [->] ({pic cs:Cs8}) [red] to ({pic cs:Ct8});
    % \draw [->] ({pic cs:Cs4}) [red] to ({pic cs:Ct4});    
    % \draw [->] ({pic cs:Cs1}) [red] to ({pic cs:Ct1});
    % \draw [->] ({pic cs:Cs11a}) [red] to ({pic cs:Ct11a});
    % \draw [->] ({pic cs:Cs11b}) [red] to ({pic cs:Ct11b});
    % \draw [->] ({pic cs:Cs3}) [red, shorten <=0.23cm] to ({pic cs:Ct3});
  \end{tikzpicture}}
\end{center}
\caption{The algebraic Atiyah--Hirzebruch spectral sequence in degree $(s,k,t) = (3,1, 3\cdot 2^{n+3}+1)$ for \Cref{lem:q01hj3}. In this chart the vertical axis is the algAH filtration and the horizontal axis is the $s$-degree. $d_1$-differentials are red and $d_3$-differentials are green.}
\label{fig:algA 3b}
\end{figure}

\begin{proof}
  We consider all elements in the $E_1$-page of algAH sseq in degree 
  $ (s, \ k, \ t)=(3, \ 1, \ 3\cdot 2^{n+3} + 2) $.
  As in \Cref{rmk:q0 remove} we can determine all such elements by beginning with the list from \Cref{candidates algAHSS} and dividing by $q_0^2$.
    
%% We consider all elements in the $E_1$-page of algAHSS that converges to the $E_2$-page of CESS in the tridegree 
%% $$(s, \ k, \ t)=(3, \ 1, \ 3\cdot 2^j + 1).$$ 
%% Call any of such an element $y$, then $q_0^2 \cdot y$ is in tri-degree 
%% $$(s, \ k, \ t)=(3, \ 3, \ 3\cdot 2^j + 3).$$
%% Let $n=j-2$, this is the same tri-degree as the one in Lemma~\ref{candidates algAHSS} (the candidates in algAHSS that could kill $q_0^3 \cdot g_n$ in CESS). 

%% Removing the ones not of the form $q_0^2 \cdot y$ in the statement of Lemma~\ref{candidates algAHSS}, we have the following candidates for $y$ (with the same index):
\begin{multicols}{2}
\begin{enumerate}
\stepcounter{enumi}%\item $q_1q_n \cdot c_n$,
\item $q_0 \cdot h_{n+1}^2h_{n+3} = q_0 \cdot h_{n+2}^3$,
\item $q_{n+1} \cdot h_0h_{n+1}h_{n+3}$,
\item $q_{n+3} \cdot h_0h_{n+1}^2$,
\stepcounter{enumi}%\item $q_{n+1}^2 \cdot h_0^2h_{n+3}$,
\stepcounter{enumi}%\item $q_{n+1}q_{n+3} \cdot h_0^2h_{n+1}$,
\stepcounter{enumi}%\item $q_{n+1}^2q_{n+3} \cdot h_0^3$,
\item $q_{n+2} \cdot h_0h_{n+2}^2$,
\stepcounter{enumi}%\item $q_{n+2}^2 \cdot h_0^2h_{n+2}$,
\stepcounter{enumi}%\item $q_{n+2}^3 \cdot h_0^3$,
\stepcounter{enumi}%\item $q_1q_{n+1} \cdot h_{n}^2h_{n+3}$,
\stepcounter{enumi}%\item $q_{n}^2 \cdot h_1h_{n+1}h_{n+3}$,
\stepcounter{enumi}%\item $q_nq_{n+1} \cdot h_1h_{n}h_{n+3}$,
\stepcounter{enumi}%\item $q_{n+1}q_{n+3} \cdot h_1h_{n}^2$,
\stepcounter{enumi}%\item $q_1q_{n}^2 \cdot h_0h_{n+1}h_{n+3}$,
\stepcounter{enumi}%\item $q_1q_nq_{n+1} \cdot h_0h_{n}h_{n+3}$,
\stepcounter{enumi}%\item $q_1q_{n+1}q_{n+3} \cdot h_0h_{n}^2$,
\stepcounter{enumi}%\item $q_{n+1}^2 \cdot h_1h_{n+2}^2$,
\stepcounter{enumi}%\item $q_{n+2}^2 \cdot h_1h_{n+1}^2$.
\stepcounter{enumi}%\item $q_1q_{n+1}^2 \cdot h_0h_{n+2}^2$,
\stepcounter{enumi}%\item $q_1q_{n+2}^2 \cdot h_0h_{n+1}^2$.
\end{enumerate}
\end{multicols}
For the cases $(3)$ and $(8)$, we have the following Atiyah-Hirzebruch $d_1$-differentials:
 \begin{align*}
(3): \ \ \ & d_1(q_{n+2} \cdot h_0h_{n+3})  = q_{n+1} \cdot h_0h_{n+1}h_{n+3}, \\
(8): \ \ \ & d_1(q_{n+3} \cdot h_0h_{n+2})  = q_{n+2} \cdot h_0h_{n+2}^2.
\end{align*}
For the case $(4)$, by \Cref{lem AH diff}(5), we have the following algAH $d_3$-differential
\begin{align*}
(4): \ \ \ d_3(q_{n+3} \cdot h_0h_{n+1}^2) & = q_{n} \cdot h_0 \langle h_n, h_{n+1}, h_{n+2}, h_{n+1}^2 \rangle \\
& = q_n \cdot h_0c_n.
\end{align*}
This completes the proof.
\end{proof}

\begin{lem}  \label{lem:skt153}
For $n\geq 3$,  $\Ext^{1,\ 3\cdot 2^{n+3} + 3}_{\P}(\F_2, \Ext_{\mathcal{Q}}^5(\F_2, \F_2)) \cong 0$.
\end{lem}

\begin{figure}[h]
\begin{center}
  \scalebox{0.75}{\renewcommand{\arraystretch}{1.1}
    \begin{tabular}{|c||c|c|c|}\hline
      & $s=0$ & $s=1$ & $s=2$ \\\hline\hline
      $5n+6$ & &  $q_{n+1}^4q_{n+2} \cdot h_{2}$ \tikzmark{s30} & \\\hline
      $5n+5$ & & $q_{n}^2q_{n+1}q_{n+2}^2 \cdot h_{2}$ \tikzmark{s26} & \tikzmark{t30} $q_{n+1}^5 \cdot h_{2}h_{n+1}$ \\\hline
      $5n+4$ & & & \tikzmark{t26} $q_n^3q_{n+2}^2 \cdot h_{2}h_{n}$ \\\hline
      $5n+3$ & & $q_{n}^4q_{n+3} \cdot h_{2}$ \tikzmark{s23} & \\\hline
      $5n+2$ & & $q_{n-1}^2q_nq_{n+1}q_{n+3} \cdot h_{2}$ \tikzmark{s18} & \tikzmark{t23} $q_n^4q_{n+2} \cdot h_{2}h_{n+2}$ \\\hline
      $5n+1$ & & & \tikzmark{t18} $q_{n-1}^3q_{n+1}q_{n+3} \cdot h_{2}h_{n-1} + q_{n-1}^2q_n^2q_{n+3} \cdot h_2h_{n}$ \\
             & & & $ + q_{n-1}^2q_nq_{n+1}q_{n+2} \cdot h_2h_{n+2}$ \\\hline        
      $\cdots$ & & & \\\hline
      $4n+8$ & $q_2q_{n+1}^2q_{n+2}^2 \cdot 1$ \tikzmark{s15} & & \\\hline
      $4n+7$ & & \tikzmark{t15} $q_1q_{n+1}^2q_{n+2}^2 \cdot h_1$ & \\
             & & $q_2q_{n+1}^3q_{n+2} \cdot h_{n+1}$ \tikzmark{s31} & \\
             & & $q_2q_nq_{n+1}q_{n+2}^2 \cdot h_{n}$ \tikzmark{s27} & \\\hline
      $4n+6$ & & $q_2q_{n+1}^4 \cdot h_{n+2}$ \tikzmark{s32} & \tikzmark{t31} $q_1q_{n+1}^3q_{n+2} \cdot h_{1}h_{n+1} + q_2q_{n+1}^4 \cdot h_{n+1}^2$ \\
             & $q_2q_n^2q_{n+1}q_{n+3} \cdot 1$ \tikzmark{sg2a}\tikzmark{sg2b}\tikzmark{sg2c} & $q_2q_{n}^2q_{n+2}^2 \cdot h_{n+1}$ \tikzmark{s28} & \tikzmark{t27} $q_1q_nq_{n+1}q_{n+2}^2 \cdot h_{1}h_{n} + q_2q_{n}^2q_{n+2}^2 \cdot h_{n}^2$ \\\hline
      $4n+5$ & & \tikzmark{tg2a} $q_1q_n^2q_{n+1}q_{n+3} \cdot h_{1}$ \tikzmark{s11a}\tikzmark{s11b} & \tikzmark{t32} $q_1q_{n+1}^4 \cdot h_{1}h_{n+2}$  \\
             & & \tikzmark{tg2b} $q_2q_n^3q_{n+3} \cdot h_{n}$ \tikzmark{s24a}\tikzmark{s24b} & \tikzmark{t28} $q_1q_n^2q_{n+2}^2 \cdot h_{1}h_{n+1}$ \\
           & & \tikzmark{tg2c} $q_2q_n^2q_{n+1}q_{n+2} \cdot h_{n+2}$ \tikzmark{s29a}\tikzmark{s29b} & \\
             & & $q_2q_{n-1}q_nq_{n+1}q_{n+3} \tikzmark{s19} \cdot h_{n-1}$  & \\\hline
      $4n+4$ & & & \tikzmark{t11a}\tikzmark{t24a} $q_1q_n^3q_{n+3} \cdot h_1h_{n}$ \\
             & & & \tikzmark{t11b}\tikzmark{t29a} $q_1q_n^2q_{n+1}q_{n+2} \cdot h_1h_{n+2}$ \\
             & & & \tikzmark{t24b}\tikzmark{t29b} $q_2q_n^3q_{n+2} \cdot h_nh_{n+2}$ \\
             & & & \tikzmark{t19} $q_1q_{n-1}q_nq_{n+1}q_{n+3} \cdot h_{1}h_{n-1} + q_2q_{n-1}^2q_{n+1}q_{n+3} \cdot h_{n-1}^2$ \\
             & & $q_2q_{n-1}^2q_{n+1}q_{n+3} \cdot h_{n}$ \tikzmark{s20} & $+ q_2q_{n-1}q_nq_{n+1}q_{n+2} \cdot h_{n-1}h_{n+2}$ \\\hline
      $4n+3$ & & $q_2q_{n-1}^2q_nq_{n+3} \cdot h_{n+1}$ \tikzmark{s21} & \tikzmark{t20} $q_1q_{n-1}^2q_{n+1}q_{n+3} \cdot h_{1}h_{n} + q_2q_{n-1}^2q_{n+1}q_{n+2} \cdot h_{n}h_{n+2}$ \\\hline
      $4n+2$ & & $q_2q_{n}^4 \cdot h_{n+3}$ \tikzmark{s25} & \tikzmark{t21} $q_1q_{n-1}^2q_nq_{n+3} \cdot h_{1}h_{n+1} + q_2q_{n-1}^3q_{n+3} \cdot h_{n-1}h_{n+1}$ \\\hline
      $4n+1$ & & $q_2q_{n-1}^2q_nq_{n+1} \cdot h_{n+3}$ \tikzmark{s22} & \tikzmark{t25} $q_1q_n^4 \cdot h_{1}h_{n+3}$ \\\hline
      $4n$ & & & \tikzmark{t22} $q_1q_{n-1}^2q_nq_{n+1} \cdot h_{1}h_{n+3} + q_2q_{n-1}^3q_{n+1} \cdot h_{n-1}h_{n+3} $ \\
             & & & $+ q_2q_{n-1}^2q_n^2 \cdot h_{n}h_{n+3}$ \\\hline
    \end{tabular}
    \begin{tikzpicture}[overlay, remember picture, shorten >=.5pt, shorten <=.5pt, transform canvas={yshift=.25\baselineskip}]
      \draw [->] ({pic cs:s30}) [red] to ({pic cs:t30});
      \draw [->] ({pic cs:s26}) [red] to ({pic cs:t26});
      \draw [->] ({pic cs:s23}) [red] to ({pic cs:t23});
      \draw [->] ({pic cs:s18}) [red] to ({pic cs:t18});

      \draw [->] ({pic cs:s15}) [red] to ({pic cs:t15});
      \draw [->] ({pic cs:s31}) [red] to ({pic cs:t31});
      \draw [->] ({pic cs:s27}) [red] to ({pic cs:t27});
      \draw [->] ({pic cs:s28}) [red, bend right=10] to ({pic cs:t28});
      \draw [->] ({pic cs:sg2a})[red] to ({pic cs:tg2a});
      \draw [->] ({pic cs:sg2b})[red] to ({pic cs:tg2b});
      \draw [->] ({pic cs:sg2c})[red] to ({pic cs:tg2c});
      \draw [->] ({pic cs:s11a})[red] to ({pic cs:t11a});
      \draw [->] ({pic cs:s11b})[red] to ({pic cs:t11b});
      \draw [->] ({pic cs:s24a})[red] to ({pic cs:t24a});
      \draw [->] ({pic cs:s24b})[red] to ({pic cs:t24b});
      \draw [->] ({pic cs:s29a})[red, bend right=15] to ({pic cs:t29a});
      \draw [->] ({pic cs:s29b})[red, bend right=15] to ({pic cs:t29b});
      \draw [->] ({pic cs:s19})[red, bend right=20, shorten <=0.2cm] to ({pic cs:t19});
      \draw [->] ({pic cs:s20})[red] to ({pic cs:t20});
      \draw [->] ({pic cs:s21})[red] to ({pic cs:t21});
      \draw [->] ({pic cs:s25})[red, bend right=10] to ({pic cs:t25});
      \draw [->] ({pic cs:s22})[red] to ({pic cs:t22});
      \draw [->] ({pic cs:s32})[red, bend right=15] to ({pic cs:t32});
  \end{tikzpicture}}
\end{center}
\caption{The algebraic Atiyah--Hirzebruch spectral sequence in degree $(s,k,t) = (1,5, 3\cdot 2^{n+3}+3)$ and algAH filtration $\geq 4n$ for \Cref{lem:skt153}. In this chart the vertical axis is the algAH filtration and the horizontal axis is the $s$-degree. $d_1$-differentials are red.}
\label{fig:algA 2a}
\end{figure}

\begin{figure}[h]
\begin{center}
  \scalebox{0.9}{\renewcommand{\arraystretch}{1.1}
    \begin{tabular}{|c||c|c|c|}\hline
      & $s=0$ & $s=1$ & $s=2$ \\\hline\hline
      $3n+8$ & $q_1^2q_{n+2}^3 \cdot 1$ \tikzmark{s16} & & \\\hline
      $3n+7$ & $q_1^2q_{n+1}^2q_{n+3} \cdot 1$ \tikzmark{s17} & \tikzmark{t16} $q_1^2q_{n+1}q_{n+2}^2 \cdot h_{n+1}$ & \\
             & & $q_0q_1q_{n+2}^3 \cdot h_{0}$ \tikzmark{s8} & \\\hline
      $3n+6$ & & \tikzmark{t17} $q_1^2q_{n+1}q_{n+2}^2 \cdot h_{n+1}$ & \tikzmark{t8} $q_0^2q_{n+2}^3 \cdot h_0^2 + q_0q_1q_{n+1}q_{n+2}^2\cdot h_0h_{n+1}$ \\

             & & $q_0q_1q_{n+1}^2q_{n+3} \cdot h_{0}$ \tikzmark{s4} & \\
             & & $q_1^2q_nq_{n+1}q_{n+3} \cdot h_{n}$ \tikzmark{s12} & \\
             & & $q_0^2q_{n+2}^3 \cdot h_{1}$ \tikzmark{s9} & \\\hline
      $3n+5$ & & & \tikzmark{t4} $ q_0^2q_{n+1}^2q_{n+3} \cdot h_{0}^2 + q_0q_1q_{n+1}^2q_{n+2} \cdot h_0h_{n+2}$ \\                   
             & & $q_0^2q_{n+1}^2q_{n+3} \cdot h_{1}$ \tikzmark{s5} & \tikzmark{t12} $q_1^2q_{n}^2q_{n+3} \cdot h_{n}^2 + q_1^2q_nq_{n+1}q_{n+2} \cdot h_nh_{n+2}$ \\
             & & $q_1^2q_{n}^2q_{n+3} \cdot h_{n+1}$ \tikzmark{s13} & \tikzmark{t9} $q_0^2q_{n+1}q_{n+2}^2 \cdot h_{1}h_{n+1}$ \\\hline
      $3n+4$ & & & \tikzmark{t5} $q_0^2q_{n+1}^2q_{n+2} \cdot h_{1}h_{n+2}$ \\\hline
      $3n+3$ & & $q_1^2q_{n}^2q_{n+1} \cdot h_{n+3}$ \tikzmark{s14} & \tikzmark{t13} $q_0^2q_n^2q_{n+3}\cdot h_1h_{n+1} + q_1^2q_n^2q_{n+1} \cdot h_{n+2}^2$ \\\hline
      $3n+2$ & & & \tikzmark{t14} $q_1^2q_{n}^3 \cdot h_{n}h_{n+3}$ \\\hline
      $\cdots$ & & & \\\hline
      $2n+6$ & $q_0^2q_1q_{n+2}q_{n+3} \cdot 1$ \tikzmark{sg1a}\tikzmark{sg1b}\tikzmark{sg1c} & & \\\hline
      $2n+5$ & &  $q_0^2q_1 \tikzmark{tg1a} q_{n+1}q_{n+3} \cdot h_{n+1}$ \tikzmark{s1} & \\
             & & \tikzmark{tg1b} $q_0^2q_1q_{n+2}^2 \cdot h_{n+2}$ \tikzmark{s6}& \\
             & & \tikzmark{tg1c} $q_0^3q_{n+2}q_{n+3} \cdot h_0$ \tikzmark{s10a}\tikzmark{s10b} & \\\hline
      $2n+4$ & & & \tikzmark{t1}\tikzmark{t10a} $q_0^3q_{n+1}q_{n+3} \cdot h_0h_{n+1}$ \\
             & & & \tikzmark{t6}\tikzmark{t10b} $q_0^3q_{n+2}^2 \cdot h_0h_{n+2}$ \\\hline
      $2n+3$ & & $q_0^2q_1q_{n+1}^2 \cdot h_{n+3}$ \tikzmark{s7} & \\\hline
      $2n+2$ & & & \tikzmark{t7} $q_0^3q_{n+1}^2 \cdot h_0h_{n+3}$ \\\hline
      $\cdots$ & & & \\\hline
      $n + 3$ & & $q_0^4q_{n+3} \cdot h_{n+2}$ \tikzmark{s2} & \\\hline
      $n + 2$ & & $q_0^4q_{n+2} \cdot h_{n+3}$ \tikzmark{s3} & \tikzmark{t2} $q_0^4q_{n+2} \cdot h_{n+2}^2$ \\\hline
      $n + 1$ & & & \tikzmark{t3} $ q_0^4q_{n+1} \cdot h_{n+1}h_{n+3}$ \\\hline
    \end{tabular}
    \begin{tikzpicture}[overlay, remember picture, shorten >=.5pt, shorten <=.5pt, transform canvas={yshift=.25\baselineskip}]
      \draw [->] ({pic cs:s16}) [red] to ({pic cs:t16});
      \draw [->] ({pic cs:s17}) [red] to ({pic cs:t17});
      \draw [->] ({pic cs:s8}) [red] to ({pic cs:t8});
      \draw [->] ({pic cs:s4}) [red] to ({pic cs:t4});
      \draw [->] ({pic cs:s9}) [red, bend right=30] to ({pic cs:t9});
      \draw [->] ({pic cs:s5}) [red, bend right=20] to ({pic cs:t5});
      \draw [->] ({pic cs:s13}) [blue] to ({pic cs:t13});
      \draw [->] ({pic cs:s14}) [red, bend right=10] to ({pic cs:t14});
      \draw [->] ({pic cs:s12}) [red] to ({pic cs:t12});

      \draw [->] ({pic cs:sg1a}) [red, bend left, shorten >=0.2cm] to ({pic cs:tg1a});
      \draw [->] ({pic cs:sg1b}) [red, bend right] to ({pic cs:tg1b});
      \draw [->] ({pic cs:sg1c}) [red, bend right] to ({pic cs:tg1c});
      \draw [->] ({pic cs:s1}) [red] to ({pic cs:t1});
      \draw [->] ({pic cs:s6}) [red, bend right=20] to ({pic cs:t6});
      \draw [->] ({pic cs:s10a}) [red] to ({pic cs:t10a});
      \draw [->] ({pic cs:s10b}) [red, bend right=20] to ({pic cs:t10b});
      \draw [->] ({pic cs:s7}) [red] to ({pic cs:t7});
      \draw [->] ({pic cs:s2}) [red] to ({pic cs:t2});
      \draw [->] ({pic cs:s3}) [red] to ({pic cs:t3});
  \end{tikzpicture}}
\end{center}
\caption{The algebraic Atiyah--Hirzebruch spectral sequence in degree $(s,k,t) = (1,5, 3\cdot 2^{n+3}+3)$ and algAH filtration $\leq 3n+8$ for \Cref{lem:skt153}. In this chart the vertical axis is the algAH filtration and the horizontal axis is the $s$-degree. $d_1$-differentials are red and $d_2$-differentials are blue.}
\label{fig:algA 2b}
\end{figure}

\begin{proof}
  From \Cref{candidates CESS} we know the $E_1$-page of the algebraic Atiyah--Hirzebruch spectral sequence in degree $(s,k,t) = (1,5,3 \cdot 2^{n+3} + 3)$. 
  Using \Cref{lem AH diff}(1) we compute the relevant algAH $d_1$-differentials.
  These differentials are displayed in Figures \ref{fig:algA 2a} and \ref{fig:algA 2b}.
  (Note that when $n=3$, the cases $(18)$ and $(19)$ are identical in \Cref{candidates CESS}, so their algAH $d_1$-differentials are identically displayed twice, in filtration $5n+2$ and $4n+5$, in Figures \ref{fig:algA 2a}.)
  
  In order to pass the $E_2$-page we must also verify that the targets of these $d_1$-differentials are all linearly independent.
  For this we note that two classes $q_{i_1}q_{i_2}q_{i_3}q_{i_4}q_{i_5} \cdot a$ and $q_{j_1}q_{j_2}q_{j_3}q_{j_4}q_{j_5} \cdot b$ (with the $i$'s and $j$'s in non-decreasing order) on the $E_1$-page can only satisfy a relation if $i_k = j_k$ for $k = 1,\dots,5$.
  Examining Figures \ref{fig:algA 2a} and \ref{fig:algA 2b} we see that no pair of classes satisfy this condition.

  On the $E_2$-page of the algAH sseq only one of the 32 candidates from Lemma~\ref{candidates CESS} is still present: (13). Using \Cref{lem AH diff}(2,3), we obtain the following $d_2$-differential.
$$(13): \ \ \  d_2(q_1^2q_{n}^2q_{n+3} \cdot h_{n+1})  = q_0^2q_n^2q_{n+3}\cdot h_1h_{n+1} + q_1^2q_n^2q_{n+1} \cdot h_{n+2}^2.$$
  As the target is non-zero on the $E_2$-page of the algAH sseq this $d_2$-differential is non-zero.
  In particular, the algAH sseq is empty in degree $(s,k,t) = (1,5,3 \cdot 2^{n+3} + 3)$ starting from the $E_3$-page.
\end{proof}

\begin{lem}  \label{lem:skt142}
For $n\geq 3$,  $\Ext^{1,\ 3\cdot 2^{n+3} + 2}_{\P}(\F_2, \Ext_{\mathcal{Q}}^4(\F_2, \F_2)) \cong 0$.
\end{lem}

\begin{figure}[h]
\begin{center}
  \scalebox{0.9}{\renewcommand{\arraystretch}{1.1}
    \begin{tabular}{|c||c|c|c|}\hline
      & $s=0$ & $s=1$ & $s=2$ \\\hline\hline
      $3n+7$ & & $q_1q_{n+2}^3 \cdot h_{0}$ \tikzmark{Ds8} & \\\hline
      $3n+6$ & & $q_0q_{n+2}^3 \cdot h_{1}$ \tikzmark{Ds9} & \tikzmark{Dt8} $q_0q_{n+2}^3 \cdot h_0^2 + q_1q_{n+1}q_{n+2}^2\cdot h_0h_{n+1}$ \\
      & & $q_1q_{n+1}^2q_{n+3} \cdot h_{0}$ \tikzmark{Ds4} & \\\hline
      $3n+5$ & & & \tikzmark{Dt9} $q_0q_{n+1}q_{n+2}^2 \cdot h_{1}h_{n+1}$ \\
      & & $q_0q_{n+1}^2q_{n+3} \cdot h_{1}$ \tikzmark{Ds5} & \tikzmark{Dt4} $ q_0q_{n+1}^2q_{n+3} \cdot h_{0}^2 + q_1q_{n+1}^2q_{n+2} \cdot h_0h_{n+2}$ \\\hline      
      $3n+4$ & & & \tikzmark{Dt5} $q_0q_{n+1}^2q_{n+2} \cdot h_{1}h_{n+2}$ \\\hline
      $\cdots$ & & & \\\hline
      $2n+6$ & $q_0q_1q_{n+2}q_{n+3} \cdot 1$ \tikzmark{Dsg1a}\tikzmark{Dsg1b}\tikzmark{Dsg1c} & & \\\hline
      $2n+5$ & &  $q_0q_1 \tikzmark{Dtg1a} q_{n+1}q_{n+3} \cdot h_{n+1}$ \tikzmark{Ds1} & \\
             & & \tikzmark{Dtg1b} $q_0q_1q_{n+2}^2 \cdot h_{n+2}$ \tikzmark{Ds6}& \\
             & & \tikzmark{Dtg1c} $q_0^2q_{n+2}q_{n+3} \cdot h_0$ \tikzmark{Ds10a}\tikzmark{Ds10b} & \\\hline
      $2n+4$ & & & \tikzmark{Dt1}\tikzmark{Dt10a} $q_0^2q_{n+1}q_{n+3} \cdot h_0h_{n+1}$ \\
             & & & \tikzmark{Dt6}\tikzmark{Dt10b} $q_0^2q_{n+2}^2 \cdot h_0h_{n+2}$ \\\hline
      $2n+3$ & & $q_0q_1q_{n+1}^2 \cdot h_{n+3}$ \tikzmark{Ds7} & \\\hline
      $2n+2$ & & & \tikzmark{Dt7} $q_0^2q_{n+1}^2 \cdot h_0h_{n+3}$ \\\hline
      $\cdots$ & & & \\\hline
      $n + 3$ & & $q_0^3q_{n+3} \cdot h_{n+2}$ \tikzmark{Ds2} & \\\hline
      $n + 2$ & & $q_0^3q_{n+2} \cdot h_{n+3}$ \tikzmark{Ds3} & \tikzmark{Dt2} $q_0^3q_{n+2} \cdot h_{n+2}^2$ \\\hline
      $n + 1$ & & & \tikzmark{Dt3} $ q_0^3q_{n+1} \cdot h_{n+1}h_{n+3}$ \\\hline
    \end{tabular}
    \begin{tikzpicture}[overlay, remember picture, shorten >=.5pt, shorten <=.5pt, transform canvas={yshift=.25\baselineskip}]
      \draw [->] ({pic cs:Ds8}) [red] to ({pic cs:Dt8});
      \draw [->] ({pic cs:Ds4}) [red] to ({pic cs:Dt4});
      \draw [->] ({pic cs:Ds9}) [red] to ({pic cs:Dt9});
      \draw [->] ({pic cs:Ds5}) [red, bend right=15] to ({pic cs:Dt5});      
      \draw [->] ({pic cs:Dsg1a}) [red, bend left, shorten >=0.2cm] to ({pic cs:Dtg1a});
      \draw [->] ({pic cs:Dsg1b}) [red, bend right] to ({pic cs:Dtg1b});
      \draw [->] ({pic cs:Dsg1c}) [red, bend right] to ({pic cs:Dtg1c});
      \draw [->] ({pic cs:Ds1}) [red] to ({pic cs:Dt1});
      \draw [->] ({pic cs:Ds6}) [red, bend right=20] to ({pic cs:Dt6});
      \draw [->] ({pic cs:Ds10a}) [red] to ({pic cs:Dt10a});
      \draw [->] ({pic cs:Ds10b}) [red, bend right=20] to ({pic cs:Dt10b});
      \draw [->] ({pic cs:Ds7}) [red] to ({pic cs:Dt7});
      \draw [->] ({pic cs:Ds2}) [red] to ({pic cs:Dt2});
      \draw [->] ({pic cs:Ds3}) [red] to ({pic cs:Dt3});
  \end{tikzpicture}}
\end{center}
\caption{The algebraic Atiyah--Hirzebruch spectral sequence in degree $(s,k,t) = (1,4, 3\cdot 2^{n+3}+2)$ for \Cref{lem:skt142}. In this chart the vertical axis is the algAH filtration and the horizontal axis is the $s$-degree. $d_1$-differentials are red.}
\label{fig:algA 4}
\end{figure}

\begin{proof}
  Dividing by $q_0$ as in \Cref{rmk:q0 remove} we can determine
  the $E_1$-page of the algAH sseq in degree $(s,k,t) = (1,4,3 \cdot 2^{n+3} + 2)$
  from \Cref{candidates CESS}.
  It contains the following classes:
  \begin{multicols}{2}
    \begin{enumerate}
    \item $q_0^2q_{j}q_{j+1} \cdot h_0$,
    \item $q_0^3q_{j+1} \cdot h_{j}$,
    \item $q_0^3q_{j} \cdot h_{j+1}$,
    \item $q_1q_{j-1}^2q_{j+1} \cdot h_0$,
    \item $q_0q_{j-1}^2q_{j+1} \cdot h_1$,
    \item $q_0q_1q_{j-1}q_{j+1} \cdot h_{j-1}$,
    \item $q_0q_1q_{j-1}^2 \cdot h_{j+1}$,
    \item $q_1q_{j}^3 \cdot h_{0}$,
    \item $q_0q_{j}^3 \cdot h_{1}$,
    \item $q_0q_1q_{j}^2 \cdot h_{j}$.
    \end{enumerate}
  \end{multicols}
  Using \Cref{lem AH diff}(1) we compute the relevant algAH $d_1$-differentials.
  These differentials are displayed in \Cref{fig:algA 4}.
  In particular, the algAH sseq is empty in degree $(s,k,t) = (1,4,3 \cdot 2^{n+3} + 2)$ starting from the $E_2$-page.  
%%   Consider algAH sseq that converges to the Cartan-Eilenberg $E_2$-page. Elements in the $E_1$-page of algAHSS that could detect $x$ are of the form $q_aq_bq_cq_d\cdot h_e$. Similar to the proof of Lemma~\ref{candidates CESS}, we consider the $t$-degree, and obtain the following 10 possibilities. 
%% Their $q_0$-multiples correspond to exactly the cases $(1) - (10)$ in Lemma~\ref{candidates CESS} (with $j = n+2$). By the proof of Lemma~\ref{candidates CESS}, all 10 elements do not survive to the $E_2$-page of algAHSS. Therefore, we have
%% $$\Ext_P^{1,3\cdot 2^{j+1} +2}(\F_2, \Ext^4_{Q}(\F_2, \F_2)) = 0,$$ 
%% and completes the proof for part $(c)$.
\end{proof}

\begin{lem}  \label{lem:skt131}
For $n\geq 3$,  $\Ext^{1,\ 3\cdot 2^{n+3} + 1}_{\P}(\F_2, \Ext_{\mathcal{Q}}^3(\F_2, \F_2)) \cong 0$.
\end{lem}

\begin{figure}[h]
\begin{center}
  \scalebox{0.9}{\renewcommand{\arraystretch}{1.1}
    \begin{tabular}{|c||c|c|c|}\hline
      & $s=0$ & $s=1$ & $s=2$ \\\hline\hline
      $3n+6$ & & $q_{n+2}^3 \cdot h_{1}$ \tikzmark{Ws9} & \\\hline
      $3n+5$ & & $q_{n+1}^2q_{n+3} \cdot h_{1}$ \tikzmark{Ws5} & \tikzmark{Wt9} $q_{n+1}q_{n+2}^2 \cdot h_{1}h_{n+1}$ \\\hline      
      $3n+4$ & & & \tikzmark{Wt5} $q_{n+1}^2q_{n+2} \cdot h_{1}h_{n+2}$ \\\hline
      $\cdots$ & & & \\\hline
      $2n+6$ & $q_1q_{n+2}q_{n+3} \cdot 1$ \tikzmark{Wsg1a}\tikzmark{Wsg1b}\tikzmark{Wsg1c} & & \\\hline
      $2n+5$ & &  $q_1 \tikzmark{Wtg1a} q_{n+1}q_{n+3} \cdot h_{n+1}$ \tikzmark{Ws1} & \\
             & & \tikzmark{Wtg1b} $q_1q_{n+2}^2 \cdot h_{n+2}$ \tikzmark{Ws6}& \\
             & & \tikzmark{Wtg1c} $q_0q_{n+2}q_{n+3} \cdot h_0$ \tikzmark{Ws10a}\tikzmark{Ws10b} & \\\hline
      $2n+4$ & & & \tikzmark{Wt1}\tikzmark{Wt10a} $q_0q_{n+1}q_{n+3} \cdot h_0h_{n+1}$ \\
             & & & \tikzmark{Wt6}\tikzmark{Wt10b} $q_0q_{n+2}^2 \cdot h_0h_{n+2}$ \\\hline
      $2n+3$ & & $q_1q_{n+1}^2 \cdot h_{n+3}$ \tikzmark{Ws7} & \\\hline
      $2n+2$ & & & \tikzmark{Wt7} $q_0q_{n+1}^2 \cdot h_0h_{n+3}$ \\\hline
      $\cdots$ & & & \\\hline
      $n + 3$ & & $q_0^2q_{n+3} \cdot h_{n+2}$ \tikzmark{Ws2} & \\\hline
      $n + 2$ & & $q_0^2q_{n+2} \cdot h_{n+3}$ \tikzmark{Ws3} & \tikzmark{Wt2} $q_0^2q_{n+2} \cdot h_{n+2}^2$ \\\hline
      $n + 1$ & & & \tikzmark{Wt3} $ q_0^2q_{n+1} \cdot h_{n+1}h_{n+3}$ \\\hline
    \end{tabular}
    \begin{tikzpicture}[overlay, remember picture, shorten >=.5pt, shorten <=.5pt, transform canvas={yshift=.25\baselineskip}]
      \draw [->] ({pic cs:Ws9}) [red] to ({pic cs:Wt9});
      \draw [->] ({pic cs:Ws5}) [red] to ({pic cs:Wt5});      
      \draw [->] ({pic cs:Wsg1a}) [red, bend left, shorten >=0.2cm] to ({pic cs:Wtg1a});
      \draw [->] ({pic cs:Wsg1b}) [red, bend right] to ({pic cs:Wtg1b});
      \draw [->] ({pic cs:Wsg1c}) [red, bend right] to ({pic cs:Wtg1c});
      \draw [->] ({pic cs:Ws1}) [red] to ({pic cs:Wt1});
      \draw [->] ({pic cs:Ws6}) [red, bend right=10] to ({pic cs:Wt6});
      \draw [->] ({pic cs:Ws10a}) [red] to ({pic cs:Wt10a});
      \draw [->] ({pic cs:Ws10b}) [red, bend right=20] to ({pic cs:Wt10b});
      \draw [->] ({pic cs:Ws7}) [red] to ({pic cs:Wt7});
      \draw [->] ({pic cs:Ws2}) [red] to ({pic cs:Wt2});
      \draw [->] ({pic cs:Ws3}) [red] to ({pic cs:Wt3});
  \end{tikzpicture}}
\end{center}
\caption{The algebraic Atiyah--Hirzebruch spectral sequence in degree $(s,k,t) = (1,3, 3\cdot 2^{n+3}+1)$ for \Cref{lem:skt131}. In this chart the vertical axis is the algAH filtration and the horizontal axis is the $s$-degree. $d_1$-differentials are red.}
\label{fig:algA 1-3}
\end{figure}

\begin{proof}
  Dividing by $q_0$ as in \Cref{rmk:q0 remove} again we can determine
  the $E_1$-page of the algAH sseq in degree $(s,k,t) = (1,3,3 \cdot 2^{n+3} + 1)$.
  It contains the following classes:
  \begin{multicols}{2}
    \begin{enumerate}
    \item $q_0q_{j}q_{j+1} \cdot h_0$,
    \item $q_0^2q_{j+1} \cdot h_{j}$,
    \item $q_0^2q_{j} \cdot h_{j+1}$,
    \item $q_{j-1}^2q_{j+1} \cdot h_1$,
    \item $q_1q_{j-1}q_{j+1} \cdot h_{j-1}$,
    \item $q_1q_{j-1}^2 \cdot h_{j+1}$,
    \item $q_{j}^3 \cdot h_{1}$,
    \item $q_1q_{j}^2 \cdot h_{j}$.
    \end{enumerate}
  \end{multicols}
  Using \Cref{lem AH diff}(1) we compute the relevant algAH $d_1$-differentials.
  These differentials are displayed in \Cref{fig:algA 1-3}.
  In particular, the algAH sseq is empty in degree $(s,k,t) = (1,3,3 \cdot 2^{n+3} + 1)$ starting from the $E_2$-page.  
\end{proof}

\begin{lem}  \label{lem:skt120}
For $n\geq 3$,  $\Ext^{1,\ 3\cdot 2^{n+3}}_{\P}(\F_2, \Ext_{\mathcal{Q}}^2(\F_2, \F_2)) \cong 0$.
\end{lem}

\begin{figure}[h]
\begin{center}
  \scalebox{0.9}{\renewcommand{\arraystretch}{1.1}
    \begin{tabular}{|c||c|c|c|}\hline
      & $s=0$ & $s=1$ & $s=2$ \\\hline\hline
      $2n+5$ & & $q_{n+2}q_{n+3} \cdot h_0$ \tikzmark{Qs10a}\tikzmark{Qs10b} & \\\hline
      $2n+4$ & & & \tikzmark{Qt10a} $q_{n+1}q_{n+3} \cdot h_0h_{n+1} + q_{n+2}^2 \cdot h_0h_{n+2}$ \\\hline
      $\cdots$ & & & \\\hline
      $n + 3$ & & $q_0q_{n+3} \cdot h_{n+2}$ \tikzmark{Qs2} & \\\hline
      $n + 2$ & & $q_0q_{n+2} \cdot h_{n+3}$ \tikzmark{Qs3} & \tikzmark{Qt2} $q_0q_{n+2} \cdot h_{n+2}^2$ \\\hline
      $n + 1$ & & & \tikzmark{Qt3} $ q_0q_{n+1} \cdot h_{n+1}h_{n+3}$ \\\hline
    \end{tabular}
    \begin{tikzpicture}[overlay, remember picture, shorten >=.5pt, shorten <=.5pt, transform canvas={yshift=.25\baselineskip}]
      \draw [->] ({pic cs:Qs10a}) [red] to ({pic cs:Qt10a});
      \draw [->] ({pic cs:Qs2}) [red] to ({pic cs:Qt2});
      \draw [->] ({pic cs:Qs3}) [red] to ({pic cs:Qt3});
  \end{tikzpicture}}
\end{center}
\caption{The algebraic Atiyah--Hirzebruch spectral sequence in degree $(s,k,t) = (1,2, 3\cdot 2^{n+3})$ for \Cref{lem:skt120}. In this chart the vertical axis is the algAH filtration and the horizontal axis is the $s$-degree. $d_1$-differentials are red.}
\label{fig:algA 1-2}
\end{figure}

\begin{proof}
  Dividing by $q_0$ as in \Cref{rmk:q0 remove} again we can determine
  the $E_1$-page of the algAH sseq in degree $(s,k,t) = (1,2,3 \cdot 2^{n+3})$.
  It contains the following classes:
  \begin{enumerate}
  \item $q_{j}q_{j+1} \cdot h_0$,
  \item $q_0q_{j+1} \cdot h_{j}$,
  \item $q_0q_{j} \cdot h_{j+1}$,
  \end{enumerate}
  Using \Cref{lem AH diff}(1) we compute the relevant algAH $d_1$-differentials.
  These differentials are displayed in \Cref{fig:algA 1-2}.
  In particular, the algAH sseq is empty in degree $(s,k,t) = (1,2,3 \cdot 2^{n+3})$ starting from the $E_2$-page.  
\end{proof}

%%% Local Variables:
%%% mode: latex
%%% TeX-master: "main"
%%% End:

%% file: cess2v2.tex
In this section we show that the Cartan--Eilenberg sseq has no differentials in a neighborhood around $h_j^3$ and use this to complete the proof of \Cref{E2 descriptions}.

% The main goal of this section is to set up the Cartan-Eilenberg spectral sequence, and to prove Theorem~\ref{E2 descriptions}~$(1)$. We will also prove Theorem~\ref{diff motivic ctau}~$(3)$ in the end of this section. 

\begin{lem} \label{lem:entering g}
  Let $j \geq 5$.
  There are no CE differentials entering the following $(a,t-a)$ degrees\footnote{Recall that $a=s+k$.}
  $(4, 3 \cdot 2^{j+1}-4)$, $(5, 3 \cdot 2^{j+1}-4)$, $(6, 3 \cdot 2^{j+1}-4)$ or $(7, 3 \cdot 2^{j+1}-4)$.
\end{lem}

\begin{proof}
  Recall that Cartan-Eilenberg $d_r$-differentials changes the tri-degrees in the following way
  $$d_r: E_r^{s,k,t} \to E_r^{s+r, k-r+1, t}.$$
  Rewriting this in the $(a,s,t-a)$ basis we obtain
  $$d_r: E_r^{a,s,t-a} \to E_r^{a+1, s+r, t-a-1}.$$
  
  From the sparsity of the CE sseq we know that all differentials have odd length.
  The following table contains a list of each differential we must rule out, and why each differential doesn't occur. 
  % ---------------------
  % entering (4 -1)
  % (1 2 0) --> (4 0 -1)         (zero source)
  % ---------------------
  % entering (5 -1)
  % (1 3 0) --> (4 1 -1)         (zero source)
  % ---------------------
  % entering (6 -1)
  % (1 4 0) --> (4 2 -1)         (zero source)
  % (3 2 0) --> (6 0 -1)         (perm source)
  % (1 4 0) --> (6 0 -1)         (zero source)
  % ---------------------
  % entering (7 -1)
  % (1 5 0) --> (4 3 -1)         (zero source)
  % (3 3 0) --> (6 1 -1)         (perm source)
  % (1 5 0) --> (6 1 -1)         (zero source)  
  \begin{center}{\renewcommand{\arraystretch}{1.2}
      \begin{tabular}{|c|c|}\hline
        $d_r : E_r^{a,s,t-a} \to E_r^{a+1,s+r,t-a-1}$ & argument \\\hline\hline
        $d_3 : E_3^{3,1, 3 \cdot 2^{j+1} -3} \to E_3^{4,4, 3 \cdot 2^{j+1} -4}$ & zero source \\\hline
        $d_3 : E_3^{4,1, 3 \cdot 2^{j+1} -3} \to E_3^{5,4, 3 \cdot 2^{j+1} -4}$ & zero source \\\hline
        $d_3 : E_3^{5,1, 3 \cdot 2^{j+1} -3} \to E_3^{6,4, 3 \cdot 2^{j+1} -4}$ & zero source \\\hline
        $d_3 : E_3^{5,3, 3 \cdot 2^{j+1} -3} \to E_3^{6,6, 3 \cdot 2^{j+1} -4}$ & source all permanent cycles \\\hline
        $d_5 : E_5^{5,1, 3 \cdot 2^{j+1} -3} \to E_5^{6,6, 3 \cdot 2^{j+1} -4}$ & zero source \\\hline
        $d_3 : E_3^{6,1, 3 \cdot 2^{j+1} -3} \to E_3^{7,4, 3 \cdot 2^{j+1} -4}$ & zero source \\\hline
        $d_3 : E_3^{6,3, 3 \cdot 2^{j+1} -3} \to E_3^{7,6, 3 \cdot 2^{j+1} -4}$ & source all permanent cycles \\\hline
        $d_5 : E_5^{6,1, 3 \cdot 2^{j+1} -3} \to E_5^{7,6, 3 \cdot 2^{j+1} -4}$ & zero source \\\hline
      \end{tabular}}
  \end{center}
  In each case \Cref{prop:weak CE E2} provides the required information about the source group.
\end{proof}

\begin{rec} \label{rec:Chen}
  Building on Lin's work on the 4-line, in \cite[Theorem~1.2]{Ext5}, Chen gives a complete description of $\Ext_{\A}$ up to Adams filtration 5.
  In particular, Chen's work provides the following information about $\Ext_{\A}$ in a neighborhood of $h_j^3$:
  \begin{center}{\renewcommand{\arraystretch}{1.2}
      \begin{tabular}{|ll|c|}\hline
        $(a,$ & $t-a)$ & $\Ext^{a,\,t}_{\A}(\mathbb{F}_2,\, \mathbb{F}_2)$ \\\hline\hline
        $(4,$ & $ 3 \cdot 2^j-4)$ & $\F_2\{g_{j-2}\}$ \\\hline
        $(5,$ & $ 3 \cdot 2^j-4)$ & $\F_2\{h_0g_{j-2}\}$ \\\hline
        $(3,$ & $ 3 \cdot 2^j-3)$ & $\F_2\{h_j^3\}$ \\\hline
        $(4,$ & $ 3 \cdot 2^j-3)$ & $\F_2\{h_0h_j^3\}$ \\\hline
        $(5,$ & $ 3 \cdot 2^j-3)$ & $\F_2\{h_0^2h_j^3, h_1g_{j-2}\}$ \\\hline    
      \end{tabular}}
  \end{center}  
  for $j \geq 8$.
  For $j=7$ degree $(a,t-s)=(5,3 \cdot 2^{7} - 3)$ contains the additional class $h_8D_3(1)$.
  For $j=6$ degree $(a,t-s)=(5,3 \cdot 2^{6} - 4)$ contains the additional class $h_7D_3(0)$.\todo{There's really this extra class and that means for 6,7 we really need to track this through the argument.}
\end{rec}

\begin{prop} \label{prop:g survives}
  Let $j \geq 5$.
  On the $E_2$-page of the classical Adams spectral sequence for $S^{0}$
  the class $h_0^3g_{j-2}$ is non-zero.
\end{prop}

\begin{proof}
  In the cases $j=5,6$, this is known from stemwise calculation of the $\Ext$-groups.
  See for example \cite{Bruner2, IWX2} for $j=5$ and \cite{Nassau} for $j=6$.
  We will prove the proposition for $j \geq 6$.
        
  %The only remaining claim we must prove is that $h_0^2g_{j-2}$ and $h_0^3g_{j-2}$ are non-zero.
  %In fact it suffices to prove that $h_0^3g_{j-2}$ is non-zero.

  The class $q_0^3$ survives the algAH sseq and CE sseq, and detects the class $h_0^3$ in the Adams $E_2$-page. From the Frobenius isomorphism between $\Ext_{\mathcal{P}}^{*,2*}(\F_2, \F_2)$ and $\Ext_{\mathcal{A}}^{*,*}(\F_2, \F_2)$, the element $g_n$ in $\Ext_{\mathcal{P}}^{4, 3\cdot 2^{n+3}}(\F_2, \F_2)$ detects $g_{n+1}$ in $\Ext_{\mathcal{A}}^{4, 3\cdot 2^{n+3}}(\F_2, \F_2)$. Therefore, the element $q_0^3 \cdot g_n$ detects the element $h_0^3 g_{n+1}$ in the Adams $E_2$-page.
  
  For this we observe that $h_0^3g_{j-2}$ is detected on the Cartan--Eilenberg $E_2$-page by $q_0^3 g_{j-3}$ and that this class is not hit by a CE differential since there are no differential which enter this tridegree by \Cref{lem:entering g}.
\end{proof}

Together \Cref{rec:Chen} and \Cref{prop:g survives} complete the proof of \Cref{E2 descriptions}(1).

\begin{lem} \label{lem:weak CE E2}
  Let $j \geq 6$. 
  The $E_2$-page of the motivic Adams spectral sequence for $S^{0,0}/\tau$
  takes the following form near $h_{j}^3$:
  \begin{center}{\renewcommand{\arraystretch}{1.2}
      \begin{tabular}{|lll|c|}\hline
        $(a,$ & $s,$ & $t-a)$ & $\Ext^{a,\,t,\,w}_{\A^\textup{mot}}(\mathbb{F}_2[\tau],\, \mathbb{F}_2)$ \\\hline\hline
        $(4,$ & $4,$ & $ 3 \cdot 2^j-4)$ & $\F_2\{g_{j-2}\}$ \\\hline
        $(5,$ & $4,$ & $ 3 \cdot 2^j-4)$ & $\F_2\{h_0g_{j-2}\} $ \\\hline
        $(3,$ & $3,$ & $ 3 \cdot 2^j-3)$ & $\F_2\{h_j^3\}$ \\\hline
        $(4,$ & $3,$ & $ 3 \cdot 2^j-3)$ & $\F_2\{h_0h_j^3\}$ \\\hline
        $(5,$ & $3,$ & $ 3 \cdot 2^j-3)$ & $\F_2\{h_0^2h_j^3\}$ \\\hline    
        $(5,$ & $5,$ & $ 3 \cdot 2^j-3)$ & $\F_2\{h_1g_{j-2}\}$ \\\hline          
      \end{tabular}}
  \end{center}
  For $j=7$ there is an additional class, $h_8D_3(1)$, in degree $(a,s,t-a)=(5,5, 3 \cdot 2^7 - 3)$.
  For $j=6$ there is possibly an additional class, $h_7D_3(0)$, in degree $(a,s,t-a)=(5,4, 3 \cdot 2^6 - 4)$.
\end{lem}

\begin{proof}
  % --------------------------------------
  % 4 1 -1    contains q0g
  % 4 2 -1    contains q02g
  % 4 3 -1    contains q03g
  % 1 2  0        0
  % 1 3  0        0
  % 1 4  0        0
  % 1 5  0        0
  % 3 1  0     q0hj3 or 0
  % 3 2  0     q0hj3 or 0
  % 3 3  0     q0hj3 or 0
  % --------------------------------------
  % 4 0 -1        g                 (chen Ext5)
  % 4 1 -1       q0g                (chen Ext5 + no fil jump)
  % 3 0  0       hj3                (chen Ext5)
  % 3 1  0      q0hj3               (chen Ext5 + no fil jump)
  % 3 2  0      q02hj3              (chen Ext5 + no fil jump)
  % 5 0  0       h1g                (chen Ext5)
  % --------------------------------------
  Recall from \Cref{tridegree} that there is an isomorphism between the $E_2$-pages of the Cartan-Eilenberg spectral sequence and the motivic Adams spectral sequence for $S^{0,0}/\tau$:
\[\xymatrix{
  \Ext^{s,t}_{\P}(\F_2, \Ext_{\mathcal{Q}}^k(\F_2, \F_2)) \ar[rr]^-{\cong} & & \Ext^{s+k, t, \frac{t-k}{2}}_{\A^{\textup{mot}}}(\F_2[\tau], \F_2),
}\]
where the element $g_{j-3}$ on the Cartan--Eilenberg $E_2$-page sends to $g_{j-2}$ on the motivic Adams $E_2$-page, and $h_{j-1}^3$ to $h_j^3$.
We will prove the statement for Cartan--Eilenberg $E_2$-page instead. 

For the information on lines 1, 3 and 6 (these are the cases $k=a-s=0$), they follow from Lin and Chen's computation of $\Ext_{\A}$ through Adams filtration 5 (see \cite[Theorem~1.2]{Ext5}) and the Frobenius isomorphism $\Ext_{\P}^{s,2t} \cong \Ext_{\A}^{s,t}$.

For $j \geq 6$ we proved in \Cref{prop:weak CE E2}
 that the CE $E_2$-page contains a non-trivial class $q_0g_{j-3}$, which corresponds to $h_0g_{j-2}$ in the motivic $E_2$-page. Suppose the CE $E_2$-page had another element $x$ besides $q_0g_{j-3}$
  in degree $(a,s,t-a) = (5,4, 3 \cdot 2^j - 4)$.
  This element would be a permanent cycle for degree reasons and would survive to the $E_\infty$-page since no differentials enter this tridegree by \Cref{lem:entering g}.
  This would imply that the rank of degree $(a,t-a) = (5, 3 \cdot 2^j - 4)$ in the Adams $E_2$-page is at least $2$.
  This contradicts Lin and Chen's computations of $\Ext_{\A}$, therefore no such $x$ can exist
  and we obtain the conclusion on line 2. In the $j=6$ case there the rank of $\Ext_{\A}$ is 2 so we may have another class detecting $h_7D_3(0)$.

For the claims on lines 4 and 5, by \Cref{prop:weak CE E2} we only need to show that  $q_0^2h_{j-1}^3 \neq 0$ on the CE $E_2$-page for $j \geq 6$.
  From Lin and Chen's computations we know that $h_0^2h_{j}^3 \neq 0$ for $j \geq 6$ in $\Ext_{\A}$, which
  would be detected on the CE $E_2$-page by $q_0^2h_{j-1}^3$ if this element was not killed by any CE differential.
  If this is not the case, then the element $h_0^2h_{j}^3$ in $\Ext_{\A}$ must be detected by some other element in a higher $s$-filtration on the  CE $E_2$-page and the only possibility left is degree $(a,s,t-a) = (5,5, 3\cdot 2^j - 3)$, where $k=a-s=0$.
  This cannot happen since the only class in this degree is $h_0g_{j-3}$ from $\Ext_{\P}$ which detects $h_1g_{j-2}$ in $\Ext_{\A}$ (in the $j=7$ case there is also $h_7D_3$ in $\Ext_{\P}$, which detects $h_8D_3(1)$ in $\Ext_{\A}$). The conclusion follows.
\end{proof}

Together \Cref{prop:weak CE E2} and \Cref{lem:weak CE E2} complete the proof of \Cref{E2 descriptions}(3).
We end the section by using the motivic CE sseq (and \Cref{cor:no entering} in particular) to prove \Cref{E2 descriptions}(2,4).

\begin{lem} \label{lem:entering hj3}
  Let $j \geq 5$.
  There are no CE differentials entering the following $(a,t-a)$ degrees
  $(3, 3 \cdot 2^{j+1}-3)$, $(4, 3 \cdot 2^{j+1}-3)$ and $(5, 3 \cdot 2^{j+1}-3)$.
\end{lem}

\begin{proof}
  Convergence of the CE sseq implies that the
  rank of the CE $E_2$-page is an upper bound on the rank of the Adams $E_2$-page
  and that this bound is sharp in degree $(a,t-a)$ exactly when there are no differentials entering or leaving degree $(a,t)$.
  Comparing the ranks from \Cref{E2 descriptions}(3) with the ranks from \Cref{E2 descriptions}(1)
  we see the bound is sharp and obtain the desired conclusion.    
  % ---------------------
  % entering (3 0)
  % (0 2 1) --> (3 0 0)          
  % ---------------------
  % entering (4 0)
  % (0 3 1) --> (3 1 0)          
  % ---------------------
  % entering (5 0)
  % (0 4 1) --> (3 2 0)
  % (0 4 1) --> (5 0 0)
  % (2 2 1) --> (5 0 0)
\end{proof}

% Now we use motivic Cartan--Eilenberg sseq to finish the proof of \Cref{E2 descriptions}.

\begin{proof}[Proof (of \Cref{E2 descriptions}(2,4)).]
  Using \Cref{cor:no entering} we combine the statements about Cartan--Eilenberg differentials from
  Lemmas \ref{lem:entering g} and \ref{lem:entering hj3} with our knowledge of the Cartan--Eilenberg $E_2$-page to prove \Cref{E2 descriptions}(2).

  In order to prove \Cref{E2 descriptions}(4) we begin by observing that \Cref{cor:no entering} lets us match up names of classes with names of lifts, so it suffices to determine where each class goes under Betti realization.
  For $j \geq 6$, in $(a,t-a)$ degrees 
  $(4, 3 \cdot 2^j - 4)$, $(5, 3 \cdot 2^j - 4)$, $(3, 3 \cdot 2^j - 3)$ and $(4, 3 \cdot 2^j - 3)$
  the target of Betti realization is a single $\F_2$ so the conclusion is automatic.
  In $(a,t-a)$ degrees $(6, 3 \cdot 2^j - 4)$, $(7, 3 \cdot 2^j - 4)$ and $(5, 3 \cdot 2^j - 3)$  we can choose our lifts of Cartan--Eilenberg permanent cycles to be the products
  $h_0^2 \cdot g_{j-2}$, $h_0^3 \cdot g_{j-2}$, $h_0 \cdot h_j^3$, $h_0^2 \cdot h_j^3$ and $h_1 \cdot g_{j-2}$
  and the compatibility of Betti realization with products lets us conclude.
\end{proof}

%% file: productv2.tex
In this section we prove the key motivic Adams differentials for $S^{0,0}/\tau$ appearing in \Cref{diff motivic ctau}.
Through \Cref{MASS algNSS} this is equivalent to the following family of algebraic Novikov differentials:

\begin{prop}\label{prop:alg-n-diff}
   In the algebraic Novikov spectral sequence we have differentials
  \[ d_2(h_j^3) = 0, \quad d_3(h_j^3) = 0, \quad d_4(h_j^3) = q_0^3 g_{j-2}\]
  for $j \geq 5$.
\end{prop}

% Without these operations we are forced to give a more direct and explicit proof.

\begin{rmk}
  Note that in Proposition~\ref{prop:alg-n-diff}
  we are using Cartan--Eilenberg names 
  rather than motivic Adams names.
  In particular, the class $q_0^3g_{j-2}$ corresponds to the class $h_0^3g_{j-1}$ in \Cref{diff motivic ctau} and the condition $j \geq 5$ corresponds to the condition $j \geq 6$.\footnote{See Section~\ref{sec:cess} for more on this translations of names.} % Also note that by Theorem~\ref{diff motivic ctau}~(3), the targets of these $d_4$-differentials are nonzero.
\end{rmk}

Speaking broadly, the proof of \Cref{prop:alg-n-diff} has two parts.
In the first part, we relate the differential on $h_j^3$ to a product $\vartheta_{j}\vartheta_{j+1}$ on the Adams--Novikov $E_2$-page. Where $\vartheta_j$ refers to a particular choice of class that maps to $h_j^2$ under the Thom reduction map 
\[ \Ext_{\BP_*\BP}^{2, 2^{j+1}}(\BP_*,\, \BP_*) \to \Ext_{\A}^{2, 2^{j+1}}(\F_2,\, \F_2) \]
whose precise definition will be given in \Cref{dfn:Tj}.
In the second part, we use explicit cocycle representatives to identify $\vartheta_{j}\vartheta_{j+1}$ in the algebraic Novikov spectral sequence and thereby prove \Cref{prop:alg-n-diff}.

\begin{rmk} 
  In \Cref{sec:kervaire}, we will explain how the product $\vartheta_{j}\vartheta_{j+1}$ is related to the Adams differentials supported by the Kervaire invariant class $h_j^2$ and is therefore of independent interest.
\end{rmk}

  % and has some bearing on  one differentials in the Adams spectral sequence. %Surprisingly, the reason our cocycle manipulations succeed is that there is a closely related stack which \emph{does} admit a lift of frobenius 

% justifying its name as a Kervaire invariant one class. 
% Kervaire invaraint one class on the Adams--Novikov $E_2$-page\footnote{By a ``Kervaire invariant one class'' we mean that the image of $\vartheta_j$ under the Thom reduction map is $h_j^2$.}, the precise definition of this class 

Speaking concretely, the section is organized as follows.
In \Cref{subsec:cocycle-diffs} we explain how algebraic Novikov differentials can be computed in terms of cocycles and reduce \Cref{prop:alg-n-diff} to understanding the product $\vartheta_j\vartheta_{j+1}$.
In \Cref{subsec:product-reln} we identify the class detecting the product $\vartheta_j\vartheta_{j+1}$ in the algebraic Novikov spectral sequence.
In \Cref{subsec:correction}, we provide material which, although not strictly necessary for the proof of \Cref{prop:alg-n-diff}, is useful for the reader looking to extend our results and methods.

% but we have included it in order to provide a more principled explanation of why the proofs in \Cref{subsec:product-reln} work.
% The remaining three subsections of this section are not strictly necessary for the material in the rest of the paper, but we have included them in order to provide a more principled explanation of why the proof of \Cref{prop:alg-n-diff} works.

\begin{rmk}
  The family of differentials we prove in \Cref{prop:alg-n-diff} 
  appears as though it might be obtainable inductively via algebraic Steenrod operations. 
  Unfortunately, as the moduli of formal groups is an integral object, 
  the operation we would like to use is not definable without a lift of Frobenius. 
  This does, however, suggest an alternative approach to proving \Cref{prop:alg-n-diff} for $j \gg 0$.
  
  Instead of studying the moduli stack of one-dimensional formal groups $\mathcal{M}_{\mathrm{fg}}$ (\cite{GoerssMfg, HopkinsMfg}) we can study another closely related object,
  \[ \mathbb{W}\left( \left( \mathcal{M}_{\mathrm{fg}} \times_{\Spec(\Z)} \Spec(\F_p) \right)^{\mathrm{perf}} \right),\]
  which \emph{does} admit a lift of Frobenius.
  This is the Witt vectors of the perfection of the reduction mod $p$ of the moduli of formal groups.
  Although there is no comparison map between these two stacks, 
  it seems that computations on the Witt vector side can be lifted back to the moduli of formal groups in an asymptotic sense (i.e. for $j \gg 0$).
  On the Witt vector side the lift of Frobenius would let us reduce the proof of the analog of \Cref{prop:alg-n-diff} to a single calculation.
\end{rmk}

\subsection{Algebraic Novikov differentials in terms of cocycles}\label{subsec:cocycle-diffs}\hfill

% % Before proceeding %proving this proposition
% We begin by explaining how algebraic Novikov differentials can be computed in terms of explicit cocycles.
In \Cref{sec:review} we introduced the algebraic Novikov spectral sequence by filtering the cobar complex for $\BP_*\BP$ by powers of the augmentation ideal $I$ of $\BP_*$. This spectral sequence takes the form
\todo{fix degrees, it seems like $k$-degree is power of augmentation ideal.}
\begin{align*}
   \Ext_{\BP_*\BP/I}^{s,t'}(\BP_*/I,\, I^k/I^{k+1}) &\cong E_2^{s,k,t'} \Longrightarrow \Ext_{\BP_*\BP}^{s,t'}(\BP_*,\, \BP_*) \\
  d_r: E_r^{s,k,t'} &\longrightarrow E_r^{s+1,k+r-1,t'}.
\end{align*}
In this subsection we explain how algebraic Novikov differentials can be calculated using cocycles and prove \Cref{prop:alg-n-diff} modulo a lemma which we defer to \Cref{subsec:product-reln}.

\begin{ntn}
  We write $\mathrm{cb}(-)$ for the cobar complex of a Hopf algebroid and use bar notation for elements of this complex. For example this means that we write $[t_1 | t_2^2]$ for the class
  $t_1 \otimes t_2^2$ in $(\BP_*\BP) \otimes_{\BP_*} (\BP_*\BP)$.
  Note that we are also identifying $\BP_*\BP$ with $\BP_*[t_1,t_2,\dots]$ in the usual way.

  We let $I \cdot \mathrm{cb}(\BP_*\BP)$ denote the (levelwise) augmentation ideal of this cosimplicial ring. The quotient by this ideal is $\mathrm{cb}(\BP_*\BP/I)$.
\end{ntn}

Suppose we have a class $\epsilon$ on the $E_2$-page of the algebraic Novikov spectral sequence.
For simplicity we will assume that the $k$-degree of $\epsilon$ is zero.
The differentials $d_*(\epsilon)$ can be computed using the following procedure:

\begin{enumerate}
\item Pick a cocycle $e$ in $\mathrm{cb}(\BP_*\BP/I)$ which represents $\epsilon$.
\item Pick a lift $e_2$ of $e$ to the cobar complex $\mathrm{cb}(\BP_*\BP)$. Note that although $e$ is a cocyle, $e_2$ may not be a cocycle.\footnote{This is the mechanism by which algebraic Novikov differentials arise.}
\item Compute $d(e_2)$ in $\mathrm{cb}(\BP_*\BP)$ and denote it by $c_2$. Note that $c_2 \in I \cdot \mathrm{cb}(\BP_*\BP)$ since $e_2$ is a cocycle modulo the augmentation ideal. Now proceed to step $(4.2)$. 
\item[(4.r)] We currently have $e_r$ with $c_r = d(e_r)$ such that $c_r \in I^{r-1} \cdot \mathrm{cb}(\BP_*\BP)$.
  In order to proceed we break into two cases:
  \begin{itemize}
  \item Suppose there exists a $w_r \in I^{r-1} \cdot \mathrm{cb}(\BP_*\BP)$ such that
    \[ d(w_r) \equiv c_r \pmod{I^{r}}. \]
    Pick such a $w_r$ and proceed to step $(4.r+1)$ with 
    $$e_{r+1} \coloneqq e_r + w_r, \ c_{r+1} \coloneqq d(e_{r+1}).$$ 
    We then have
   $$c_{r+1}  = d(e_{r+1}) = d(e_r) + d(w_r) \equiv 2c_r \pmod{I^{r}}$$ 
Since $c_r \in I^{r-1} \cdot \mathrm{cb}(\BP_*\BP)$, we have $c_{r+1} \in I^{r} \cdot \mathrm{cb}(\BP_*\BP)$.   

  In this case we have the statement that $d_r(\epsilon) = 0$.\footnote{More specifically, the class $w_r$ can be viewed as a class which supports a differential pre-empting the one on $\epsilon$. Note that this implies that in order to compute the $d_r$-differential one needs knowledge of all shorter length differentials.}
  \item If no such $w_r$ exists, then we may conclude that there is a class
    \[ \zeta \in \Ext_{\BP_*\BP/I}^{*,*}(\BP_*/I, I^{r-1}/I^{r}) \]
    which is detected by $c_r$ and survives to the $E_r$-page of the algebraic Novikov spectral sequence where it is hit by the differential $d_r(\epsilon) = \zeta$.
  \end{itemize}
\end{enumerate}

%%%%% As an illustration, recompute the Hopf invariant 1 differential here? %%%%%

In order to apply this procedure to the classes $h_j^3$ there is one more wrinkle we need to iron out. In the statement of \Cref{prop:alg-n-diff} both the source and target of the differential we wish to prove were stated in terms of the names for classes coming from the \emph{Cartan--Eilenberg $E_2$-page}. This means we will have to provide a precise procedure for translating between these naming schemes.

From \Cref{sec:review} we recall that the isomorphism
$$\xymatrix{
\Ext^{s,t}_{\P}(\F_2, \F_2[q_0,q_1,\dots]) \ar[rr]^-{\cong} & & \bigoplus_{t'+k=t} \Ext_{\BP_*\BP/I}^{s,t'}(\BP_*/I, I^k/I^{k+1})
}$$
sends a $q_i$ to $v_i$ and uses the fact that $\BP_*\BP/I \cong \P$ in order to provide an isomorphism at the level of cobar complexes. The most subtle point here is that this isomorphism sends $t_i$ to $\chi(\xi_i^2)$ where $\chi$ is the antipode on $\P$ (see also \cite{AndrewsMiller, MillerSquare, RavenelGreenBook} for this isomorphism).

% \begin{exm}
%   The cocycle $[t_1]$ which detects $\eta$ in the Adams--Novikov spectral sequence maps to $[\xi_1^2]$ in $\mathcal{P}$ which upon including into $\mathcal{A}$ detects $\eta$ in the Adams spectral sequence.
% \end{exm}

\begin{dfn} \label{dfn:Tj}
  Let $T_j$ denote the following class in the cobar complex $\mathrm{cb}(\BP_*\BP)$:
  \[ T_j \coloneqq \frac{1}{2} d(t_1^{2^j}) = \frac{1}{2} \left( \Delta( t_1^{2^{j}} ) - t_1^{2^j}|1 - 1|t_1^{2^j} \right) = \frac{1}{2} \sum_{i=1}^{2^{j} -1} \binom{2^{j}}{i} [t_1^i | t_1^{2^{j} - i}]. \]
  The class $T_j$ is a cocycle since $d^2=0$ and the $\BP_*\BP$-cobar complex is torsion free.
  We will let $\vartheta_j$ denote the class on the Adams--Novikov $E_2$-page detected by $T_j$.
  In \Cref{exm:vartheta-hj2} we will check that $\vartheta_j$ maps to $h_{j-1}^2$ under the quotient map to the cohomology of $\P$.
\end{dfn}

%%%%% Refer later to explain why this element maps to h_j^2.  %%%%%%%

For \Cref{prop:alg-n-diff} we will need the following lemma whose proof we defer to the next subsection.

\begin{lem} \label{lem:prod-small}
  For $j \geq 5$, the product $ \vartheta_j\vartheta_{j+1} $ is detected by $q_0^2g_{j-2}$ in the algebraic Novikov spectral sequence.
\end{lem}

Now we prove \Cref{prop:alg-n-diff}.

\begin{proof}[Proof of \Cref{prop:alg-n-diff}]
  We follow the procedure outlined above.
  The first step is picking a lift of $h_j^3$ to the the cobar complex $\mathrm{cb}(\BP_*\BP)$. Under the isomorphism between $\BP_*\BP/I$ and $\P$ followed by the isomorphism with $\A$ the class
  $[t_1^{2^{j-1}}]$ is detected by $h_{j-1}$.  
  We pick the class:
  \[ e_2 \coloneqq t_1^{2^{j}} | T_{j+1} \]
  where $T_{j+1}$ is the cocycle given in \Cref{dfn:Tj} which detects $\vartheta_{j+1}$. By \Cref{exm:vartheta-hj2}, $T_{j+1}$ is a cobar representative of a lift of $h_{j}^2$, therefore $t_1^{2^{j}} | T_{j+1}$ is a cobar representative of a lift of $h_j^3$.
  % and $c_j$ is the correction term from \Cref{lem:prod-big}.
  We can compute the algebraic Novikov differential on $h_j^3$ by applying the cobar differential to this cobar representative $e_2$. In particular, we have
  \[ c_2 \coloneqq d(e_2) = d( t_1^{2^{j}} | T_{j+1} ) =  -d(t_1^{2^{j}}) | T_{j+1} + t_1^{2^{j}} | d(T_{j+1}) = -2T_j|T_{j+1} \in I \cdot \mathrm{cb}(\BP_*\BP). \]

  Now we need to produce the correction terms for step (4.2).
  
  We make use of \Cref{lem:prod-small}, from which we know for $j \geq 5$, $\vartheta_j\vartheta_{j+1}$ is detected by $q_0^2g_{j-2}$ in the algebraic Novikov spectral sequence. In terms of cocycles this means that there exists a correction term $c$ such that 
  $$T_j|T_{j+1} + d(c) \in I^2 \cdot \mathrm{cb}(\BP_*\BP)$$ 
  %%%%%% Explain where I^3 comes from %%%%%%
  and the image of this cocycle in the cohomology of $I^2/I^3$ is detected by $q_0^2g_{j-2}$.

  Therefore, we have
  \[ d(t_1^{2^{j}} | T_{j+1}   - 2c) =  -d(t_1^{2^{j}})|T_{j+1} + t_1^{2^{j}} | d(T_{j+1}) - 2d(c) = -2T_j|T_{j+1} - 2d(c) \]
  which is detected by $q_0^3g_{j-2}$ in the cohomology of $I^3/I^4$. 
  
  Choose $w_2 $ to be $-2c \in I \cdot \mathrm{cb}(\BP_*\BP)$. Then
  $$c_2 - d(w_2) = -2T_j|T_{j+1} - 2d(c) \in I^{3} \cdot \mathrm{cb}(\BP_*\BP),.$$
  This proves that $d_2(h_j^3) = 0$. Next, we go to step (4.3), we have
  $$e_3 = e_2 + w_2 = t_1^{2^{j}} | T_{j+1}   - 2c,$$ 
  $$c_3 = d(e_3) = -2T_j|T_{j+1} - 2d(c).$$
  Since $c_3$ already belongs to $I^{3} \cdot \mathrm{cb}(\BP_*\BP)$, we may choose $w_3=0$. This proves that $d_3(h_j^3) = 0$.  Next, we go to step (4.4), we have 
  $$e_4=e_3= t_1^{2^{j}} | T_{j+1}   - 2c,$$
  $$c_4 = c_3 = -2T_j|T_{j+1} - 2d(c).$$
  Since $c_4$ is detected by $q_0^3g_{j-2}$, this proves that $d_4(h_j^3) = q_0^3 g_{j-2}$.
 
  (Note that we do not discuss the existence a $w_4$ such that 
  $d(w_4) - c_4 \in I^{4} \cdot \mathrm{cb}(\BP_*\BP).$ Therefore we do not conclude $q_0^3 g_{j-2}$ is nonzero even on the $E_2$-page. For this we need the material from \Cref{sec:cess}.)
  
  In sum we have for $j \geq 5$,
  \[ d_2(h_j^3) = 0, \quad d_3(h_j^3) = 0, \quad d_4(h_j^3) = q_0^3 g_{j-2}.\]
  
  % Now we can use the conclusion of \Cref{lem:prod-big} to say that this cocyle is in $I^3$ and 
  % its image in the cohomology of $I^3/I^4$ is detected by $q_0^3g_{j-2}$.
  
  % In order to finish the proof we must show that $q_0^3g_{j-2}$ is non-zero on the $E_4$-page of the algebraic Novikov spectral sequence.
  % From \Cref{???} we know enough about the $E_2$-page of the algebraic Novikov spectral sequence to conclude that this class survives as desired.
\end{proof}

% At this point we can make a first pass at running the algorithm laid out above (enough to motivate the next steps at least). For our lift of $h_j^3$ we pick the following class: $t_1^{2^{j}} | T_{j+1} $. 
% From this we obtain
% \begin{align} d(t_1^{2^{j}} | T_{j+1}) = 2 T_j | T_{j+1} \label{eqn:diff-simple} \end{align} %\TODO{signs}
% which allows us to read off that if the product $\vartheta_j\vartheta_{j+1}$ is detected by a class $c$ in the algebraic Novikov spectral sequence, then $d_*(h_j^3) = q_0c$. Note that this may not in fact determine the differential since on the relevant page of the algebraic Novikov spectral sequence this class may already be zero\footnote{Thankfully this isn't how things shake out.}.

% ------------------------------------------

% With all the preliminaries out of the way we can now prove the main result of this section.
% Our proof  and uses the identification of the product $\vartheta_j\vartheta_{j+1}$ given in \Cref{subsec:product-reln} in an essential way.

% \begin{prop}
%   For $j \geq 5$, there are algebraic Novikov differentials
%   \[ d_4(h_j^3) = q_0^3 g_{j-2}. \]
% \end{prop}

\subsection{A key product relation}
\label{subsec:product-reln}\ 

% Through \Cref{eqn:diff-simple} we have related the algebraic Novikov differential we wish to prove to the fate of the product $\vartheta_j\vartheta_{j+1}$ in the algebraic Novikov spectral sequence. In this subsection we identify the class which detects this cocycle. 

% Before proceeding further it is profitable to unwind the precise meaning of this lemma a bit.
% Recall that we have chosen cocycles $T_j$ and $T_{j+1}$ representing $\vartheta_j$ and $\vartheta_{j+1}$.
% What we will actually prove is the following more precise version of \Cref{lem:prod-small}
% The product $T_j|T_{j-1}$ is non-zero modulo the augmentation ideal, so the statement that $\vartheta_j\vartheta_{j-1}$ is detected by $q_0^2g_{j-2}$ means that there exists a correction term $c$ such that
% $ T_j|T_{j-1} + d(c) $ lifts to the square of the augmentation ideal and it is detected on the associated graded by $q_0^2g_{j-2}$.

In this subsection we prove \Cref{lem:prod-small}, identifying the class detecting the product $\vartheta_j\vartheta_{j+1}$ in the algebraic Novikov spectral sequence in the following refined form in \Cref{lem:prod-big}.

\begin{lem} \label{lem:prod-big}
  For $j \geq 5$, there exists a correction term $c_j$ such that $T_j|T_{j+1} + d(c_j)$ is zero modulo $(4, v_1^4)$.
  Moreover, the image of this cocycle in $I^2/I^3$ is detected by $q_0^2g_{j-2}$.
\end{lem}

%\begin{rmk}
%  This lemma is sufficient to conclude that $q_0^2g_{j-2}$ is a permanent cycle in the algebraic Novikov spectral sequence, but does not guarantee that this class is non-zero (even on the $E_2$-page).
%  For this we need the material from \Cref{sec:cess}.
%\end{rmk}

  %more specifically \Cref{Ext h03gn} and \Cref{d3 mot ass ctau group}. 
  % However, from our computation of the $E_2$-page of the algebriac Novikov spectral sequence in \Cref{} we 
  % We will return to this point in the next subsection.\todo{Update this with proper reference.}

\begin{dfn} \label{dfn:cj}
The correction term $c_j$ we use is
\begin{align*}
  c_j &\coloneqq [t_1^{2^{j-1}} | t_1^{3\cdot2^{j-1}} | t_1^{2^{j}}]
        + 7 [t_2^{2^{j-1}} | t_1^{2^{j-1}} | t_1^{2^{j}}] \\
      & + 2 [t_2^{2^{j-2}} | t_2^{2^{j-2}} | t_2^{2^{j-1}}] \\
      &+ 2 [t_1^{2^{j-2}} | t_1^{2^{j-1}} | t_2^{2^{j-1}}] \cdot \left( [t_2^{2^{j-2}}|1|1] + [1|t_2^{2^{j-2}}|1] \right) \\
      & + 2 \left([t_1^{2^{j-1}}|t_1^{2^{j-1}}|t_1^{2^{j-1}}] + [t_1^{2^{j-2}}|t_1^{2^{j-2}}|t_1^{2^{j}}] \right) \cdot \left( [t_2^{2^{j-1}}|1|1] + [1|t_2^{2^{j-1}}|1] + [1|1|t_2^{2^{j-1}}] \right) \\
      & + 2 [t_1^{3 \cdot 2^{j-2}}|t_1^{5 \cdot 2^{j-2}}|t_1^{2^{j}}] \\
      & + 2 [t_1^{2^{j}}|t_1^{3 \cdot 2^{j-1}}|t_1^{2^{j-1}}] + 2 [t_1^{3 \cdot 2^{j-1}}|t_1^{2^{j}}|t_1^{2^{j-1}}] \\
      & + 2 [t_1^{2^{j-2}}|t_2^{2^{j-2}}|t_1^{2^{j+1}}] + 2 [t_1^{2^{j-1}}|t_2^{2^{j-1}}|t_1^{2^{j}}].
\end{align*}
\end{dfn}

%     + 7 t_1^{2^{n+2}} | t_1^{2^{n+1}} | t_2^{2^{n+1}} \\
%   & + 2 t_1^{2^{n+1}} t_2^{2^{n+1}} | t_1^{2^{n+1}} | t_1^{2^{n+1}} 
%     + 2 t_1^{2^{n+1}} | t_1^{2^{n+1}} t_2^{2^{n+1}} | t_1^{2^{n+1}} 
%     + 2 t_1^{2^{n+1}} | t_1^{2^{n+1}} | t_1^{2^{n+1}} t_2^{2^{n+1}} \\
%   & + 2 t_1^{2^{n+1}} | t_1^{2^{n+2}} | t_2^{2^{n+1}} 
%    + 2 t_1^{2^{n+1}} | t_2^{2^{n+1}} | t_1^{2^{n+2}} 
%    + 2 t_1^{2^{n+1}} | t_1^{3\cdot2^{n+1}} | t_1^{2^{n+2}} \\
%   & + 2 t_2^{2^{n+1}} | t_1^{2^{n+1}} t_2^{2^n} | t_1^{2^n} 
%    + 2 t_2^{2^{n+1}} | t_1^{2^{n+1}} | t_2^{2^n} t_1^{2^n} 
%    + 2 t_2^{2^{n+1}} | t_2^{2^n} | t_2^{2^n} \\
%   & + 2 t_1^{2^{n+2}} t_2^{2^{n+1}} | t_1^{2^n} | t_1^{2^n} 
%    + 2 t_1^{2^{n+3}} | t_2^{2^n} | t_1^{2^n} 
%    + 2 t_1^{2^{n+2}} | t_1^{2^n} t_2^{2^{n+1}} | t_1^{2^n} \\
%   & + 2 t_1^{2^{n+2}} | t_1^{2^n} | t_2^{2^{n+1}} t_1^{2^n}
%    + 2 t_1^{2^{n+2}} | t_1^{5\cdot2^n} | t_1^{3\cdot 2^n}.
% \end{align*}

We will return to the issue of how we produced this correction term in Subsection~\ref{subsec:correction} and at that point the reason for our somewhat unnatural choice of line breaks will become clear. The remainder of the proof involves computing $T_j|T_{j+1} + d(c_j)$. We will do this by showing that this cocycle stabilizes for $j \gg 0$ and then either using a small computer program or by hand to verify the desired conclusion holds in a single (now universal) case directly.

We will need to give formulas for the cobar differential on certain powers of $t_1$ and $t_2$ modulo $(8,v_1^4)$. The key observation here is the following congruence of binomial coefficients
\[ \binom{a \cdot 2^{n+k}}{b \cdot 2^{n} + c} \equiv \begin{cases} \binom{a \cdot 2^k}{b} & c=0 \\ 0 & c=1,\dots,2^n-1 \end{cases} \pmod{2^{k+1}}. \]
This congruence allows us to replace cocycles which, a priori, would have a number of terms depending on $j$ with cocycles that are independent of $j$ in a certain sense. 

\begin{exm} \label{exm:t1powermod16}
As an example to illustrate the point we consider the following calculation mod 16 (for $k=3$):

\begingroup
\allowdisplaybreaks
\begin{align*}
  d(t_1^{2^{n+3}}) &= \sum_{i=1}^{2^{n+3} -1} \binom{2^{n+3}}{i} [t_1^i | t_1^{2^{n+3} - i}] \\
  &\equiv \sum_{j=1}^{7} \binom{2^{n+3}}{j \cdot 2^n} [t_1^{j \cdot 2^n} | t_1^{2^{n+3} - j \cdot 2^n}] \\
  &\equiv \binom{2^{n+3}}{1 \cdot 2^n} [t_1^{1 \cdot 2^{n}} | t_1^{7 \cdot 2^{n}}] + \binom{2^{n+3}}{2 \cdot 2^n} [t_1^{2 \cdot 2^{n}} | t_1^{6 \cdot 2^{n}}] + \binom{2^{n+3}}{3\cdot 2^n} [t_1^{3 \cdot 2^{n}} | t_1^{5 \cdot 2^{n}}] + \binom{2^{n+3}}{4\cdot 2^n} [t_1^{4 \cdot 2^{n}} | t_1^{4 \cdot 2^{n}}] \\
  &\quad + \binom{2^{n+3}}{5\cdot 2^n} [t_1^{5 \cdot 2^{n}} | t_1^{3 \cdot 2^{n}}] + \binom{2^{n+3}}{6\cdot 2^n} [t_1^{6 \cdot 2^{n}} | t_1^{2 \cdot 2^{n}}] + \binom{2^{n+3}}{7\cdot 2^n} [t_1^{7 \cdot 2^{n}} | t_1^{1 \cdot 2^{n}}]  \\
                   &\equiv \binom{8}{1} [t_1^{1 \cdot 2^{n}} | t_1^{7 \cdot 2^{n}}] + \binom{8}{2} [t_1^{2 \cdot 2^{n}} | t_1^{6 \cdot 2^{n}}] + \binom{8}{3} [t_1^{3 \cdot 2^{n}} | t_1^{5 \cdot 2^{n}}] + \binom{8}{4} [t_1^{4 \cdot 2^{n}} | t_1^{4 \cdot 2^{n}}] \\
                   &\quad + \binom{8}{5} [t_1^{5 \cdot 2^{n}} | t_1^{3 \cdot 2^{n}}] + \binom{8}{6} [t_1^{6 \cdot 2^{n}} | t_1^{2 \cdot 2^{n}}] + \binom{8}{7} [t_1^{7 \cdot 2^{n}} | t_1^{1 \cdot 2^{n}}]  \\
                 &\equiv 8[t_1^{1 \cdot 2^{n}} | t_1^{7 \cdot 2^{n}}] + 12[t_1^{2 \cdot 2^{n}} | t_1^{6 \cdot 2^{n}}] + 8[t_1^{3 \cdot 2^{n}} | t_1^{5 \cdot 2^{n}}] + 6 [t_1^{4 \cdot 2^{n}} | t_1^{4 \cdot 2^{n}}] \\
                 &\quad + 8 [t_1^{5 \cdot 2^{n}} | t_1^{3 \cdot 2^{n}}] + 12 [t_1^{6 \cdot 2^{n}} | t_1^{2 \cdot 2^{n}}] + 8 [t_1^{7 \cdot 2^{n}} | t_1^{1 \cdot 2^{n}}]  \pmod{16}
\end{align*}
  % &\equiv 8 [t_1^{1 \cdot 2^{n-3}} | t_1^{7 \cdot 2^{n-3}}] + 12 [t_1^{2^{n-2}} | t_1^{3 \cdot 2^{n-2}}] + 8 [t_1^{3 \cdot 2^{n-3}} | t_1^{5 \cdot 2^{n-3}}] + 6 [t_1^{2^{n-1}} | t_1^{2^{n-1}}] \\ &\quad + 8 [t_1^{5 \cdot 2^{n-3}} | t_1^{3 \cdot 2^{n-3}}] + 12 [t_1^{3 \cdot 2^{n-2}} | t_1^{2^{n-2}}] + 8 [t_1^{7 \cdot 2^{n-3}} | t_1^{1 \cdot 2^{n-3}}] \pmod{16}

We make two observations from this calculation which holds for $n \geq 0$: 
\begin{itemize}
\item For $m \geq 3$, to obtain the expression $d(t_1^{2^{m+1}})$ mod 16 from $d(t_1^{2^{m}})$ mod 16, we can take a termwise square for each monomial (with coefficient 1)  in $t_1$ and then take the sum. We say this expansion \emph{stabilizes} for $m \geq 3$.

\item We can also read off the expression $d(t_1^{2^{m}})$ mod 8, and that it stabilizes for $m \geq 2$.
\end{itemize}  
  
\end{exm}

\begin{exm}\label{exm:vartheta-hj2}
  Consider the mod 4 ($k=1$) version of these congruences:
$$d(t_1^{2^{j}}) = \sum_{i=1}^{2^{j} -1} \binom{2^{j}}{i} [t_1^i | t_1^{2^{j} - i}] \equiv \binom{2^{j}}{2^{j-1}}[t_1^{2^{j-1}} | t_1^{2^{j-1}}] \equiv 2[t_1^{2^{j-1}} | t_1^{2^{j-1}}] \pmod{4}$$
  we can determine the image of $\vartheta_j$ under the quotient map to the cohomology of $\P$.
  \[ T_j = \frac{1}{2} d(t_1^{2^j}) \equiv [t_1^{2^{j-1}} | t_1^{2^{j-1}}] \pmod{2} \]
  Under the isomorphism between $\BP_*\BP/I$ and $\P$ followed by the isomorphism with $\A$ the class
  $[t_1^{2^{j-1}}]$ is detected by $h_{j-1}$ and therefore the image of $\vartheta_j$ maps to $h_{j-1}^2$.
\end{exm}

% From the fact that $2T_j$ is killed by the algebraic Novikov differential on $h_{j+1}$ we can infer that $T_j$ provides a choice of Kervaire invariant one class.
% More formally, we have
% \[ T_j \equiv \frac{1}{2} \binom{2^j}{2^{j-1}} [t_1^{2^{j-1}} | t_1^{2^{j-1}}] \equiv [t_1^{2^{j-1}} | t_1^{2^{j-1}}] \pmod{2} \]
% which implies that $\vartheta_j$ maps to $h_j^2$ under the Thom reduction map---justifying the name.

\begin{exm} \label{exm:tjtj1mod8}
This same formula also allows us to rewrite $T_j|T_{j+1}$ mod $8$ in terms of powers of $t_1$ for $j \geq 3$. From the formula in \Cref{exm:t1powermod16}, we have
 \begin{align*}
T_j & = \frac{1}{2} d(t_1^{2^{j}}) \\
&\equiv 4[t_1^{1 \cdot 2^{j-3}} | t_1^{7 \cdot 2^{j-3}}] + 6[t_1^{2 \cdot 2^{j-3}} | t_1^{6 \cdot 2^{j-3}}] + 4[t_1^{3 \cdot 2^{j-3}} | t_1^{5 \cdot 2^{j-3}}] + 3[t_1^{4 \cdot 2^{j-3}} | t_1^{4 \cdot 2^{j-3}}] \\
                 &\quad + 4[t_1^{5 \cdot 2^{j-3}} | t_1^{3 \cdot 2^{j-3}}] + 6[t_1^{6 \cdot 2^{j-3}} | t_1^{2 \cdot 2^{j-3}}] + 4[t_1^{7 \cdot 2^{j-3}} | t_1^{1 \cdot 2^{j-3}}]  \pmod{8}
     \end{align*}             
                 
{\footnotesize
  \begin{align*}
    T_j|T_{j+1} & \equiv 4 [t_1^{1 \cdot 2^{j-3}} | t_1^{7 \cdot 2^{j-3}}|T_{j+1}] &+ 6 [t_1^{2 \cdot 2^{j-3}} | t_1^{6 \cdot 2^{j-3}}|T_{j+1}] &+ 4 [t_1^{3 \cdot 2^{j-3}} | t_1^{5 \cdot 2^{j-3}}|T_{j+1}] \\ 
    &+ 3 [t_1^{4 \cdot 2^{j-3}} | t_1^{4 \cdot 2^{j-3}}|T_{j+1}] \\
                 &+ 4 [t_1^{5 \cdot 2^{j-3}} | t_1^{3 \cdot 2^{j-3}}|T_{j+1}] &+ 6 [t_1^{6 \cdot 2^{j-3}} | t_1^{2 \cdot 2^{j-3}}|T_{j+1}] &+ 4 [t_1^{7 \cdot 2^{j-3}} | t_1^{1 \cdot 2^{j-3}}|T_{j+1}]  \\   
   & \equiv 4 [t_1^{1 \cdot 2^{j-3}} | t_1^{7 \cdot 2^{j-3}} | t_1^{8 \cdot 2^{j-3}} | t_1^{8 \cdot 2^{j-3}} ] \\
    &+ (4 [t_1^{2 \cdot 2^{j-3}} | t_1^{6 \cdot 2^{j-3}} | t_1^{4 \cdot 2^{j-3}} | t_1^{12 \cdot 2^{j-3}}]
    &+ 2 [t_1^{2 \cdot 2^{j-3}} | t_1^{6 \cdot 2^{j-3}} | t_1^{8 \cdot 2^{j-3}} | t_1^{8 \cdot 2^{j-3}}]
    &+ 4 [t_1^{2 \cdot 2^{j-3}} | t_1^{6 \cdot 2^{j-3}} | t_1^{12 \cdot 2^{j-3}} | t_1^{4 \cdot 2^{j-3}}]) \\
    &+ 4 [t_1^{3 \cdot 2^{j-3}} | t_1^{5 \cdot 2^{j-3}} | t_1^{8 \cdot 2^{j-3}} | t_1^{8 \cdot 2^{j-3}}] \\
    &+ (4 [t_1^{4 \cdot 2^{j-3}} | t_1^{4 \cdot 2^{j-3}} | t_1^{2 \cdot 2^{j-3}} | t_1^{14 \cdot 2^{j-3}}]
    &+ 2 [t_1^{4 \cdot 2^{j-3}} | t_1^{4 \cdot 2^{j-3}} | t_1^{4 \cdot 2^{j-3}} | t_1^{12 \cdot 2^{j-3}}]
    &+ 4 [t_1^{4 \cdot 2^{j-3}} | t_1^{4 \cdot 2^{j-3}} | t_1^{6 \cdot 2^{j-3}} | t_1^{10 \cdot 2^{j-3}}] \\
    &+ 1 [t_1^{4 \cdot 2^{j-3}} | t_1^{4 \cdot 2^{j-3}} | t_1^{8 \cdot 2^{j-3}} | t_1^{8 \cdot 2^{j-3}}] \\
    &+ 4 [t_1^{4 \cdot 2^{j-3}} | t_1^{4 \cdot 2^{j-3}} | t_1^{10 \cdot 2^{j-3}} | t_1^{6 \cdot 2^{j-3}}]
    &+ 2 [t_1^{4 \cdot 2^{j-3}} | t_1^{4 \cdot 2^{j-3}} | t_1^{12 \cdot 2^{j-3}} | t_1^{4 \cdot 2^{j-3}}]
    &+ 4 [t_1^{4 \cdot 2^{j-3}} | t_1^{4 \cdot 2^{j-3}} | t_1^{14 \cdot 2^{j-3}} | t_1^{2 \cdot 2^{j-3}}])\\
    &+ 4 [t_1^{5 \cdot 2^{j-3}} | t_1^{3 \cdot 2^{j-3}} | t_1^{8 \cdot 2^{j-3}} | t_1^{8 \cdot 2^{j-3}}] \\
    &+ (4 [t_1^{6 \cdot 2^{j-3}} | t_1^{2 \cdot 2^{j-3}} | t_1^{4 \cdot 2^{j-3}} | t_1^{12 \cdot 2^{j-3}}]
    &+ 2 [t_1^{6 \cdot 2^{j-3}} | t_1^{2 \cdot 2^{j-3}} | t_1^{8 \cdot 2^{j-3}} | t_1^{8 \cdot 2^{j-3}}]
    &+ 4 [t_1^{6 \cdot 2^{j-3}} | t_1^{2 \cdot 2^{j-3}} | t_1^{12 \cdot 2^{j-3}} | t_1^{4 \cdot 2^{j-3}}])\\
    &+ 4 [t_1^{7 \cdot 2^{j-3}} | t_1^{1 \cdot 2^{j-3}} | t_1^{8 \cdot 2^{j-3}} | t_1^{8 \cdot 2^{j-3}}] & &\pmod{8}.
    \end{align*}}

\end{exm}

\begin{exm} \label{exm:formulas}
Some further formulas that can be similarly verified for $n \geq 0$ are:

\begin{align*}
  d(t_1^{3 \cdot 2^{n+2}}) &= \sum_{i=1}^{3 \cdot 2^{n+2} - 1} \binom{3 \cdot 2^{n+2}}{i} [t_1^i | t_1^{3 \cdot 2^{n+2} - i}] 
  \equiv \sum_{j=1}^{3 \cdot 2^{2} -1} \binom{3 \cdot 2^{n+2}}{j} [t_1^{j \cdot 2^{n}} | t_1^{(3 \cdot 2^{2} - j) \cdot 2^{n}}] \\
  &\equiv 4 [t_1^{1 \cdot 2^{n}} | t_1^{11 \cdot 2^{n}}]
  + 2 [t_1^{2 \cdot 2^{n}} | t_1^{10 \cdot 2^{n}}]
  + 4 [t_1^{3 \cdot 2^{n}} | t_1^{9 \cdot 2^{n}}]
  + 7 [t_1^{4 \cdot 2^{n}} | t_1^{8 \cdot 2^{n}}]
  + 4 [t_1^{6 \cdot 2^{n}} | t_1^{6 \cdot 2^{n}}] \\
  &\quad 
  + 7 [t_1^{8 \cdot 2^{n}} | t_1^{4 \cdot 2^{n}}]
  + 4 [t_1^{9 \cdot 2^{n}} | t_1^{3 \cdot 2^{n}}]
  + 2 [t_1^{10 \cdot 2^{n}} | t_1^{2 \cdot 2^{n}}]
  + 4 [t_1^{11 \cdot 2^{n}} | t_1^{1 \cdot 2^{n}}] \pmod{8},
\end{align*}

We observe that the expression $d(t_1^{3 \cdot 2^{m}})$ mod 8 stabilizes for $m \geq 2$.

\begin{align*}
  d(t_1^{5 \cdot 2^{n+1}})
  &= \sum_{i=1}^{5 \cdot 2^{n+1} - 1} \binom{5 \cdot 2^{n+1}}{i} [t_1^i | t_1^{5 \cdot 2^{n+1} - i}]
  \equiv \sum_{j=1}^{5 \cdot 2 -1} \binom{5 \cdot 2}{j} [t_1^{j \cdot 2^{n}} | t_1^{(5 \cdot 2 - j) \cdot 2^{n}}] \\
  &\equiv 2 [t_1^{1 \cdot 2^{n}} | t_1^{9 \cdot 2^{n}}]
  + [t_1^{2 \cdot 2^{n}} | t_1^{8 \cdot 2^{n}}]
  + 2 [t_1^{4 \cdot 2^{n}} | t_1^{6 \cdot 2^{n}}]
  + 2 [t_1^{6 \cdot 2^{n}} | t_1^{4 \cdot 2^{n}}] \\
  &\quad + [t_1^{8 \cdot 2^{n}} | t_1^{2 \cdot 2^{n}}]
  + 2 [t_1^{9 \cdot 2^{n}} | t_1^{1 \cdot 2^{n}}] \pmod{4},
\end{align*}

We observe that the expression $d(2 \cdot t_1^{5 \cdot 2^{m}})$ mod 8 stabilizes for $m \geq 1$.

{\footnotesize
\begin{align*}
  \Delta (t_2^{2^{n+4}}) &= \left( [t_2|1] - [t_1|t_1^2] + v_1 [t_1|t_1] + [1|t_2] \right)^{2^{n+4}} \\
                         &= \sum_{|J| = 4,\ \mathrm{deg}(J) = 2^{n+4}} (-1)^{j_2} \binom{2^{n+4}}{ J } v_1^{j_3} [ t_1^{j_2 + j_3} t_2^{j_1} | t_1^{2j_2 + j_3} t_2^{j_4} ] \\
                         &\equiv  \sum_{|I|=4,\ \mathrm{deg}(I) = 4} (-1)^{i_2 \cdot 2^{n+2}} \binom{4}{I} v_1^{i_3 \cdot 2^{n+2}} [ t_1^{(i_2 + i_3) \cdot 2^{n+2}} t_2^{i_1 \cdot 2^{n+2}} | t_1^{(2i_2 + i_3) \cdot 2^{n+2}} t_2^{i_4 \cdot 2^{n+2}} ] \\
                         &\equiv  \sum_{|I|=4,\ \mathrm{deg}(I) = 4,\ i_3=0} \binom{4}{I} [ t_1^{i_2 \cdot 2^{n+2}} t_2^{i_1 \cdot 2^{n+2}} | t_1^{i_2 \cdot 2^{n+3}} t_2^{i_4 \cdot 2^{n+2}} ] \\
                         &\equiv 1 [t_2^{16 \cdot 2^{n}}|1] + 4 [t_1^{4 \cdot 2^{n}}t_2^{12 \cdot 2^{n}}| t_1^{8 \cdot 2^{n}}] + 4 [t_2^{12 \cdot 2^{n}} | t_2^{4 \cdot 2^{n}}] + 6 [t_1^{8 \cdot 2^{n}}t_2^{8 \cdot 2^{n}}| t_1^{16 \cdot 2^{n}}] \\
                         &\quad+ 4 [t_1^{4 \cdot 2^{n}}t_2^{8\cdot 2^{n}} | t_1^{8 \cdot 2^{n}} t_2^{4 \cdot 2^{n}}] + 6 [t_2^{8 \cdot 2^{n}} | t_2^{8 \cdot 2^{n}}] + 4 [t_1^{12 \cdot 2^{n}}t_2^{4 \cdot 2^{n}}| t_1^{24 \cdot 2^{n}}] + 4 [t_1^{8 \cdot 2^{n}}t_2^{4 \cdot 2^{n}} | t_1^{16 \cdot 2^{n}} t_2^{4 \cdot 2^{n}}] \\
                         &\quad+ 4 [t_1^{4 \cdot 2^{n}}t_2^{4 \cdot 2^{n}} | t_1^{8 \cdot 2^{n}} t_2^{8 \cdot 2^{n}}] + 4 [t_2^{4 \cdot 2^{n}} | t_2^{12 \cdot 2^{n}}] + 1 [t_1^{16 \cdot 2^{n}}| t_1^{32 \cdot 2^{n}}] + 4 [t_1^{12 \cdot 2^{n}}| t_1^{24 \cdot 2^{n}} t_2^{4 \cdot 2^{n}}] \\
                         &\quad+ 6 [t_1^{8 \cdot 2^{n}}| t_1^{16 \cdot 2^{n}} t_2^{8 \cdot 2^{n}}] + 4 [t_1^{4 \cdot 2^{n}}| t_1^{8 \cdot 2^{n}} t_2^{12 \cdot 2^{n}}] + 1 [1| t_2^{16 \cdot 2^{n}}] \pmod{8,v_1^4}.
\end{align*}}
Here $J=(j_1,j_2,j_3,j_4), I = (i_1,i_2,i_3,i_4)$. Since for each monomial the minimal 2-adic evaluation of the power of $t_1, t_2$ in the above expression is $n+2$, and to get ride of $(-1)^{i_2}\cdot 2^{n+2}$ we would like to have $n+2\geq 1$, we observe that the expression $d(t_2^{2^m})$ mod $(8, v_1^4)$ stabilizes for $m \geq 3$.
\end{exm}

\begin{lem}\label{lem:cobar-stabilize}
  The value of $T_j|T_{j+1} + d(c_j)$ mod $(8,v_1^4)$ stabilizes for $j \geq 5$ in the sense that after expanding this cocycle as a sum of monomials (with coefficient 1) in $t_1$ and $t_2$ we can increase $j$ by taking a termwise square. 
\end{lem}

\begin{proof}
From \Cref{exm:tjtj1mod8} we see that the expression $T_j|T_{j+1}$ mod 8 (and therefore mod $(8,v_1^4)$) stabilizes for $j \geq 3$. From \Cref{dfn:cj}, we see that $c_j$ can be expressed as the sum of products and tensors of the following terms (powers of $t_1,t_2$ are written in the form $a \cdot 2^{j-2}$ for easier comparison):
\[ {t_1^{1 \cdot 2^{j-2}}},\ t_1^{2 \cdot 2^{j-2}},\ {t_1^{3 \cdot 2^{j-2}}},\ t_1^{4 \cdot 2^{j-2}},\ 2 \cdot {t_1^{5 \cdot 2^{j-2}}},\ t_1^{6 \cdot 2^{j-2}},\ {t_1^{8 \cdot 2^{j-2}}},\ {t_2^{1 \cdot 2^{j-2}}},\ t_2^{2 \cdot 2^{j-2}}. \]
 % \[ \underline{t_1^{1 \cdot 2^{j-2}}},\ t_1^{2 \cdot 2^{j-2}},\ \underline{t_1^{3 \cdot 2^{j-2}}},\ t_1^{4 \cdot 2^{j-2}},\ \underline{t_1^{5 \cdot 2^{j-2}}},\ t_1^{6 \cdot 2^{j-2}},\ \underline{t_1^{8 \cdot 2^{j-2}}},\ \underline{t_2^{1 \cdot 2^{j-2}}},\ t_2^{2 \cdot 2^{j-2}} \]
%  where the underlined terms only appear with coefficient 2.
Therefore, $d(c_j)$ can be expressed as the sum of products and tensors of these terms and differentials on them. These terms stabilize for $j \geq 2$. As for their differentials,
for powers of $t_1$, from \Cref{exm:t1powermod16}, we see that the expression $d(t_1^{2^{j-2}})$ mod 8 stabilizes for $j \geq 4$; 
  From \Cref{exm:formulas}, the expressions $d(t_1^{3 \cdot 2^{j-2}})$ and $d(2 \cdot t_1^{5 \cdot 2^{j-2}})$ mod 8 stabilizes for $j \geq 4$.  
  For powers of $t_2$, from \Cref{exm:formulas}, the expression $d(t_2^{2^{j-2}})$ mod $(8, v_1^4)$ stabilizes for $j \geq 5$. 
 
  In sum, we have that the expression $T_j|T_{j+1} + d(c_j)$ mod $(8,v_1^4)$ stabilizes for $j \geq 5$.
 
\end{proof}

\begin{proof}[Proof of \Cref{lem:prod-big}.]
  Using either the program provided in \Cref{app:code} or directly by hand one may use the formulas above to give an explicit expansion of the cocycle $T_5|T_{6} + d(c_5)$ modulo $(8,v_1^4)$. Applying \Cref{lem:cobar-stabilize} we can then covert this into an explicit expansion of $T_j|T_{j+1} + d(c_j)$ valid for $j \geq 5$.\footnote{Note that if one does this expansion by hand it is actually easier to skip over \Cref{lem:cobar-stabilize} and just use the formulas given prior to that lemma directly.}

  In the tables below we provide expansions of the terms which make up $T_j|T_{j+1} + d(c_j)$ modulo $(8,v_1^4)$.
  The columns are organized by the May filtration of the monomials (see below for a review of our use of the May filtration).
  
  Throughout these tables we use the following compact notation:
  we write $k_1$ for $t_1^{k \cdot 2^n}$ and $k_2$ for $t_2^{k \cdot 2^n}$ and here $n=j-3$. So for example $4[4_2|1_1|3_1|8_1]$ stands for 
  $$4[t_2^{4\cdot 2^n}|t_1^{1\cdot 2^n}|t_1^{3\cdot 2^n}|t_1^{8\cdot 2^n}] = 4[t_2^{4\cdot 2^{j-3}}|t_1^{1\cdot 2^{j-3}}|t_1^{3\cdot 2^{j-3}}|t_1^{8\cdot 2^{j-3}}] = 4[t_2^{2^{j-1}}|t_1^{2^{j-3}}|t_1^{3\cdot 2^{j-3}}|t_1^{2^{j}}] .$$

\input{tablev2.tex}
  Explicit calculation of the first row $T_j|T_{j+1}$ mod 8 was already carried out in \Cref{exm:tjtj1mod8}.
  We provide detailed calculation of the second row and the fifth row as typical examples and let the reader to check the rest of the table. 
 
\begin{itemize}  
\item The second row $d( [ 4_1 | 12_1 | 8_1 ] )$ mod 8: 
\end{itemize}
\begin{align*}
d( [ 4_1 | 12_1 | 8_1 ] ) & = - [d(4_1) | 12_1 | 8_1 ] &+ [4_1 | d(12_1) | 8_1] &- [ 4_1 | 12_1 | d(8_1) ] \\
 & \equiv - 4 [1_1 | 3_1 | 12_1 | 8_1] &- 6 [2_1 | 2_1 | 12_1 | 8_1] &- 4 [3_1 | 1_1 | 12_1 | 8_1] && (by \ \Cref{exm:t1powermod16}) \\
 & + 4 [4_1 | 1_1 | 11_1 | 8_1] &+ 2 [4_1 | 2_1 | 10_1 | 8_1] &+ 4 [4_1 | 3_1 | 9_1 | 8_1] \\
 & + 7 [4_1 | 4_1 | 8_1 | 8_1] &+ 4 [4_1 | 6_1 | 6_1 | 8_1] &+ 7 [4_1 | 8_1 | 4_1 | 8_1] \\
 & + 4 [4_1 | 9_1 | 3_1 | 8_1] &+ 2 [4_1 | 10_1 | 2_1 | 8_1] &+ 4 [4_1 | 11_1 | 1_1 | 8_1]
 && (by \ \Cref{exm:formulas}) \\
 & - 4 [ 4_1 | 12_1 | 2_1 | 6_1 ] &- 6 [ 4_1 | 12_1 | 4_1 | 4_1 ] &- 4 [ 4_1 | 12_1 | 6_1 | 2_1 ] && (by \ \Cref{exm:t1powermod16})
\end{align*}
Re-organize these 15 terms by May filtration, we have the following in May filtration 4:
$$7 [4_1 | 4_1 | 8_1 | 8_1], \ 7 [4_1 | 8_1 | 4_1 | 8_1],$$
in May filtration 5:
$$2 [2_1 | 2_1 | 12_1 | 8_1], \  2 [4_1 | 2_1 | 10_1 | 8_1], \ 2 [4_1 | 10_1 | 2_1 | 8_1], \ 2 [ 4_1 | 12_1 | 4_1 | 4_1 ],$$
and in May filtration 6:
$$4 [1_1 | 3_1 | 12_1 | 8_1], \  4 [3_1 | 1_1 | 12_1 | 8_1], \ 4 [4_1 | 1_1 | 11_1 | 8_1], \ 4 [4_1 | 3_1 | 9_1 | 8_1], \ 4 [4_1 | 6_1 | 6_1 | 8_1],$$  
$$4 [4_1 | 9_1 | 3_1 | 8_1], \ 4 [4_1 | 11_1 | 1_1 | 8_1], \ 4 [ 4_1 | 12_1 | 2_1 | 6_1 ], \ 4 [ 4_1 | 12_1 | 6_1 | 2_1 ].$$

\ \ \\

\begin{itemize}  
\item The fifth row $d(2[2_1 2_2 | 4_1 | 4_2])$ mod $(8, v_1^4)$:
\end{itemize}

We first expand $d(2_1 2_2)$ mod $(4, v_1^4)$:
\begin{align*}
d(2_1 2_2) &= \Delta(2_1 2_2) - [1|2_1 2_2] - [2_1 2_2|1]\\
 & = \Delta(2_1) \cdot \Delta(2_2) - [1|2_1 2_2] - [2_1 2_2|1]\\
 & \equiv ([1|2_1] + 2[1_1|1_1] + [2_1|1]) \cdot ([1|2_2] + 2[1_11_2|2_1] + 2[1_2|1_2] + [2_1|4_1]+2[1_1|2_11_2]+[2_2|1])\\
 & - [1|2_1 2_2] - [2_1 2_2|1] \ \ \ (by \ \Cref{exm:t1powermod16} \ and \ \Cref{exm:formulas})\\
 & \equiv (2[1_11_2|4_1] + 2[1_2|1_22_1] + [2_1|6_1]+2[1_1|4_11_2]+[2_2|2_1]) \\
 &+ (2[1_1|1_12_2] + 2[3_1|5_1] + 2[1_12_2|1_1]) \\
 &+ ([2_1|2_2] + 2[3_11_2|2_1] + 2[2_11_2|1_2] + [4_1|4_1]+2[3_1|2_11_2])
\end{align*}
Then we have $d(2[2_1 2_2 | 4_1 | 4_2])$ mod $(8, v_1^4)$:
{\footnotesize
\begin{align*}
d(2[2_1 2_2 | 4_1 | 4_2]) & = - [2\cdot d([2_12_2]) | 4_1 | 4_2] + [2_12_2 | 2\cdot d(4_1) | 4_2] - [2_12_2 | 4_1| 2\cdot d(4_2)] \\
&  \equiv (-4[1_11_2|4_1| 4_1 | 4_2] -4[1_2|2_11_2| 4_1 | 4_2] -2 [2_1|6_1| 4_1 | 4_2]-4[1_1|4_11_2| 4_1 | 4_2]-2[2_2|2_1| 4_1 | 4_2]) \\
 &+ (-4[1_1|1_12_2| 4_1 | 4_2] -4[3_1|5_1| 4_1 | 4_2] -4[1_12_2|1_1]| 4_1 | 4_2) \\
 &+ (-2[2_1|2_2| 4_1 | 4_2] -4[3_11_2|2_1| 4_1 | 4_2] -4[2_11_2|1_2| 4_1 | 4_2] -2 [4_1|4_1| 4_1 | 4_2]-4[3_1|2_11_2| 4_1 | 4_2]) \\
 & + 4[2_12_2 | 2_1|2_1 | 4_2] \ \ (by \ \Cref{exm:t1powermod16})\\
 & - (4[2_12_2 | 4_1| 2_12_2|4_1] + 4[2_12_2 | 4_1| 2_2|2_2] +2[2_12_2 | 4_1| 4_1|8_1]+4[2_12_2 | 4_1| 2_1|4_12_2)]) \ \ (by \ \Cref{exm:formulas}).
\end{align*}}
Re-organize these 18 terms by May filtration, we have the following in May filtration 6:
$$6[4_1|4_1| 4_1 | 4_2],$$
in May filtration 7:
$$6[2_1|6_1| 4_1 | 4_2], \ 6[2_12_2 | 4_1| 4_1|8_1],$$
in May filtration 8:
$$6[2_2|2_1| 4_1 | 4_2], \ 4[3_1|5_1| 4_1 | 4_2], \ 6[2_1|2_2| 4_1 | 4_2],$$
and 4-multiples in May filtrations 9, 10 and 11.\\

The proof of \Cref{lem:prod-big} follows essentially from the following two facts about $T_j|T_{j+1} + d(c_j)$, which we extract by examining these tables:
  \begin{itemize}
  \item[(a)] $T_j|T_{j+1} + d(c_j)$ is zero modulo $(4,v_1^4)$ and
  \item[(b)] modulo terms of May filtration $< 12$ it is given by
    \[ [t_2^{2^{j-2}} | t_2^{2^{j-2}} | t_2^{2^{j-2}} | t_2^{2^{j-2}}]
    + [t_2^{2^{j-2}} | t_2^{2^{j-3}} | t_2^{2^{j-3}} | t_2^{2^{j-1}}]
    + [t_2^{2^{j-3}} | t_2^{2^{j-3}} | t_2^{2^{j-2}} | t_2^{2^{j-1}}]. \]
  \end{itemize}

  Part (a) is exactly the first part of the lemma.
  To check that part (a) is true, we sum within each column (fixed May filtration) and check that the sum is divisible by 4.
  
  In order to determine which class detects this cocycle in $I^2/I^3$ we first note that no $v_n$'s appear other than $v_0=2$. This means that if we divide this class by $4$, view it as an element of the cobar complex of $\P$, and at that point it detects a class $H$, then the original cocycle was detected by $q_0^2H$ in the algebraic Novikov spectral sequence.

  In order to determine the relevant class in the cohomology of $\P$ we make two observations.
  First, as a consequence of computations of $\Ext_{\A}^4$ (\cite{LinExt}) the only nontrivial element of the cohomology of $\P$ in this bidegree is $g_{j-2}$, which under the double up isomorphism corresponds to $g_{j-1}$ in $\Ext_{\A}^4$. Therefore it suffices to show that this class is detected by something nontrivial.
  Second, in the May spectral sequence $g_{j-1}$ is detected by $h_{2,j-1}^4$ (see \cite{TangoraExt, RavenelGreenBook} for example).
  In order to identify our class we would like to pass to the associated graded of the May filtration.
  
  Recall that $\Ext_{\A}$ can be computed via the cobar complex over the dual Steenrod algebra $\A$. The May filtration we use here is an increasing filtration that filters the cobar complex over $\A$ and in particular it assigns $\xi_i^{2^j}$ filtration $2i-1$ (see Page 69 of \cite{RavenelGreenBook}). It has the advantage that 
  $$\Delta(\xi_i^{2^j}) = 1|\xi_i^{2^j} + \xi_i^{2^j}|1 + \text{lower May filtration terms},$$
  so $\xi_i^{2^j}$ is primitive in the associated graded $E^0 \A$. 
We view the associated graded cobar complex as the May $E_0$-page, and the May $E_1$-page has the form
$$E_1 \cong \mathbb{F}_2[h_{i,j}| i\geq 1, j \geq 0],$$
where $h_{i,j}$ corresponds to $\xi_i^{2^j}$ with May filtration $2i-1$.

Although the $t_i$'s don't directly correspond to the $\xi_i^2$'s in $\P$---they differ by an action of the anti-involution, this action is the identity on the May $E_0$-page since the rest of the terms are in lower May filtration. 
  Using (b) we find that our class
      \[ [t_2^{2^{j-2}} | t_2^{2^{j-2}} | t_2^{2^{j-2}} | t_2^{2^{j-2}}]
    + [t_2^{2^{j-2}} | t_2^{2^{j-3}} | t_2^{2^{j-3}} | t_2^{2^{j-1}}]
    + [t_2^{2^{j-3}} | t_2^{2^{j-3}} | t_2^{2^{j-2}} | t_2^{2^{j-1}}]\]
   is sent to (up to lower May filtration terms)
    \[ [\xi_2^{2^{j-1}} | \xi_2^{2^{j-1}} | \xi_2^{2^{j-1}} | \xi_2^{2^{j-1}}]
    + [\xi_2^{2^{j-1}} | \xi_2^{2^{j-2}} | \xi_2^{2^{j-2}} | \xi_2^{2^{j}}]
    + [\xi_2^{2^{j-2}} | \xi_2^{2^{j-2}} | \xi_2^{2^{j-1}} | \xi_2^{2^{j}}] \]
   and is detected in the May $E_1$-page by 
  \[ h_{2,j-1}^4 + h_{2,j}h_{2,j-2}^2h_{2,j-1}  + h_{2,j}h_{2,j-1}h_{2,j-2}^2 \]
  which is equal to $h_{2,j-1}^4$, finishing the proof.
\end{proof}
\endgroup

% In the associated graded of the May filtration $t_2^{2^n}$ is a polynomial generator $h_{2,n}$. Since the $E_1$ page is commutative the first and third term cancel and we are left with the single term $h_{2,n}^4$. This is also the class in the May spectral sequence which detects $g_{j-2}$(?) so we may conclude that the image of $1/4 * \vartheta_j \vartheta_j-1$ in the Adams spectral sequence is $g_{j-2}$.

\begin{cor}
  The product $\vartheta_j\vartheta_{j+1}$ is nontrivial on the Adams--Novikov $E_2$-page.
\end{cor}

\begin{proof}
  From \Cref{lem:prod-small} we know that $\vartheta_j\vartheta_{j+1}$ is detected by $q_0^2g_{j-2}$ in the algebraic Novikov spectral sequence. This means that it will suffice to show this class isn't hit by a differential.
  
  Examining the information about the $E_2$ page of the algebraic Novikov spectral sequence from \Cref{E2 descriptions}(3) the only potential differentials which could hit $q_0^2g_{j-2}$ are a $d_2$ differential on $q_0h_j^3$ or a $d_3$ differential on $h_j^3$.
  We showed that $d_2(h_j^3)$ and $d_3(h_j^3)$ are zero in the Proposition~\ref{prop:alg-n-diff}.
\end{proof}

\subsection{Constructing the correction term}\ 
\label{subsec:correction}

In this short subsection we digress and discuss the choice of the correction term $c_j$. 
Although $c_j$ was originally produced by inspection we explain our idea of how to produce such a term. We hope that this serves as a useful guide to the reader trying to make similar calculations in the future. 

First note that we have rigged things so that at each step no $v_i$ appears (other than $v_0 = 2$).
What this means is that the associated graded of the algebraic Novikov filtration becomes, as far as we see, the same as the cobar complex for $\P$. What gain does this provide us? Well, we can now filter things by the May filtration and successively add terms to work our way up the May filtration. As an example: The cocycle we started with $T_j|T_{j+1}$ is quite complicated, but mod 2 things are manageable
\[ T_j | T_{j+1} \equiv [t_1^{2^{j-1}}|t_1^{2^{j-1}}|t_1^{2^{j}}|t_1^{2^{j}}] \pmod{2} \]
We use the term $[t_1^{2^{j-1}}|t_1^{3 \cdot 2^{j-1}}|t_1^{2^{j}}]$ to swap the middle pair of powers of $t_1$.
Next we add the term $[t_2^{2^{j-1}}|t_1^{2^{j-1}}|t_1^{2^{j}}]$ in order to use the May $d_1$ differential killing $h_{j-1}h_j$.

At this point, after adding $d(-)$ of these two correction terms we get something which is zero mod 2.
This means we get to move up one stage in the algebraic Novikov filtration.
Now we look at things mod 4, since we get a cocycle mod 4 which is divisible by 2 we consider it again as a cocycle in $\P$. We also have to make a choice of lift of our coefficients on the correction term and we use the ones that appear in $c_j$ (i.e. one and seven). The reasoning behind this choice is related to how we initially wrote down this cocycle (by inspection) and it likely makes little difference in the end what coefficients were chosen.

The next stage involves examining the cocycle
\[ T_j|T_{j+1} + d\left([t_1^{2^{j-1}}|t_1^{3 \cdot 2^{j-1}}|t_1^{2^{j}}] + 7[t_2^{2^{j-1}}|t_1^{2^{j-1}}|t_1^{2^{j}}] \right) \pmod{4}. \]
The leading term in the May filtration of this cocycle is $2[t_2^{2^{j-2}}|t_2^{2^{j-2}}|t_1^{2^{j-1}}|t_1^{2^{j}}]$ in May filtration 8. In the May $E_1$ page this cocycle is detected by $h_{2,j-2}^2h_{1,j-1}h_{1,j}$ and this class is hit by a May $d_1$ differential coming off of $h_{2,j-2}^2h_{2,j-1}$. Using this we conclude that the next correction term we want to add is $2[t_2^{2^{j-2}}|t_2^{2^{j-2}}|t_2^{2^{j-1}}]$.

The leading May filtration of the cocycle
\[ T_j|T_{j+1} + d\left([t_1^{2^{j-1}}|t_1^{3 \cdot 2^{j-1}}|t_1^{2^{j}}] + 7[t_2^{2^{j-1}}|t_1^{2^{j-1}}|t_1^{2^{j}}] + 2[t_2^{2^{j-2}}|t_2^{2^{j-2}}|t_2^{2^{j-1}}] \right) \pmod{4} \]
is $2[t_2^{2^{j-2}}|t_1^{2^{j-2}}|t_1^{2^{j-1}}|t_2^{2^{j-1}}] + 2[t_1^{2^{j-2}}|t_1^{2^{j-1}}|t_2^{2^{j-2}}|t_2^{2^{j-1}}]$ in May filtration 8. In the May $E_1$ page this cocycle is already zero so we add a correction term that swaps the $t_1$'s and $t_2$'s so that they cancel (see the third line of the formula for $c_j$). After adding this correction the top May filtration is now reduced from 8 to 6.

Repeating this process we eventually eliminate the entire thing---proving that $\vartheta_j\vartheta_{j+1}$ is divisible by 4 (with the desired correction term $c_j$ providing a witness to divisibility in the cobar complex).
In \Cref{app:code} we have annotated the corrections terms we add with the May lengths of the corresponding May differentials we are using at each step. One of the reasons this process was relatively straightforward was that we used only May $d_0$ and $d_1$ differentials.

%% file: tablev2.tex
\begin{figure}
  \begin{sideways}
    \centering
    \scalebox{0.8}{
      \begin{tabular}{|c||m{2.5cm}|m{2.5cm}|m{2.5cm}|m{2.5cm}|m{2.5cm}|m{2.5cm}|m{2.5cm}|m{2.5cm}|m{2.5cm}|}
        \hline
        & $4$ & $5$ & $6$ & $7$ & $8$ & $9$ & $10$ & $11$ & $12$ \\\hline\hline    
        $T_j|T_{j+1}$
        & % 4  
        $ 1 [ {4 }_1  | {4 }_1  | {8 }_1  | {8 }_1  ] $ 
        & % 5  
        $ 2 [ {2 }_1  | {6 }_1  | {8 }_1  | {8 }_1  ] $
        $ 2 [ {4 }_1  | {4 }_1  | {4 }_1  | {12}_1  ] $
        $ 2 [ {4 }_1  | {4 }_1  | {12}_1  | {4 }_1  ] $ 
        $ 2 [ {6 }_1  | {2 }_1  | {8 }_1  | {8 }_1  ] $ 
        & % 6  
        $ 4 [ {1 }_1  | {7 }_1  | {8 }_1  | {8 }_1  ] $ 
        $ 4 [ {2 }_1  | {6 }_1  | {4 }_1  | {12}_1  ] $
        $ 4 [ {2 }_1  | {6 }_1  | {12}_1  | {4 }_1  ] $ 
        $ 4 [ {3 }_1  | {5 }_1  | {8 }_1  | {8 }_1  ] $
        $ 4 [ {4 }_1  | {4 }_1  | {2 }_1  | {14}_1  ] $ 
        $ 4 [ {4 }_1  | {4 }_1  | {6 }_1  | {10}_1  ] $ 
        $ 4 [ {4 }_1  | {4 }_1  | {10}_1  | {6 }_1  ] $
        $ 4 [ {4 }_1  | {4 }_1  | {14}_1  | {2 }_1  ] $
        $ 4 [ {5 }_1  | {3 }_1  | {8 }_1  | {8 }_1  ] $
        $ 4 [ {6 }_1  | {2 }_1  | {4 }_1  | {12}_1  ] $
        $ 4 [ {6 }_1  | {2 }_1  | {12}_1  | {4 }_1  ] $
        $ 4 [ {7 }_1  | {1 }_1  | {8 }_1  | {8 }_1  ] $
        & % 7  
        & % 8  
        & % 9  
        & % 10  
        & % 11  
        & % 12  
        \\\hline    
        $d( [ 4_1 | 12_1 | 8_1 ] )$
        & % 4   
        $ 7 [ {4 }_1  | {4 }_1  | {8 }_1  | {8 }_1  ] $ 
        $ 7 [ {4 }_1  | {8 }_1  | {4 }_1  | {8 }_1  ] $ 
        & % 5
        $ 2 [ {2 }_1  | {2 }_1  | {12}_1  | {8 }_1  ] $ 
        $ 2 [ {4 }_1  | {2 }_1  | {10}_1  | {8 }_1  ] $ 
        $ 2 [ {4 }_1  | {10}_1  | {2 }_1  | {8 }_1  ] $ 
        $ 2 [ {4 }_1  | {12}_1  | {4 }_1  | {4 }_1  ] $ 
        & % 6
        $ 4 [ {1 }_1  | {3 }_1  | {12}_1  | {8 }_1  ] $ 
        $ 4 [ {3 }_1  | {1 }_1  | {12}_1  | {8 }_1  ] $ 
        $ 4 [ {4 }_1  | {1 }_1  | {11}_1  | {8 }_1  ] $ 
        $ 4 [ {4 }_1  | {3 }_1  | {9 }_1  | {8 }_1  ] $ 
        $ 4 [ {4 }_1  | {6 }_1  | {6 }_1  | {8 }_1  ] $ 
        $ 4 [ {4 }_1  | {9 }_1  | {3 }_1  | {8 }_1  ] $ 
        $ 4 [ {4 }_1  | {11}_1  | {1 }_1  | {8 }_1  ] $ 
        $ 4 [ {4 }_1  | {12}_1  | {2 }_1  | {6 }_1  ] $ 
        $ 4 [ {4 }_1  | {12}_1  | {6 }_1  | {2 }_1  ] $ 
        & % 7   
        & % 8   
        & % 9   
        & % 10   
        & % 11   
        & % 12   
        \\\hline    
        $d( 7 [ 4_2 | 4_1 | 8_1 ] )$
        & % 4  
        $ 1 [ {4 }_1  | {8 }_1  | {4 }_1  | {8 }_1  ] $ 
        & % 5  
        & % 6  
        $ 2 [  {4 }_2 | {2 }_1  | {2 }_1  | {8 }_1  ] $ 
        $ 6 [  {4 }_2 | {4 }_1  | {4 }_1  | {4 }_1  ] $ 
        & % 7  
        $ 4 [  {4 }_2 | {1 }_1  | {3 }_1  | {8 }_1  ] $ 
        $ 4 [  {4 }_2 | {3 }_1  | {1 }_1  | {8 }_1  ] $ 
        $ 4 [  {4 }_2 | {4 }_1  | {2 }_1  | {6 }_1  ] $ 
        $ 4 [  {4 }_2 | {4 }_1  | {6 }_1  | {2 }_1  ] $ 
        $ 6 [ {2 }_1  | {4 }_1 {2 }_2 | {4 }_1  | {8 }_1  ] $ 
        $ 6 [ {2 }_1 {2 }_2 | {4 }_1  | {4 }_1  | {8 }_1  ] $ 
        & % 8
        $ 6 [  {2 }_2 |  {2 }_2 | {4 }_1  | {8 }_1  ] $ 
        & % 9  
        $ 4 [ {3 }_1  | {6 }_1 {1 }_2 | {4 }_1  | {8 }_1  ] $ 
        $ 4 [ {3 }_1 {1 }_2 | {6 }_1  | {4 }_1  | {8 }_1  ] $ 
        & % 10
        $ 4 [ {1 }_1  | {2 }_1 {3 }_2 | {4 }_1  | {8 }_1  ] $ 
        $ 4 [ {1 }_1 {1 }_2 | {2 }_1 {2 }_2 | {4 }_1  | {8 }_1  ] $ 
        $ 4 [ {1 }_1 {2 }_2 | {2 }_1 {1 }_2 | {4 }_1  | {8 }_1  ] $ 
        $ 4 [ {1 }_1 {3 }_2 | {2 }_1  | {4 }_1  | {8 }_1  ] $ 
        $ 4 [ {2 }_1 {1 }_2 | {4 }_1 {1 }_2 | {4 }_1  | {8 }_1  ] $ 
        & % 11
        $ 4 [  {1 }_2 |  {3 }_2 | {4 }_1  | {8 }_1  ] $ 
        $ 4 [  {3 }_2 |  {1 }_2 | {4 }_1  | {8 }_1  ] $ 
        & % 12  
        \\\hline    
        $d( 2 [ 2_2 | 2_2 | 4_2 ] )$
        & % 4  
        & % 5  
        & % 6  
        & % 7  
%        $ 6 [ {2 }_1 {2 }_2 | {4 }_1  | {4 }_1  | {8 }_1  ] $ 
        & % 8  
        $ 6 [ {2 }_1  | {4 }_1  |  {2 }_2 |  {4 }_2 ] $ 
        $ 2 [  {2 }_2 | {2 }_1  | {4 }_1  |  {4 }_2 ] $
        $ 6 [  {2 }_2 |  {2 }_2 | {4 }_1  | {8 }_1  ] $ 
        & % 9  
        & % 10
        & % 11
        $ 4 [ {1 }_1  | {2 }_1 {1 }_2 |  {2 }_2 |  {4 }_2 ] $ 
        $ 4 [ {1 }_1 {1 }_2 | {2 }_1  |  {2 }_2 |  {4 }_2 ] $ 
        $ 4 [  {2 }_2 | {1 }_1  | {2 }_1 {1 }_2 |  {4 }_2 ] $ 
        $ 4 [  {2 }_2 | {1 }_1 {1 }_2 | {2 }_1  |  {4 }_2 ] $ 
        $ 4 [  {2 }_2 |  {2 }_2 | {2 }_1  | {4 }_1 {2 }_2 ] $ 
        $ 4 [  {2 }_2 |  {2 }_2 | {2 }_1 {2 }_2 | {4 }_1  ] $ 
        & % 12  
        $ 4 [  {2 }_2 |  {2 }_2 |  {2 }_2 |  {2 }_2 ] $ 
        $ 4 [  {2 }_2 |  {1 }_2 |  {1 }_2 |  {4 }_2 ] $ 
        $ 4 [  {1 }_2 |  {1 }_2 |  {2 }_2 |  {4 }_2 ] $ 
        \\\hline    
      \end{tabular}
    }
  \end{sideways}
\end{figure}

\begin{figure}
  \begin{sideways}
    \centering
    \scalebox{0.8}{
      \begin{tabular}{|c||m{2.5cm}|m{2.5cm}|m{2.5cm}|m{2.5cm}|m{2.5cm}|m{2.5cm}|m{2.5cm}|m{2.5cm}|m{2.5cm}|}
        \hline
        & $4$ & $5$ & $6$ & $7$ & $8$ & $9$ & $10$ & $11$ & $12$ \\\hline\hline    
        $d(2[2_1 2_2 | 4_1 | 4_2])$
        & % 4  
        & % 5  
        & % 6  
        $ 6 [ {4 }_1  | {4 }_1  | {4 }_1  |  {4 }_2 ] $ 
        & % 7  
        $ 6 [ {2 }_1  | {6 }_1  | {4 }_1  |  {4 }_2 ] $ 
        $ 6 [ {2 }_1 {2 }_2 | {4 }_1  | {4 }_1  | {8 }_1  ] $ 
        & % 8   
        $ 6 [  {2 }_2 | {2 }_1  | {4 }_1  |  {4 }_2 ] $ 
        $ 4 [ {3 }_1  | {5 }_1  | {4 }_1  |  {4 }_2 ] $ 
        $ 6 [ {2 }_1  |  {2 }_2 | {4 }_1  |  {4 }_2 ] $
        & % 9  
        $ 4 [ {1 }_1  | {1 }_1 {2 }_2 | {4 }_1  |  {4 }_2 ] $ 
        $ 4 [ {1 }_1 {2 }_2 | {1 }_1  | {4 }_1  |  {4 }_2 ] $ 
        $ 4 [ {2 }_1 {2 }_2 | {2 }_1  | {2 }_1  |  {4 }_2 ] $ 
        $ 4 [ {1 }_1  | {4 }_1 {1 }_2 | {4 }_1  |  {4 }_2 ] $ 
        $ 4 [ {1 }_1 {1 }_2 | {4 }_1  | {4 }_1  |  {4 }_2 ] $ 
        & % 10
        $ 4 [ {3 }_1  | {2 }_1 {1 }_2 | {4 }_1  |  {4 }_2 ] $ 
        $ 4 [ {3 }_1 {1 }_2 | {2 }_1  | {4 }_1  |  {4 }_2 ] $ 
        $ 4 [ {2 }_1 {2 }_2 | {4 }_1  | {2 }_1  | {4 }_1 {2 }_2 ] $ 
        $ 4 [ {2 }_1 {2 }_2 | {4 }_1  | {2 }_1 {2 }_2 | {4 }_1  ] $ 
        & % 11
        $ 4 [  {1 }_2 | {2 }_1 {1 }_2 | {4 }_1  |  {4 }_2 ] $ 
        $ 4 [ {2 }_1 {1 }_2 |  {1 }_2 | {4 }_1  |  {4 }_2 ] $ 
        $ 4 [ {2 }_1 {2 }_2 | {4 }_1  |  {2 }_2 |  {2 }_2 ] $ 
        & % 12  
        \\\hline    
        $d(2[ 2_1 | 4_1 2_2 | 4_2 ])$
        & % 4  
        & % 5  
        & % 6  
        $ 2 [ {2 }_1  | {2 }_1  | {8 }_1  |  {4 }_2 ] $ 
        & % 7  
        $ 4 [ {2 }_1  | {4 }_1  | {6 }_1  |  {4 }_2 ] $ 
        $ 2 [ {2 }_1  | {6 }_1  | {4 }_1  |  {4 }_2 ] $ 
        $ 6 [ {2 }_1  | {4 }_1 {2 }_2 | {4 }_1  | {8 }_1  ] $ 
        & % 8  
        $ 2 [ {2 }_1  | {4 }_1  |  {2 }_2 |  {4 }_2 ] $ 
        $ 2 [ {2 }_1  |  {2 }_2 | {4 }_1  |  {4 }_2 ] $ 
        & % 9  
        $ 4 [ {1 }_1  | {1 }_1  | {4 }_1 {2 }_2 |  {4 }_2 ] $ 
        $ 4 [ {2 }_1  | {2 }_1  | {2 }_1 {2 }_2 |  {4 }_2 ] $ 
        $ 4 [ {2 }_1  | {2 }_1 {2 }_2 | {2 }_1  |  {4 }_2 ] $ 
        & % 10
        $ 4 [ {2 }_1  | {1 }_1  | {6 }_1 {1 }_2 |  {4 }_2 ] $ 
        $ 4 [ {2 }_1  | {5 }_1  | {2 }_1 {1 }_2 |  {4 }_2 ] $ 
        $ 4 [ {2 }_1  | {1 }_1 {1 }_2 | {6 }_1  |  {4 }_2 ] $ 
        $ 4 [ {2 }_1  | {5 }_1 {1 }_2 | {2 }_1  |  {4 }_2 ] $ 
        $ 4 [ {2 }_1  | {4 }_1 {2 }_2 | {2 }_1  | {4 }_1 {2 }_2 ] $ 
        $ 4 [ {2 }_1  | {4 }_1 {2 }_2 | {2 }_1 {2 }_2 | {4 }_1  ] $ 
        & % 11
        $ 4 [ {2 }_1  |  {1 }_2 | {4 }_1 {1 }_2 |  {4 }_2 ] $ 
        $ 4 [ {2 }_1  | {4 }_1 {1 }_2 |  {1 }_2 |  {4 }_2 ] $ 
        $ 4 [ {2 }_1  | {4 }_1 {2 }_2 |  {2 }_2 |  {2 }_2 ] $ 
        & % 12  
        \\\hline    
        $d(2[ 4_1 4_2 | 4_1 | 4_1 ])$
        & % 4  
        $ 6 [ {8 }_1  | {8 }_1  | {4 }_1  | {4 }_1  ] $ 
        & % 5  
        $ 6 [ {4 }_1  | {12}_1  | {4 }_1  | {4 }_1  ] $ 
        & % 6  
        $ 6 [ {4 }_1  |  {4 }_2 | {4 }_1  | {4 }_1  ] $ 
        $ 6 [  {4 }_2 | {4 }_1  | {4 }_1  | {4 }_1  ] $ 
        $ 4 [ {6 }_1  | {10}_1  | {4 }_1  | {4 }_1  ] $ 
        & % 7  
        $ 4 [ {2 }_1  | {2 }_1 {4 }_2 | {4 }_1  | {4 }_1  ] $ 
        $ 4 [ {2 }_1 {4 }_2 | {2 }_1  | {4 }_1  | {4 }_1  ] $ 
        $ 4 [ {4 }_1 {4 }_2 | {2 }_1  | {2 }_1  | {4 }_1  ] $ 
        $ 4 [ {4 }_1 {4 }_2 | {4 }_1  | {2 }_1  | {2 }_1  ] $ 
        $ 4 [ {2 }_1  | {8 }_1 {2 }_2 | {4 }_1  | {4 }_1  ] $ 
        $ 4 [ {2 }_1 {2 }_2 | {8 }_1  | {4 }_1  | {4 }_1  ] $ 
        & % 8  
        $ 4 [ {6 }_1  | {4 }_1 {2 }_2 | {4 }_1  | {4 }_1  ] $ 
        $ 4 [ {6 }_1 {2 }_2 | {4 }_1  | {4 }_1  | {4 }_1  ] $ 
        & % 9  
        $ 4 [  {2 }_2 | {4 }_1 {2 }_2 | {4 }_1  | {4 }_1  ] $ 
        $ 4 [ {4 }_1 {2 }_2 |  {2 }_2 | {4 }_1  | {4 }_1  ] $ 
        & % 10  
        & % 11  
        & % 12  
        \\\hline    
        $d(2[ 4_1 | 4_1 4_2 | 4_1 ])$
        & % 4  
        $ 2 [ {4 }_1  | {8 }_1  | {8 }_1  | {4 }_1  ] $ 
        & % 5  
        $ 2 [ {4 }_1  | {4 }_1  | {12}_1  | {4 }_1  ] $ 
        & % 6  
        $ 2 [ {4 }_1  | {4 }_1  |  {4 }_2 | {4 }_1  ] $ 
        $ 2 [ {4 }_1  |  {4 }_2 | {4 }_1  | {4 }_1  ] $ 
        $ 4 [ {4 }_1  | {6 }_1  | {10}_1  | {4 }_1  ] $ 
        & % 7  
        $ 4 [ {2 }_1  | {2 }_1  | {4 }_1 {4 }_2 | {4 }_1  ] $ 
        $ 4 [ {4 }_1  | {2 }_1  | {2 }_1 {4 }_2 | {4 }_1  ] $ 
        $ 4 [ {4 }_1  | {2 }_1 {4 }_2 | {2 }_1  | {4 }_1  ] $ 
        $ 4 [ {4 }_1  | {4 }_1 {4 }_2 | {2 }_1  | {2 }_1  ] $ 
        $ 4 [ {4 }_1  | {2 }_1  | {8 }_1 {2 }_2 | {4 }_1  ] $ 
        $ 4 [ {4 }_1  | {2 }_1 {2 }_2 | {8 }_1  | {4 }_1  ] $ 
        & % 8  
        $ 4 [ {4 }_1  | {6 }_1  | {4 }_1 {2 }_2 | {4 }_1  ] $ 
        $ 4 [ {4 }_1  | {6 }_1 {2 }_2 | {4 }_1  | {4 }_1  ] $ 
        & % 9  
        $ 4 [ {4 }_1  |  {2 }_2 | {4 }_1 {2 }_2 | {4 }_1  ] $ 
        $ 4 [ {4 }_1  | {4 }_1 {2 }_2 |  {2 }_2 | {4 }_1  ] $ 
        & % 10  
        & % 11  
        & % 12  
        \\\hline    
        $d(2[ 4_1 | 4_1 | 4_1 4_2])$
        & % 4  
        $ 6 [ {4 }_1  | {4 }_1  | {8 }_1  | {8 }_1  ] $ 
        & % 5  
        $ 6 [ {4 }_1  | {4 }_1  | {4 }_1  | {12}_1  ] $ 
        & % 6  
        $ 6 [ {4 }_1  | {4 }_1  | {4 }_1  |  {4 }_2 ] $ 
        $ 6 [ {4 }_1  | {4 }_1  |  {4 }_2 | {4 }_1  ] $ 
        $ 4 [ {4 }_1  | {4 }_1  | {6 }_1  | {10}_1  ] $ 
        & % 7  
        $ 4 [ {2 }_1  | {2 }_1  | {4 }_1  | {4 }_1 {4 }_2 ] $ 
        $ 4 [ {4 }_1  | {2 }_1  | {2 }_1  | {4 }_1 {4 }_2 ] $ 
        $ 4 [ {4 }_1  | {4 }_1  | {2 }_1  | {2 }_1 {4 }_2 ] $ 
        $ 4 [ {4 }_1  | {4 }_1  | {2 }_1 {4 }_2 | {2 }_1  ] $ 
        $ 4 [ {4 }_1  | {4 }_1  | {2 }_1  | {8 }_1 {2 }_2 ] $ 
        $ 4 [ {4 }_1  | {4 }_1  | {2 }_1 {2 }_2 | {8 }_1  ] $ 
        & % 8  
        $ 4 [ {4 }_1  | {4 }_1  | {6 }_1  | {4 }_1 {2 }_2 ] $ 
        $ 4 [ {4 }_1  | {4 }_1  | {6 }_1 {2 }_2 | {4 }_1  ] $ 
        & % 9  
        $ 4 [ {4 }_1  | {4 }_1  |  {2 }_2 | {4 }_1 {2 }_2 ] $ 
        $ 4 [ {4 }_1  | {4 }_1  | {4 }_1 {2 }_2 |  {2 }_2 ] $ 
        & % 10  
        & % 11  
        & % 12  
        \\\hline    
        $d(2[ 2_1 4_2 | 2_1 | 8_1 ])$
        & % 4  
        & % 5  
        $ 6 [ {4 }_1  | {10}_1  | {2 }_1  | {8 }_1  ] $ 
        $ 6 [ {6 }_1  | {8 }_1  | {2 }_1  | {8 }_1  ] $ 
        & % 6  
        $ 6 [ {2 }_1  |  {4 }_2 | {2 }_1  | {8 }_1  ] $ 
        $ 6 [  {4 }_2 | {2 }_1  | {2 }_1  | {8 }_1  ] $ 
        $ 4 [ {5 }_1  | {9 }_1  | {2 }_1  | {8 }_1  ] $ 
        & % 7  
        $ 4 [ {1 }_1  | {1 }_1 {4 }_2 | {2 }_1  | {8 }_1  ] $ 
        $ 4 [ {1 }_1 {4 }_2 | {1 }_1  | {2 }_1  | {8 }_1  ] $ 
        $ 4 [ {2 }_1 {4 }_2 | {1 }_1  | {1 }_1  | {8 }_1  ] $ 
        $ 4 [ {2 }_1 {4 }_2 | {2 }_1  | {4 }_1  | {4 }_1  ] $ 
        $ 4 [ {4 }_1  | {4 }_1 {2 }_2 | {2 }_1  | {8 }_1  ] $ 
        $ 4 [ {4 }_1 {2 }_2 | {4 }_1  | {2 }_1  | {8 }_1  ] $ 
        & % 8  
        $ 4 [ {2 }_1  | {6 }_1 {2 }_2 | {2 }_1  | {8 }_1  ] $ 
        $ 4 [ {2 }_1 {2 }_2 | {6 }_1  | {2 }_1  | {8 }_1  ] $ 
        & % 9  
        $ 4 [  {2 }_2 | {2 }_1 {2 }_2 | {2 }_1  | {8 }_1  ] $ 
        $ 4 [ {2 }_1 {2 }_2 |  {2 }_2 | {2 }_1  | {8 }_1  ] $ 
        & % 10  
        & % 11  
        & % 12  
        \\\hline    
      \end{tabular}
    }
  \end{sideways}
\end{figure}

\begin{figure}
  \begin{sideways}
    \centering
    \scalebox{0.8}{
      \begin{tabular}{|c||m{2.5cm}|m{2.5cm}|m{2.5cm}|m{2.5cm}|m{2.5cm}|m{2.5cm}|m{2.5cm}|m{2.5cm}|m{2.5cm}|}
        \hline
        & $4$ & $5$ & $6$ & $7$ & $8$ & $9$ & $10$ & $11$ & $12$ \\\hline\hline    
        $d(2[ 2_1  | 2_1 4_2 | 8_1 ])$
        & % 4  
        & % 5  
        $ 2 [ {2 }_1  | {4 }_1  | {10}_1  | {8 }_1  ] $ 
        $ 2 [ {2 }_1  | {6 }_1  | {8 }_1  | {8 }_1  ] $ 
        & % 6  
        $ 2 [ {2 }_1  | {2 }_1  |  {4 }_2 | {8 }_1  ] $ 
        $ 2 [ {2 }_1  |  {4 }_2 | {2 }_1  | {8 }_1  ] $ 
        $ 4 [ {2 }_1  | {5 }_1  | {9 }_1  | {8 }_1  ] $ 
        & % 7  
        $ 4 [ {1 }_1  | {1 }_1  | {2 }_1 {4 }_2 | {8 }_1  ] $ 
        $ 4 [ {2 }_1  | {1 }_1  | {1 }_1 {4 }_2 | {8 }_1  ] $ 
        $ 4 [ {2 }_1  | {1 }_1 {4 }_2 | {1 }_1  | {8 }_1  ] $ 
        $ 4 [ {2 }_1  | {2 }_1 {4 }_2 | {4 }_1  | {4 }_1  ] $ 
        $ 4 [ {2 }_1  | {4 }_1  | {4 }_1 {2 }_2 | {8 }_1  ] $ 
        $ 4 [ {2 }_1  | {4 }_1 {2 }_2 | {4 }_1  | {8 }_1  ] $ 
        & % 8  
        $ 4 [ {2 }_1  | {2 }_1  | {6 }_1 {2 }_2 | {8 }_1  ] $ 
        $ 4 [ {2 }_1  | {2 }_1 {2 }_2 | {6 }_1  | {8 }_1  ] $ 
        & % 9  
        $ 4 [ {2 }_1  |  {2 }_2 | {2 }_1 {2 }_2 | {8 }_1  ] $ 
        $ 4 [ {2 }_1  | {2 }_1 {2 }_2 |  {2 }_2 | {8 }_1  ] $ 
        & % 10  
        & % 11  
        & % 12  
        \\\hline    
        $d(2[ 2_1  | 2_1  | 8_1 4_2 ])$
        & % 4  
        $ 6 [ {2 }_1  | {2 }_1  | {4 }_1  | {16}_1  ] $ 
        & % 5  
        $ 4 [ {2 }_1  | {2 }_1  | {8 }_1  | {12}_1  ] $ 
        $ 6 [ {2 }_1  | {2 }_1  | {12}_1  | {8 }_1  ] $ 
        & % 6  
        $ 6 [ {2 }_1  | {2 }_1  | {8 }_1  |  {4 }_2 ] $ 
        $ 6 [ {2 }_1  | {2 }_1  |  {4 }_2 | {8 }_1  ] $ 
        & % 7  
        $ 4 [ {1 }_1  | {1 }_1  | {2 }_1  | {8 }_1 {4 }_2 ] $ 
        $ 4 [ {2 }_1  | {1 }_1  | {1 }_1  | {8 }_1 {4 }_2 ] $ 
        $ 4 [ {2 }_1  | {2 }_1  | {4 }_1  | {4 }_1 {4 }_2 ] $ 
        $ 4 [ {2 }_1  | {2 }_1  | {4 }_1 {4 }_2 | {4 }_1  ] $ 
        & % 8  
        $ 4 [ {2 }_1  | {2 }_1  | {2 }_1  | {12}_1 {2 }_2 ] $ 
        $ 4 [ {2 }_1  | {2 }_1  | {10}_1  | {4 }_1 {2 }_2 ] $ 
        $ 4 [ {2 }_1  | {2 }_1  | {2 }_1 {2 }_2 | {12}_1  ] $ 
        $ 4 [ {2 }_1  | {2 }_1  | {10}_1 {2 }_2 | {4 }_1  ] $ 
        & % 9  
        $ 4 [ {2 }_1  | {2 }_1  |  {2 }_2 | {8 }_1 {2 }_2 ] $ 
        $ 4 [ {2 }_1  | {2 }_1  | {8 }_1 {2 }_2 |  {2 }_2 ] $ 
        & % 10  
        & % 11  
        & % 12  
        \\\hline    
        $d(2[ 6_1  | 10_1  | 8_1  ])$
        & % 4  
        & % 5  
        $ 2 [ {2 }_1  | {4 }_1  | {10}_1  | {8 }_1  ] $ 
        $ 2 [ {4 }_1  | {2 }_1  | {10}_1  | {8 }_1  ] $ 
        $ 2 [ {6 }_1  | {2 }_1  | {8 }_1  | {8 }_1  ] $ 
        $ 2 [ {6 }_1  | {8 }_1  | {2 }_1  | {8 }_1  ] $ 
        & % 6  
        $ 4 [ {1 }_1  | {5 }_1  | {10}_1  | {8 }_1  ] $ 
        $ 4 [ {5 }_1  | {1 }_1  | {10}_1  | {8 }_1  ] $ 
        $ 4 [ {6 }_1  | {1 }_1  | {9 }_1  | {8 }_1  ] $ 
        $ 4 [ {6 }_1  | {4 }_1  | {6 }_1  | {8 }_1  ] $ 
        $ 4 [ {6 }_1  | {6 }_1  | {4 }_1  | {8 }_1  ] $ 
        $ 4 [ {6 }_1  | {9 }_1  | {1 }_1  | {8 }_1  ] $ 
        $ 4 [ {6 }_1  | {10}_1  | {4 }_1  | {4 }_1  ] $ 
        & % 7  
        & % 8  
        & % 9  
        & % 10  
        & % 11  
        & % 12  
        \\\hline    
        $d(2[ 8_1  | 12_1  | 4_1  ])$
        & % 4  
        $ 6 [ {8 }_1  | {4 }_1  | {8 }_1  | {4 }_1  ] $ 
        $ 6 [ {8 }_1  | {8 }_1  | {4 }_1  | {4 }_1  ] $ 
        & % 5  
        $ 4 [ {4 }_1  | {4 }_1  | {12}_1  | {4 }_1  ] $ 
        $ 4 [ {8 }_1  | {2 }_1  | {10}_1  | {4 }_1  ] $ 
        $ 4 [ {8 }_1  | {10}_1  | {2 }_1  | {4 }_1  ] $ 
        $ 4 [ {8 }_1  | {12}_1  | {2 }_1  | {2 }_1  ] $ 
        & % 6 
        & % 7  
        & % 8  
        & % 9  
        & % 10  
        & % 11  
        & % 12  
        \\\hline    
        $d(2[ 12_1  | 8_1  | 4_1  ])$
        & % 4  
        $ 2 [ {4 }_1  | {8 }_1  | {8 }_1  | {4 }_1  ] $ 
        $ 2 [ {8 }_1  | {4 }_1  | {8 }_1  | {4 }_1  ] $ 
        & % 5  
        $ 4 [ {2 }_1  | {10}_1  | {8 }_1  | {4 }_1  ] $ 
        $ 4 [ {10}_1  | {2 }_1  | {8 }_1  | {4 }_1  ] $ 
        $ 4 [ {12}_1  | {4 }_1  | {4 }_1  | {4 }_1  ] $ 
        $ 4 [ {12}_1  | {8 }_1  | {2 }_1  | {2 }_1  ] $ 
        & % 6  
        & % 7  
        & % 8  
        & % 9  
        & % 10  
        & % 11  
        & % 12  
        \\\hline    
        $d(2[ 2_1  | 2_2  | 16_1  ])$
        & % 4  
        $ 2 [ {2 }_1  | {2 }_1  | {4 }_1  | {16}_1  ] $ 
        & % 5  
        & % 6  
        $ 4 [ {1 }_1  | {1 }_1  |  {2 }_2 | {16}_1  ] $ 
        $ 4 [ {2 }_1  |  {2 }_2 | {8 }_1  | {8 }_1  ] $ 
        & % 7  
        $ 4 [ {2 }_1  | {1 }_1  | {2 }_1 {1 }_2 | {16}_1  ] $ 
        $ 4 [ {2 }_1  | {1 }_1 {1 }_2 | {2 }_1  | {16}_1  ] $ 
        & % 8  
        $ 4 [ {2 }_1  |  {1 }_2 |  {1 }_2 | {16}_1  ] $ 
        & % 9  
        & % 10  
        & % 11  
        & % 12  
        \\\hline    
        $d(2[ 4_1  | 4_2  | 8_1  ])$
        & % 4  
        $ 2 [ {4 }_1  | {4 }_1  | {8 }_1  | {8 }_1  ] $ 
        & % 5  
        & % 6  
        $ 4 [ {2 }_1  | {2 }_1  |  {4 }_2 | {8 }_1  ] $ 
        $ 4 [ {4 }_1  |  {4 }_2 | {4 }_1  | {4 }_1  ] $ 
        & % 7  
        $ 4 [ {4 }_1  | {2 }_1  | {4 }_1 {2 }_2 | {8 }_1  ] $ 
        $ 4 [ {4 }_1  | {2 }_1 {2 }_2 | {4 }_1  | {8 }_1  ] $ 
        & % 8  
        $ 4 [ {4 }_1  |  {2 }_2 |  {2 }_2 | {8 }_1  ] $ 
        & % 9  
        & % 10  
        & % 11  
        & % 12  
        \\\hline    
      \end{tabular}
    }
  \end{sideways}
\end{figure}

%% file: kervairev2.tex
In this section, which can for the most part be read independently of the rest of the paper, we investigate the fate of the Kervaire invariant one classes in the Adams spectral sequence.
The celebrated Hill--Hopkins--Ravenel theorem on Kervaire invariant one tells us that $h_j^2$ supports a non-trivial Adams differential for every $j \geq 7$ \cite{HHR}. However, their work provides no further identifying information about these differentials.\footnote{For example, it remains possible that the lengths of these differentials grows without bound as $j$ increases.}
% Although the authors find it unlikely that this is the case, we must acknowledge that if this is the case, then it is probable that we will never know their precise length or target classes.}
As a corollary of our study of these classes we provide a new lower bound on the length of the HHR differentials, showing that $ d_4(h_j^2) = 0 $.

It is in this section that the distinction between $\Theta_j$, $\theta_j$ and $h_j^2$ becomes important and for this reason we pause to introduce to the reader of our conventions.
  %For the attentive reader who has already noticed our varying use of $\Theta_j$ and $\theta_j$ we provide the following explanation.
  We use $\Theta_j$ for the Kervaire invariant one classes in the classical stable homotopy groups of spheres.
  We use $h_j^2$ for the class on the Adams $E_2$-page.
  We use $\theta_j$ for a choice of class in the synthetic homotopy groups of $\clambda^k$ (which exists when $h_j^2$ survives to the $E_{k+1}$-page of the Adams spectral sequence).

We obtain our lower bound by revisiting one of the most promising proposals for constructing $\Theta_j$, the inductive approach of Barratt--Jones--Mahowald \cite{Inductive}. %\todo{is this the right attribution?}.
The crux of this approach is a construction which proceeds
\[ \begin{pmatrix} \Theta_j \text{ exists } \\ 2\Theta_j  = 0 \\ \Theta_j^2 = 0 \end{pmatrix} \Longrightarrow \begin{pmatrix} \Theta_{j+1} \text{ exists } \\ 2\Theta_{j+1} = 0 \end{pmatrix}. \]

In light of the HHR theorem this approach must break down at some point, meaning at least one of $\Theta_5^2$ or $\Theta_6^2$ is non-zero. The idea we pursue in this section is that it should be possible to run the inductive approach internal to the Adams spectral sequence on a fixed page. The main result of this section, \Cref{thm:inductive}, can be informally summarized as saying that,
\[ \begin{pmatrix} \theta_j \text{ survives to the } E_r \text{-page }\\ 2\theta_j = 0 \text{ on the } E_r \text{-page } \\ \theta_j^2 = 0 \text{ on the } E_{r-2} \text{-page } \end{pmatrix} \Longrightarrow \begin{pmatrix} \theta_j \text{ survives to the } E_r \text{-page }\\ 2\theta_j = 0 \text{ on the } E_r \text{-page } \end{pmatrix}. \]

The power of this result lies in the fact that the inductive hypothesis we must verify lies 2 pages prior to the conclusion. This means that if we proceed by induction on $r$ (rather than induction on $j$) the classes $\theta_j^2$ are already defined for all $j$ at the outset. This opens the possibility that $\theta_j^2$ might be identified on the Adams $E_r$-page simultaneously for all $j$ and that being non-zero we might, then read off the HHR differentials. In order to make the expected outcome as concrete and precise as possible we make the following conjecture, which is a special case of Conjecture~\ref{conj: stable length}.

\begin{cnj}[Uniform Kervaire differentials conjecture]
  There exists a $Sq^0$-family of classes $(\mathrm{HHR})_j$ on the Adams $E_2$-page (defined for $j \gg 0$) and another class $x$ on the $E_2$-page such that
  \[ d_r(h_j^2) = x \cdot (\mathrm{HHR})_j \neq 0. \]
\end{cnj}

In order to make the idea of ``working on a fixed page of the Adams spectral sequence'' rigorous we will work in the category of $\F_2$-synthetic spectra introduced in \cite{Pstragowski}. In the first subsection of this section we provide a lightning introduction to this category and its connection to the Adams spectral sequence. In the second subsection we run the Barratt--Jones--Mahowald inductive argument internal to the category of $\F_2$-synthetic spectra. In the final subsection we investigate a certain product related to the differential on $h_j^3$.

%\begin{rmk}
%  It is in this section that the distinction between $\Theta_j$, $\theta_j$ and $h_j^2$ becomes important and for this reason we pause to remind the reader of our conventions.
  %For the attentive reader who has already noticed our varying use of $\Theta_j$ and $\theta_j$ we provide the following explanation.
%  We use $\Theta_j$ for the Kervaire invariant one classes in the stable homotopy groups of spheres.
%  We use $h_j^2$ for the class on the Adams $E_2$-page.
%  We use $\theta_j$ for a choice of class in the synthetic homotopy groups of $\clambda^k$ (which exists when $h_j^2$ survives to the $E_{k+1}$-page of the Adams spectral sequence).
%\end{rmk}

% ---------------------------------------------------------------------------------%
% ---------------------------------------------------------------------------------%
\subsection{A lightning introduction to synthetic spectra}\ 
\label{subsec:syn-intro}

% In order to prove the main results of this section we will need to be able to understand how the Adams filtration and Adams differentials interact with power operations. 

% \todo{This needs to be rewritten}

% This subject begins with work of Adams, Barratt and Mahowald and attained a relatively stable form with Bruner's foundational work in \cite{BMMS}. Pairing this with a modern understanding of the Adams spectral sequence as provided by Pstragowski's category of synthetic spectra \cite{Pstragowski} we will obtain a robust tool for manipulating differentials in Adams type spectral sequences.

In this subsection we provide a minimal introduction to the category of $\F_2$-synthetic spectra.
For a more complete introduction focusing on the construction of this category see \cite{Pstragowski}.
For a short introduction tailored to using synthetic spectra for computational purposes see \cite[Sections 9 and A]{boundaries}.
For a more comprehensive account of synthetic homotopy theory see the (forthcoming) book \cite{cookware}. Since our main application is only in the case of $\F_2$ we will stick to this case throughout our exposition.

% ---------------------------------------------------------------------------------%
\subsubsection{Synthetic spectra}\hfill

Let $\mathrm{Stable}_{\A}$ be Hovey's stable $\infty$-category of comodules \cite{Hovey} over the Steenrod algebra $\A$.

\begin{cnstr}[Pstr\k{a}gowski]\label{cnstr:syn}
  There is a stable, presentably symmetric monoidal $\infty$-category $\mathrm{Syn}_{\F_2}$
  which fits into a diagram of symmetric monoidal functors
  \begin{center}
    \begin{tikzcd}
      & \Sp \ar[ddl, swap, "1", bend right] \ar[d, "\nu"] \ar[ddr, "(\F_2)_*(-)", bend left] & \\
      & \mathrm{Syn}_{\F_2} \ar[dl, "\lambda^{-1}"] \ar[dr, "- \otimes \clambda"'] & \\
      \Sp & & \mathrm{Stable}_{\A}
    \end{tikzcd}
  \end{center}
  with the following properties:
  \begin{enumerate}
  \item $\nu$ commutes with filtered colimits and sends a cofiber sequence to a cofiber sequence if and only if it induces a short exact sequence on $\F_2$-homology.\footnote{This condition on cofiber sequences is exactly the minimal one so that we have a long exact sequence at the level of the Adams $E_2$-page.}
  \item The functors $- \otimes \clambda$ and $\lambda^{-1}$ each commute with colimits.%\footnote{We will return to the actual meaning of the recurring term $\lambda$ shortly.}.
  \item The functor $\lambda^{-1}$ is a localization. 
  \end{enumerate}  
\end{cnstr}
  
For each spectrum $X$ the object $\nu X$ records all information present in the Adams spectral sequence for $X$. In fact, one can reasonably think about the object $\nu X$ as \emph{being} the Adams spectral sequence for $X$.

% ---------------------------------------------------------------------------------%
\subsubsection{Synthetic homotopy groups}\hfill
 
Before we can justify our claim that $\nu X$ records the data present in the Adams spectral sequence for $X$ we will need to introduce a way to measure a synthetic spectrum.
For this our preferred method is via the \emph{bigraded synthetic homotopy groups} and the action of the canonical bigraded homotopy element $\lambda$ on them. For the reader familiar with motivic spectra over $\mathbb{C}$ this pattern should be relatively familiar.

\begin{dfn}[{\cite[Definitions 4.6 and 4.9]{Pstragowski}}]
  The \emph{bigraded synthetic sphere} $\Ss^{k,s}$ is defined\footnote{We warn the reader that our conventions differ from those in some previous references. With the conventions used here the two indices of $\Ss^{k,s}$ are the $x$ and $y$ coordinates in an Adams chart. We have found that in practice these are the most user-friendly conventions.} to be $\Sigma^{-s} \nu S^{k+s}$.
  As is standard we will omit the superscripts in the case $(0,0)$, using the symbol $\Ss$ for the monoidal unit of $\Syn_{\F_2}$.
  For any synthetic spectrum $X$, the \emph{bigraded homotopy group} $\pi_{k,s}(X)$ is defined to be the abelian group of homotopy classes of maps with source $\Ss^{k,s}$ and target $X$.
\end{dfn}
  
\begin{dfn}[{\cite[Definition 4.27]{Pstragowski}}]
  For each spectrum $X$ we have an assembly map $\Sigma \nu X \to \nu \Sigma X$.
  In the case of $S^{-1}$ this provides us with a canonical map\footnote{Typically this class is referred to as $\tau$. Since we make use of both the synthetic and motivic categories, we denote it by $\lambda$ instead.}%figure prominently in this paper we have decided to break with standard notation in order to disambiguate.}
  \[ \lambda : \Ss^{0,-1} \simeq \Sigma \nu S^{-1} \longrightarrow \nu \Sigma S^{-1} \simeq \Ss^{0,0}. \]
  % which we will call $\lambda$.
\end{dfn}

The class $\lambda \in \pi_{0,-1} \Ss$ and its behavior is the most important feature of the category of synthetic spectra. As an example say that $X$ is \emph{$\lambda$-invertible} (\cite[Definition 4.32]{Pstragowski}) if the map $$\lambda: \Sigma^{0,-1} X \to X$$ is an equivalence. Then, the symmetric monoidal localization functor associated to the map $\lambda$ is the functor $\lambda^{-1}$ given above. In particular, the full subcategory of $\lambda$-invertible objects is equivalent to the category of spectra.
At the opposite extreme,
let the symbol $\clambda$ denotes the cofiber of $\lambda$,
then we have a simple description of the homotopy groups of $\clambda \otimes X$. 

\begin{lem}[{\cite[Lemma 4.56]{Pstragowski}}] \label{lem:ctau-E2}
  For any spectrum $X$, there is a natural isomorphism of bigraded abelian groups
  \[ \pi_{t-s,s}(\clambda \otimes \nu X) \cong \Ext_{\A}^{s,t}(\F_2,(\F_2)_*X) \]
  between homotopy groups mod $\lambda$ and the $E_2$-page of the Adams spectral sequence for $X$.
\end{lem}

% For the moment we also remark that $\clambda$ can be given a commutative algebra structure and that we will return to this point later.

In fact, this lemma is only a prelude to \cite[Theorem 9.19]{boundaries} which gives a precise translation\footnote{Again we warn the reader that our grading differs by a shearing from those used in loc. cit. That said, $k$ and $s$ are the same.} between Adams spectral sequence data for $X$ and bigraded synthetic homotopy groups for $\nu X$. Again for the readers familiar with motivic spectra over $\mathbb{C}$ this pattern should be relatively familiar.

\subsubsection{The $\lambda$-bockstein tower}\hfill

With the class $\lambda$ in hand we can refine our understanding of the diagram in \Cref{cnstr:syn}.

\begin{thm}[Pstr\k{a}gowski] \label{thm:tau-inv}\ 

  \begin{enumerate}
  \item The cofiber of $\lambda$, which we denote $\clambda$, admits the structure of a commutative algebra in $\Syn_{\F_p}$.
  \item There is a natural symmetric monoidal equivalence
    \[ \mathrm{Mod}(\Syn_{\F_p}; \clambda) \to \mathrm{Stable}_{\A}, \]
    from the stable $\infty$-category of modules over $\clambda$ to Hovey's stable $\infty$-category of comodules over the Steenrod algebra $\A$.
    The composite of $\nu$ with $\clambda \otimes - $ and this equivalence is naturally equivalent to the $\F_p$-homology functor.
  \end{enumerate}
\end{thm}

% \begin{rmk}
%   Echoing \cite[Remark 9.14]{boundaries} we point out that \Cref{thm:tau-inv} suggests the following algebro-geometric picture of the category of synthetic spectra:
%   $\Syn_{\F_p}$ is a family of categories over a copy of $\mathbb{A}^1$ with coordinate $\lambda$.
%   This family of categories is $\mathbb{G}_m$-equivariant and its fiber at $0$ is a category of comodules while the generic fiber is the category of spectra.
%   In this way we see that $\Syn_{\F_p}$ is a degeneration of $\Sp$ to an algebra central fiber\footnote{The reader should imagine this like one imagines a toric degeneration such as $(yt)^2 = x^3 + ax + b$ (with $t$ being the parameter).}.
% \end{rmk}

In fact, we can go beyond just $\clambda$ and study $\clambda^n$ for varying $n$.
To do this in a coherent way we will need \cite[Example C.15]{rmot} and the surrounding material.
As a consequence of this $\Syn_{\F_p}$ has the structure of a locally filtered category in the sense of \cite{rotation}
where the canonical shift map is $\lambda$.
Through this we make the following construction. %synthetic spectra inherit the following from the universal example of a locally filtered category (filtered spectra):

\begin{cnstr}\label{cnstr:bock-maps}
  The locally filtered structure on $\Syn_{\F_p}$ comes from a symmetric monoidal left adjoint
  \[ i : \Sp^{\Fil} \to \Syn_{\F_p} \]
where $\Sp^{\Fil}$ is the category of filtered spectra and $i$ sends the shifted sphere spectrum $S(1)$ to $\Ss^{0,1}$ (see \cite[Definition B.3]{rmot} for the definition of filtered spectra and the natural automorphism $(1)$).
  This provides us with a tower of commutative algebras
  \[ \Ss \to \cdots \to \clambda^3 \to \clambda^2 \to \clambda. \]
  Using $r_{n,m}$ to denote these restriction maps, we have cofiber sequences
  \[ \Sigma^{0,m-n}\clambda^{n-m} \xrightarrow{\underline{\lambda^{m}}}  \clambda^n \xrightarrow{r_{n,m}} \clambda^m \xrightarrow{\delta_{n,m}} \Sigma^{1, m-n-1}\clambda^{n-m}. \]
  \[ \Sigma^{0,-m}\Ss \xrightarrow{{\lambda^{m}}}  \Ss \xrightarrow{r_{m}} \clambda^m \xrightarrow{\delta_{m}} \Sigma^{1, -m-1} \Ss. \]
  % Note that since the maps $\delta_{n,m}$ encode Adams differentials their linearity over $\clambda^n$ implies a Liebniz rule for Adams differentials.   
\end{cnstr}

The commutative algebra structure on $\clambda^n$ provides us with a symmetric monoidal category of modules over this object. In view of the connection between the $\lambda$-bockstein and Adams differentials we think of the category of $\clambda^n$-modules as providing a way to work at the level of the Adams $E_{n+1}$-page.

\begin{wrn}
  Outside the $n=1$ case, 
  the bigraded homotopy groups of $\clambda^n$ are not literally given by the Adams $E_{n+1}$-page.
  Instead, thinking through the mechanics of the $\lambda$-bockstein spectral sequence one finds that there is a surjective map
  \[ \pi_{k,s}(\clambda^n) \longrightarrow E_{n+1}^{s, k+s} \]
  with kernel generated by the image of $\underline{\lambda}$ together with the $\lambda^{n-1}$-torsion classes. Here $\underline{\lambda}$ is the boundary map in the cofiber sequence in \Cref{cnstr:bock-maps} for $m=1, n \geq 2$.
\end{wrn}

%---------------------------------------------------------------------------------%
\subsubsection{Examples of synthetic homotopy groups}\hfill

%Although \cite[Theorem 9.19]{boundaries} gives a complete interpretation of the data in the Adams spectral sequence in terms of synthetic homotopy groups its statement is quite long and it takes some effort to digest properly. 
In order to prepare for later subsections we work through several examples of synthetic homotopy groups. Each of our examples will correspond to understanding a small region of the Adams spectral sequence for the sphere near $h_j^2$. Before proceeding, we suggest the reader new to synthetic spectra look at \cite[Section A.2]{boundaries} which works through an example chosen for its instructiveness in full detail.

\input{kerv-chart.tex}

\begin{figure}[t]
  \centering
  The $\F_2$-Adams spectral sequence near $h_j^2$ \\\vspace{6pt}
  \scalebox{1.00}{
    \printpage[ name = thetaASS, page = 2 ]
  }
  \caption{
    The Adams spectral sequence for the sphere near $h_j^2$ for $j \geq 10$.
    We have indexed the chart so that $0$ on the $x$-axis corresponds to topological degree $2^{j+1}$.
    % The indicated differentials are Adams' Hopf invariant one differentials \cite{TODO} and the class $\eta_{j+1}$ is a permanent cycle constructed by Mahowald \cite{TODO}. 
    % Note that in this chart we have drawn some, but not all differentials and multiplications that extend beyond the boundary of the page. For example, we are not claiming that $h_2h_j^2$ is zero.
  }
  \label{fig:hj2}    
\end{figure}

In \Cref{fig:hj2} we display the Adams spectral sequence for the sphere near the class $h_j^2$ for\footnote{In the $j=9$ case which is not pictured there is a single extra dot, $V_2'$, located at $(-1,5)$.} $j \geq 10$.
The structure of the $E_2$-page in this range can be obtained from \cite{Ext5} and the indicated differentials are the Hopf invariant one differentials proved by Adams \cite{Adams}.
The class $h_1h_{j+1}$ is a permanent cycle detecting the class $\eta_{j+1}$ in the homotopy groups of spheres constructed by Mahowald \cite{etaj}.
%Applying \cite[Theorem 9.19]{boundaries} we obtain the following lemma.

Let $\wt{2}$ denote the synthetic homotopy class that is detected by $h_0$ in the Adams spectral sequence for $\mathbb{S}$ and that $\lambda \cdot \wt{2} = 2$.

\begin{lem} \label{lem:near-hj2}
  Near $h_{j}^2$ for $j \geq 9$ the synthetic homotopy groups satisfy
  \begin{align*}    
    \pi_{2^{j+1}-4,4}(\clambda^2) &= 0,
    & \pi_{2^{j+1}-3,4}(\clambda^2) &= 0,
    & \pi_{2^{j+1}-2,4}(\clambda^2) &= \F_2\{ \wt{2}^2 \theta_j \}, \\
    \pi_{2^{j+1}-4,3}(\clambda^3) &= 0,
    &  \pi_{2^{j+1}-3,3}(\clambda^3) &= 0,
    & \pi_{2^{j+1}-2,3}(\clambda^3) &= \F_2\{ \wt{2} \theta_j \}, \\
    &
    & &
    & \pi_{2^{j+1}-2,2}(\clambda^3) &= \F_2\{ \theta_j \},
  \end{align*}
  where $\theta_j$ is a choice of lift of $h_j^2$ to $\clambda^3$.
\end{lem}

% \begin{proof}
%   Applying \cite[Theorem 9.19]{boundaries} to the information displayed in \Cref{fig:hj2} we obtain the lemma.
% \end{proof}

\begin{figure}[t]
  \centering
  The $\F_2$-Adams spectral sequence near $h_7^2$ and $h_8^2$ \\\vspace{6pt}
  \scalebox{1.00}{
    \printpage[ name = theta7ASS, page = 2 ]
    \printpage[ name = theta8ASS, page = 2, no y ticks, x axis tail = 0cm ]
  }
  \caption{
    Left: The $\F_2$-Adams spectral sequence for the sphere near $h_7^2$.
    Right: Near $h_8^2$. Note that $d_3(h_8^2) = 0$ by \Cref{rmk:theta-ctau-3}.
    %As in the previous chart we have drawn some, but not all differentials and multiplications that extend beyond the boundary of the page. For example, we are not claiming that $d_2(V_0') = 0$ nor are we claiming that we know whether $h_0V_0'$ is non-trivial.
  }
  \label{fig:theta7-8}    
\end{figure}

In \Cref{fig:theta7-8} we display the Adams spectral sequence for the sphere around the classes $h_7^2$ and $h_8^2$. As above, the structure of the $E_2$-page is obtained from \cite{Ext5} and the indicated differentials are the Hopf invariant one differentials. Additionally, we will show in \Cref{rmk:theta-ctau-3} that $d_3(h_8^2) = 0$.
% Applying \cite[Theorem 9.19]{boundaries} we obtain the following lemma.

\begin{lem} \label{lem:theta7-theta8-nearby}
  Near $h_7^2$ and $h_8^2$ the synthetic homotopy groups satisfy
  \begin{align*}    
    \pi_{252,4}(\clambda^2) &\cong \F_2\{ \underline{\lambda} V_0' \},
    & \pi_{253,4}(\clambda^2) &= 0,
    & \pi_{254,4}(\clambda^2) &= \F_2\{ \wt{2}^2 \theta_7 \}, \\
    \pi_{252,3}(\clambda^2) &= 0,
    &  \pi_{253,3}(\clambda^3) &= 0,
    & \pi_{254,3}(\clambda^3) &= \F_2\{ \wt{2} \theta_7 \}, \\
    \pi_{252,3}(\clambda^3) &\cong \F_2\{ \underline{\lambda^2} V_0' \},
    & &
    & \pi_{254,2}(\clambda^3) &= \F_2\{ \theta_7 \},
  \end{align*}
  where $\theta_7$ is a choice of lift of $h_7^2$ to $\clambda^3$
  \begin{align*}
    \pi_{508,4}(\clambda^2) &= 0, &
    \pi_{509,4}(\clambda^2) &= \F_2\{ \underline{\lambda} V_1' \}, &
    \pi_{510,4}(\clambda^2) &= \F_2\{ \wt{2}^2 \theta_8 \}, \\
    \pi_{508,3}(\clambda^3) &= 0, &
    \pi_{509,3}(\clambda^2) &= 0, &
    \pi_{510,3}(\clambda^3) &= \F_2\{ \wt{2} \theta_8 \}, \\
    & & 
    & &
    \pi_{510,2}(\clambda^2) &= \F_2\{ \theta_8 \}, 
  \end{align*}
  and $\theta_8$ is a choice of lift of $h_8^2$ to $\clambda^3$. % and the notation $\underline{\lambda}$ will be explained in the next subsection.
\end{lem}

% \begin{proof}
%   As in \Cref{lem:near-hj2} we apply \cite[Theorem 9.19]{boundaries} to the information displayed in \Cref{fig:theta7-8}. The final key input is that $d_3(h_8^2) = 0$.
% \end{proof}

We leave it as an instructive exercise for the reader to work out the proofs of Lemmas~\ref{lem:near-hj2}, \ref{lem:theta7-theta8-nearby} using \cite[Theorem 9.19]{boundaries}.
% Note that by Lemma~\ref{classical d3 hj2}, $d_3(h_8^2) = 0$.

%why the strength of our conclusions near $h_7^2$ and $h_8^2$ differ and how ruling out a potential $d_3$ on $h_8^2$ would repair things\footnote{This potential differential will be ruled out by \Cref{cor:theta-ctau-4}.}. Note that we only state $\pi_{509,3}(\clambda^2) = 0$ instead of $ \pi_{509,3}(\clambda^3) &= 0$.

% \begin{figure}[t]
%   \centering
%   The $\F_2$-Adams spectral sequence near $h_6^2$ \\\vspace{6pt}
%   \scalebox{1.00}{
%     \printpage[ name = theta6ASS, page = 2 ]
%   }
%   \caption{
%     Black dots denote a copy of $\F_2$ on the $E_2$ page.
%     Explain indexing.
%   }
%   \label{fig:theta6ASS}    
% \end{figure}

\begin{figure}[t]
  \centering
  The $E_2$-page of the $\F_2$-Adams spectral sequence near $h_6^2$ \\\vspace{6pt}
  \setlength{\columnsep}{0cm}
  \begin{multicols}{2}
    \frame{\includegraphics[scale=1.15]{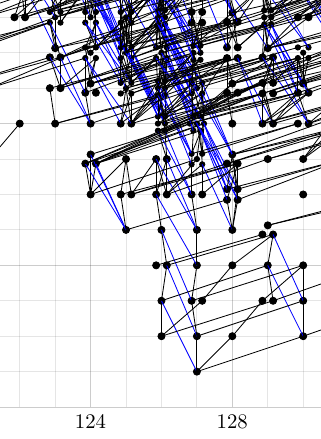}}
    \columnbreak
    
    \frame{\includegraphics[trim={0cm 0cm 0cm 0.01cm}, scale=1.15]{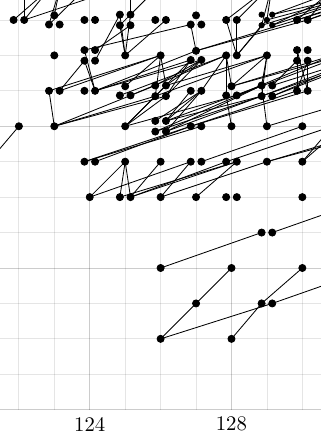}}
  \end{multicols}

  \caption{
    In this figure, reproduced from \cite{dexter-chart}, we display the $E_2$-page of the Adams spectral sequence near $h_6^2$ with all $d_2$ differentials and the corresponding $E_3$-page.
  }
  \label{fig:theta6ASS}    
\end{figure}

In \Cref{fig:theta6ASS} we display the Adams spectral sequence for the sphere near the class $h_6^2$.
In this case our knowledge of the $E_2$-page comes from \cite{Bruner2} and the $d_2$ differentials were computed in \cite{dexter} using an implementation of an algorithm of Nassau refining Baues' work on algorithmic computation of Adams $d_2$ differentials. 
% off of the $h_0$-tower on $h_7$ are due to Adams. The $d_2$ differential on $K_0$ can be obtained by applying Bruner's power operation formulas \cite{TODO}
% \[ d_2(D_3(1)) = d_2(\Sq^0(D_3)) = 0, \]
% \[ d_2(K_0) = d_2(\Sq^1(D_3)) = h_0\Sq^2(D_3) = h_0x_{6,94}. \]
% The computation of the relevant Steenrod squares comes from \cite{TODO}.
Applying \cite[Theorem 9.19]{boundaries} we obtain the following lemma.

\begin{lem} \label{lem:theta6-nearby}
  Near $h_{6}^2$ the synthetic homotopy groups satisfy
  \begin{align*}    
    \pi_{124,4}(\clambda^2) &= 0,
    & \pi_{125,4}(\clambda^2) &= \F_2\{ \underline{\lambda} K_0 \},
    & \pi_{126,4}(\clambda^2) &= \F_2\{ \wt{2}^2 \theta_j,\ [D_3(1)] \}, \\
    \pi_{124,3}(\clambda^3) &= 0,
    &  \pi_{125,3}(\clambda^2) &= 0,
    & \pi_{126,3}(\clambda^3) &= \F_2\{ \wt{2} \theta_j,\ \underline{\lambda} [D_3(1)] \}, \\
    &
    & \pi_{125,3}(\clambda^3) &= \F_2\{ \underline{\lambda^2} K_0 \},
    & \pi_{126,2}(\clambda^3) &= \F_2\{ \theta_j,\ \underline{\lambda^2} D_3(1) \},
  \end{align*}
  where $\theta_6$ is a choice of lift of $h_6^2$ to $\clambda^3$, $[D_3(1)]$ is a choice of lift of $D_3(1)$ to $\clambda^2$.
\end{lem}

\begin{rmk}
  The reader may have noticed that, somewhat counter-intuitively, we have chosen to present our examples in the opposite of the usual order. Our reason for doing this is to emphasize that the complexity of $E_2$ page is minimal for the generic element in a $\Sq^0$-family. We expect that this phenomenon is robust.
\end{rmk}

% ---------------------------------------------------------------------------------%
\subsubsection{Power operations}\hfill

The study of power operations in the Adams spectral sequence began with work of Adams, Barratt and Mahowald on the quadratic construction. These ideas were developed further by several others over the next two decades attaining a relatively stable form with Bruner's treatment in \cite{BMMS}.
In order to give our synthetic inductive argument we will need to introduce the synthetic incarnation of this material.\footnote{For a more relaxed introduction see the corresponding chapter in \cite{cookware}}

For the purposes of this paper we only need the simplest example of a power operation: the quadratic construction. This construction, which can be performed in any symmetric monoidal category, takes a map $f$ with target the unit and sends it to the composite
\[ X^{\otimes 2}_{hC_2} \xrightarrow{f^{\otimes 2}_{hC_2}} \o^{\otimes 2}_{hC_2} \xrightarrow{\mu} \o \]
where $\mu$ is the multiplication map provided by the symmetric monoidal structure.

The quadratic construction uses three things: colimits, the symmetric monoidal structure and the commutative algebra structure on the unit. Since a symmetric monoidal left adjoint preserves each of these we learn that the quadratic construction is compatible with such functors. Now consider the following span of symmetric monoidal left adjoints.

\begin{center}
  \begin{tikzcd}
    & \Syn_{\F_p} \ar[r, "\clambda^n \otimes -"] \ar[dl, "\lambda^{-1}"] & \mathrm{Mod}(\Syn_{\F_p}; \clambda^n) \ar[dr, "\clambda \otimes_{\clambda^n} -"] \\
    \Sp & & & \mathrm{Stable}(\A)
  \end{tikzcd}
\end{center}

Although not immediately obvious, the existence of this span essentially encodes (nearly) all known compatibilites between power operations and the Adams spectral sequence. As a demonstration of this we consider the example of a map of synthetic spheres.

\begin{exm} \label{exm:quadratic-basics}
  Suppose we have a map $\alpha : \Ss^{k,s} \to \Ss^{0,0}$.
  Applying the quadratic construction we obtain an associated map
  \[ \Sq(\alpha) : (\Ss^{k,s})^{\otimes 2}_{hC_2} \to \Ss^{0,0}. \]
  The usual filtration on the $C_2$ orbits provides us with a filtration on the source whose associated graded is given by $\Ss^{2k+i,2s-i}$ in its $i^{\text{th}}$ piece. The bottom piece of this filtration is a copy of $\Ss^{2k,2s}$ and composing with the associated inclusion gives $\alpha^2$. Inverting $\lambda$ on $(\Ss^{k,s})^{\otimes 2}_{hC_2}$ gives us $(S^{k})^{\otimes 2}_{hC_2} \simeq \Sigma^k \mathbb{R}\mathrm{P}^\infty_k$ which allows us to work out the attaching maps.\footnote{Since all the attaching maps decrease $s$ they are uniquely determined by what happens on inverting $\lambda$, i.e. by the attaching maps in a stunted projective space.} If we think about $\alpha$ in terms of the class it represents on inverting $\lambda$, then the (changing) values of the $s$-degree of the cells of the quadratic construction are recording lower bounds on the Adams filtrations of power operations.

  Now let's tensor down to $\clambda$.
  The map $\alpha$ now becomes some class $a$ on the Adams $E_2$-page.
  The induced filtration on the quadratic construction now splits because there are no attaching maps of the appropriate degree on the $E_2$-page, so
  \[ \clambda \otimes (\Ss^{k,s})^{\otimes 2}_{hC_2} \simeq \oplus_{i \geq 0} \Sigma^{2k+i,2s-i} \clambda.  \]
  This provides a family of power operations $Q_i$ for $i \geq 0$ where $Q_0(a) = a^2$.
  Examining \cite{May-power-ops} we learn that $Q_i(a) = \Sq^{s-i}(a)$ and $Q_i(a) = 0$ for $i > s$. This tells us that the quadratic construction in the synthetic category unifies the algebraic Steenrod operations on the Adams $E_2$-page with the quadratic construction in the category of spectra.

  Bruner's formulas for Adams differentials on algebraic Steenrod squares can now be read off by examining the $\lambda$-bockstein spectral sequence for $(\Ss^{k,s})^{\otimes 2}_{hC_2}$ and pushing these differentials forward using $\Sq(\alpha)$.
\end{exm}

In general, if we think about $X$ in terms of its bigraded cells, then $\pi_{**}(\nu\F_p/\lambda \otimes X)$ is a bigraded vector space which records the locations of these cells. Then using the fact that $\nu\F_p/\lambda \otimes -$ is symmetric monoidal we can work out the cells needed for $X^{\otimes 2}_{hC_2}$.

\begin{exm}
  Suppose that $X$ has two cells $e_1$ and $e_2$ where $e_j$ is in degree $(k_j,s_j)$.
  Then, $X^{\otimes 2}_{hC_2}$ has cells $Q_i(e_j)$ for $i \geq 0$ in degree $(2k_j+i, 2s_j-i)$ and a single cell $e_1e_2$ in degree $(k_1+k_2,s_1+s_2)$.
\end{exm}

\subsection{The inductive approach}
\label{subsec:inductive}\hfill

We now come to the main task of this section: giving a synthetic refinement of the inductive approach to constructing $\Theta_j$. In order to orient the reader we give a brief outline of the main points of the original version of this argument. %\footnote{We also suggest that reader unfamiliar with the original paper \cite{Inductive} read }

\begin{enumerate}
\item Using the hypothesis that $2\Theta_j = 0$ construct a map
  \[ \widehat{\Theta}_j : \Sigma^{2^{j+1}-2} S^0/2 \to S^0 \]
  which is $\Theta_j$ on the bottom cell.
\item Apply the quadratic construction to $\widehat{\Theta}_j$.
  A cell diagram of the source of the resulting map in shown below.
  \begin{center} \begin{tikzcd}[cells={nodes={draw=gray}}, sep=huge]
      {Q_3(a)} \ar[dd, bend left=50, no head, "\eta"] & & Q_1(b) \ar[ddll, no head, "\eta"'] \ar[d, "2", no head] \\
      {Q_2(a)} \ar[dd, bend right=50, no head, "\eta"'] \ar[d, "2", no head] & & {\color{blue} Q_0(b)} \ar[ddll, no head, "\eta"'] \\
      {Q_1(a)}  & & {ab} \ar[dll, "2", no head] \\
      {Q_0(a)}
    \end{tikzcd} \end{center}
  Each cell in this diagram is labelled by its associated Dyer--Lashof operation where $a$ is the bottom cell of the copy of $S^0/2$ and $b$ is the top cell.
\item Composing $D_2(\widehat{\Theta}_j)$ with the inclusion of the bottom cell gives $\Theta_j^2$ and assuming we have a nullhomotopy of this map, the induced map out of the cell indicated in blue will be $\Theta_{j+1}$ and the difference of the pair of cells above it will record a nullhomotopy of $2\Theta_{j+1}$.
\end{enumerate}

Before proceeding further we remark that since inverting $\lambda$ is symmetric monoidal and commutes with colimits any lift of this procedure to the synthetic category (using maps between spheres) will produce something lying over $\Theta_j$. In essence, all that this would do is record an Adams filtration bound on $\Theta_{j+1}$. The next theorem shows that if we linearize with respect to $\clambda^r$, then the same argument works though with a weaker output and an inductive hypothesis which is easier to check (a substantial victory)!

\begin{thm} \label{thm:inductive}
  Fix an $r \geq 2$ and suppose that $\theta_j$ is a lift of $h_j^2$ from $\clambda$ to $\clambda^r$.
  If $2\theta_j = 0$ and $\lambda^2 \theta_j^2 = 0$ in $\pi_{**}(\clambda^r)$,
  then there exists a class $\theta_{j+1}$ lifting $h_{j+1}^2$ to $\clambda^r$ such that $2\theta_{j+1} = 0$ in $\pi_{**}(\clambda^r)$.
  % As a corollary we show that
  % $ d_4(h_j^2) = 0 $
  % for all $j \geq 0$.
  % Suppose that $\theta_j$ is a lift of $h_j^2$ to $\clambda^r$
  % such that $2\theta_j = 0$ and $\lambda^2 \theta_j^2 = 0$,
  % then there exists a $\theta_{j+1}$ in $\clambda^r$ lifting $h_{j+1}^2$
  % such that $2\theta_{j+1} = 0$.
\end{thm}

Throughout the proof of \Cref{thm:inductive} we will work in the symmetric monoidal category of $\clambda^r$-modules.
In order to simplify notation we will use $\o$ to denote $\clambda^r$ (since this is the monoidal unit of the category) and $\o^{k,s}$ to denote $\Sigma^{k,s}\o$.
Before proceeding we pause to point out the key ideas in this proof:
\begin{itemize}
\item Although we work in $\clambda^r$-modules, we can often maintain control over maps at the level of their image under the functor $\clambda \otimes_{\clambda^r} -$.% tensoring down to $\clambda$.
\item In $\clambda$-modules we have good control over the quadratic construction since it induces the action of algebraic Steenrod squares on the cohomology of the Steenrod algebra.% via the algebraic Steenrod squares.
\end{itemize}

\begin{proof}      
  Using a choice of nullhomotopy of $2\theta_j$ we can write down a map
  \[ \widehat{\theta}_j : \o^{2^{j+1}-2,2}/2 \to \o \]
  which is $\theta_j$ on the bottom cell.
  For later use we record the following property (which all choices of $\widehat{\theta}_j$ share):
  \begin{itemize}
  \item[($*$)] Upon tensoring down to $\clambda$ the source of $\widehat{\theta}_j$ splits into two cells.
    The bottom cell maps to $\o$ by $h_j^2$ while the top cell maps to $\o$ by $h_{j+1}$ (as long as $r \geq 2$ and $j \geq 3$).
  \end{itemize}
  To verify ($*$) consider the map of $\lambda$-bockstein spectral spectral sequences induced by $\widehat{\theta}_j$. Label the bottom cell of the source by $a$ and the top cell by $b$. Then, we have $\lambda$-bockstein differential $d_1(b) = h_0 a$. By naturality we learn $d_1(\widehat{\theta}_j(b)) = h_0h_j^2 \neq 0$, which implies $\widehat{\theta}_j(b)$ is the unique non-zero class in $\clambda$ in this degree---$h_{j+1}$\footnote{Note that in this argument we have \emph{proved} the Hopf invariant differentials rather than citing them. The reader should compare this argument with the one which can be made using Bruner's power operations formulas.}.
    
  Applying the quadratic construction (i.e. the functor $(-)^{\otimes 2}_{hC_2}$)
  and postcomposing with the multiplication map on $\o$
  we obtain a map
  \[ \mathrm{Sq}(\widehat{\theta}_j) : (\o^{2^{j+1}-2,2}/2)^{\otimes 2}_{hC_2} \to \o \]
  which is $\theta_j^2$ on the bottom cell.
  Now we make a couple of observations.
  First, we can read off from \Cref{exm:quadratic-basics} that the cell diagram for $(\Ss^{2^{j+1} -2}/2)^{\otimes 2}_{hC_2}$ given above is also a cell structure for $(\o^{2^{j+1}-2,2}/2)^{\otimes 2}_{hC_2}$ where the bigradings of the cells have constant weight (the difference of the two coordinates).
  Second, after tensoring down to $\clambda$ the source splits as a sum of (shifts of) copies of the unit and the Dyer--Lashof operations the cells were labelled by now correspond to the Dyer-Lashof operations (i.e. Steenrod square) they record in $\A$-comodules.
  
  % From the Hopf invariant one differentials we can read off that 
  % any null homotopy of $2\theta_j$ will give a $\widehat{\theta}_j$ that has a $h_j^2$ on the bottom cell and a $h_{j+1}$ on the top cell after tensoring down to $\clambda$.
  As above, we will label the bottom and top cell of $\o^{2^{j+1}-2,2}/2$ as $a$ and $b$ respectively.
  Using our identification of what $b$ does on tensoring down to $\clambda$ from ($*$) we can read off that the cell $Q_0(b)$ in $(\o^{2^{j+1}-2,2}/2)^{\otimes 2}_{hC_2}$ will map to $\o$ via $h_{j+1}^2$.
  This means that if we can remove the cells below this one, then we will have a construction of $\theta_{j+1}$.  
  After including a subcomplex of $(\o^{2^{j+1}-2,2}/2)^{\otimes 2}_{hC_2}$ we obtain a map
  \begin{center} \begin{tikzcd}[cells={nodes={draw=gray}}]
      {Q_1(b) - Q_3(a)} \ar[d, "\lambda\wt{2}", no head] \\
      {Q_0(b)} \ar[dd, no head, "\lambda^2\eta"', bend right] \\
      & \\
      {Q_0(a)} \ar[drr, "\theta_j^2"] \\
      & & {\o}.
    \end{tikzcd} \end{center}
  Note the powers of $\lambda$ which appear in the attaching maps of this cell structure due to the bigradings of the cells. Next we compose this map with the map that pushes that the $\lambda^2$ off the bottom as indicated,
  \begin{center} \begin{tikzcd}[cells={nodes={draw=gray}}]
      {\o^{2^{j+2}-1,1}} \ar[rr, "1"] \ar[d, "\lambda\wt{2}", no head] & & {Q_1(b) - Q_3(a)} \ar[d, "\lambda\wt{2}", no head] \\
      {\o^{2^{j+2}-2,2}} \ar[rr, "1"] \ar[dd, no head, "\eta"', bend right] & & {Q_0(b)} \ar[dd, no head, "\lambda^2\eta"', bend right] \\
      & \\
      {\o^{2^{j+2}-4,2}} \ar[rr, "\lambda^2"] & & {Q_0(a)} \ar[drr, "\theta_j^2"] \\
      & & & & {\o}.
    \end{tikzcd} \end{center}
  Using our assumption that $\lambda^2\theta_j^2$ is nullhomotopic we can factor this composite through $\o^{2^{j+2}-2,2}/2$ via a map which is $h_j^2$ on the bottom cell if we tensor down to $\clambda$. This map is a good choice of $\widehat{\theta}_{j+1}$ and its construction completes the proof.
\end{proof}

\begin{rmk} \label{rmk:theta-ctau-3}
  The condition that $\lambda^2\theta_{j}^2 = 0$ in $\pi_{**}(\clambda^r)$ is implied by asking that $\theta_j^2 =0$ in $\pi_{**}(\clambda^{r-2})$. This means that as a corollary of the fact that $h_j^4 = 0$ for $j \geq 1$ we learn that $h_j^2$ lifts to $\clambda^3$. In particular, $d_2(h_j^2) = 0$ and $d_3(h_j^2) = 0$.
\end{rmk}

\begin{prop} \label{cor:theta-ctau-4}
  There exist lifts $\theta_j$ of the classes $h_j^2$ to $\clambda^4$ with the property that $2\theta_j = 0$ and $\lambda^2 \theta_j^2 = 0$.\footnote{As this is the farthest we lift $h_j^2$ in this paper we will hereafter fix our notation so that $\theta_j$ refers to the classes constructed in this proposition.} In more classical language the first claim is equivalent to saying that $d_4(h_j^2) = 0 $.
\end{prop}

\begin{proof}  
  We proceed by induction on $j$ using \Cref{thm:inductive}.
  The inductive step requires that we check at each step that $\lambda^2 \theta_j^2 = 0$ in $\pi_{**}(\clambda^4)$. In fact we will check the stronger statement that $\theta_j^2 = 0$ in $\pi_{**}(\clambda^2)$.

  For $j \leq 4$ this follows from the fact that $\Theta_j^2 = 0$.
  If $j \neq 7$, then the conclusion follows from the fact that $\Ext_{\A}^{5,2^{j+2} + 1} = 0$ (see Figures~\ref{fig:hj2}, \ref{fig:theta7-8} and \ref{fig:theta6ASS}).

  In the case $j=7$ we must rule out the possibility that $\theta_6^2 = \underline{\lambda} V_0'$ in $\pi_{252,4}(\clambda^2)$.
  In \Cref{lem:theta-4-fold}, in the next subsection, we will give a Toda bracket expression
  $ \theta_6 \in \langle 2, \theta_5, \lambda \theta_5, \wt{2} \rangle$ in $\pi_{126,2}(\clambda^2)$.
  Multiplying by $\theta_6$ and shuffling we have
  \[ \theta_6^2 \in \theta_6 \cdot \langle 2, \theta_5, \lambda \theta_5, \wt{2} \rangle \subseteq \langle \langle \theta_6, 2, \theta_5 \rangle, \lambda \theta_5, \wt{2} \rangle. \]
  
  To proceed we would like to determine the value of the 3-fold $\langle \theta_6, 2, \theta_5 \rangle \in \pi_{189,3}(\clambda^2)$.
  From \Cref{E2 descriptions}(1) we can read off that $\pi_{189,3}(\clambda^2) \cong \Z/4\{[h_6^3]\}$.
  Meanwhile, the synthetic analog of Moss' theorem (see \cite{cookware}) allows us to conclude that the image of $\langle \theta_6, 2, \theta_5 \rangle$ in $\pi_{189,3}(\clambda)$ is zero.
  As a consequence we have  
  \[ \langle \theta_6, 2, \theta_5 \rangle = 2a[h_6^3] \]
  for some $a \in \{0,1\}$.
  Now, with some further shuffling we may conclude that
  \begin{align*}
    \theta_6^2
    &= \langle 2a[h_6^3],\ \lambda \theta_5,\ \wt{2} \rangle 
      \subseteq \langle \wt{2},\ a[h_6^3] \lambda^2 \theta_5,\ \wt{2} \rangle = \langle \wt{2}, 0, \wt{2} \rangle
      = (\pi_{252,3}\clambda^2) \cdot \wt{2} = 0             
  \end{align*}
  where the final step uses \Cref{lem:theta7-theta8-nearby} which tells us $\pi_{252,3}\clambda^2 = 0$.
  \qedhere
  
  %The image of the 3-fold $\langle \theta_6, 2, \theta_5 \rangle$ in $\pi_{189,3}(\clambda)$ is zero by the synthetic analog of Moss' theorem (see \cite{cookware}). As a consequence there is an $x$ in $\pi_{189,4}(\clambda)$ such that $x \underline{\lambda} = \langle \theta_6, 2, \theta_5 \rangle$.
  %Using that we are working $\clambda^2$-linearly, shuffling and simplifying we now get\footnote{Recall that $r_{2,1}$ and $\underline{\lambda}$ are as defined in \Cref{cnstr:bock-maps}.}
  %\begin{align*}
  %  \theta_6^2
  %  &= \langle x \underline{\lambda},\ r_{2,1} \underline{\lambda} %\theta_5,\ \wt{2} \rangle 
  %    \subseteq \langle x,\ \underline{\lambda}r_{2,1} \underline{\lambda} %\theta_5,\ \wt{2} \rangle \\
  %  &= \langle x, 0, \wt{2} \rangle
  %    = x \cdot [\Sigma^{63,0} \clambda,\ \clambda^2]_{\clambda^2} + (\pi_{252,3}\clambda^2) \cdot \wt{2} \\
  %  &= x \cdot 0 + 0 \cdot \wt{2} = 0             
  %\end{align*}
  %where we have used $[-,-]_{\clambda^2}$ to denote homotopy classes of $\clambda^2$-linear maps.
  %The indentifications on the last line come from \Cref{lem:theta7-theta8-nearby} which tells us $\pi_{252,3}\clambda^2 = 0$ and the isomorphism
  %\[ [\Sigma^{63,0} \clambda, \clambda^2]_{\clambda^2} \cong [\Sigma^{63,0} \clambda^2, \mathbb{D}\clambda ]_{\clambda^2} \cong \pi_{64,-2}\clambda = 0. \]
  %Finally, we can use our conclusion that $\theta_j^2$ is zero in $\pi_{**}(\clambda^2)$ to conclude that $\lambda^2\theta_j^2 = 0$ in $\pi_{**}(\clambda^4)$ allowing us to apply \Cref{thm:inductive}.
\end{proof}

Recall from \Cref{cnstr:bock-maps} that $\delta_{4,1}: \clambda \rightarrow \Sigma^{1, -4}\clambda^{3}$ is the cofiber of the restriction map $r_{4,1}: \clambda^4 \rightarrow \clambda$.
% \clambda^n \xrightarrow{r_{n,m}} \clambda^m \xrightarrow{\delta_{n,m}} \Sigma^{1, m-n-1}\clambda^{n-m}

\begin{cor} \label{cor:Hopf-diff}
  For $j \geq 4$ we have $\delta_{4,1}(h_{j+1}) = \wt{2} \theta_{j}$ in $\pi_{**}(\clambda^3)$.
\end{cor}

\begin{proof}
  Examining the map induced by $\widehat{\theta}_j$ on $\lambda$-Bockteins in view of ($*$) in the proof of \Cref{thm:inductive} provides the desired conclusion.
\end{proof}

In the case of $h_6^2$ where we have a spherical class $\Theta_5$ to work with we can provide a more refined statement.

\begin{prop}
  The class $h_6^2$ survives to the $E_{r+3}$-page of the Adams spectral sequence if and only if $\eta\theta_5^2 = 0$ in $\pi_{**}(\clambda^{r})$. 
\end{prop}

\begin{proof}
  Recall from \Cref{cnstr:bock-maps} that we denote the cofiber of the $\lambda$-multiplication map by $\delta_1:\clambda \rightarrow \Sigma^{1, -2} \Ss $, which is the total $\lambda$-bockstein.
   We examine the bottom two cells in the final diagram in the proof of \Cref{thm:inductive}.
  The total $\lambda$-bockstein in the source of this map records the attaching map so
  \[ \delta_1(Q_0(b)) = \lambda \eta Q_0(a). \]
  By naturality we now learn that
  \[ \delta_1(h_6^2) = \lambda \eta \Theta_5^2. \]
  This implies that $h_6^2$ is a permanent cycle if and only if $\lambda \eta \Theta_5^2 = 0$ and more specifically that the $\lambda$-power divisibility of this product records the length of the Adams differential on $h_6^2$.
  % conclude that $\lambda^2\eta\Theta_5^2 = 0$. If the product $\eta\Theta_5^2$ has non-trivial image in $\clambda^{r-2}$, then this would force a $\lambda$-Bockstein differential enforcing this relation. $h_6^2$ is the only possible source of such a differential.

  % For the ``if'' direction we consider the same diagram.
  % Instead of pushing the $\lambda^2$ off the bottom we could have pushed the $\eta$ off the bottom. 
\end{proof}

% \begin{rmk}
%   This proof of this proposition doesn't actually use that we're working in the ``$j=6$'' case.
%   We've 
  
% \end{rmk}
Using this proposition and our knowledge of the synthetic stable stems through topological degree $95$ can give an analysis of the class $\eta \Theta_5^2$ and a corresponding lower bound on the length of an Adams differential on $h_6^2$.

\begin{thm}
  The class $\eta \theta_5^2$ is divisible by $\lambda^6$. In particular, through the previous proposition we learn that $h_6^2$ survives to the Adams $E_9$-page.
\end{thm}

\begin{proof}
Classically, by Corollary~1.10 in \cite{WX}, we have a strictly defined 4-fold Toda bracket for $\Theta_5$:
$$\Theta_5 \in \langle 2, \Theta_4, \Theta_4, 2 \rangle.$$
by Corollary~1.3 in \cite{Xu}, we have $2 \Theta_5 = 0$, therefore
$$\Theta_5^2 \in \Theta_5 \cdot \langle 2, \Theta_4, \Theta_4, 2 \rangle  \subseteq \langle \langle \Theta_5, 2, \Theta_4 \rangle, \Theta_4, 2 \rangle.$$
  
Working synthetically, we also have 
$$2\theta_4 = 0 \in \pi_{30,2}\Ss, \ \theta_4^2=0 \in \pi_{60,4}\Ss, \ 2\theta_5=0 \in \pi_{62,2}\Ss,$$
$$\langle 2, \theta_4, \theta_4\rangle = 0 \in \pi_{61,3}\Ss , \ \theta_5 \in \langle 2, \theta_4, \theta_4, 2 \rangle \subseteq \pi_{62,2}\Ss,$$
by comparing with classical statements and using that these bidegrees are $\lambda$-torsion free. Therefore, 
$$\theta_5^2 \in \theta_5 \cdot \langle 2, \theta_4, \theta_4, 2 \rangle  \subseteq \langle \langle \theta_5, 2, \theta_4 \rangle, \theta_4, 2 \rangle \subseteq \pi_{124,4}\Ss.$$
In order to run the main argument that proves the theorem we now need six additional facts whose proofs we defer for the moment.

\begin{enumerate}
\item $ \pi_{93,3}\Ss \subseteq \lambda^2 \cdot \pi_{93,5}\Ss $.
  In particular, any element in $\langle \theta_5, 2, \theta_4 \rangle \subseteq \pi_{93,3}\Ss$ can be written in the form $\lambda^2 x$, where $x \in \pi_{93,5}\Ss$.\vspace{4pt}
\item For any $x$ as in $(1)$, we have $x \cdot \theta_4 = 0$ in $\pi_{123,7}\Ss$.\vspace{4pt}
\item For any $x$ as in $(1)$, any element in $\langle x, \theta_4, 2 \rangle \subseteq \pi_{124,6}\Ss$ can be written in the form $\lambda^2 z$, where $z \in \pi_{124, 8}\Ss$.\vspace{4pt}
\item $\pi_{124, 8}\Ss \subseteq \lambda \cdot \pi_{124, 9}\Ss$.\vspace{4pt}
\item $\eta \cdot \pi_{124, 9}\Ss \subseteq \lambda \cdot \pi_{125, 11}\Ss$.\vspace{4pt}
\item $\eta \cdot \pi_{31, 1}\Ss \subseteq \lambda^4 \cdot \pi_{32, 6}\Ss$.\vspace{4pt}
\end{enumerate} 

Fact $(1)$ lets us rewrite $\theta_5^2 \in \langle \langle \theta_5, 2, \theta_4 \rangle, \theta_4, 2 \rangle$ as 
$$\theta_5^2 \in \langle \lambda^2 x, \theta_4, 2 \rangle \ \text{for some} \ x \in \pi_{93,5}\Ss.$$
Now, by $(2)$, the bracket $\langle x, \theta_4, 2 \rangle$ is well-defined.
Choose a $y$ in this bracket, then we have
$$\lambda^2 \cdot y \in \lambda^2 \cdot \langle x, \theta_4, 2 \rangle \subseteq \langle \lambda^2 x, \theta_4, 2 \rangle.$$
Since both $\theta_5^2$ and $\lambda^2 \cdot y$ are contained in $\langle \lambda^2 x, \theta_4, 2 \rangle$, their difference belongs to its indeterminacy, which is $\lambda^2 x \cdot \pi_{31,1}\Ss + 2 \cdot \pi_{124,4}\Ss$. Therefore, we have
$$\theta_5^2 \in \lambda^2 \cdot y  + \lambda^2 x \cdot \pi_{31,1}\Ss + 2 \cdot \pi_{124,4}\Ss.$$
Using $(3)$ and $(4)$, we may write $y = \lambda^2 \cdot z = \lambda^2 \cdot \lambda \cdot w$
for some $z \in \pi_{124,8}\Ss$ and $w \in \pi_{124,9}\Ss$. 
Multiplying our expression for $\theta_5^2$ by $\eta$ we obtain
$$\eta \theta_5^2 \in \lambda^2 \cdot \eta \cdot y + \lambda^2 x \cdot \eta \cdot \pi_{31,1} =  \lambda^5 \cdot \eta \cdot w + \lambda^2 x \cdot \eta \cdot \pi_{31,1}\Ss.$$
By $(5)$, $\eta \cdot w$ is $\lambda$-divisible; by $(6)$, $\eta \cdot \pi_{31, 1}$ is $\lambda^4$-divisible. Therefore, $\eta \theta_5^2$ is $\lambda^6$-divisible. 

We now return to proving the six facts from above.\vspace{4pt}
\begin{enumerate}
\item In the 93-stem of the classical Adams sseq both $h_5^3$ and $h_0 h_5^3$ support non-trivial differentials (see \cite{IWX}).
  Synthetically, this means any element in $\pi_{93,3}\Ss$ is divisible by $\lambda^2$.\vspace{4pt}
  %This follows immediately from the classical fact (proved in \cite{IWX}) that in the 93-stem both $h_5^3$ and $h_0 h_5^3$ do not survive to the $E_\infty$-page, so the Toda bracket $\langle \Theta_4, 2, \Theta_5 \rangle$ is detected in Adams filtration at least 5. Synthetically, this means any element in $\langle \theta_4, 2, \theta_5 \rangle \subseteq \pi_{93,3}\Ss$ can be written in the form $\lambda^2 x$, where $x \in \pi_{93,5}\Ss$.
\item Classically, $\langle \Theta_5, 2, \Theta_4 \rangle \cdot \Theta_4 = \Theta_5 \cdot \langle 2, \Theta_4, \Theta_4 \rangle \subseteq \Theta_5 \cdot \pi_{61} = 0$, where the last equality is due to $\pi_{61}S^0 = 0$ (see Theorem~1.9 in \cite{WX}). This means that synthetically $\lambda^2 x \cdot \theta_4$ is $\lambda$-torsion.

From \Cref{fig:theta6ASS}, we know that on the classical Adams $E_2$-page
$$\Ext_{\A}^{s, 123+s} = 0 \ \text{for} \ s \leq 7,$$
$$\Ext_{\A}^{s, 124+s} = 0 \ \text{for} \ s \leq 5.$$
These facts means that synthetically $\pi_{123, 7}\Ss$ contains no non-trivial $\lambda$-torsion elements. Therefore, we must have $x \cdot \theta_4 = 0$ in $\pi_{123,7}\Ss$.\vspace{4pt} 
\item It is clear from the classical Adams $E_2$ and $E_3$-pages in \Cref{fig:theta6ASS} that the synthetic homotopy group 
$$\pi_{124, 6}(\clambda^2) \cong \F_2\{a_{124, 6}, \ \underline{\lambda} a_{124, 7}, \ \underline{\lambda} b_{124, 7}\},$$ 
where $a_{124, 6}$, $a_{124, 7}$, $b_{124, 7}$ are the generators of the classical $E_3$-page in the corresponding degrees. We need to rule out these three elements in $\pi_{124, 6}(\clambda^2)$.

For $a_{124, 6}$, note that $x$ is detected by $h_1g_3$ or it is further $\lambda$-divisible. Then it can be ruled out by the corresponding Massey product 
$$\langle h_1g_3, h_4^2, h_0 \rangle = h_1h_5g_3 = 0 \neq a_{124,6} $$ 
in the Adams $E_3$-page and Moss's theorem (there are no crossing differentials).

For the other two elements, note that both $a_{124, 7}$, $b_{124, 7}$ on the $E_3$-page have nonzero $h_2$-multiples. Therefore, we only need to show that 
$$\nu \cdot \langle x, \theta_4, 2 \rangle \subseteq \pi_{127, 7}(\clambda^2)$$ 
does not contain (a linear combination of) $\nu \cdot \underline{\lambda} a_{124, 7}$ and $\nu \cdot \underline{\lambda} b_{124, 7}$.

In fact, since $h_1g_3 \cdot h_2 = 0$ in $\Ext_{\A}^{6, 96+6}$, and $\Ext_{\A}^{7, 96+7} = 0$, we know that $\nu \cdot x$ is $\lambda^2$-divisible. Therefore, in $\pi_{127, 7}(\clambda^2)$, 
$$\nu \cdot \langle x, \theta_4, 2 \rangle \subseteq  \langle \nu \cdot x, \theta_4, 2 \rangle = \langle 0, \theta_4, 2 \rangle = 2 \cdot \pi_{127, 7}(\clambda^2).$$
Factoring $2$ as $\lambda \wt{2}$ we may use the isomorphism between the Adams $E_3$-page and the image of multiplication by $\lambda$ on $\pi_{**}\clambda^2$ to analyze $2$-divisibility.
In particular, we can read off that $\nu \cdot \underline{\lambda} a_{124, 7}$ and $\nu \cdot \underline{\lambda} b_{124, 7}$ are not 2-divisible in $\pi_{127, 7}(\clambda^2)$
from the fact that $h_2a_{124,7}$ and $h_2b_{124,7}$ are not $h_0$-divisible on the $E_3$-page.\vspace{4pt}
%Inspection of  \Cref{fig:theta6ASS} shows that both $\nu \cdot \underline{\lambda} a_{124, 7}$ and $\nu \cdot \underline{\lambda} b_{124, 7}$ are not 2-divisible in $\pi_{127, 7}(\clambda^2)$. This completes the proof that $\langle x, \theta_4, 2 \rangle$ is $\lambda^2$-divisible.\\

\item This is clear from \Cref{fig:theta6ASS} since the only nonzero element in $\Ext_{\A}^{8, 124+8}$ supports a nonzero $d_2$-differential. Therefore, we have $\pi_{124, 8} \subseteq \lambda \cdot \pi_{124, 9}\Ss$.\vspace{4pt}

\item This is clear from \Cref{fig:theta6ASS} since $h_1 \cdot \Ext_{\A}^{9, 124+9} = 0$. Therefore, we have $\eta \cdot \pi_{124, 9}\Ss \subseteq \lambda \cdot \pi_{125, 11}\Ss$.\vspace{4pt}

\item It is clear from the classical Adams spectral sequence that $\pi_{31, 1}\Ss$ is generated by $\lambda^2 \eta \theta_4$ and $\lambda^4$-multiples. So we only need to show that $\eta \cdot \lambda^2 \eta \theta_4$ is $\lambda^4$-divisible. This is true since it is zero. In fact, classically $\eta^2 \Theta_4 = 0$, so synthetically $\eta^2 \theta_4$ must be a $\lambda$-torsion or zero. For degree reasons, it must be zero. Therefore, we have $\eta \cdot \pi_{31, 1}\Ss \subseteq \lambda^4 \cdot \pi_{32, 6}\Ss$.

\end{enumerate} 
This completes the proof.
\end{proof}

\subsection{A (synthetic) product of $\theta_j$'s}\ 
\label{subsec:product-relation}

Having constructed synthetic lifts $\theta_j$ in $\clambda^4$ we make a short investigation of products among them. More specifically, we show that there is a (non-trivial) product
\[ \theta_j\theta_{j-1} = \underline{\lambda^2} h_0^2g_{j-2} \neq 0 \]
on the Adams $E_4$-page (i.e. in the homotopy of $\clambda^3$)\footnote{We have passed from $\clambda^4$ to $\clambda^3$ here in order to avoid potential extra terms in the product.}. % {and our use of brackets around $4\overline{\kappa}_{j-2}$ indicates that we do not yet know that $\overline{\kappa}_{j-2}$ survives to the $E_4$ page.}.
Though investigating this product may seem unmotivated, as it will turn out, 
knowledge of this product is essentially equivalent to \Cref{thm:main} less the assertion that the differential is non-trivial. % In point of fact, this subsection provides a key input in the proof of \Cref{thm:main} and the section that contains it was created to provide the inputs necessary for this subsection.

\begin{lem}\label{lem:theta-4-fold}
  For $j \geq 5$ and not $7$
  the $\clambda^3$-linear 4-fold
  $ \langle \wt{2}, \lambda \theta_j, \theta_j, 2 \rangle $
  is defined, contains $\theta_{j+1}$ and has no indeterminacy (except in the $j=5$ case where $\underline{\lambda}^2 D_3(1)$ may be in the indeterminacy).
\end{lem}

\begin{proof}
  We start with showing that the 4-fold is defined.
  In \Cref{cor:theta-ctau-4} we showed that $2\theta_j =0$ and $\lambda^2\theta_j^2 = 0$ in $\clambda^4$.
  This means that the 3-folds $\langle \wt{2}, \lambda \theta_j, \theta_j \rangle$ and $\langle \lambda \theta_j, \theta_j, \wt{2} \rangle$ are defined. After tensoring down to $\clambda^3$ we observe that since $\pi_{2^{j+2} - 3, 3}(\clambda^3) = \pi_{2^{j+2} - 3, 2}(\clambda^3) = 0$ as proved in \Cref{lem:near-hj2} and \Cref{lem:theta7-theta8-nearby} both 3-folds are automatically zero.

  In order to unambiguously identify the value of this 4-fold we begin by observing that the reduction map $ \pi_{2^{j+2} -2, 2}(\clambda^3) \to \pi_{2^{j+2} -2, 2}(\clambda)$ is an isomorphism (surjective with kernel $\underline{\lambda}^2D_3(1)$ in the $j=5$ case). As a consequence it will suffice to understand the value of this bracket after reduction to $\clambda$. The desired conclusion can now be obtained using the synthetic version of Moss' theorem for 4-folds \cite{cookware}. 
\end{proof}

As it would be much more involved to argue that $\langle \wt{2}, \lambda \theta_7, \theta_7 \rangle$ is zero in $\pi_{**}(\clambda^3)$ we instead prove a slightly modified form of \Cref{lem:theta-4-fold} in the $j=7$ case.

% In the case where $j = 7$ we 
% Although it seems likely that \Cref{lem:theta-4-fold} remains true in the case $j=7$, since the weaker form suffices for our application of interest we will refrain from tackling the issue of whether $\langle \wt{2}, \lambda \theta_7, \theta_7 \rangle$ is zero.

\begin{lem} \label{lem:theta-8-4-fold}
  In $\clambda^3$ the matric 4-fold
  \[ \left\langle \wt{2}, \begin{pmatrix} \lambda \theta_7 & 0 \end{pmatrix}, \begin{pmatrix} \theta_7 \\ \underline{\lambda^2} \end{pmatrix}, 2 \right\rangle \]
  is defined, contains $\theta_{8}$ and has zero indeterminacy.
\end{lem}

\begin{proof}
  The proof of this lemma is quite similar to previous one, so we will mostly indicate the necessary modification. Each of the 2-fold products are zero as before. 
  The right 3-fold is then defined and lands in a zero group, therefore it suffices to check that the left 3-fold contains zero.
  From \Cref{lem:theta7-theta8-nearby} we know that $\pi_{509,3}(\clambda^3) \cong \Z/2\{ \underline{\lambda^2} V_1'  \}$ which implies that the left 3-fold is equal to its own indeterminacy and therefore it contains zero.
  In order to evaluate the value of this 4-fold we begin with some shuffling,
  \[ \left\langle \wt{2}, \begin{pmatrix} \lambda \theta_7 & 0 \end{pmatrix}, \begin{pmatrix} \theta_7 \\ \underline{\lambda^2} \end{pmatrix}, 2 \right\rangle \subseteq \left\langle \wt{2}, \begin{pmatrix} \lambda \theta_7 & 0 \end{pmatrix}, \begin{pmatrix} \theta_7\lambda \\ 0 \end{pmatrix}, \wt{2} \right\rangle. \]
  
  Next we remove the pair of zero maps through $\Sigma^{1,-3}\clambda$ reducing to an ordinary 4-fold.
  Note that in doing this we must take into account the indeterminacy coming from composites which factor through $\Sigma^{1,-3}\clambda$.
  \[ \left\langle \wt{2}, \begin{pmatrix} \lambda \theta_7 & 0 \end{pmatrix}, \begin{pmatrix} \theta_7\lambda \\ 0 \end{pmatrix}, \wt{2} \right\rangle \subseteq \left\langle \wt{2}, \lambda \theta_7 , \theta_7\lambda , \wt{2} \right\rangle + [\Sigma^{510,2} \clambda^3, \Sigma^{1,-3} \clambda]_{\clambda^3} \cdot [\Sigma^{1,-3} \clambda, \clambda^3]_{\clambda^3} \]
  Using the fact $\clambda^3$ is a free $\clambda^3$-module the adjunction between $\mathbb{S}$-modules and $\clambda^3$-modules allows us to read off that for $f \in [\Sigma^{510,2} \clambda^3, \Sigma^{1,-3} \clambda]_{\clambda^3}$ and $g \in [\Sigma^{1,-3} \clambda, \clambda^3]_{\clambda^3}$ the composite $g \circ f$ depends only on the underlying $\mathbb{S}$-linear composite of $f$ and $g$. 
  From the cofiber sequence 
  $\Ss^{1,-3} \to \Sigma^{1,-3}\clambda \to \Ss^{2,-5}$
  and the fact that $\pi_{1,-3}\clambda^3 = 0$ and $\pi_{2,-5}\clambda^3 = 0$ we can read off that 
  $[\Sigma^{1,-3} \clambda, \clambda^3] = 0$.
  Thus, we may conclude that the composite $g \circ f$ is zero
  and our previous expression simplifies to
  \[ \left\langle \wt{2}, \begin{pmatrix} \lambda \theta_7 & 0 \end{pmatrix}, \begin{pmatrix} \theta_7 \\ \underline{\lambda^2} \end{pmatrix}, 2 \right\rangle \subseteq \left\langle \wt{2}, \lambda \theta_7 , \theta_7\lambda , \wt{2} \right\rangle.\]
  
  % \begin{align*}
  %   &\subseteq \left\langle \wt{2}, \lambda \theta_7 , \theta_7\lambda , \wt{2} \right\rangle \\
  %   &\quad + [\Sigma^{510,2} \clambda^3, \Sigma^{1,-3} \clambda]_{\clambda^3} \cdot [\Sigma^{1,-3} \clambda, \clambda^3]_{\clambda^3} \\
  %   &= \left\langle \wt{2}, \lambda \theta_7 , \theta_7\lambda , \wt{2} \right\rangle + [\Sigma^{510,2} \clambda^3, \Sigma^{1,-3} \clambda]_{\clambda^3} \cdot 0 \\
  %   &= \left\langle \wt{2}, \lambda \theta_7 , \theta_7\lambda , \wt{2} \right\rangle 
  % \end{align*}
  At point this we can conclude as in the previous lemma by observing that the value of this bracket is uniquely identified after reduction to $\clambda$ and then applying the synthetic version of Moss' theorem for 4-folds \cite{cookware}.
\end{proof}

\begin{lem} \label{lem:theta-product-divisible}
  In $\clambda^3$ the product $\theta_{j}\theta_{j+1}$ is divisible by $\wt{2}$ for $j \geq 5$.  
\end{lem}

\begin{proof}
  We begin by handling the case $j \neq 5,7$.
  From \Cref{lem:theta7-theta8-nearby} and \Cref{lem:near-hj2} we know that $\pi_{2^{j+2}-3, 3}\clambda^3 = 0$. This implies that $\langle \theta_j, 2, \theta_j \rangle$ is strictly zero. As a consequence of this relation we are able to shuffle the expression for $\theta_j$ given in \Cref{lem:theta-4-fold} to produce the desired divisibility\footnote{Note that shuffling gives a one-to-one correspondence of values of brackets so indeterminacy issues do not arise in this argument.}
  \[ \langle \wt{2}, \lambda \theta_j, \theta_j, 2 \rangle \theta_j = \wt{2} \langle \lambda \theta_j, \theta_j, 2,  \theta_j \rangle. \]

  In the cases $j=5,7$ we need a different argument to conclude that
  $\langle \theta_j, 2, \theta_j \rangle$ contains zero.
  Since $2\theta_j = 0$ and $\clambda^3$ is an $\E_\infty$-algebra (see \Cref{cnstr:bock-maps})
  %\footnote{As we could not find a construction of the $\E_\infty$-algebra structure on $\clambda^3$ in the literature we include the following brief argument: Using \cite[Remark 4.31]{Pstragowski} we can write $\clambda^3$ as $\tau_{\leq 2}\Ss$ where the truncation with respect to the natural $t$-structure on synthetic spectra (see \cite[\S4.2]{Pstragowski}). The $\E_\infty$-algebra structure on $\clambda^3$ now comes from the fact that the natural $t$-structure is compatible with the tensor product and the fact that the unit is a connective $\E_\infty$-algebra.}
  this bracket contains the product $2 \cdot Q_1(\theta_j)$ where here we are referring to the spherical power operation $Q_1$ on even degree classes first identified by Toda \cite[p.27-28]{TodaBook}. Since all elements of $\pi_{125,3}(\clambda^3)$ and $\pi_{509,3}(\clambda^3)$ are simple 2-torsion we know $2 Q_1(\theta_j) =0$ for $j=5,7$. This finishes the $j=5$ case.
  
  In the case $j=7$ we follow the same basic strategy replacing \Cref{lem:theta-4-fold} with \Cref{lem:theta-8-4-fold}.
  In order to shuffle the 4-fold in this case we need the pair of 3-folds
  \[ \langle \underline{\lambda^2}, 2, \theta_7 \rangle \qquad \text{ and } \qquad  \langle \theta_7, 2, \theta_7 \rangle\]
  to be zero.\footnote{More precisely, we need that we can always pick the nullhomotopy of the product on the right to make the bracket zero.} The first of these brackets is zero since it lands in a zero group and we have already checked the second contains zero.
\end{proof}

The restriction placed on the product of $\theta_j$ and $\theta_{j-1}$ by \Cref{lem:theta-product-divisible} allows us to obtain a corresponding restriction on the Adams differentials on $h_j^3$, which proves \Cref{prop:diff classical}.

\begin{prop} \label{prop:d4-divisible}
  $\delta_{4,1}(h_j^3)$ is divisible by $\wt{2}^2$ for $j \geq 6$.
  Consequently, for $j \geq 6$,
  \begin{enumerate}
  \item $d_2(h_j^3) = 0$,
  \item $d_3(h_j^3)$ is either $0$ or $h_0^2g_{j-2}$ and
  \item $d_4(h_j^3)$ is either $0$ or $h_0^3g_{j-2}$ (if defined).
  \end{enumerate}  
\end{prop}

\begin{proof}
  From \Cref{cnstr:bock-maps} we know that $\delta_{4,1}$ arises as the cofiber of the map of rings $r_{4,1} : \clambda^4 \to \clambda$ and therefore that this map is $\clambda^4$-linear. Using that $\theta_j$ lifts to $\clambda^4$ (by \Cref{thm:inductive}) we may now perform the following manipulations
  \begin{align} 
    \delta_{4,1}( h_j^3 ) = \delta_{4,1}( \theta_j h_j ) = \theta_j \delta_{4,1}(h_j) = \theta_j \cdot \wt{2} \theta_{j-1} \label{eqn:test}
  \end{align}
  where we have used \Cref{cor:Hopf-diff} to identify $\delta_{4,1}(h_j)$.
  Plugging the $\wt{2}$ divisibility of $\theta_j\theta_{j-1}$ from \Cref{lem:theta-product-divisible}
  into this equation completes the proof of the first claim.

  Let $\delta_{4,1}(h_j^3) = \wt{2}^2 \cdot y$ with $y \in \pi_{ 3 \cdot 2^j - 4, 3}\clambda^3$.
  Using our knowledge of the $E_2$-page from \cite{Ext5} (cf. \Cref{E2 descriptions}(1)) together with \Cref{exm:Bruner formula} which lets us lift $g_{j-2}$ to $\clambda^2$ (we denote the lift by $[g_{j-2}]$) we can read off that
  $\pi_{ 3 \cdot 2^j - 4, 3}\clambda^3 \cong \Z/4\{\underline{\lambda}[g_{j-2}]\}$.\footnote{In the case $j=6$ there is an additional summand $\Z/2\{\underline{\lambda^2}(h_7D_3(0))\}$, but this class is $\wt{2}$-torsion and so does not affect the value of $\wt{2}^2 \cdot y$.}
  It follows that $d_2(h_j^3) = 0$ and depending on the value of $y \in \Z/4\{\underline{\lambda}[g_{j-2}]\}$ we fall into one of the following cases
  \begin{itemize}
  \item[(a)] $y$ is a generator and $d_3(h_j^3) = h_0^2g_{j-2}$,
  \item[(b)] $y = 2\underline{\lambda}[g_{j-2}]$, $d_3(h_j^3) = 0$ and $d_4(h_j^3) = h_0^3g_{j-2}$.
  \item[(c)] $y = 0$, $d_2(h_j^3) = 0$, $d_3(h_j^3) = 0$ and $d_4(h_j^3) = 0$.
  \end{itemize}
\end{proof}

Now we can reverse the flow of information.
From \Cref{thm:main} we know that
\[ d_4(h_j^3) = h_0^3g_{j-2} \neq 0, \]
which, when combined with \Cref{eqn:test}, implies:

%\footnote{The reason for ordering things in this somewhat obtuse way is that \Cref{prop:d4-divisible} is used in the proof of \Cref{thm:main}.} %that $\theta_j\theta_{j-1} = \underline{\lambda^2} h_0^2g_{j-2} \neq 0 $ in $\pi_{3\cdot 2^j- 4, 4}\clambda^3$.
% We also know form \Cref{lem:theta-product-divisible} that this product is divisible by $\wt{2}$ and from  our knowledge of the $E_2$-page and \Cref{???} \todo{need that d2 on g is non-zero here.} we can work out that \todo{Make sure this is right.} $\pi_{3\cdot 2^j -4, 3}\clambda^3 \cong \Z/4\{\lambda \overline{\kappa}_{j-2}\}$. Putting all the pieces together we obtain the promised product.

\begin{cor}
  In $\clambda^3$, there is a product relation, $\theta_j\theta_{j-1} = \underline{\lambda^2} h_0^2g_{j-2} \neq 0 $ for $j \geq 6$.
\end{cor}

The family of products in this corollary are one of the first infinite families of hidden products in the Adams spectral sequence. It seems likely that they bear a close relations to the hidden products $\theta_j^2$ which identify the HHR differentials.

%%% Local Variables:
%%% mode: latex
%%% TeX-master: "main"
%%% End:

%% file: kerv-chart.tex
\begin{sseqdata}[ name = thetaASS, xscale=1.25, yscale=1.25, x range = {-5}{0}, y range = {0}{5}, x tick step = 2, y tick step = 2, class labels = {left}, classes = fill, grid = crossword, Adams grading, lax degree]

  \class["h_j^2"](-2,2)
  \class(-2,3) \structline
  \class(-2,4) \structline
  \class(-2,5) \structline
  \class(-2,6) \structline
  \class(-2,7) \structline

  \class(-1,3) \structline(-2,2)
  \class(0,4) \structline
  \class(1,5) \structline 
  % \class(1,5) \structline \structline(-2,4)
  % \class(1,3) \structline(-2,2)
  % \class(1,4) \structline(-2,3) \structline \structline(1,5)
  
  \class["h_{j+1}"](-1,1)
  \class(-1,2) \structline
  \class(-1,3) \structline
  \class(-1,4) \structline
  \class(-1,5) \structline
  \class(-1,6) \structline

  % \class["\eta_{j+1}" below](0,2) \structline(-1,1)
  \class(0,2) \structline(-1,1)
  \class(1,3) \structline
  % \class(2,4) \structline \structline(-1,3,-1)
  % \class(2,2) \structline(-1,1)
  % \class(2,3) \structline(-1,2) \structline \structline(2,4)

  \d[red]2(-1,1)(-2,3)
  \d[red]2(-1,2)(-2,4)
  \d[red]2(-1,3,-1)(-2,5)
  \d[red]2(-1,4)(-2,6)
  \d[red]2(-1,5)(-2,7)
  
\end{sseqdata}

%% local 256
\begin{sseqdata}[ name = theta7ASS, xscale=1.0, yscale=1.0, x range = {251}{256}, y range = {0}{5}, x tick step = 2, y tick step = 2, class labels = {left}, classes = fill, grid = crossword, Adams grading, lax degree]

  \class["h_7^2"](256-2,2)
  \class(256-2,3) \structline
  \class(256-2,4) \structline
  \class(256-2,5) \structline
  \class(256-2,6) \structline
  \class(256-2,7) \structline

  \class(256-1,3) \structline(256-2,2)
  \class(256,4) \structline
  \class(256+1,5) \structline 
  % \class(1,5) \structline \structline(-2,4)
  % \class(1,3) \structline(-2,2)
  % \class(1,4) \structline(-2,3) \structline \structline(1,5)
  
  \class["h_{8}"](256-1,1)
  \class(256-1,2) \structline
  \class(256-1,3) \structline
  \class(256-1,4) \structline
  \class(256-1,5) \structline
  \class(256-1,6) \structline

  \class(256,2) \structline(256-1,1)
  \class(256+1,3) \structline
  % \class(2,4) \structline \structline(-1,3,-1)
  % \class(2,2) \structline(-1,1)
  % \class(2,3) \structline(-1,2) \structline \structline(2,4)

  \class["V_0'" below](256-4,5)
  \class["K_1" below](256-1,5)
  \class["D_3(2)" below](256,4)
  \class(256,5) \structline

  \d[red]2(256-1,1)(256-2,3)
  \d[red]2(256-1,2)(256-2,4)
  \d[red]2(256-1,3,-1)(256-2,5)
  \d[red]2(256-1,4)(256-2,6)
  \d[red]2(256-1,5)(256-2,7)
  
\end{sseqdata}

%% local 512
\begin{sseqdata}[ name = theta8ASS, xscale=1.0, yscale=1.0, x range = {507}{512}, y range = {0}{5}, x tick step = 2, y tick step = 2, class labels = {left}, classes = fill, grid = crossword, Adams grading, lax degree]
  \i = 512
  
  \class["h_8^2"](510,2)
  \class(512-2,3) \structline
  \class(512-2,4) \structline
  \class(512-2,5) \structline
  \class(512-2,6) \structline
  \class(512-2,7) \structline

  \class(512-1,3) \structline(512-2,2)
  \class(512,4) \structline
  \class(512+1,5) \structline 
  % \class(1,5) \structline \structline(-2,4)
  % \class(1,3) \structline(-2,2)
  % \class(1,4) \structline(-2,3) \structline \structline(1,5)
  
  \class["h_{9}"](512-1,1)
  \class(512-1,2) \structline
  \class(512-1,3) \structline
  \class(512-1,4) \structline
  \class(512-1,5) \structline
  \class(512-1,6) \structline

  \class(512,2) \structline(512-1,1)
  \class(512+1,3) \structline
  % \class(2,4) \structline \structline(-1,3,-1)
  % \class(2,2) \structline(-1,1)
  % \class(2,3) \structline(-1,2) \structline \structline(2,4)

  \class["V_1'" below](512-3,5)

  \d[red]2(512-1,1)(512-2,3)
  \d[red]2(512-1,2)(512-2,4)
  \d[red]2(512-1,3,-1)(512-2,5)
  \d[red]2(512-1,4)(512-2,6)
  \d[red]2(512-1,5)(512-2,7)
  
\end{sseqdata}

%% local 128
\begin{sseqdata}[ name = theta6ASS, xscale=1.0, yscale=1.0, x range = {123}{127}, y range = {0}{7}, x tick step = 2, y tick step = 2, class labels = {left}, classes = fill, grid = crossword, Adams grading, lax degree]

  \class["x_{6,94}"](124,6)
  \class(124,7) \structline
  \class(124,6)
  \class(124,6) 

  \class["D_3(1)"](126,4)

  \class["K_0"](125,5)
  \class(125,6) \structline
  \class(125,7) \structline \structline(124,6,1)
  \class["h_6H_1"](125,6)
  
  \class(126,6)
  \class(126,7) \structline(125,6,-1)

  \class["h_6^2"](128-2,2)
  \class(128-2,3) \structline
  \class(128-2,4) \structline
  \class(128-2,5) \structline
  \class(128-2,6) \structline
  \class(128-2,7) \structline
  \class(128-2,8) \structline
  \class(128-2,9) \structline

  \class(128-1,3) \structline(128-2,2)
  \class(128,4) \structline
  \class(128+1,5) \structline 
  % \class(1,5) \structline \structline(-2,4)
  % \class(1,3) \structline(-2,2)
  % \class(1,4) \structline(-2,3) \structline \structline(1,5)
  
  \class["h_{7}"](128-1,1)
  \class(128-1,2) \structline
  \class(128-1,3) \structline
  \class(128-1,4) \structline
  \class(128-1,5) \structline
  \class(128-1,6) \structline
  \class(128-1,7) \structline
  \class(128-1,8) \structline
  
  \class(128,2) \structline(128-1,1)
  \class(128+1,3) \structline
  % \class(2,4) \structline \structline(-1,3,-1)
  % \class(2,2) \structline(-1,1)
  % \class(2,3) \structline(-1,2) \structline \structline(2,4)

  \class(127,6)

  \class(127,7) \structline(124,6)
  \class(127,7) \structline(126,6)
  \class(127,7)
  \class(127,7)

  \d[red]2(128-1,1)(128-2,3,)
  \d[red]2(128-1,2)(128-2,4,-1)
  \d[red]2(128-1,3,-1)(128-2,5)
  \d[red]2(128-1,4)(128-2,6,-1)
  \d[red]2(128-1,5)(128-2,7,-1)
  \d[red]2(128-1,6)(128-2,8,-1)
  \d[red]2(128-1,7)(128-2,9,-1)

  \d[red]2(125,5)(124,7)
  
\end{sseqdata}

%%% Local Variables:
%%% mode: latex
%%% TeX-master: "main"
%%% End:

%% file: code.tex
In this appendix we provide a short script which gives an alternative way of describing an explicit expansion of the cocycle $T_5|T_{6} + d(c_4)$ modulo $(8,v_1^4)$ was used in the proof (of \Cref{lem:prod-big}).
This script runs within the computer algebra system Sage and
evaluates the cobar differential in the Hopf algebroid $(\BP_*, \BP_*\BP)$ modulo $v_1$ and a power of $2$ on terms which include only $t_1$ and $t_2$.
Ultimately, we know this script gives the correct answer in our case of interest since we have checked the outputs by hand.

\begin{itemize}
\item The ring $K$ is the base ring for our computation. We use $\Z/8$ since that is what is relevant for us, but other choices work as well.
\item The rings $A$ and $B$ represent the third and fourth layers of the cobar complex respectively.
\item The modifiers $a,b,c$ or $w,x,y,z$ indicate which tensor factor a given variable comes from.
\item The cobar complex is a cosimplicial ring and the various $di$'s are relevant face maps of this cosimplicial ring.
\item The term $E$ corresponds to the cocycle $T_j|T_{j+1}$ from \Cref{sec:alg-nov-diff}.
\item The final line (which is expanded and printed on evaluation) is the cocycle $T_j|T_{j+1} + d(c_j)$. \item We have annotated the various terms which make up $c_j$ with the length of the May differential they are involved in (see \Cref{subsec:correction}).
\end{itemize}

\begin{verbatim}

K = Zmod(8);
A.<t1a, t1b, t1c, t2a, t2b, t2c> = K[];
B.<t1w, t1x, t1y, t1z, t2w, t2x, t2y, t2z> = K[];

d0 = A.hom([t1x, t1y, t1z, t2x, t2y, t2z], B);
d1 = A.hom([t1w + t1x, t1y, t1z, t2w - t1w * t1x^2 + t2x, t2y, t2z], B);
d2 = A.hom([t1w, t1x + t1y, t1z, t2w, t2x - t1x * t1y^2 + t2y, t2z], B);
d3 = A.hom([t1w, t1x, t1y + t1z, t2w, t2x, t2y - t1y * t1z^2 + t2z], B);
d4 = A.hom([t1w, t1x, t1y, t2w, t2x, t2y], B);

def cobar_diff(x):
    return d0(x) - d1(x) + d2(x) - d3(x) + d4(x);

def make_theta_left(k):
    return 4 * t1w^(1 * 2^(k)) * t1x^(7 * 2^(k)) 
         + 6 * t1w^(2 * 2^(k)) * t1x^(6 * 2^(k)) 
         + 4 * t1w^(3 * 2^(k)) * t1x^(5 * 2^(k)) 
         + 3 * t1w^(4 * 2^(k)) * t1x^(4 * 2^(k)) 
         + 4 * t1w^(5 * 2^(k)) * t1x^(3 * 2^(k)) 
         + 6 * t1w^(6 * 2^(k)) * t1x^(2 * 2^(k)) 
         + 4 * t1w^(7 * 2^(k)) * t1x^(1 * 2^(k));

def make_theta_right(k):
    return 4 * t1y^(1 * 2^(k)) * t1z^(7 * 2^(k)) 
         + 6 * t1y^(2 * 2^(k)) * t1z^(6 * 2^(k)) 
         + 4 * t1y^(3 * 2^(k)) * t1z^(5 * 2^(k)) 
         + 3 * t1y^(4 * 2^(k)) * t1z^(4 * 2^(k)) 
         + 4 * t1y^(5 * 2^(k)) * t1z^(3 * 2^(k)) 
         + 6 * t1y^(6 * 2^(k)) * t1z^(2 * 2^(k)) 
         + 4 * t1y^(7 * 2^(k)) * t1z^(1 * 2^(k));

# The n which appears here is related to j by n = j - 3.
n = 2;

theta_left = make_theta_left(n);
theta_right = make_theta_right(n+1);

E = theta_left * theta_right;

term1 = t1a^(4 * 2^n) * t1b^(12 * 2^n) * t1c^(8 * 2^n);
term2 = 7 * t2a^(4 * 2^n) * t1b^(4 * 2^n) * t1c^(8 * 2^n);
term3 = 2 * t2a^(2 * 2^n) * t2b^(2 * 2^n) * t2c^(4 * 2^n);
term4 = 2 * t1a^(2 * 2^n) * t1b^(4 * 2^n) * t2c^(4 * 2^n) 
          * (t2a^(2 * 2^n) + t2b^(2 * 2^n));
term5a = 2 * t1a^(4 * 2^n) * t1b^(4 * 2^n) * t1c^(4 * 2^n);
term5b = 2 * t1a^(2 * 2^n) * t1b^(2 * 2^n) * t1c^(8 * 2^n);
term5 = (term5a + term5b) * (t2a^(4 * 2^n) + t2b^(4 * 2^n) + t2c^(4 * 2^n));
term6 = 2 * t1a^(6 * 2^n) * t1b^(10 * 2^n) * t1c^(8 * 2^n);
term7 = 2 * t1a^(8 * 2^n) * t1b^(12 * 2^n) * t1c^(4 * 2^n) 
      + 2 * t1a^(12 * 2^n) * t1b^(8 * 2^n) * t1c^(4 * 2^n); 
term8 = 2 * t1a^(2 * 2^n) * t2b^(2 * 2^n) * t1c^(16 * 2^n) 
      + 2 * t1a^(4 * 2^n) * t2b^(4 * 2^n) * t1c^(8 * 2^n);

correction = term1 + term2; # first row
correction += term3; # may d1
correction += term4; # may d0
correction += term5; # may d0
correction += term6; # may d0
correction += term7; # may d0
correction += term8; # may d1
E + cobar_diff(correction);

\end{verbatim}

%% file: mainv2.bbl
\begin{thebibliography}{BKWX22}

\bibitem[Ada58]{AdamsSseq}
J.~F. Adams.
\newblock On the structure and applications of the {S}teenrod algebra.
\newblock {\em Comment. Math. Helv.}, 32:180--214, 1958.

\bibitem[Ada60]{Adams}
J.~F. Adams.
\newblock On the non-existence of elements of {H}opf invariant one.
\newblock {\em Ann. of Math. (2)}, 72:20--104, 1960.

\bibitem[AM17]{AndrewsMiller}
Michael Andrews and Haynes Miller.
\newblock Inverting the {H}opf map.
\newblock {\em J. Topol.}, 10(4):1145--1168, 2017.

\bibitem[BHS20]{rmot}
Robert Burklund, Jeremy Hahn, and Andrew Senger.
\newblock Galois reconstruction of {A}rtin--{T}ate {$\mathbb{R}$}-motivic
  spectra.
\newblock 2020.
\newblock \href{https://arxiv.org/abs/2010.10325}{arXiv:2010.10325}.

\bibitem[BHS23]{boundaries}
Robert Burklund, Jeremy Hahn, and Andrew Senger.
\newblock On the boundaries of highly connected, almost closed manifolds.
\newblock {\em Acta Math.}, 231(2):205--344, 2023.

\bibitem[BJM83]{Inductive}
M.~G. Barratt, J.~D.~S. Jones, and M.~E. Mahowald.
\newblock The {K}ervaire invariant problem.
\newblock In {\em Proceedings of the {N}orthwestern {H}omotopy {T}heory
  {C}onference ({E}vanston, {I}ll., 1982)}, volume~19 of {\em Contemp. Math.},
  pages 9--22. Amer. Math. Soc., Providence, RI, 1983.

\bibitem[BJM84]{BJM62}
M.~G. Barratt, J.~D.~S. Jones, and M.~E. Mahowald.
\newblock Relations amongst {T}oda brackets and the {K}ervaire invariant in
  dimension {$62$}.
\newblock {\em J. London Math. Soc. (2)}, 30(3):533--550, 1984.

\bibitem[BKWX22]{BKWX}
Tom Bachmann, Hana~Jia Kong, Guozhen Wang, and Zhouli Xu.
\newblock The {C}how {$t$}-structure on the {$\infty$}-category of motivic
  spectra.
\newblock {\em Ann. of Math. (2)}, 195(2):707--773, 2022.

\bibitem[BMMS86]{BMMS}
R.~R. Bruner, J.~P. May, J.~E. McClure, and M.~Steinberger.
\newblock {\em {$H_\infty $} ring spectra and their applications}, volume 1176
  of {\em Lecture Notes in Mathematics}.
\newblock Springer-Verlag, Berlin, 1986.

\bibitem[BMT70]{BMT}
M.~G. Barratt, M.~E. Mahowald, and M.~C. Tangora.
\newblock Some differentials in the {A}dams spectral sequence. {II}.
\newblock {\em Topology}, 9:309--316, 1970.

\bibitem[Bro69]{Browder}
William Browder.
\newblock The {K}ervaire invariant of framed manifolds and its generalization.
\newblock {\em Ann. of Math. (2)}, 90:157--186, 1969.

\bibitem[Bru]{Bruner2}
Robert Bruner.
\newblock The cohomology of the mod 2 {S}teenrod algebra: A computer
  calculation.
\newblock \href{http://www.rrb.wayne.edu/papers/cohom.pdf}{available online}.

\bibitem[Bur23]{cookware}
Robert Burklund.
\newblock Synthetic cookware.
\newblock 2023.
\newblock To appear.

\bibitem[Che11]{Ext5}
Tai-Wei Chen.
\newblock Determination of {${\rm Ext}^{5,*}_{A}(Z/2, Z/2)$}.
\newblock {\em Topology Appl.}, 158(5):660--689, 2011.

\bibitem[Chu21a]{dexter-chart}
Dexter Chua.
\newblock Computer calculated adams $d_2$ differentials.
\newblock v1.0.0, 2021.
\newblock DOI: \href{10.5281/zenodo.4766457}{10.5281/zenodo.4766457}.

\bibitem[Chu21b]{dexter}
Dexter Chua.
\newblock The ${E}_3$ page of the adams spectral sequence.
\newblock 2021.
\newblock \href{https://arxiv.org/abs/2105.07628}{arXiv:2105.07628}.

\bibitem[DFHH14]{TMFbook}
Christopher~L. Douglas, John Francis, Andr\'e{}~G. Henriques, and Michael~A.
  Hill, editors.
\newblock {\em Topological modular forms}, volume 201 of {\em Mathematical
  Surveys and Monographs}.
\newblock American Mathematical Society, Providence, RI, 2014.

\bibitem[DI10]{DuggerIsaksen}
Daniel Dugger and Daniel~C. Isaksen.
\newblock The motivic {A}dams spectral sequence.
\newblock {\em Geom. Topol.}, 14(2):967--1014, 2010.

\bibitem[Goe08]{GoerssMfg}
Paul Goerss.
\newblock Quasi-coherent sheaves on the moduli stack of formal groups.
\newblock 2008.
\newblock \href{https://arxiv.org/pdf/0802.0996}{arXiv:0802.0996}.

\bibitem[GWX21]{GWX}
Bogdan Gheorghe, Guozhen Wang, and Zhouli Xu.
\newblock The special fiber of the motivic deformation of the stable homotopy
  category is algebraic.
\newblock 2021.
\newblock Acta Mathematica. Vol. 226, No. 2 (2021), pp. 319-407.

\bibitem[HHR16]{HHR}
M.~A. Hill, M.~J. Hopkins, and D.~C. Ravenel.
\newblock On the nonexistence of elements of {K}ervaire invariant one.
\newblock {\em Ann. of Math. (2)}, 184(1):1--262, 2016.

\bibitem[Hil15]{Hill}
Michael~A. Hill.
\newblock On the fate of $\eta^3$ in higher analogues of real bordism.
\newblock 2015.
\newblock \href{https://arxiv.org/abs/1507.08083}{arXiv:1507.08083}.

\bibitem[HKO11]{HuKrizOrmsby}
Po~Hu, Igor Kriz, and Kyle Ormsby.
\newblock Remarks on motivic homotopy theory over algebraically closed fields.
\newblock {\em J. K-Theory}, 7(1):55--89, 2011.

\bibitem[Hop99]{HopkinsMfg}
Mike Hopkins.
\newblock Complex oriented cohomology theories and the language of stacks.
\newblock 1999.
\newblock
  \href{https://people.math.rochester.edu/faculty/doug/otherpapers/coctalos.pdf}{available
  at https://people.math.rochester.edu/faculty/doug/otherpapers/coctalos.pdf}.

\bibitem[Hov04]{Hovey}
Mark Hovey.
\newblock Homotopy theory of comodules over a {H}opf algebroid.
\newblock In {\em Homotopy theory: relations with algebraic geometry, group
  cohomology, and algebraic {$K$}-theory}, volume 346 of {\em Contemp. Math.},
  pages 261--304. Amer. Math. Soc., Providence, RI, 2004.

\bibitem[Isa19]{Isaksen}
Daniel~C. Isaksen.
\newblock Stable stems.
\newblock {\em Mem. Amer. Math. Soc.}, 262(1269):viii+159, 2019.

\bibitem[IWX20]{IWX2}
Daniel~C. Isaksen, Guozhen Wang, and Zhouli Xu.
\newblock Stable homotopy groups of spheres.
\newblock {\em Proceedings of the National Academy of Sciences},
  117(40):24757--24763, 2020.

\bibitem[IWX23]{IWX}
Daniel~C. Isaksen, Guozhen Wang, and Zhouli Xu.
\newblock Stable homotopy groups of spheres: from dimension 0 to 90.
\newblock {\em Publ. Math. Inst. Hautes \'Etudes Sci.}, 137:107--243, 2023.

\bibitem[Lau00]{Laures}
Gerd Laures.
\newblock On cobordism of manifolds with corners.
\newblock {\em Trans. Amer. Math. Soc.}, 352(12):5667--5688, 2000.

\bibitem[Lin98]{Lin}
Wen-Hsiung Lin.
\newblock A differential in the {A}dams spectral sequence for spheres.
\newblock In {\em Stable and unstable homotopy ({T}oronto, {ON}, 1996)},
  volume~19 of {\em Fields Inst. Commun.}, pages 205--239. Amer. Math. Soc.,
  Providence, RI, 1998.

\bibitem[Lin08]{LinExt}
Wen-Hsiung Lin.
\newblock {${\rm Ext}^{4,\ast}_A(\Bbb Z/2,\Bbb Z/2)$} and {${\rm
  Ext}^{5,\ast}_A(\Bbb Z/2,\Bbb Z/2)$}.
\newblock {\em Topology Appl.}, 155(5):459--496, 2008.

\bibitem[LSWX19]{LSWX}
Guchuan Li, XiaoLin~Danny Shi, Guozhen Wang, and Zhouli Xu.
\newblock Hurewicz images of real bordism theory and real {J}ohnson-{W}ilson
  theories.
\newblock {\em Adv. Math.}, 342:67--115, 2019.

\bibitem[Lur15]{rotation}
Jacob Lurie.
\newblock Rotation invariance in algebraic {$K$}-theory.
\newblock 2015.
\newblock \href{http://www.math.ias.edu/~lurie/}{available online}.

\bibitem[Mah77]{etaj}
Mark Mahowald.
\newblock A new infinite family in {${}_{2}\pi_{*}{}^s$}.
\newblock {\em Topology}, 16(3):249--256, 1977.

\bibitem[May70]{May-power-ops}
J.~Peter May.
\newblock A general algebraic approach to {S}teenrod operations.
\newblock In {\em The {S}teenrod {A}lgebra and its {A}pplications ({P}roc.
  {C}onf. to {C}elebrate {N}. {E}. {S}teenrod's {S}ixtieth {B}irthday,
  {B}attelle {M}emorial {I}nst., {C}olumbus, {O}hio, 1970)}, Lecture Notes in
  Mathematics, Vol. 168, pages 153--231. Springer, Berlin, 1970.

\bibitem[Mil81]{MillerSquare}
Haynes~R. Miller.
\newblock On relations between {A}dams spectral sequences, with an application
  to the stable homotopy of a {M}oore space.
\newblock {\em J. Pure Appl. Algebra}, 20(3):287--312, 1981.

\bibitem[Mil09]{Miller}
Haynes~R. Miller.
\newblock Browder's theorem and manifolds with corners.
\newblock
  \href{https://math.mit.edu/~hrm/papers/kervaire-corners.pdf}{available
  online}, 2009.

\bibitem[Min95]{Minami}
Norihiko Minami.
\newblock The {A}dams spectral sequence and the triple transfer.
\newblock {\em Amer. J. Math.}, 117(4):965--985, 1995.

\bibitem[Nas]{Nassau}
Christian Nassau.
\newblock Computer ext calculation.
\newblock \href{http://www.nullhomotopie.de/charts/bigpng.png}{available
  online}.

\bibitem[Pst23]{Pstragowski}
Piotr Pstr\k{a}gowski.
\newblock Synthetic spectra and the cellular motivic category.
\newblock {\em Invent. Math.}, 232(2):553--681, 2023.

\bibitem[Rav86]{RavenelGreenBook}
Douglas~C. Ravenel.
\newblock {\em Complex cobordism and stable homotopy groups of spheres}, volume
  121 of {\em Pure and Applied Mathematics}.
\newblock Academic Press, Inc., Orlando, FL, 1986.

\bibitem[Tan70]{TangoraExt}
Martin~C. Tangora.
\newblock On the cohomology of the {S}teenrod algebra.
\newblock {\em Math. Z.}, 116:18--64, 1970.

\bibitem[Tod62]{TodaBook}
Hirosi Toda.
\newblock {\em Composition methods in homotopy groups of spheres}.
\newblock Annals of Mathematics Studies, No. 49. Princeton University Press,
  Princeton, N.J., 1962.

\bibitem[Voe03a]{Voe032}
Vladimir Voevodsky.
\newblock Motivic cohomology with {${\bf Z}/2$}-coefficients.
\newblock {\em Publ. Math. Inst. Hautes \'{E}tudes Sci.}, (98):59--104, 2003.

\bibitem[Voe03b]{Voe03}
Vladimir Voevodsky.
\newblock Reduced power operations in motivic cohomology.
\newblock {\em Publ. Math. Inst. Hautes \'{E}tudes Sci.}, (98):1--57, 2003.

\bibitem[Wu13]{Wu}
Pomin Wu.
\newblock The differential $d_4(h_6^3)$ in the adams spectral sequence for
  spheres.
\newblock 2013.
\newblock \href{https://arxiv.org/abs/1307.0064}{arXiv:1307.0064}.

\bibitem[WX17]{WX}
Guozhen Wang and Zhouli Xu.
\newblock The triviality of the 61-stem in the stable homotopy groups of
  spheres.
\newblock {\em Ann. of Math. (2)}, 186(2):501--580, 2017.

\bibitem[Xu16]{Xu}
Zhouli Xu.
\newblock The strong {K}ervaire invariant problem in dimension 62.
\newblock {\em Geom. Topol.}, 20(3):1611--1624, 2016.

\end{thebibliography}
